\documentclass[11pt]{amsart}
\usepackage[margin=1in]{geometry}
\usepackage{amssymb,xcolor,colortbl,diagbox,comment,graphicx,mathpazo,cite}
\usepackage{MnSymbol}
\usepackage{color,float,fullpage,mathrsfs,mathtools}
\usepackage{enumerate,enumitem}
\usepackage[colorlinks=true, pdfstartview=FitV, linkcolor=blue, citecolor=blue, urlcolor=blue]{hyperref}
\usepackage{multirow,array}
\usepackage{makecell}

\def\bse{\begin{subequations}}
\def\ese{\end{subequations}}
\def\eq{\begin{equation}}
\def\endeq{\end{equation}}
\def\bbm{\begin{bmatrix}}
\def\ebm{\end{bmatrix}}
\def\bpm{\begin{pmatrix}}
\def\epm{\end{pmatrix}}
\def\bvm{\begin{vmatrix}}
\def\evm{\end{vmatrix}}
\def\Real{\mathbb{R}}
\def\Complex{\mathbb{C}}
\def\A{\mathbf{A}}

\def\C{\mathbf{C}}
\def\E{\mathbf{E}}

\def\F{\mathbf{F}}

\def\G{\mathbf{G}}

\def\I{\mathbb{I}}

\def\K{\mathbf{K}}
\def\M{\mathbf{M}}
\def\N{\mathbf{N}}
\def\O{\mathcal{O}}
\def\P{\mathbf{P}}
\def\Q{\mathbf{Q}}
\def\S{\mathcal{S}}
\def\bS{\mathbf{S}}

\def\V{\mathbf{V}}
\def\W{\mathbf{W}}

\def\bphi{\boldsymbol{\phi}}
\def\brho{\boldsymbol{\rho}}

\def\tr{{\mathrm{tr}}}
\def\b#1{\overline{#1}}

\def\dbar{\overline\partial}
\def\sech{\mathrm{sech}}

\def\Xint#1{\mathchoice
   {\XXint\displaystyle\textstyle{#1}}%
   {\XXint\textstyle\scriptstyle{#1}}%
   {\XXint\scriptstyle\scriptscriptstyle{#1}}%
   {\XXint\scriptscriptstyle\scriptscriptstyle{#1}}%
   \!\int}
\def\XXint#1#2#3{{\setbox0=\hbox{$#1{#2#3}{\int}$}
     \vcenter{\hbox{$#2#3$}}\kern-.5\wd0}}

\def\dashint{\Xint-}

\newtheorem{rhp}{Riemann-Hilbert Problem}
\newtheorem{rhdp}{Riemann-Hilbert-$\b{\partial}$ Problem}
\newtheorem{dbp}{$\b{\partial}$ Problem}

\newtheorem{theorem}{Theorem}
\newtheorem{corollary}{Corollary}
\newtheorem{lemma}{Lemma}
\theoremstyle{definition}
\newtheorem{definition}{Definition}
\newtheorem{assumption}{Assumption}
\newtheorem{remark}{Remark}

\numberwithin{equation}{section}
\makeatletter
\@addtoreset{equation}{section}
\makeatother

\newcommand{\ee}{\mathrm{e}}
\newcommand{\ii}{\mathrm{i}}
\newcommand{\dd}{\,\mathrm{d}}

\title[Sharp-line Maxwell-Bloch System]{On the Maxwell-Bloch System in the Sharp-Line Limit Without Solitons}
\author{Sitai Li}
\address[S. Li]{Department of Mathematics\\ University of Michigan\\East Hall\\ 530 Church St.\\ Ann Arbor, MI 48109.}
\email{sitaili@umich.edu}
\author{Peter D. Miller}
\address[P. D. Miller]{Department of Mathematics\\University of Michigan\\East Hall\\ 530 Church St.\\Ann Arbor, MI 48109.}
\email{millerpd@umich.edu}
\urladdr{http://www.math.lsa.umich.edu/~millerpd/}
\thanks{The second author was partly supported by the National Science Foundation under grant number DMS-1812625.}
\dedicatory{For Hermann Flaschka and Charlie Doering, pioneers of nonlinear science who left us too soon.}

\begin{document}

\begin{abstract}
We study the (characteristic) Cauchy problem for the Maxwell-Bloch equations of light-matter interaction via asymptotics,
under assumptions that prevent the generation of solitons.
Our analysis clarifies some features of the sense in which physically-motivated initial/boundary conditions are satisfied.
In particular,
we present a proper Riemann-Hilbert problem that generates the unique \textit{causal} solution to the Cauchy problem,
that is, the solution vanishes outside of the light cone.
Inside the light cone,
we relate the leading-order asymptotics to self-similar solutions that satisfy a system of ordinary differential equations related to the Painlev\'e-III (PIII) equation.
We identify these solutions and show that they are related to a family of PIII solutions recently discovered in connection with several limiting processes involving the focusing nonlinear Schr\"odinger equation.
We fully explain a resulting boundary layer phenomenon in which,
even for smooth initial data (an incident pulse),
the solution makes a sudden transition over an infinitesimally small propagation distance.
At a formal level,
this phenomenon has been described by other authors in terms of the PIII self-similar solutions.
We make this observation precise and for the first time we relate the PIII self-similar solutions to the Cauchy problem.
Our analysis of the asymptotic behavior satisfied by the optical field and medium density matrix reveals slow decay of the optical field in one direction that is actually inconsistent with the simplest version of scattering theory.
Our results identify a precise generic condition on an optical pulse incident on an initially-unstable medium sufficient for the pulse to stimulate the decay of the medium to its stable state.
\end{abstract}
\maketitle

\tableofcontents

\section{Introduction}
\label{s:intro}
\subsection{The Maxwell-Bloch equations and their self-similar solutions}
The scalar Maxwell-Bloch system of equations (MBEs),
first derived in 1965~\cite{ab1965},
describes light-matter interactions in a two-level active medium,
and is completely integrable in certain limits~\cite{ae1975,as1981}.
The system has attracted great interest since then,
due to its important role in the successful explanation of self-induced transparency~\cite{mh1967,mh1969,mh1970}
and the closely-related phenomenon of superfluorescence~\cite{z1980,gzm1983,gzm1984,gzm1985}.

The Cauchy problem of the system models the injection of a known incident optical pulse through a boundary point $z=0$ into a finite or semi-infinitely long medium with a known initial state.
The modeling assumes that atoms in the medium have two states: a ground state and an excited state.
Macroscopically, the medium can be initially in a pure ground state (the \emph{initially-stable case}), a pure excited state (the \emph{initially-unstable case}) or a mixed state.
Assuming the incident pulse vanishes in the distant past and future,
the Cauchy problem of an initially-stable medium was among the first few systems analyzed by the inverse scattering transform (IST) in the 1970's~\cite{akn1974},
and the initially-unstable case was studied via IST a few years later~\cite{gzm1985}.
There are many recent works formulating ISTs for the MBE and other related systems,
and using the transforms to study the behavior of solutions; a sampling of these works (not intended to be exhaustive) includes~\cite{fkm2017,physrep1990,bgk2003,CPA2014,gk2,gk1,msr1988,rt1993}.
The case of a medium initially in a mixed state requires a compatible nonvanishing optical pulse in the distant past ($t\to -\infty$).
Further assuming a nonvanishing pulse as $t\to+\infty$,
the mixed-state case was studied recently by IST methods~\cite{lbkg2018,bgkl2019}.
In general, as assumed in the aforementioned works,
the medium exhibits inhomogeneous broadening due to the Doppler effect or other physical phenomena (e.g. static crystalline electric and
magnetic fields in solids)~\cite{mh1969}.
The distribution of atoms is characterized by the spectral line shape function $g(\lambda)$ where $\lambda$ is the difference between the atomic transition frequency and the resonant frequency.
Mathematically, $g(\lambda)$ can be an arbitrary probability density function (or distribution).

This paper concerns the aftereffect in an active medium of the passage of an optical pulse as modeled by the MBEs.
Physically, one is interested in the form of any residual optical pulse and the remaining state after a long time at any given point in the active medium.
This problem is addressed by calculating the asymptotic behavior of solutions as $t\to+\infty$ for a fixed position $z > 0$.
We take a reasonably large function space for the incident optical pulse
and consider propagation in media both initially stable and initially unstable, but neglecting inhomogeneous broadening.
A simple part of our study that is nonetheless crucial from the point of view of uniqueness is the analysis of the solution outside of the light cone, i.e., for $t < 0$, but more interesting phenomena appear within the light cone, as $t\to+\infty$.
Our analysis is based on applying the Deift-Zhou steepest descent method~\cite{dz1993-1,dz1993-2} and the $\overline{\partial}$ approach~\cite{mm2006,mm2008,dmm2019} to a suitable Riemann-Hilbert problem (RHP) encoding a particular solution of the Cauchy problem that is most relevant for physical applications,
assuming no discrete eigenvalues or spectral singularities are present.
The combination of the Deift-Zhou nonlinear steepest descent method and the $\overline{\partial}$ approach allows one to avoid assuming any analyticity of the reflection coefficient,
without the need to use complicated rational approximations.
Our results are novel,
applicable to a wide variety of incident pulses,
provide rigorous proofs (in some cases of results obtained at a physical level of rigor in other papers),
and come with precise error estimates.

In terms of results obtained earlier by other authors,
in a series of works~\cite{z1980,gzm1983,gzm1984,gzm1985,mn1986},
formal asymptotic analysis of solutions of the MBE system
suggested the importance of self-similar solutions, and such solutions also appear in our rigorous analysis.
Later, the Deift-Zhou nonlinear steepest descent method was applied to a related problem~\cite{fm1999},
but only initially-stable media were considered and the results were somewhat incomplete in the sense that (i) error estimates were omitted (although in principle they are accessible via the methodology employed) and (ii) the leading-order term was given implicitly in terms of the solution of a singular integral equation that is difficult to compare with the Riemann-Hilbert characterization we offer below.
Very recently,
assuming periodic incident pulses injected into an initially-stable medium,
the large-$t$ asymptotic problem was revisited and analyzed by the nonlinear steepest descent method~\cite{fk2020}.

In the setting that inhomogeneous broadening is absent from the system
(equivalently taking the spectral line shape to be the Dirac delta $g(\lambda) = \delta(\lambda)$) and that
the optical pulse vanishes in the distant past (so the initial state of the medium is one of the two pure states, stable or unstable),
the Cauchy problem for the MBEs takes the form
\begin{equation}
\label{e:mbe}
\begin{aligned}
q_z(t,z) & = -P(t,z)\,,\qquad
P_t(t,z) =  - 2q(t,z)D(t,z)\,,\qquad
D_t(t,z) = 2\Re(\overline{q(t,z)} P(t,z))\,,\\
q(t,0) & = q_{0}(t)\,,\qquad t\in\Real\,,\\
\lim_{t\to-\infty}q(t,z) & = 0\,,\qquad
D_- \coloneq \lim_{t\to-\infty}D(t,z) = \pm1\,,\qquad
P_- \coloneq \lim_{t\to-\infty} P(t,z) = 0\,,\qquad z\ge0\,.
\end{aligned}
\end{equation}
where the subscripts $t$ and $z$ denote partial derivatives,
and $\b{q}$ denotes the complex conjugate of $q$.
The variables $z = z_{\mathrm{lab}}$ and $t = t_{\mathrm{lab}} - z_{\mathrm{lab}}/c$ are the propagation distance and retarded time,
respectively, with $c$ denoting the speed of light in the vacuum ($(z_\mathrm{lab},t_\mathrm{lab})$ denote space and time coordinates in a fixed laboratory frame).
The unknowns are the optical pulse $q(t,z)\in\Complex$,
the population inversion $D(t,z)\in\Real$ of the medium,
and its polarization $P(t,z)\in\Complex$.
We refer to the evolution equation on $q$ in~\eqref{e:mbe} as the Maxwell equation,
and to the two equations on $P$ and $D$ as the Bloch subsystem.
Even though the MBE system is completely integrable,
there is only one global conservation law~\cite{as1981}, namely that $D^2 + |P|^2$
is independent of $t$, and for the given values of $D_-$ and $P_-$ in \eqref{e:mbe} we have $D^2+|P|^2=1$ for all $t\in\mathbb{R}$ and $z\ge 0$.
The quantity $D_-$ is the initial population inversion,
with $D_- = -1$ (resp., $D_- = 1$) denoting an initially-pure stable (resp., unstable) medium.

Although we think of $t$ and $z$ as mathematical spatial and temporal variables respectively,
the asymptotic behavior of $D(t,z)$ as $t\to +\infty$ need not be specified.
In fact, because of the first-order nature of the Bloch subsystem in MBEs~\eqref{e:mbe} one cannot arbitrarily specify the behavior of $D(t,z)$ in both limits $t\to\pm\infty$.
An influential early work~\cite{z1980} considers the IST and solutions of the MBEs,
in the hope of analyzing the physical phenomenon of superfluorescence.
In that paper, it is assumed that \emph{both} $D_- = 1$ and $D(t,z) \to -1$ as $t\to+\infty$,
i.e., that the medium is initially in the unstable excited state and decays to the stable ground state in the future.
Although such an assumption is natural
from the physical perspective,
it is not clear mathematically how one can enforce two asymptotic values for $D(t,z)$ at $t = \pm\infty$ simultaneously due to the first-order nature of the Bloch subsystem.
In fact, it is recognized in~\cite{z1980} that imposing two asymptotic conditions on $D(t,z)$ may be mathematically incorrect,
but the resolution proposed --- a causality requirement --- is also not fully justified.
In this paper, we prove that \textit{under the causality requirement
and other mild assumptions on the incident optical pulse,
an unstable excited medium indeed decays naturally to the stable ground state as $t\to+\infty$}.
Hence, by fully rigorous arguments
we validate the causality requirement originally proposed in~\cite{z1980}.

If solutions to the MBE system are restricted to real-valued functions, under the substitutions $P = \sin(\Theta)$, $D = \cos(\Theta)$ and $q = -\Theta_t/2$
the system~\eqref{e:mbe} becomes the sine-Gordon equation in characteristic coordinates:  $\Theta_{tz} = 2\sin(\Theta)$.
The asymptotic behavior of solutions of the sine-Gordon equation for large values of the independent variables was studied in 1999~\cite{cvz1999} and again quite recently~\cite{hl2018,cll2020}.
However, even if real solutions of the MBEs are considered, our work goes in quite a different direction for two related reasons:
\begin{itemize}
\item
The Cauchy problem considered in \cite{cvz1999,hl2018,cll2020} is the second-order initial-value problem for the sine-Gordon equation in the form $\Theta_{\tau\tau}-\Theta_{\chi\chi}+\sin(\Theta)=0$ with two initial conditions given at $\tau=0$. In this setting, the reflection coefficient $r(\lambda)$ comes from the Faddeev-Takhtajan scattering problem,
which automatically yields $r(0) = 0$.
However, no such condition is guaranteed for a given incident pulse $q_0(t)$ in the context of the MBEs (or the characteristic sine-Gordon equation).
This is because for the latter system the reflection coefficient comes instead from the non-selfadjoint Zakharov-Shabat problem, which gives
$r(0)\ne0$ in general.
\item
The analysis of sine-Gordon given in \cite{cvz1999,hl2018,cll2020} concerns the limit $\tau\to\infty$ in which $\chi=v \tau$.  The hyperbolic nature of the sine-Gordon equation is exhibited in the asymptotic confinement of the solution to the light cone $|v|<1$.  As $|v|\uparrow 1$, the solution decays, a result that is mathematically a direct consequence of the condition $r(0)=0$.  Since we cannot generally assume $r(0)=0$ for the MBE system, the boundary of the light cone becomes the most interesting regime for the asymptotic behavior, and hence we assume exclusively in this paper that $z/t\to 0$ as $t\to+\infty$ and we show that the generally-nonzero quantity $r(0)$ plays a crucial role in this regime.
\end{itemize}
In this paper, we show that in the aforementioned regime a boundary layer phenomenon occurs for the MBE system:
for a variety of incident pulses $q_0(t)$,
the solutions exhibit a sudden transition between the boundary of the medium $z = 0$ and the interior $z > 0$.
Roughly speaking, no matter how fast the incident pulse $q_0(t)$ decays as $t\to+\infty$,
after an infinitesimal propagation distance the optical pulse $q(t,z)$ always decays at a fixed slow rate as $t\to+\infty$.
Physically, the residual pulse remains in the active medium for a long time,
due to strong nonlinear interactions between light and the active medium.
The decay rate is slowest when $r(0)$, a spectral quantity we call below the ``moment'' of the incident pulse $q_0(t)$ (see Definition~\ref{def:moments}), is nonzero.

The slow decay of the optical pulse within the boundary layer is resolved at the leading order by a family of universal profiles expressible in terms of
a family of certain Painlev\'e-III (PIII) solutions.
This resolution occurs most clearly in the limit $t\to+\infty$ with $z = \O(1/t)$.
The PIII solutions that occur are closely related to PIII solutions appearing in some other recent works:
\begin{itemize}
\item
A particular PIII solution was uncovered by Suleimanov~\cite{Suleimanov2017} (along with its dilations by a scaling transformation) through a formal analysis of weakly-dispersive corrections to a self-similar singular solution (Talanov pulse) of the dispersionless focusing NLS (nonlinear Schr\"odinger) equation.  In work in progress, Buckingham, Jenkins, and Miller~\cite{bjm2021} are proving Suleimanov's observation rigorously and also generalizing its applicability to the whole family of Talanov pulses that are not necessarily self-similar.
\item
The Suleimanov solution was shown by Bilman, Ling, and Miller~\cite{blm2020} to describe the near-field/high-order limit of fundamental rogue-wave solutions of the focusing NLS with a nonzero background, in which context it was called the \emph{rogue wave of infinite order}.
\item
A one-parameter family of PIII solutions generalizing the Suleimanov solution was shown by Bilman and Buckingham~\cite{bb2019} to describe the near-field high-order limit of multiple-pole soliton solutions of focusing NLS with a zero background.
\end{itemize}

The solutions of PIII that arise in this problem are determined from spectral properties of the incident pulse $q_0(t)$ and from the initial state of the medium, and unlike in some earlier works we provide asymptotic formul\ae\ for the Bloch (medium) fields $D$ and $P$, as well as for the optical pulse $q$.
Although the PIII solutions appear just in the leading terms of an asymptotic expansion, these terms alone constitute an exact self-similar solution of the MBE system; hence such self-similar solutions appear naturally and universally just inside the light cone in media both initially stable and initially unstable.
Self-similar solutions of the MBE system are known to be connected with the PIII equation, having been derived via asymptotic analysis at various levels of rigor in several earlier papers~\cite{bc1969,lamb1969,z1980,gzm1983,gzm1984,gzm1985,fm1999}.
For $z\ge 0$ and $t\ge 0$, a natural similarity variable is $x=\sqrt{2tz}$, and it is straightforward to see that the Maxwell equation and Bloch subsystem in \eqref{e:mbe} admit exact solutions for which $tq$, $P$, and $D$ are real-valued functions of $x\ge 0$ alone.  Writing $x=X$ and assuming
\begin{equation}
q=t^{-1}y(X),\quad P=\frac{2}{X}s(X),\quad D=1-\frac{2}{X}U(X),
\label{e:similarity-form}
\end{equation}
one easily obtains the coupled ordinary differential equations
\begin{equation}
\begin{split}
y'(X)&=-2s(X)\\
Xs'(X)&=s(X)-2Xy(X)+4y(X)U(X)\\
XU'(X)&=U(X)-4y(X)s(X).
\end{split}
\label{e:coupled-self-similar}
\end{equation}
If one analytically continues a real-valued solution of the coupled system \eqref{e:coupled-self-similar} from the positive $X$-axis to the negative imaginary axis, then replacing $X=x$ with $X=-\ii x$ and $(P,D)$ with $(-P,-D)$ in \eqref{e:similarity-form}, one obtains another similarity solution of the MBEs provided that the fields $q$, $P$, and $D$ remain real under this continuation.
From either of these, complex-valued self-similar solutions can be obtained by the symmetry $(q,P,D)\mapsto (\xi q,\xi^{-1}P,D)$ of the original MBE system, where $\xi$ is any complex constant of unit modulus:  $|\xi|=1$.

Another system is closely related to \eqref{e:coupled-self-similar}, namely
\begin{equation}
\begin{split}
y'(X)&=-2s(X)\\
Xs'(X)&=s(X)-2Xy(X)+4y(X)U(X)\\
XU'(X)&= U(X)-4X\frac{y(X)U(X)}{s(X)}+4\frac{y(X)U(X)^2}{s(X)}.
\end{split}
\label{e:coupled-PIII-0-0}
\end{equation}
Indeed, the latter system has the first integral
\begin{equation}
J:=\frac{U(X)(U(X)-X)}{s(X)^2}
\label{e:J-constant}
\end{equation}
and for the specific value $J=-1$ the systems \eqref{e:coupled-self-similar} and \eqref{e:coupled-PIII-0-0} coincide.  Regardless of the value of $J$, it is straightforward to deduce from \eqref{e:coupled-PIII-0-0} that the quantity $u(X):=-y(X)/s(X)$ satisfies
\begin{equation}
u''(X)=\frac{u'(X)^2}{u(X)}-\frac{u'(X)}{X}+\frac{4}{X}+4u(X)^3-\frac{4}{u(X)},
\label{e:PIII}
\end{equation}
which is a form of the Painlev\'e-III equation.  Generally, solutions of these equations exhibit branch-point type singularities at $X=0$, but they also admit solutions that are analytic at the origin and are determined by the first few Taylor coefficients.  If one assumes without justification that these are the solutions of interest (such as in \cite{gzm1983}) then it becomes possible to match the self-similar solutions with other information such as experimental data or assumed asymptotic behavior of solutions in another regime and determine the solutions uniquely.  One of the aims of our paper is to provide the rigorous justification behind such assumptions.  We prove the desired properties of the self-similar solutions by deducing a Riemann-Hilbert representation for the solutions.  This ``spectral'' representation both allows us to explicitly relate the relevant initial values to the incident pulse profile generating the self-similar response and to simultaneously clarify the connection with the family of PIII solutions obtained in \cite{bb2019}.  The Riemann-Hilbert representation is also preferable to a purely local one (like specifying initial conditions), in that it allows us to obtain the asymptotic behavior of the self-similar solution for large $X$.
Again making contact with the literature, the PIII solution obtained for a related problem in \cite{fm1999} is also specified spectrally, via a system of singular integral equations.  However that system is equivalent to a Riemann-Hilbert problem with the real line as a jump contour, and we do not see how it can be identified or compared with the one we formulate below, which instead has the unit circle as the jump contour.  Indeed, from the isomonodromy point of view, our solutions of PIII correspond to trivial Stokes phenomenon and nontrivial connection between solutions near two irregular singular points whereas the solutions in \cite{fm1999} appear to instead have trivial connection and nontrivial Stokes phenomenon.

\begin{remark}[On notation]
In the rest of this paper, we use boldface fonts to denote $2\times 2$ matrices, with the exception of the identity matrix $\I$ and the Pauli matrices
\begin{equation}
\sigma_1\coloneq\bpm 0&1\\1&0\epm,\quad\sigma_2\coloneq\bpm 0 & -\ii\\\ii & 0\epm,\quad\text{and}\quad
\sigma_3\coloneq\bpm 1&0\\0 & -1\epm.
\end{equation}
The imaginary unit is denoted $\ii$, and complex conjugation is indicated with a bar:  $\b{\lambda}$.  We denote the characteristic function of a set $S$ by $\chi_S$.
\end{remark}

\subsection{Assumptions and causality}
We now start to make our assumptions more precise to lay the groundwork for us to present our results.
The causality requirement imposed in some earlier works~\cite{z1980,gzm1983,gzm1984,gzm1985} is that
the optical pulse $q(t,z)$ should vanish identically in the past outside of the light cone in order to obtain a unique solution from the Gel'fand-Levitan-Marchenko equation of inverse scattering.  In particular upon taking $z=0$, it should hold that $q_0(t)$ is supported on a positive half-line which we may take without loss of generality to be $t\ge 0$.
Combining causality with a level of smoothness and decay that is both convenient and natural from the point of view of the IST, we make the following basic assumption on the incident optical pulse $q_0(\cdot)$.
\begin{assumption}[Basic condition on $q_0$]
The incident optical pulse $q_0:\Real\to\Complex$ satisfies $q_0\in\mathscr{S}(\Real)$ with $q_0(t)\equiv 0$ for $t<0$.
\label{ass:q0-assumption}
\end{assumption}
Here $\mathscr{S}(\Real)$ denotes the Schwartz class.

%

\begin{definition}[Causal solutions]
\label{def:causal}
A solution of the Cauchy problem~\eqref{e:mbe} for a given incident pulse $q_0$ satisfying Assumption~\ref{ass:q0-assumption} is called \emph{causal} if $q(z,t)=0$ holds for all $t<0$ and $z\ge 0$.
From the Bloch subsystem in \eqref{e:mbe} it is clear that for a causal solution it also holds that $D(z,t)=D_-$ and that $P(z,t)=0$ for all $t<0$ and $z\ge 0$.
\end{definition}
Many of the most familiar solutions of the MBE system are non-causal, for instance the soliton solutions.
However, a key result is the following.
\begin{theorem}
Given an incident pulse $q_0$ satisfying Assumption~\ref{ass:q0-assumption}, there exists at most one causal solution of the MBE Cauchy problem~\eqref{e:mbe}.
\label{thm:causal-uniqueness}
\end{theorem}
The proof does not rely on integrability and is given in Appendix~\ref{s:causal-uniqueness}.  It is equally important to note that \emph{in general the Cauchy problem~\eqref{e:mbe} is ill-posed} in the sense that it admits multiple non-causal solutions for the same data.  This point will be discussed in more detail later (see Corollary~\ref{cor:no-uniqueness-without-causality} below), as it follows from our asymptotic results.

\subsection{Integrability and Riemann-Hilbert representation of causal solutions}
The MBEs~\eqref{e:mbe} can be equivalently written in matrix form,
\begin{equation}
\everymath{\displaystyle}
\label{e:mbe-matrix}
\begin{aligned}
\brho_t & = [\Q,\brho]\,,\qquad
\Q_z = -\frac{1}{2}[\sigma_3,\brho]\,,\\
\Q(t,z) & = \bpm 0 & q(t,z) \\ -\b{q(t,z)} & 0 \epm\,,\qquad
\brho(t,z) = \bpm D(t,z) & P(t,z) \\ \b{P(t,z)} & -D(t,z) \epm\,,\qquad
\sigma_3 = \bpm 1 & 0 \\ 0 & -1 \epm\,,
\end{aligned}
\end{equation}
where $[\cdot,\cdot]$ is the matrix commutator.
The matrix $\brho$ is called the \emph{density matrix}, and it satisfies the identities $\brho^\dagger=\brho$, $\tr(\brho)=0$ and $\det(\brho) = -1$,
where the superscript $\dagger$ denotes conjugate transpose.
The Lax pair for the MBEs in the form~\eqref{e:mbe-matrix} is given by
\begin{gather}
\label{e:LP}
\bphi_t= (\ii\lambda \sigma_3 + \Q)\,\bphi\,,\qquad
\bphi_z= -\frac{\ii}{2\lambda}\brho\,\bphi\,,
\end{gather}
where $\bphi=\bphi(t,z;\lambda)$; in other words, the system~\eqref{e:mbe-matrix} is the compatibility condition under which there exists a basis of simultaneous solutions of the two equations~\eqref{e:LP} for all $\lambda\in\mathbb{C}\setminus\{0\}$.  Since $t$ plays the mathematical role of a spatial variable, the differential equation with respect to $t$ in~\eqref{e:LP} is called the scattering problem (here, the well-known non-selfadjoint Zakharov-Shabat problem), whose scattering data evolves in mathematical ``time'' $z$ according to the other equation in the Lax pair~\eqref{e:LP}.  The IST for the system~\eqref{e:mbe-matrix} is therefore based on the direct and inverse problems for the Zakharov-Shabat equation, which has been systematically and rigorously studied in several papers such as~\cite{bc1984,z1998,bjm2017}.  For the direct problem, one takes $\Q=\Q(t,0)$ in terms of the incident pulse $q_0(t)$ and defines
the scattering matrix $\bS(\lambda)\coloneq \bphi_-(t;\lambda)^{-1}\bphi_+(t;\lambda)$ (independent of $t\in\Real$)
for $\lambda\in\mathbb{R}$ in terms of $\bphi_\pm(t;\lambda)$, the Jost eigenfunctions of the Zakharov-Shabat equation for $z=0$ normalized at $t\to\pm\infty$,
i.e., $\bphi_\pm(t;\lambda)=\ee^{\ii\lambda t\sigma_3}+o(1)$ as $t\to\pm\infty$.
It is worth noting that due to $q_0(t)\equiv0$ for $t < 0$,
we have $\bphi_-(t;\lambda) \equiv \ee^{\ii\lambda t\sigma_3}$ for $t \le 0$ and $\bphi_-(0;\lambda) = \I$,
so by taking $t=0$ without loss of generality, the scattering matrix is simply $\bS(\lambda) = \bphi_+(0;\lambda)$. It satisfies the basic identities
\begin{equation}
\det(\bS(\lambda))=1 \quad\text{and}\quad \bS(\lambda)=\sigma_2\b{\bS(\lambda)}\sigma_2.
\label{e:scattering-matrix}
\end{equation}
The \emph{reflection coefficient} $r(\lambda)$ defined by
\begin{equation}
r(\lambda) \coloneq \frac{S_{2,1}(\lambda)}{S_{1,1}(\lambda)} = \frac{\phi_{+,2,1}(0;\lambda)}{\phi_{+,1,1}(0;\lambda)}
\label{e:reflection-coefficient}
\end{equation}
plays a crucial role in the IST and consequently the long-time asymptotics.
In general (i.e., without the cutoff assumption for $t<0$ in Assumption~\ref{ass:q0-assumption}), $r(\lambda)$ is only defined on the continuous spectrum $\Real$,
but $S_{1,1}(\lambda)$ admits continuation into the upper half $\lambda$ plane as an analytic function continuous up to the real line.
The zeros of $S_{1,1}(\lambda)$ in the open upper half plane are the discrete eigenvalues corresponding to solitons,
whereas real zeros are called spectral singularities, i.e., poles of the reflection coefficient.  Under some assumptions that are difficult to justify fully, the $z$-equation in the Lax pair \eqref{e:LP} then defines an explicit evolution of the scattering data in $z$, and for the inverse problem one constructs $\Q=\Q(t,z)$ from the $z$-evolved scattering data.
It is well known that under some additional conditions the incident pulse $q_0(t)$ is encoded completely in the reflection coefficient $r(\lambda)$, which subsequently determines the solution for all $z\ge 0$.
In this direction, we have the following result.

\begin{lemma}[Properties of the reflection coefficient]
\label{lemma:reflection}
Suppose that the incident pulse $q_0$ satisfies Assumption~\ref{ass:q0-assumption}, and that there exist no discrete eigenvalues, i.e., $S_{1,1}(\lambda)\neq 0$ for $\lambda$ in the upper half-plane.
Then, the reflection coefficient $r(\lambda)$ for the non-selfadjoint Zakharov-Shabat equation admits continuation to the open upper half-plane as an analytic function.  If also $q_0$ generates no spectral singularities (i.e., $S_{1,1}(\lambda)\neq 0$ for $\lambda\in\Real$), then $r(\cdot)\in\mathscr{S}(\Real)$.
\end{lemma}

%
\begin{proof}
The statement that $q_0(\cdot)\in\mathscr{S}(\Real)$ and $S_{1,1}(\lambda)\neq 0$ for $\lambda\in\Real$ implies $r(\cdot)\in\mathscr{S}(\Real)$ is a standard result,
so one only needs to prove the analyticity of $r(\lambda)$ in the upper half plane.
Using the cutoff condition in Assumption~\ref{ass:q0-assumption}, the expression in \eqref{e:reflection-coefficient} in terms of the Jost matrix $\bphi_+(t;\lambda)$ evaluated at the finite point $t=0$ immediately gives this result. Indeed, since the first column of $\bphi_+(t;\lambda)$ is analytic in the open upper half-plane and continuous in the closed upper half-plane for each fixed $t\in\mathbb{R}$ as is well known (see
Lemma~\ref{lem:r-asymptotics} and its proof in Appendix~\ref{s:proof-reconstruction} below), in particular upon evaluation at $t=0$ the desired analyticity follows provided that $\phi_{+,1,1}(0;\lambda)=S_{1,1}(\lambda)\neq 0$, a condition that is guaranteed for $\Im(\lambda)\ge 0$ by hypothesis.
%
\end{proof}


\begin{remark}
Assumption~\ref{ass:q0-assumption} is quite strong, but the reader will see that our asymptotic results require far less, just the existence of sufficiently many continuous and absolutely integrable derivatives of the reflection coefficient associated with $q_0$, along with some auxiliary conditions related to discrete spectrum.  It is difficult to give simple conditions on $q_0$ sufficient to control a given number of derivatives of the reflection coefficient in the $L^1$ sense, although a weighted $L^2$-Sobolev bijection result has been proven by Zhou \cite{z1998}.
\label{rem:Schwartz-too-strong}
\end{remark}

A derivation of an IST for the MBE system based on the direct/inverse scattering theory for the non-selfadjoint Zakharov-Shabat equation can be found in numerous papers going back to~\cite{akn1974}.  This derivation is fundamentally problematic, because it presumes that $q(\cdot,z)\in L^1(\Real)$ for all $z\ge 0$ to define the relevant eigenfunctions and scattering data; however after the fact it can be shown that even if this condition holds at $z=0$ (as is guaranteed by Assumption~\ref{ass:q0-assumption}) it is generically violated for all $z>0$ (see Corollary~\ref{cor:not-L1} below).  Rather than repeat these arguments, we will simply postulate a well-posed Riemann-Hilbert problem (RHP) whose solution, by a dressing-method argument, encodes the unique causal solution of the Cauchy problem.

%
\begin{figure}
\includegraphics[width = 0.56\textwidth]{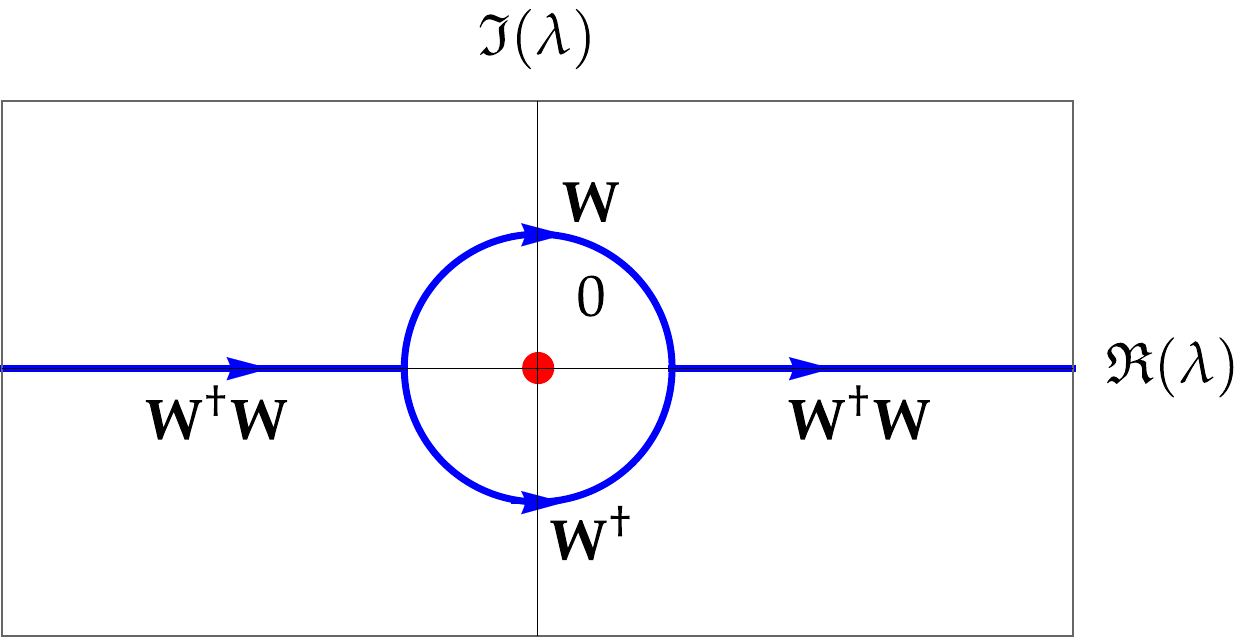}
\caption{The jump contour $\Sigma_\mathbf{M}$ for RHP~\ref{rhp:M}, consisting of the circle $|\lambda|=\gamma>0$ and the real intervals $|\lambda|>\gamma$, oriented as shown.  The jump matrix on each arc of $\Sigma_\mathbf{M}$ is as indicated.}
\label{f:M}
\end{figure}
\begin{rhp}
\label{rhp:M}
Let $\gamma>0$ be fixed and consider the contour $\Sigma_\mathbf{M}$ shown in Figure~\ref{f:M}.  For a given
Schwartz-class function $r(\lambda)$
that is the boundary value of a function analytic for $\Im(\lambda)>0$ (denoted also by $r(\lambda)$), for a given sign $D_-:=\pm 1$, and for given $(t,z)\in\Real^2$, seek
a $2\times 2$ matrix-valued function $\lambda\mapsto\M(\lambda;t,z)$ that is analytic for $\lambda\in\Complex\setminus\Sigma_\mathbf{M}$; that satisfies $\M\to\I$ as $\lambda\to\infty$;
and that takes continuous boundary values on $\Sigma_\M$ from each component of the complement related by the following jump conditions
\begin{equation}
\begin{aligned}
\M^+(\lambda;t,z) & = \M^-(\lambda;t,z)\W^\dagger(\lambda;t,z)\W(\lambda;t,z)\,,\qquad&&
\lambda\in(-\infty,-\gamma)\cup(\gamma,+\infty)\,,\\
\M^+(\lambda;t,z) & = \M^-(\lambda;t,z) \W(\lambda;t,z)\,,\qquad&&
|\lambda| = \gamma\,,\qquad \Im(\lambda) > 0\,,\\
\M^+(\lambda;t,z) & = \M^-(\lambda;t,z) \W^\dagger(\lambda;t,z)\,,\qquad&&
|\lambda| = \gamma\,,\qquad \Im(\lambda) < 0\,,
\end{aligned}
\end{equation}
where
a matrix $\W(\lambda;t,z)$ is defined for $\Im(\lambda)\ge 0$ by
\begin{equation}
\everymath{\displaystyle}
\W(\lambda;t,z) = \bpm 1 & 0 \\ r(\lambda)\ee^{-2\ii\theta(\lambda;t,z)} & 1\epm\,,\qquad
\Im(\lambda)\ge 0\,,\qquad
\theta(\lambda;t,z):= \lambda t - \frac{D_-z}{2\lambda}\,,
\label{e:phase-def-general}
\end{equation}
and where $\W^\dagger(\lambda;t,z)$ denotes the Schwarz reflection $\b{\W(\b{\lambda};t,z)}^\top$.
\end{rhp}

\begin{remark}
The jump matrix satisfies the conditions of Zhou's vanishing lemma \cite{z1989}, implying that RHP~\ref{rhp:M} is uniquely solvable for all $(t,z)\in\mathbb{R}^2$.
\end{remark}

\begin{remark}
Here, and in the rest of the paper, we use the convention that a superscript ``$+$'' (resp., ``$-$'') denotes a boundary value taken from the left (resp., right) by orientation.
There is an essential singularity at $\lambda = 0$ in the exponential factors
$\ee^{\pm2\ii\theta(\lambda;t,z)}$
that is avoided by the jump contour $|\lambda| = \gamma$ with arbitrary radius $\gamma>0$.
We observe that \emph{if the medium is initially stable ($D_- = -1$)}
we can pass to the limit $\gamma=0$,
because these factors decay as $\lambda\to 0$ from within in the respective half-disks.
This yields an equivalent RHP with the real line as the only jump contour.  In~\cite{z1980} the inverse problem was formulated instead as a system of Gel'fand-Levitan-Marchenko equations corresponding to a RHP on the real line, and it was suggested that the essential singularity at $\lambda=0$ which appears to require careful interpretation is responsible for the observed slow decay of the optical pulse as $t\to+\infty$ and leading to a loss of $L^1(\Real)$ integrability.  However, this phenomenon is also generated from RHP~\ref{rhp:M} whose contour completely avoids the origin.
\end{remark}

We then have the following result, on which the rest of our paper is based.
\begin{theorem}
\label{thm:reconstruction}
Let $q_0$ be an incident pulse satisfying Assumption~\ref{ass:q0-assumption}, and suppose further that $q_0$ generates no discrete eigenvalues or spectral singularities under the direct transform associated with the Zakharov-Shabat equation.  Then
the unique causal solution to the Cauchy problem~\eqref{e:mbe} can be reconstructed from the solution of RHP~\ref{rhp:M} in which $r(\lambda)$ denotes the reflection coefficient for $q_0$ by the
following formul\ae:
\begin{equation}
\label{e:reconstruction}
q(t,z) = -2\ii\lim_{\lambda\to\infty}\lambda M_{1,2}(\lambda;t,z)\,,\qquad
\brho(t,z) = D_- \M(0;t,z)\sigma_3 \M(0;t,z)^{-1}\,,\quad t\in\mathbb{R}\,,\quad z\ge 0\,.
\end{equation}
\end{theorem}
Theorem~\ref{thm:reconstruction} is proved in Appendix~\ref{s:proof-reconstruction}.
The solution generated from RHP~\ref{rhp:M} is necessarily causal
because this problem
can be solved exactly and trivially when $t<0$,
which is a direct consequence of the analyticity of the reflection coefficient $r(\lambda)$
resulting from the cutoff condition for $q_0(\cdot)$ in Assumption~\ref{ass:q0-assumption}.
The argument applies regardless of the initial state of the medium
because for large $\lambda$ the phase $\theta(\lambda;t,z)$
becomes independent of $D_-=\pm 1$; the fact that the contour $\Sigma_\M$ avoids the origin then allows
one to ``bypass'' the fact that $\theta(\lambda;t,z)$ is strongly dependent on $D_-$ for small $\lambda$.

%
%
%

In order to present our results,
we now introduce the notion of the ``moments'' of the incident optical pulse.
\begin{definition}[nonlinear moments]
The nonlinear \textit{moment} with index $m\ge0$ of an incident optical pulse $q_0(t)$ is defined via the reflection coefficient,
\begin{equation}
r_0^{(m)}\coloneq \frac{\dd^m r(\lambda)}{\dd \lambda^m}\bigg|_{\lambda = 0}\,,\qquad m\ge 0\,. 
\end{equation}
If the index $m$ is unspecified, the term ``moment'' refers to the zeroth moment which we denote for brevity by $r_0=r_0^{(0)}$.
We also denote the index of the first nonzero moment by $M$,
i.e., $r_0^{(M)}\ne 0$ and $r_0^{(m)} = 0$ for all $m = 0,1,\dots,M-1$.
\label{def:moments}
\end{definition}
\begin{remark}
The moment $r_0$ cannot be calculated explicitly for a generic incident pulse
$q_0(t)\in\mathscr{S}(\Real)$.
However, if $q_0(t)$ is real-valued, when $\lambda=0$ the Jost solution $\bphi_+(t;\lambda)$ is given explicitly by
\begin{equation}
\bphi_+(t;0)=\begin{pmatrix}\cos(I(t)) & \;-\sin(I(t))\\\sin(I(t)) & \;\cos(I(t))\end{pmatrix},\quad I(t)\coloneq\int_t^{+\infty}q_0(\tau)\dd\tau,
\label{e:JostPlus0}
\end{equation}
so from \eqref{e:reflection-coefficient} we have:
\begin{equation}
\label{e:r0-real}
r_0 = \tan(I(0))=\tan\bigg(\int_0^\infty q_0(t)\dd t\bigg)\,,\quad q_0(\cdot)\;\text{real-valued and supported on $\Real_+$}.
\end{equation}
In particular, this shows that when the total integral of a real-valued $q_0$ is an odd half-integer multiple of $\pi$,
a spectral singularity appears at the origin.  One can also calculate higher derivatives of $r(\lambda)$ at $\lambda=0$ assuming sufficient decay of $q_0$.  For instance, letting $\dot{\bphi}_+(t;\lambda)$ denote the partial derivative of the Jost solution with respect to $\lambda$, differentiation of \eqref{e:reflection-coefficient} gives
\begin{equation}
r'(0)=\frac{\dot{\phi}_{+,2,1}(0;0)\phi_{+,1,1}(0;0)-\phi_{+,2,1}(0;0)\dot{\phi}_{+,1,1}(0;0)}{\phi_{+,1,1}(0;0)^2} = \frac{\cos(I(0))\dot{\phi}_{+,2,1}(0;0)-\sin(I(0))\dot{\phi}_{+,1,1}(0;0)}{\cos^2(I(0))},
\end{equation}
and differentiation of the Zakharov-Shabat equation for $\bphi_+(t;\lambda)$ with respect to $\lambda$ at $\lambda=0$ gives
\begin{equation}
\dot{\bphi}_{+,t}(t;0)=q_0(t)\begin{pmatrix} 0&1\\-1&0\end{pmatrix}\dot{\bphi}_+(t;0) + \ii\sigma_3\bphi_+(t;0).
\end{equation}
Assuming for simplicity that $\mathrm{supp}(q_0)=[0,T]$ for some $T>0$, we can use $\bphi_+(t;0)$ given by \eqref{e:JostPlus0} as a fundamental solution matrix for the homogeneous equation and, since $\bphi_+(t;0)=\mathbb{I}$ and $\dot{\bphi}_+(t;0)=\ii t\sigma_3$ both hold for $t\ge T$, we get by variation of parameters:
\begin{equation}
\dot{\bphi}_+(t;0)=\bphi_+(t;0)\left[\ii T\sigma_3 -\ii\int_t^T\bphi_+(s;0)^{-1}\sigma_3\bphi_+(s;0)\dd s\right],
\end{equation}
so that when $t=0$,
\begin{equation}
\dot{\bphi}_+(0;0)=\begin{pmatrix}\cos(I(0)) & \;-\sin(I(0))\\\sin(I(0)) & \;\cos(I(0))\end{pmatrix}
\left[\ii T\sigma_3-\ii\int_0^T\bphi_+(s;0)^{-1}\sigma_3\bphi_+(s;0)\dd s\right].
\end{equation}
From \eqref{e:JostPlus0} we obtain
\begin{equation}
\begin{split}
\bphi_+(s;0)^{-1}\sigma_3\bphi_+(s;0)&=\begin{pmatrix}\cos^2(I(s))-\sin^2(I(s)) & -2\sin(I(s))\cos(I(s))\\
-2\sin(I(s))\cos(I(s)) & \sin^2(I(s))-\cos^2(I(s))\end{pmatrix}\\
&=\begin{pmatrix}\cos(2I(s)) & -\sin(2I(s))\\-\sin(2I(s)) & -\cos(2I(s))\end{pmatrix}.
\end{split}
\end{equation}
It then follows that if $q_0$ is real-valued and supported on $[0,T]$,
\begin{equation}
r'(0)=\ii\sec^2(I(0))\int_0^T\sin(2I(s))\dd s = \ii\sec^2\left(\int_0^Tq_0(t)\dd t\right)\int_0^T\sin\left(2\int_t^Tq_0(s)\dd s\right)\dd t.
\label{e:rprime-zero}
\end{equation}
The compact support assumption can then be dropped due to rapid decay of $q_0(t)$ as $t\to+\infty$ for $q_0\in\mathscr{S}(\Real)$; one simply sets $T=+\infty$ in \eqref{e:rprime-zero}.
\end{remark}
Another quantity we need that is related to the reflection coefficient is the following.
\begin{definition}[phase $\aleph$]
A real phase $\aleph$ is defined by the principal value integral
\begin{equation}
\aleph\coloneq \frac{1}{\pi}\,\,\dashint_\Real\ln(1+|r(\lambda)|^2)\frac{\dd\lambda}{\lambda}.
\label{e:aleph-def}
\end{equation}
Note that $\aleph$ is finite when $r(\cdot)\in\mathscr{S}(\Real)$ as guaranteed under some conditions by Lemma~\ref{lemma:reflection}.
\label{def:aleph}
\end{definition}
\begin{remark}
The quantity $\aleph$ vanishes when the incident pulse $q_0(t)$ is a real function.
This is because for a real potential,
the corresponding non-selfadjoint Zakharov-Shabat reflection coefficient $r(\lambda)$ enjoys an additional symmetry, namely that $\b{r(\lambda)} = r(-\lambda)$, making the integrand in
\eqref{e:aleph-def} an odd function.
\label{rem:aleph}
\end{remark}
%

When $z=0$, the density matrix $\brho(t,0)$ satisfying the Bloch subsystem $\brho_t=[\Q,\brho]$ for $t>0$ with initial condition $\brho(0,0)=D_-\sigma_3$ can be expressed explicitly in terms of the Jost solutions of the Zakharov-Shabat system with potential $q_0(t)$, evaluated at the origin $\lambda=0$.  Indeed, defining
\begin{equation}
\brho(t,0)\coloneq D_-\bphi_-(t;0)\sigma_3\bphi_-(t;0)^{-1},\quad t>0,
\label{e:density-matrix-z-zero}
\end{equation}
one checks easily that for $z=0$,
\begin{equation}
\frac{\dd\bphi_-}{\dd t}=\begin{pmatrix}0&q_0(t)\\-\b{q_0(t)}&0\end{pmatrix}\bphi_-\implies\frac{\dd\brho}{\dd t}=\left[\begin{pmatrix}0&q_0(t)\\-\b{q_0(t)}&0\end{pmatrix},\brho\right],
\end{equation}
and $\brho(0,0)=D_-\sigma_3$ because $\bphi_-(0;0)=\mathbb{I}$.  The formula \eqref{e:density-matrix-z-zero} allows one to determine the asymptotic behavior of $\brho(t,0)$ as $t\to+\infty$.  For this purpose, we recall the defining identity $\bphi_+(t;\lambda)=\bphi_-(t;\lambda)\bS(\lambda)$ for the scattering matrix to obtain the equivalent representation
\begin{equation}
\brho(t,0)=D_-\bphi_+(t;0)\bS(0)^{-1}\sigma_3\bS(0)\bphi_+(t;0)^{-1},\quad t>0.
\end{equation}
Since $\bphi_+(t;0)\to\mathbb{I}$ as $t\to+\infty$, the following limit evidently exists:
\begin{equation}
\lim_{t\to+\infty}\brho(t,0)=D_-\bS(0)^{-1}\sigma_3\bS(0).
\end{equation}
Using the identities \eqref{e:scattering-matrix},
and looking at the first-row elements then gives
\begin{equation}
\lim_{t\to+\infty}P(t,0)=-D_-\frac{2\b{r_0}}{1+|r_0|^2}\frac{S_{2,2}(0)}{S_{1,1}(0)}\quad\text{and}\quad
\lim_{t\to+\infty}D(t,0)=D_-\frac{1-|r_0|^2}{1+|r_0|^2}.
\end{equation}
The fraction $S_{2,2}(0)/S_{1,1}(0)$ has unit modulus, and under the assumptions of Theorem~\ref{thm:reconstruction} (absence of eigenvalues or spectral singularities) this fraction can be expressed via a \emph{trace identity}, which we now recall.  Because $S_{1,1}(\lambda)$ is analytic for $\Im(\lambda)>0$, continuous for $\Im(\lambda)\ge 0$, bounded away from zero for $\Im(\lambda)\ge 0$, and satisfies $S_{1,1}(\lambda)\to 1$ as $\lambda\to\infty$, one can write $S_{1,1}(\lambda)=\ee^{F^+(\lambda)}$, where $F^+(\lambda)$ denotes the boundary value of a function analytic for $\Im(\lambda)>0$ and continuous for $\Im(\lambda)\ge 0$ that vanishes as $\lambda\to\infty$.  Likewise $S_{2,2}(\lambda)=\ee^{-F^-(\lambda)}$, where $F^-(\lambda)$ is the boundary value of a function analytic for $\Im(\lambda)<0$ and continuous for $\Im(\lambda)\le 0$ that vanishes as $\lambda\to\infty$.  The identities \eqref{e:scattering-matrix} and the definition \eqref{e:reflection-coefficient} of the reflection coefficient $r(\lambda)$ then imply that for real $\lambda$, $F^+(\lambda)-F^-(\lambda)=-\ln(1+|r(\lambda)|^2)$.  From the Plemelj formula it then follows that $F^\pm(\lambda)$ are the boundary values of the following function analytic for $\lambda\in\Complex\setminus\Real$:
\begin{equation}
F(\lambda)=-\frac{1}{2\pi\ii}\int_\Real \frac{\ln(1+|r(s)|^2)\dd s}{s-\lambda},\quad\lambda\in\Complex\setminus\Real.
\label{e:f-lambda}
\end{equation}
Evaluating the sum of the boundary values at $\lambda=0$ and comparing with Definition~\ref{def:aleph} gives $F^+(0)+F^-(0)=\ii\aleph$.  Therefore under the assumptions of Theorem~\ref{thm:reconstruction}
we obtain
\begin{equation}
\lim_{t\to+\infty}P(t,0)=-D_- \frac{2\b{r_0}\ee^{-\ii\aleph}}{1+|r_0|^2}\quad\text{and}\quad
\lim_{t\to+\infty}D(t,0)=D_-\frac{1-|r_0|^2}{1+|r_0|^2}
\label{e:final-state-z-zero}
\end{equation}
for the final state of the active medium induced by the incident optical pulse $q_0(\cdot)$ exactly at the edge $z=0$.  Obviously, the medium is not in a pure state as $t\to+\infty$ for $z=0$, unless $r_0=0$, in which case the medium returns to its initial pure state (which could be stable or unstable).  This is the clearest demonstration so far that for the medium to decay to the stable state asymptotically as $t\to+\infty$ for all $z>0$, some kind of boundary layer is generally needed to resolve the transition near the edge.

\subsection{Precise definition of self-similar solutions}
\label{s:precise-self-similar}
The differential equation \eqref{e:PIII} is a special case of the general four-parameter family of Painlev\'e-III equations for which (in the notation of \cite[Chapter 32]{dlmf}) $\alpha=0$, $\beta=4$, $\gamma=4$ and $\delta=-4$. It also fits into the isomonodromy scheme of Jimbo and Miwa (described for instance in \cite{fikn2006}) with parameters $\Theta_0=\Theta_\infty=0$.
Most solutions have a branch point at the origin, which is the unique fixed singular point for \eqref{e:PIII}.  However, there are two one-parameter families of solutions that are analytic at $X=0$.  Indeed, given any nonzero complex number $u_0\in\mathbb{C}\setminus\{0\}$, there is a unique solution analytic at $X=0$ with $u(0)=u_0$.  A second family consists of analytic solutions vanishing at $X=0$.  Here, the equation \eqref{e:PIII} determines $u'(0)=1$ and $u''(0)=0$, but $u'''(0)=\omega\in\mathbb{C}$ is arbitrary, after which all subsequent Taylor coefficients are uniquely determined in terms of $\omega$. The latter solutions are the ones relevant here, with $\omega$ restricted to the open interval $(-3,3)$, and we denote them by $u=u(X;\omega)$.  These solutions are all odd functions of $X\in\mathbb{C}$ analytic at the origin and globally meromorphic, with Taylor expansion
\begin{equation}
u(X;\omega)=X+\omega\frac{X^3}{3!} + 40\frac{X^5}{5!} + \O(X^7),\quad X\to 0.
\label{e:u-series}
\end{equation}
For $\omega\in (-3,3)$ and $X$ on the real (resp., imaginary) axis in the complex plane, these solutions are real-valued (resp., purely imaginary).  The family of solutions $u(X;\omega)$ of the PIII equation \eqref{e:PIII} for $\omega\in (-3,3)$ coincides with that appearing in~\cite[Theorem 2]{bb2019}.

Given such a solution $u=u(X;\omega)$ of \eqref{e:PIII}, we consider the auxiliary functions $y(X)$, $U(X)$, and $s(X)$ satisfying the related first-order system \eqref{e:coupled-PIII-0-0}.  Using the relation $u=-y/s$, the latter system can be written in the form
\begin{equation}
\frac{\dd y}{\dd X}  = \frac{2y}{u}\,,\qquad
X \frac{\dd s}{\dd X} = 
s+2Xus-4usU\,,\qquad
X \frac{\dd U}{\dd X} = -4u U^2 + (4uX + 1) U.
\label{e:PIII-2}
\end{equation}
From \eqref{e:u-series}, upon substituting power series for $y(X)$, $s(X)$, and $U(X)$, one easily sees that solutions of these equations analytic at the origin necessarily vanish there:  $y(0)=U(0)=s(0)=0$, and $y'(0)=0$ must also hold.  The values $y''(0)=y_2$, $U'(0)=U_1$, and $s'(0)=s_1$ are free, and then all subsequent Taylor coefficients are determined in terms of $y_2$, $s_1$, $U_1$, and $\omega$ (via the Taylor coefficients of $u(X;\omega)$) by \eqref{e:PIII-2} and \eqref{e:u-series}.
For instance, from the first equation in \eqref{e:PIII-2} and the expansion \eqref{e:u-series}, one finds that the solution with $y(0)=y'(0)=0$ and $y''(0)=y_2$ has the series
\begin{equation}
y(X)=y_2\frac{X^2}{2!} - 2y_2\omega\frac{X^4}{4!} + \O(X^6),\quad X\to 0.
\label{e:y-series}
\end{equation}
In fact, the values $y_2$, $s_1$, and $U_1$ are not independent.  Indeed,
since $s=-y/u$, we can combine the series \eqref{e:u-series} and \eqref{e:y-series} to find that $s_1$ is not actually arbitrary:
\begin{equation}
s(X)=-\frac{y(X)}{u(X;\omega)} = -\frac{1}{2}y_2X + y_2\omega\frac{X^3}{3!}+\O(X^5),\quad X\to 0.
\label{e:s-series}
\end{equation}
Likewise, using the series \eqref{e:u-series} 
and \eqref{e:s-series} in the second equation in \eqref{e:PIII-2} shows that $U_1$ is also not arbitrary:
\begin{equation}
U(X)=
\frac{s(X)+2Xu(X;\omega)s(X)-Xs'(X)}{4u(X;\omega)s(X)}=
\left(\frac{1}{6}\omega+\frac{1}{2}\right)X + \O(X^3),\quad X\to 0.
\label{e:U-series}
\end{equation}
Finally, to ensure that the functions $y(X)$, $s(X)$, and $U(X)$ solve not only the system \eqref{e:coupled-PIII-0-0} but also the self-similar MBEs in the form \eqref{e:coupled-self-similar}, we need to impose the condition $J=-1$ (see \eqref{e:J-constant}), which in light of the series \eqref{e:s-series}--\eqref{e:U-series} means that
\begin{equation}
y_2^2=1-\left(\frac{\omega}{3}\right)^2.
\label{e:y2-u3}
\end{equation}
For $\omega\in (-3,3)$, selecting the positive square root in \eqref{e:y2-u3} uniquely determines a one-parameter family of solutions of \eqref{e:coupled-self-similar} that we denote by $y(X;\omega)$, $s(X;\omega)$, and $U(X;\omega)$.
These three functions are also globally meromorphic for $X\in\Complex$, but they are analytic on the real and imaginary axes, and $y(X;\omega)$ is even while $U(X;\omega)$ and $s(X;\omega)$ are both odd functions (as is $u(X;\omega)$).  All four of the functions are real-valued for real $X$; $y(X;\omega)$ is also real for imaginary $X$ while $s(X;\omega)$, $U(X;\omega)$, and $u(X;\omega)$ are purely imaginary there.  An additional symmetry is that
\begin{equation}
y(-\ii X;\omega)=-y(X;-\omega),\quad
s(-\ii X;\omega)=-\ii s(X;-\omega),\quad\text{and}\quad
U(-\ii X;\omega)=\ii U(X;-\omega)-\ii X,
\label{e:PIII-rotate-X-identities}
\end{equation}
implying that also $u(-\ii X;\omega)=-y(-\ii X;\omega)/s(-\ii X;\omega)= -\ii u(X;-\omega)$.  All of the properties of the functions $u(X;\omega)$, $y(X;\omega)$, $s(X;\omega)$, and $U(X;\omega)$ described above are proved rigorously in Section~\ref{s:Properties-of-Y} below.
With this setup, we may now define two families of self-similar solutions of the MBE system.
\begin{definition}[Particular self-similar solutions]
Let $\omega\in (-3,3)$ and $\xi=\ee^{\ii\kappa}$, $\kappa\in\mathbb{R}\pmod{2\pi}$.  Let $y(X;\omega)$, $s(X;\omega)$, and $U(X;\omega)$ be the unique solutions of \eqref{e:coupled-self-similar} analytic at $X=0$ and satisfying the initial conditions
\begin{equation}
y(0;\omega)=0,\quad y'(0;\omega)=0,\quad y''(0;\omega)=\sqrt{1-\left(\frac{\omega}{3}\right)^2},
\label{e:y-Taylor}
\end{equation}
\begin{equation}
s(0;\omega)=0,\quad s'(0;\omega)=-\frac{1}{2}\sqrt{1-\left(\frac{\omega}{3}\right)^2},\quad\text{and}
\label{e:s-Taylor}
\end{equation}
\begin{equation}
U(0;\omega)=0,\quad U'(0;\omega)=\frac{1}{6}\omega+\frac{1}{2}.
\label{e:U-Taylor}
\end{equation}
Then with $x=\sqrt{2tz}\ge 0$, one real-valued self-similar solution of the MBE system is $q(t,z)=q_\mathrm{u}(t,z;\omega,\xi)$, $P(t,z)=P_\mathrm{u}(t,z;\omega,\xi)$, and $D(t,z)=D_\mathrm{u}(t,z;\omega)$ where
\begin{equation}
\begin{split}
q_\mathrm{u}(t,z;\omega,\xi)&:=t^{-1}\xi y(x;\omega)\\
P_\mathrm{u}(t,z;\omega,\xi)&:=2\xi x^{-1}s(x;\omega)\\
D_\mathrm{u}(t,z;\omega)&:=1-2x^{-1}U(x;\omega),
\end{split}
\label{e:unstable-self-similar}
\end{equation}
and another is $q(t,z)=q_\mathrm{s}(t,z;\omega,\xi)$, $P(t,z)=P_\mathrm{s}(t,z;\omega,\xi)$, and $D(t,z)=D_\mathrm{s}(t,z;\omega)$ where
\begin{equation}
\begin{split}
q_\mathrm{s}(t,z;\omega,\xi)&:=t^{-1}\xi y(-\ii x;\omega)\\
P_\mathrm{s}(t,z;\omega,\xi)&:=-2\ii \xi x^{-1}s(-\ii x;\omega)\\
D_\mathrm{s}(t,z;\omega)&:= -1+2\ii x^{-1}U(-\ii x;\omega).
\end{split}
\label{e:stable-self-similar}
\end{equation}
\label{def:self-similar}
\end{definition}
Although for given $\omega\in (-3,3)$ and $\xi=\ee^{\ii\kappa}$ these two self-similar solutions are derived from exactly the same solution of the system \eqref{e:coupled-self-similar}, the fact that that solution is sampled along two orthogonal axes in the complex plane leads in general to very different behavior.
Plots of these solutions are shown for representative values of $\omega$ and $\xi=1$ in Figures~\ref{f:unstable-self-similar}--\ref{f:stable-self-similar}.  Note that the plots for $\omega=0$ in the two cases are comparable due to \eqref{e:PIII-rotate-X-identities}.
\begin{figure}[h]
\centering
\includegraphics[width = 0.49\textwidth]{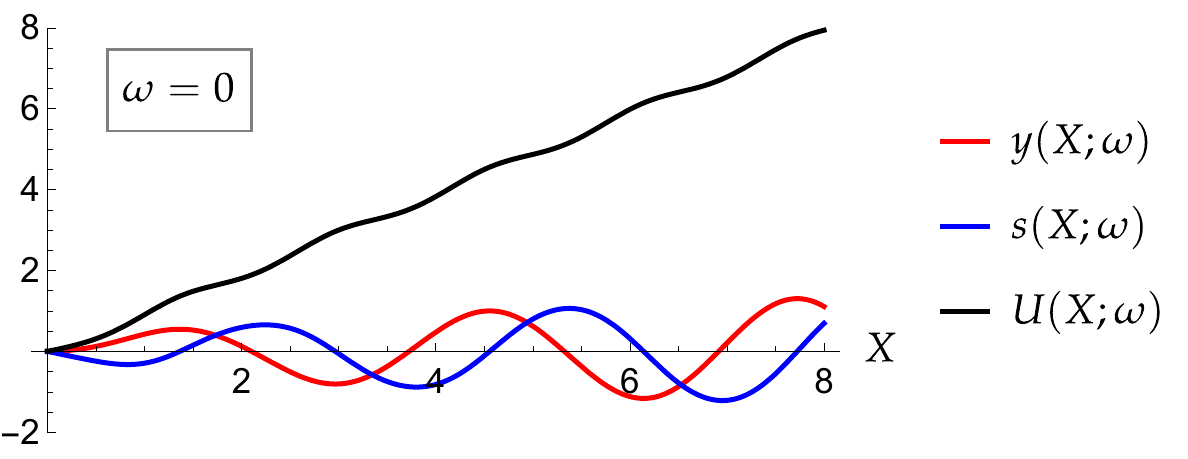}
\includegraphics[width = 0.48\textwidth]{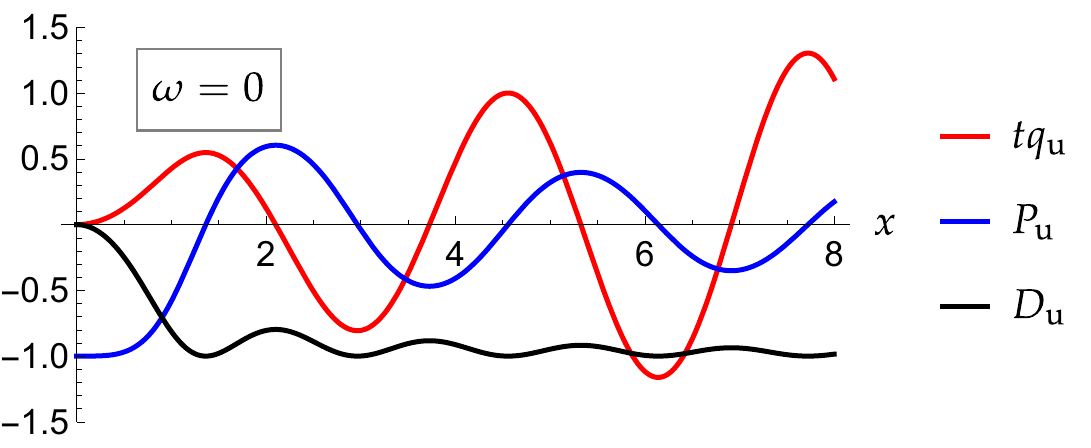}\\
\includegraphics[width = 0.49\textwidth]{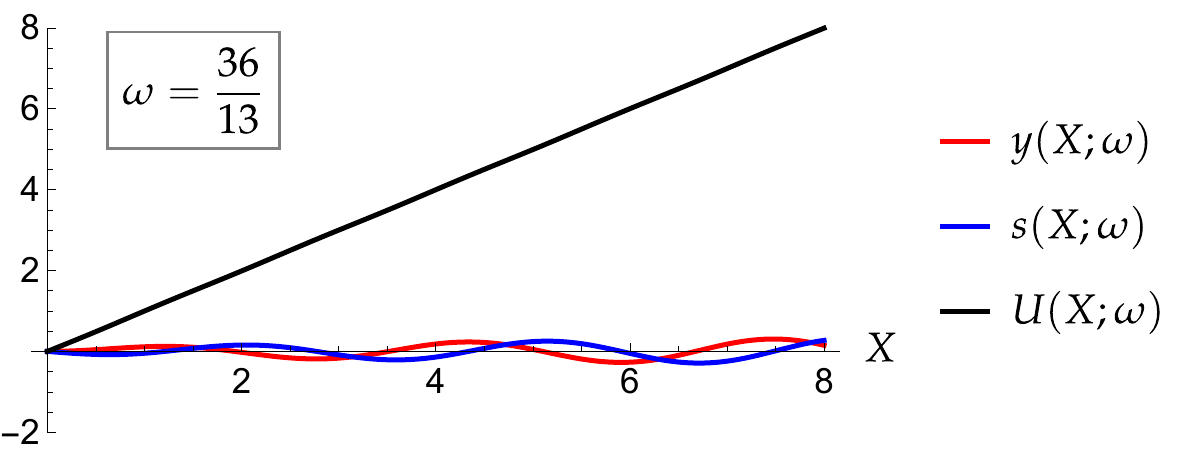}
\includegraphics[width = 0.48\textwidth]{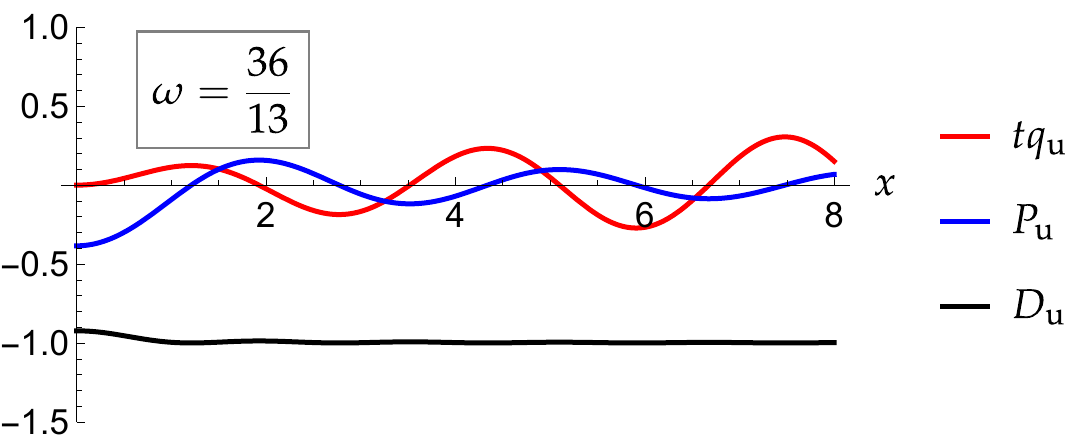}
\caption{
The PIII functions $y(x;\omega)$, $s(x;\omega)$ and $U(x;\omega)$
from Definition~\ref{def:self-similar} evaluated for $x>0$ (left column) and the corresponding self-similar solutions $tq_\mathrm{u}$, $P_\mathrm{u}$, and $D_\mathrm{u}$ plotted as functions of $x$ for $\xi=1$ (right column).  The parameter $\omega$ is $\omega=0$ (top row) and $\omega=\tfrac{36}{13}$ (bottom row), corresponding to $r_0=1$ and $r_0=5$ respectively.
}
\label{f:unstable-self-similar}
\end{figure}
\begin{figure}[h]
\centering
\includegraphics[width = 0.52\textwidth]{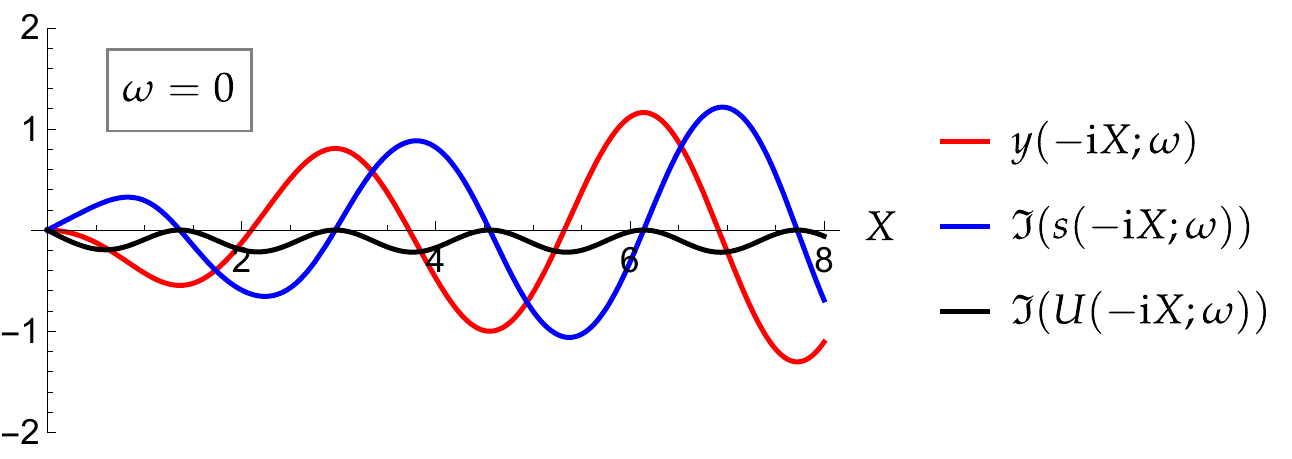}
\includegraphics[width = 0.46\textwidth]{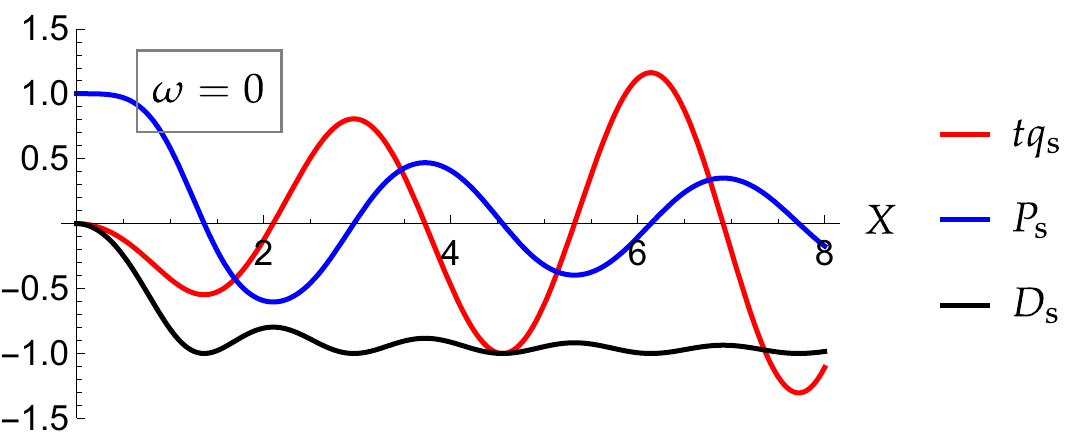}\\
\includegraphics[width = 0.52\textwidth]{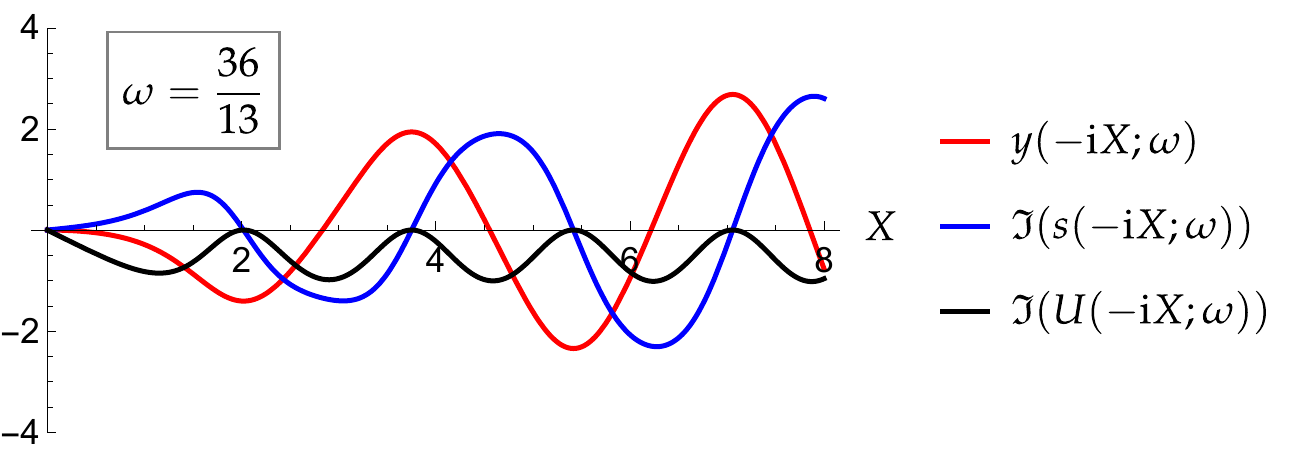}
\includegraphics[width = 0.46\textwidth]{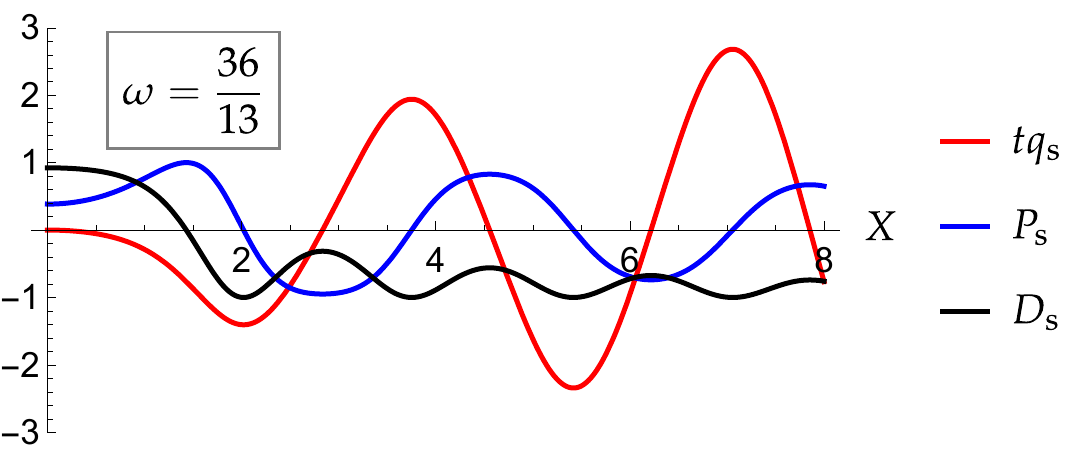}
\caption{
The ``rotated'' PIII functions $y(-\ii x;\omega)$, $s(-\ii x;\omega)$ and $U(-\ii x;\omega)$ from
Definition~\ref{def:self-similar} evaluated for $x>0$ (left column) and the corresponding self-similar solutions $tq_\mathrm{s}$, $P_\mathrm{s}$, and $D_\mathrm{s}$ plotted as functions of $x$ for $\xi=1$ (right column).
The parameter $\omega$ is $\omega=0$ (top row) and $\omega=\tfrac{36}{13}$ (bottom row), corresponding to $r_0=1$ and $r_0=5$ respectively.  
}
\label{f:stable-self-similar}
\end{figure}
\subsection{Asymptotic regimes within the light cone}
Finally, we describe the portion of the light cone $z\ge 0$, $t\ge 0$ in which our asymptotic results are valid (by causality asserted in Theorem~\ref{thm:reconstruction},
all solutions to the Cauchy problem~\eqref{e:mbe} considered in this paper are trivial outside of the light cone).
\begin{definition}[asymptotic regimes within the light cone]
\label{def:regions}
For $\alpha<1$, consider the relation between the coordinates $(t,z)$ given by
\begin{equation}
\label{e:asymptotics-curve}
z = C t^{\alpha}\,,\qquad C > 0 \mbox{ fixed.}
\end{equation}
Three asymptotic regimes within the light cone $z\ge 0$, $t\ge 0$ are defined as follows:
\begin{itemize}
\item the \emph{medium-edge regime} corresponds to $t\to+\infty$ subject to \eqref{e:asymptotics-curve} with $\alpha < -1$;
\item the \emph{transition regime} corresponds to $t\to+\infty$ subject to \eqref{e:asymptotics-curve} with $\alpha = -1$; 
\item the \emph{medium-bulk regime} corresponds to $t\to+\infty$ subject to \eqref{e:asymptotics-curve} with $-1 < \alpha < 1$.
\end{itemize}
\end{definition}
Since $\alpha<1$, the condition $z =o(t)$ as $t\to+\infty$ is met in all three regimes; this is the principal condition under which our analysis is valid.
Note that in the medium-edge regime $tz \to 0$ as $t\to+\infty$, in the transition regime $tz$ is fixed, and in the medium-bulk regime $tz\to+\infty$ as $t\to+\infty$.  The three asymptotic regimes within the light cone are illustrated in Figure~\ref{f:asymptotics}.
\begin{figure}[h]
\includegraphics[width = 0.56\textwidth]{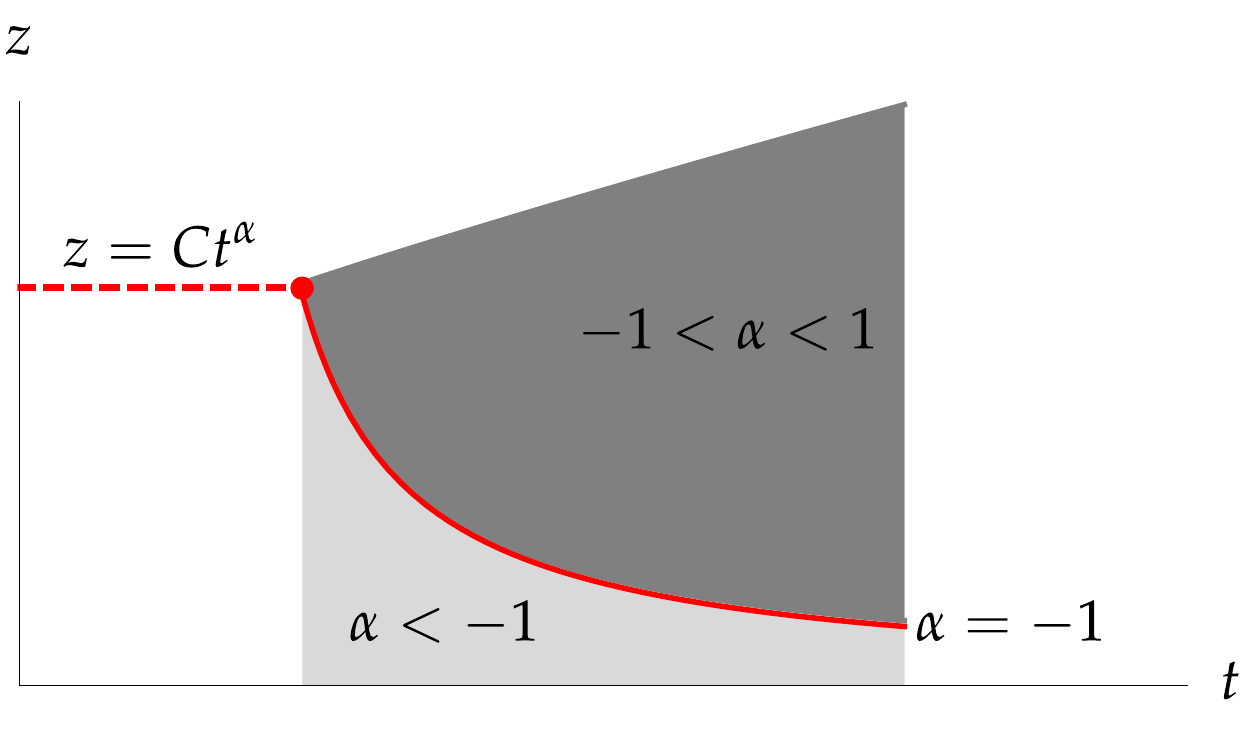}
\caption{
With $(t,z)$ related by \eqref{e:asymptotics-curve},
the asymptotic regimes within the light cone are indicated as follows:
(i) light gray shading denotes the medium-edge regime with $\alpha < -1$;
(ii) the red solid curve denotes the transition regime with $\alpha = -1$;
(iii) dark gray shading denotes the medium-bulk regime with $-1 < \alpha < 1$.
}
\label{f:asymptotics}
\end{figure}

\subsection{Results}
\label{s:results}
Our main result is the following.
\begin{theorem}[Global asymptotics --- generic case]
Suppose that the incident pulse $q_0$ satisfies the hypotheses of Theorem~\ref{thm:reconstruction}, that $M=0$, i.e., $r_0\neq 0$, and that $\omega$ is defined by
\begin{equation}
\omega\coloneq 3\frac{|r_0|^2-1}{|r_0|^2+1}\in (-3,3).
\label{e:omega-r0}
\end{equation}
Then, for every $N\in\mathbb{Z}_{\ge 2}$ the causal solution of the Cauchy problem \eqref{e:mbe} in a stable medium ($D_-=-1$) satisfies
\begin{equation}
\begin{split}
q(t,z)&=q_\mathrm{s}(t,z;\omega,\ee^{-\ii(\aleph+\arg(r_0))})+\O(z/t)+\O(t^{1-N}),\\
P(t,z)&=P_\mathrm{s}(t,z;\omega,\ee^{-\ii(\aleph+\arg(r_0))}) + \O((z/t)^\frac{1}{2})+\O(t^{\frac{3}{2}-N}),\\
D(t,z)&=D_\mathrm{s}(t,z;\omega)+\O((z/t)^\frac{1}{2}) + \O(t^{\frac{3}{2}-N}),
\end{split}
\label{e:qPD-generic-global-stable}
\end{equation}
as $t\to+\infty$ for $z\ge 0$ with $z=o(t)$.  In the same limit, the causal solution of the Cauchy problem \eqref{e:mbe} in an unstable medium ($D_-=1$) satisfies
\begin{equation}
\begin{split}
q(t,z)&=q_\mathrm{u}(t,z;\omega,\ee^{-\ii(\aleph+\arg(r_0))})+\O(z/t)+\O(t^{1-N}),\\
P(t,z)&=P_\mathrm{u}(t,z;\omega,\ee^{-\ii(\aleph+\arg(r_0))}) + \O((z/t)^\frac{1}{2})+\O(t^{\frac{3}{2}-N}),\\
D(t,z)&=D_\mathrm{u}(t,z;\omega)+\O((z/t)^\frac{1}{2}) + \O(t^{\frac{3}{2}-N}).
\end{split}
\label{e:qPD-generic-global-unstable}
\end{equation}
In these formul\ae, the phase $\aleph$ is given in Definition~\ref{def:aleph} and the explicit terms are the self-similar solutions from Definition~\ref{def:self-similar}.
\label{thm:global-M0-stable-and-unstable}
\end{theorem}
\begin{remark}
The index $N\ge 2$ is an artifact of our method of proof, which, roughly speaking, exploits a finite level of smoothness and decay of $q_0$ via the reflection coefficient $r(\lambda)$.  We leave this index in the statements of our results for readers interested to see what can be proved if weaker assumptions are taken on $q_0$.
\end{remark}
The explicit terms in the asymptotic formul\ae\ of Theorem~\ref{thm:global-M0-stable-and-unstable} are, aside from a factor of $t^{-1}$ in the optical field $q(t,z)$, functions of the similarity variable $x=\sqrt{2tz}$.  This variable becomes fixed exactly in the transition regime of Definition~\ref{def:regions}, and when $z=Ct^{-1}$ the error terms simplify as follows:
\begin{itemize}
\item in the asymptotic formul\ae\ for $q(t,z)$, the error terms take the form $\O(t^{-2}) +\O(t^{1-N})$;
\item in the asymptotic formul\ae\ for $P(t,z)$ and $D(t,z)$, the error terms are $\O(t^{-1}) + \O(t^{\frac{3}{2}-N})$.
\end{itemize}
Hence in this regime there is barely any dependence in the size of the error terms on the  index $N$, being of a different form for $N=2$ than for $N\ge 3$.

The self-similar solutions in turn simplify when $x\to 0$ and when $x\to+\infty$, which correspond to the medium-edge and medium-bulk regimes respectively.  We have the following corollaries of Theorem~\ref{thm:global-M0-stable-and-unstable}.

\begin{corollary}[Medium-edge asymptotics --- generic case]
Under the assumptions of Theorem~\ref{thm:global-M0-stable-and-unstable}, the causal solution of the Cauchy problem \eqref{e:mbe} in a stable ($D_-=-1$) or unstable ($D_-=1$) medium satisfies, for every $N\in\mathbb{Z}_{\ge 2}$,
\begin{equation}
\begin{split}
q(t,z)&=D_-\frac{2\b{r_0}\ee^{-\ii\aleph}}{1+|r_0|^2}z + \O(t^{2\alpha+1})+\O(t^{\alpha-1})+\O(t^{1-N}),\\
P(t,z)&=-D_-\frac{2\b{r_0}\ee^{-\ii\aleph}}{1+|r_0|^2} +\O(t^{\alpha+1})+\O(t^{\frac{1}{2}(\alpha-1)})+\O(t^{\frac{3}{2}-N}),\\
D(t,z)&=D_-\frac{1-|r_0|^2}{1+|r_0|^2}+\O(t^{\alpha+1}) + \O(t^{\frac{1}{2}(\alpha-1)})+\O(t^{\frac{3}{2}-N}),
\end{split}
\end{equation}
as $t\to+\infty$ with $z$ related to $t$ by \eqref{e:asymptotics-curve} with $\alpha<-1$.  \label{cor:medium-edge-M0-stable-and-unstable}
\end{corollary}
\begin{proof}
Since $y(X;\omega)$ is even while $s(X;\omega)$ and $U(X;\omega)$ are odd analytic functions of $X$, using \eqref{e:y-Taylor}--\eqref{e:U-Taylor} to expand the leading terms in \eqref{e:qPD-generic-global-stable}--\eqref{e:qPD-generic-global-unstable} as given in Definition~\ref{def:self-similar} and using \eqref{e:asymptotics-curve} in the error terms proves this result.
\end{proof}

Comparing with \eqref{e:final-state-z-zero} and taking into account the Maxwell equation $q_z=-P$, this result is satisfying because it is consistent with the state of the active medium for large $t>0$ exactly at the edge $z=0$, as computed directly from the given incident optical pulse $q_0(\cdot)$.

\begin{corollary}[Medium-bulk asymptotics --- generic case]
Under the assumptions of Theorem~\ref{thm:global-M0-stable-and-unstable}, the causal solution of the Cauchy problem \eqref{e:mbe} in a stable ($D_-=-1$) or unstable ($D_-=1$) medium satisfies
\begin{equation}
\begin{split}
q(t,z)&=D_-\ee^{-\ii(\aleph+\arg(r_0))}\frac{1}{t}\left(\frac{tz}{2}\right)^\frac{1}{4}
A\sin(\varphi(\sqrt{2tz}))+ \O(t^{-\frac{1}{4}(\alpha+5)}) + \O(t^{\alpha-1}) + \O(t^{1-N}),\\
P(t,z)&=-D_-\ee^{-\ii(\aleph+\arg(r_0))}\left(\frac{tz}{2}\right)^{-\frac{1}{4}}A\cos(\varphi(\sqrt{2tz}))
+\O(t^{-\frac{3}{4}(\alpha+1)}) + \O(t^{\frac{1}{2}(\alpha-1)})+ \O(t^{\frac{3}{2}-N}),\\
D(t,z)&=-1+\frac{1}{2}\left(\frac{tz}{2}\right)^{-\frac{1}{2}}A^2\cos^2(\varphi(\sqrt{2tz}))
+\O(t^{-(\alpha+1)}) + \O(t^{\frac{1}{2}(\alpha-1)}) + \O(t^{\frac{3}{2}-N}),
\end{split}
\label{e:qPD-medium-bulk-M0-stable-and-unstable}
\end{equation}
as $t\to+\infty$ with $z$ related to $t$ by \eqref{e:asymptotics-curve} with $\alpha\in (-1,1)$, where $\aleph$ is given in Definition~\ref{def:aleph}, and where we define
\begin{equation}
\begin{split}
\nu&\coloneq\frac{1}{2\pi}\ln(1+|r_0|^{-2D_-})>0,\\
A&\coloneq \sqrt{\frac{2}{\pi}}\frac{|\Gamma(1+\ii\nu)|}{|r_0|^{\frac{1}{2}D_-}(1+|r_0|^{2D_-})^\frac{1}{4}}>0,\\
\varphi(x)&\coloneq 2x-\nu\ln(x)-\frac{\pi}{4}+\arg(\Gamma(1+\ii\nu))-3\nu\ln(2).
\end{split}
\label{e:parameters-medium-bulk-M0-stable-and-unstable}
\end{equation}
\label{cor:medium-bulk-M0-stable-and-unstable}
\end{corollary}
Although it follows from Theorem~\ref{thm:global-M0-stable-and-unstable}, the proof will be given later after large-$x$ asymptotic formul\ae\ for the PIII functions appearing in the leading terms are derived.
In particular, Corollary~\ref{cor:medium-bulk-M0-stable-and-unstable} applies to the limit $t\to+\infty$ with $z>0$ fixed, which corresponds to $\alpha=0$ in \eqref{e:asymptotics-curve}.  In this case, the error terms simplify as follows, taking also into account that the index $N$ satisfies $N\ge 2$:
\begin{itemize}
\item in the asymptotic formula for $q(t,z)$, the error terms simplify to $\O(t^{-1})$;
\item in the asymptotic formul\ae\ for $P(t,z)$ and $D(t,z)$, the error terms simplify to $\O(t^{-\frac{1}{2}})$, and the explicit term in $D(t,z)+1$ is also of this order.
\end{itemize}
This result therefore shows that, unlike the situation near the edge of the active medium $z=0$, \emph{for every fixed $z>0$, the active medium decays as $t\to+\infty$ to the stable pure state ($P=0$ and $D=-1$), regardless of whether the initial state was stable or unstable}.  In the unstable case, this may be regarded as a decay process stimulated by the incident optical pulse.  In the stable case it instead provides mathematical justification for the heuristic terminology of ``stability'' for the active medium with $D_-=-1$.  The decay to the stable pure state is quite slow, with explicit leading terms, a fact that leads to two insights that are important enough to state explicitly as corollaries.
\begin{corollary}
Under the assumptions of Theorem~\ref{thm:global-M0-stable-and-unstable}, for every $z>0$ the optical pulse function $t\mapsto q(t,z)$ does not lie in $L^1(\Real)$.  However, the limit (improper integral)
\begin{equation}
\lim_{T\to+\infty}\int_0^Tq(t,z)\dd t
\label{e:q-improperly-integrable}
\end{equation}
exists.
\label{cor:not-L1}
\end{corollary}
\begin{proof}
The lack of absolute integrability of $t\mapsto q(t,z)$ is obvious because the leading term in \eqref{e:qPD-medium-bulk-M0-stable-and-unstable} is a sinusoidal oscillation of frequency proportional to $t^{-\frac{1}{2}}$ and amplitude proportional to $t^{-\frac{3}{4}}$ (so in fact the optical pulse is in $L^2(\Real)$).

The existence of the improper integral \eqref{e:q-improperly-integrable} is proved by applying the Fundamental Theorem of Calculus to the differential equation $P_t=-2qD$.  Fixing $z>0$, we have
\begin{equation}
\begin{split}
P(T,z) - P(0,z)
& = \int_0^T P_t(t,z)\dd t\\
& = -2\int_0^T q(t,z)D(t,z)\dd t\\
& = 2\int_0^Tq(t,z)\dd t -2\int_0^Tq(t,z)[D(t,z)+1]\dd t\,.
\end{split}
\end{equation}
By causality, $P(0,z)=0$.  Applying Corollary~\ref{cor:medium-bulk-M0-stable-and-unstable} with $z>0$ fixed then implies that $P(T,z)\to 0$ as $T\to+\infty$, and that $q(t,z)=\O(t^{-\frac{3}{4}})$ and $D(t,z)+1=\O(t^{-\frac{1}{2}})$ as $t\to+\infty$. Hence $q(\cdot,z)[D(\cdot,z)+1]\in L^1(\Real_+)$, so we deduce that
\begin{equation}
\lim_{T\to+\infty}\int_0^Tq(t,z)\dd t = \int_0^{+\infty} q(t,z)[D(t,z)+1]\dd t\,,
\end{equation}
where the integral on the right-hand side is absolutely convergent.
\end{proof}
This result is important because it proves that the most important assumption in IST theory is violated under the evolution in $z$, even if it is assumed to hold at $z=0$ (or, for that matter, even if $q_0$ has compact support); however using the existence of the improper integral \eqref{e:q-improperly-integrable} or other related interpretations of divergent integrals it may indeed be possible to recover the existence of Jost solutions for almost all $\lambda\in\Real$ through rigorous analysis.  The next result proves the ill-posedness of the Cauchy problem \eqref{e:mbe} in the initially-unstable case if causality is not imposed.
\begin{corollary}
There exist incident pulses $q_0$ satisfying the hypotheses of Theorem~\ref{thm:reconstruction} for which the Cauchy problem \eqref{e:mbe} for the Maxwell-Bloch equations with an initially-unstable medium $D_-=1$ has (other) solutions that are not causal and that decay to both stable and unstable pure states as $t\to+\infty$.
\label{cor:no-uniqueness-without-causality}
\end{corollary}
\begin{proof}
Let $q_0$ be an incident pulse satisfying the hypotheses and the following additional properties:  $\mathrm{supp}(q_0)=[0,T]$ for some $T>0$, $r_0\neq 0$, and $q_0(t)=q_0(T-t)$ for all $t\in\mathbb{R}$.  First consider propagation in an initially-unstable medium, $D_-=1$, and let $q(t,z)$ denote the causal optical pulse for $t\in\mathbb{R}$ and $z\ge 0$ corresponding to the incident pulse $q_0$.
We apply an elementary symmetry of the MBE system to generate another solution, namely $\mathcal{S}: (q(t,z),P(t,z),D(t,z))\mapsto (\mathcal{S}q(t,z),\mathcal{S}P(t,z),\mathcal{S}D(t,z))\coloneq(q(T-t,z),P(T-t,z),-D(T-t,z))$.  Then $\mathcal{S}q(t,z)$ is an optical pulse for a noncausal solution with the same incident pulse $q_0$ in an initially-unstable medium. Indeed, according to Corollary~\ref{cor:medium-bulk-M0-stable-and-unstable} we have $D(t,z)\to -1$ as $t\to+\infty$ and hence also $\mathcal{S}D(t,z)\to 1$ as $t\to -\infty$ for all $z\ge 0$, so it is also a solution of the same Cauchy problem.  However by the same result, $q(t,z)$ is definitely not supported on the half-line $t\le T$, so $\mathcal{S}q(t,z)$ is not supported on $t\ge 0$, proving that the solution is not causal.  This noncausal solution also has the property that $\mathcal{S}D(t,z)=-1$ holds for all $t\ge T$, so like the causal solution it exhibits decay to the stable state as $t\to+\infty$.  Now let $(\widetilde{q},\widetilde{P},\widetilde{D})$ denote the causal solution for the same $q_0$, now incident on an initially-stable medium with $\widetilde{D}_-=-1$.
Again applying the symmetry $\mathcal{S}$, we see that $\mathcal{S}\widetilde{q}(t,z)$ is an optical pulse for a solution with the same incident pulse $q_0$ in an unstable medium because Corollary~\ref{cor:medium-bulk-M0-stable-and-unstable} gives $\widetilde{D}(t,z)\to -1$ as $t\to+\infty$ and so also $\mathcal{S}\widetilde{D}(t,z)\to 1$ as $t\to -\infty$ for all $z\ge 0$.  However since $\widetilde{q}(t,z)$ is not supported on the half-line $t\le T$, $\mathcal{S}\tilde{q}(t,z)$ is not supported on $t\ge 0$, so the solution is again noncausal.  Unlike the previously constructed noncausal solution, this one satisfies $\mathcal{S}\widetilde{D}(t,z)=1$ for all $t\ge T$, so it exhibits decay to the \emph{unstable} state as $t\to+\infty$.
\end{proof}

\begin{remark}[On time translation symmetry]
If $q_0(\cdot)$ satisfies the assumptions of Theorem~\ref{thm:global-M0-stable-and-unstable}, then so does the time translate $\widetilde{q}_0(\cdot)\coloneq q_0(\cdot-\Delta t)$ for every $\Delta t>0$.  If $\bphi_\pm(t;\lambda)$ are the Jost solution matrices for $q_0(\cdot)$, then $\widetilde{\bphi}_\pm(t;\lambda)=\bphi_\pm(t;\lambda)\ee^{\ii\lambda\Delta t\sigma_3}$ are those corresponding to $\widetilde{q}_0(\cdot)$, from which it follows that the reflection coefficients are related by $\widetilde{r}(\lambda)=\ee^{2\ii\lambda\Delta t}r(\lambda)$.  Hence, both $r_0$ and $\aleph$ are completely insensitive to time translation, and these are the only quantities on which Theorem~\ref{thm:global-M0-stable-and-unstable} and its corollaries depend.  We conclude that exactly the same asymptotic formul\ae\ describe the causal solutions for both incident pulses $q_0(\cdot)$ and $\widetilde{q}_0(\cdot)$.  The apparent paradox is resolved upon noting that the results all require the limit $t\to+\infty$, in which case $t-\Delta t=t(1+\O(t^{-1}))$ and, for fixed $z>0$, $\varphi(\sqrt{2(t-\Delta t)z})=\varphi(\sqrt{2tz})+\O(t^{-\frac{1}{2}})$, so time-translation of the leading terms can always be absorbed into the error terms.
\end{remark}

Our final results concern incident pulses that are not generic in that the first moment $r_0$ vanishes.  The first result applies to the case of propagation in an initially-stable medium ($D_-=-1$), and it displays an interesting dependence at the leading order on the index $M\ge 1$ of the first nonzero moment of the reflection coefficient.
\begin{theorem}[Global asymptotics --- nongeneric case for a stable medium]
Suppose that the incident pulse $q_0$ satisfies the hypotheses of Theorem~\ref{thm:reconstruction} and that $r_0=0$, so that the index $M$ of the first nonzero moment of the reflection coefficient is strictly positive.  For every integer $N\ge M+2$, the causal solution of the Cauchy problem \eqref{e:mbe} in a
stable medium ($D_-=-1$) satisfies
\begin{equation}
\begin{split}
q(t,z)&=-\ii^M\frac{\b{r_0^{(M)}}\ee^{-\ii\aleph}}{M!}2^{\frac{1}{2}(1-M)}\left(\frac{z}{t}\right)^{\frac{1}{2}(M+1)}J_{M+1}(2\sqrt{2tz})+ \O((z/t)^{\frac{1}{2}(M+2)}) + \O(t^{1-N}),\\
P(t,z)&=\ii^M\frac{\b{r_0^{(M)}}\ee^{-\ii\aleph}}{M!}2^{1-\frac{1}{2}M}\left(\frac{z}{t}\right)^{\frac{1}{2}M}J_M(2\sqrt{2tz}) + \O((z/t)^{\frac{1}{2}(M+1)})+\O(t^{\frac{3}{2}-N}),\\
D(t,z)&=-1+\left(\frac{|r_0^{(M)}|}{M!}\right)^22^{1-M}\left(\frac{z}{t}\right)^MJ_M(2\sqrt{2tz})^2+\O((z/t)^{\frac{1}{2}(M+1)})+\O(t^{\frac{3}{2}-N}),
\end{split}
\label{e:qPD-global-Mpos-stable}
\end{equation}
as $t\to+\infty$ for $z\ge 0$ with $z=o(t)$.  Here, the moment $r_0^{(M)}$ is given in Definition~\ref{def:moments}, $\aleph$ is given in Definition~\ref{def:aleph}, and $J_n(\cdot)$ denotes the Bessel function of the first kind of order $n$ \cite[Section 10.2]{dlmf}.
\label{thm:global-Mpos-stable}
\end{theorem}
The leading terms are easily seen to be consistent with the conservation law $|P|^2+D^2=1$ and, via the identity \cite[Eqn.\@ 10.6.2]{dlmf}, the Maxwell equation $q_z=-P$.
Analogues of Corollaries~\ref{cor:medium-edge-M0-stable-and-unstable} and \ref{cor:medium-bulk-M0-stable-and-unstable} are easily extracted from this result by expansion of the Bessel functions for small and large positive $x$, respectively.  Indeed, from \cite[Eqn.\@ 10.2.2]{dlmf} we get
\begin{equation}
J_{n}(2x)=\frac{x^{n}}{n!}(1+\O(x^2)),\quad x\to 0.
\label{e:Bessel-origin}
\end{equation}
This implies that in the regime that $t\to+\infty$ while $z=o(t^{-1})$ the optical field $q(t,z)$ is proportional to $z^{M+1}$, a result consistent with Corollary~\ref{cor:medium-edge-M0-stable-and-unstable} applying for $M=0$.
Likewise, from \cite[Eqn.\@ 10.17.3]{dlmf} we get
\begin{equation}
J_{n}(2x)=\frac{1}{\sqrt{\pi x}}\left(\cos\left(2x-\frac{1}{2}n\pi-\frac{1}{4}\pi\right)+\O(x^{-1})\right),\quad x\to +\infty.
\label{e:Bessel-expand}
\end{equation}
Applying this formula in the situation that $z>0$ is fixed shows that $q(t,z)=\O(t^{-1-\frac{1}{2}M})$ as $t\to+\infty$, so $L^1(\Real)$ integrability of the optical pulse is recovered for each $z\ge 0$ under the nongeneric condition that $r_0=0$.  Therefore, in some sense, a nongeneric incident pulse $q_0$ produces a smaller optical field within the active medium than does a generic pulse; since a generic incident pulse returns an initially-stable medium to its stable state, it is not surprising that the same occurs for the weaker pulse since $D(t,z)\to -1$ as $t\to+\infty$.

Passing now to the case of an initially-unstable medium, it would be very interesting to determine if a nongeneric incident pulse is strong enough to stimulate the decay of an unstable active medium to its stable state.  Indeed, the trivial incident pulse $q_0(t)\equiv 0$ satisfies the hypotheses of Theorem~\ref{thm:reconstruction} and clearly the corresponding unique causal solution yields $D(t,z)\equiv 1$ for all $t\ge 0$ and $z\ge 0$ if $D_-=1$, so at least one (trivial) pulse with $r_0=0$ fails to stimulate the decay of an initially-unstable medium!  Moreover, for an initially-unstable medium a result qualitatively different from that given in Corollary~\ref{cor:medium-bulk-M0-stable-and-unstable} might be expected if $r_0=0$, since one can verify using \cite[Eqn.\@ 5.11.9]{dlmf} that the amplitude $A>0$ defined in \eqref{e:qPD-medium-bulk-M0-stable-and-unstable} is proportional to $\sqrt{\ln(|r_0|^{-1})}$ for small $|r_0|$ when $D_-=1$ and hence blows up as $r_0\to 0$.  We can give a version of Theorem~\ref{thm:global-Mpos-stable} applicable to an initially-unstable medium but we have to restrict to the medium-edge and transition regimes.
\begin{theorem}[Medium-edge and transition regime asymptotics --- nongeneric case for an unstable medium]
Suppose that the incident pulse $q_0$ satisfies the hypotheses of Theorem~\ref{thm:reconstruction} and that $r_0=0$, so that the index $M$ of the first nonzero moment of the reflection coefficient is strictly positive.  For every integer $N\ge M+2$, the causal solution of the Cauchy problem \eqref{e:mbe} in a
unstable medium ($D_-=1$) satisfies
\begin{equation}
\begin{split}
q(t,z)&=\ii (-1)^{M+1}\frac{\b{r_0^{(M)}}\ee^{-\ii\aleph}}{M!}2^{\frac{1}{2}(1-M)}\left(\frac{z}{t}\right)^{\frac{1}{2}(M+1)}J_{M+1}(2\ii\sqrt{2tz})+\O(t^{-\frac{1}{2}(M+2)(1-\alpha)})+\O(t^{1-N}),\\
P(t,z)&=(-1)^{M+1}\frac{\b{r_0^{(M)}}\ee^{-\ii\aleph}}{M!}2^{1-\frac{1}{2}M}\left(\frac{z}{t}\right)^{\frac{1}{2}M}J_M(2\ii\sqrt{2tz}) + \O(t^{-\frac{1}{2}(M+1)(1-\alpha)}) + \O(t^{\frac{3}{2}-N}),\\
D(t,z)&=1+(-1)^{M+1}\left(\frac{|r_0^{(M)}|}{M!}\right)^22^{1-M}\left(\frac{z}{t}\right)^MJ_M(2\ii\sqrt{2tz})^2 + \O(t^{-\frac{1}{2}(M+1)(1-\alpha)})+\O(t^{\frac{3}{2}-N}),
\end{split}
\label{e:qPD-edge-Mpos-unstable}
\end{equation}
as $t\to+\infty$ with $z$ related to $t$ by \eqref{e:asymptotics-curve} with $\alpha\le -1$.
\label{thm:edge-Mpos-unstable}
\end{theorem}
Again, the leading terms are consistent with $|P|^2+D^2=1$ (here we may use the identity $J_n(2\ii x)^2=(-1)^n|J_n(2\ii x)|^2$ for $x>0$) and $q_z=-P$.  An analogue of Corollary~\ref{cor:medium-edge-M0-stable-and-unstable} is available for the medium-edge regime, by means of the formula \eqref{e:Bessel-origin} which is valid for complex $x$.  However unlike for the stable case, there is no analogue of Corollary~\ref{cor:medium-bulk-M0-stable-and-unstable} since Theorem~\ref{thm:edge-Mpos-unstable} is not valid in the medium-bulk regime.  This is more than a mere technical difficulty, since the Bessel functions grow exponentially along the imaginary axis, implying that the formul\ae\ \eqref{e:qPD-edge-Mpos-unstable} must become invalid as the similarity variable $x=\sqrt{2tz}$ becomes large.  The dynamics in the latter regime would resolve the interesting question of whether the medium decays to the stable state as $t\to+\infty$ for fixed $z>0$, but their description remains out of reach by the methods used in this paper.

The proofs of these results are somewhat different for an initially-unstable medium than for an initially-stable one.  The results concerning an initially-stable medium will be proved in Section~\ref{s:stable-positive}, and the modifications necessary to handle the initially-unstable case will be described in Section~\ref{s:unstable-positive}.

\subsection{Numerical verification}
\label{s:numerical-verification}
We compared the numerical solution of the Cauchy problem \eqref{e:mbe} with the explicit leading terms in the approximate formul\ae\  in order to verify and illustrate our analytic results.  We used a numerical method that enforces the causality of the solution, which is briefly described along with the numerical method used to construct the Painlev\'e-III solutions in
Appendix~\ref{s:numerics}.  We show the results for several choices of the incident pulse $q_0(\cdot)$ as given along with the auxiliary data $M$, $r_0^{(M)}$, $\omega$ (for $M=0$ only), and $\aleph$ in Table~\ref{tab:ic-numerics}.
\begin{table}[h]
\caption{Three incident pulses for numerical experiments and associated data.}
\begin{tabular}{|l|l|l|l|l|l|}
\hline
Pulse & $\vphantom{\Big|}q_0(t)$ & $M$ & $r_0^{(M)}$ & $\omega$ & $\aleph$\\
\hline\hline
(a) & $\vphantom{\Bigg|}\displaystyle 0.5\ee^{-\frac{1}{10t}-\frac{1}{10(3.5-t)}}\chi_{[0,3.5]}(t)$  & $0$ & $4.7157$ & $2.7418$ & $0$\\
\hline
(b) & $\vphantom{\Bigg|}\displaystyle 0.5\ee^{\ii t-\frac{1}{10t}-\frac{1}{10(3.5-t)}}\chi_{[0,3.5]}(t)$  & $0$ &
$-0.50723 - 0.47903\ii$ & $-1.03564$ & $1.26854$\\
\hline
(c) & $\vphantom{\Bigg|}\displaystyle 0.5\ee^{-\frac{1}{10t}-\frac{1}{10(6-t)}}\tanh(t-3)\chi_{[0,6]}(t)$ &
$1$ & $4.26238\ii$ & N/A & $0$\\
\hline
(d) & $\vphantom{\Bigg|}\displaystyle \frac{0.7\ee^{3\ii t - \frac{1}{10t}}}{(t - 10)^4 + \ii}\chi_{\{t\ge 0\}}(t)$ &  $0$ & $-0.076833 - 0.269224\ii$ & $-2.56388$ & $0.691048$\\
\hline
\end{tabular}
\label{tab:ic-numerics}
\end{table}
For making a strong comparison with our analytical results, an important property of the incident pulses  that is clear from Table~\ref{tab:ic-numerics} is that the value of $r_0^{(M)}$ is not too small.
The four pulses are plotted in Figure~\ref{fig:ic-numerics}.
\begin{figure}[h]
\begin{center}
\includegraphics[width=0.45\linewidth]{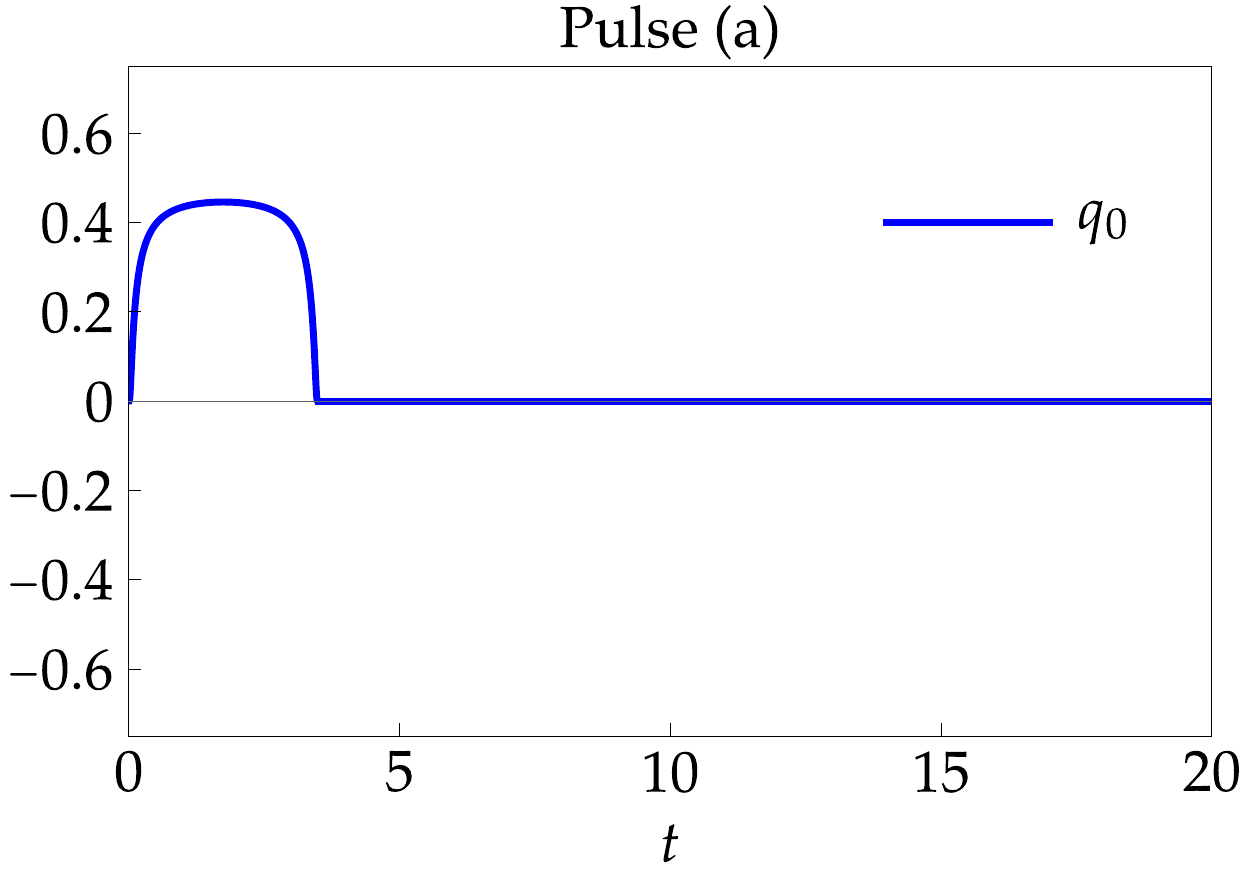}\hfill%
\includegraphics[width=0.45\linewidth]{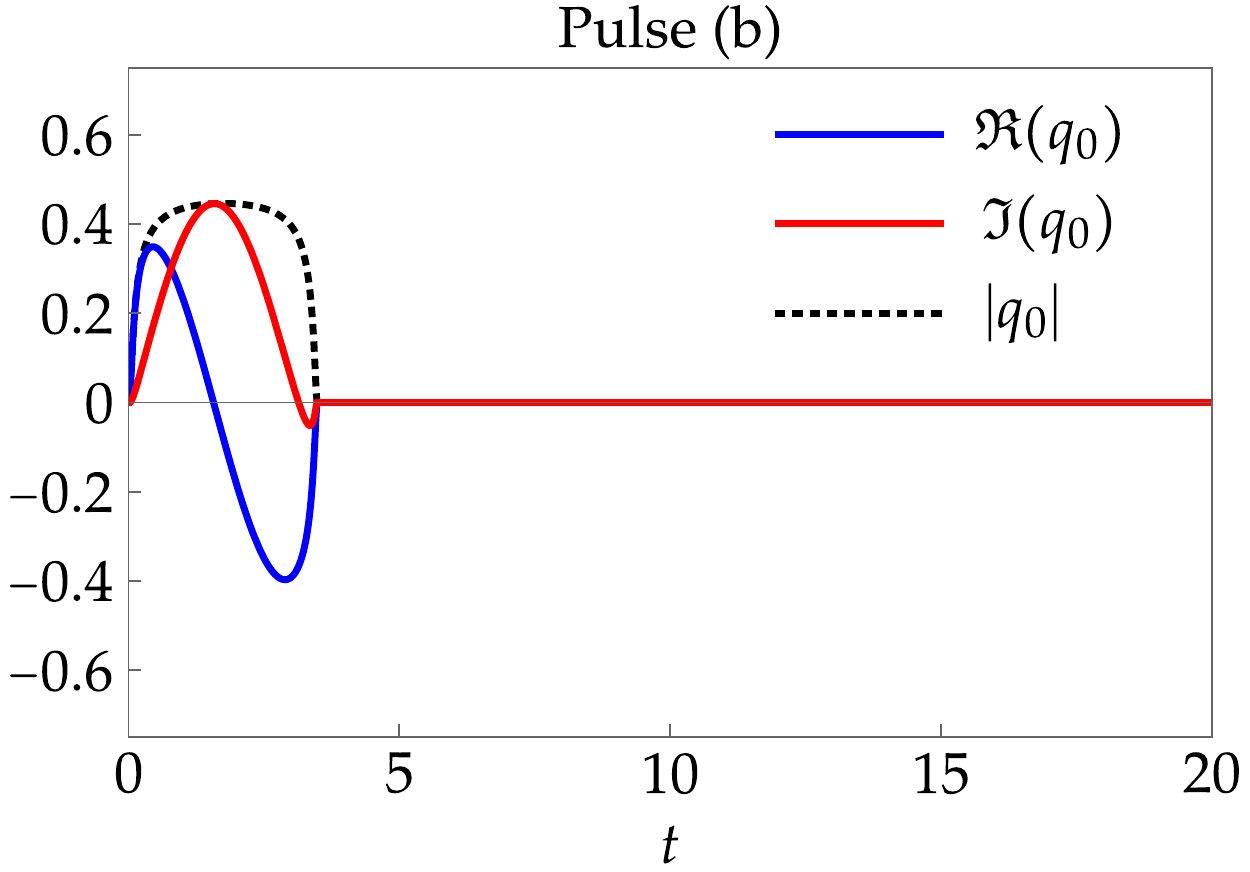}\\
\includegraphics[width=0.45\linewidth]{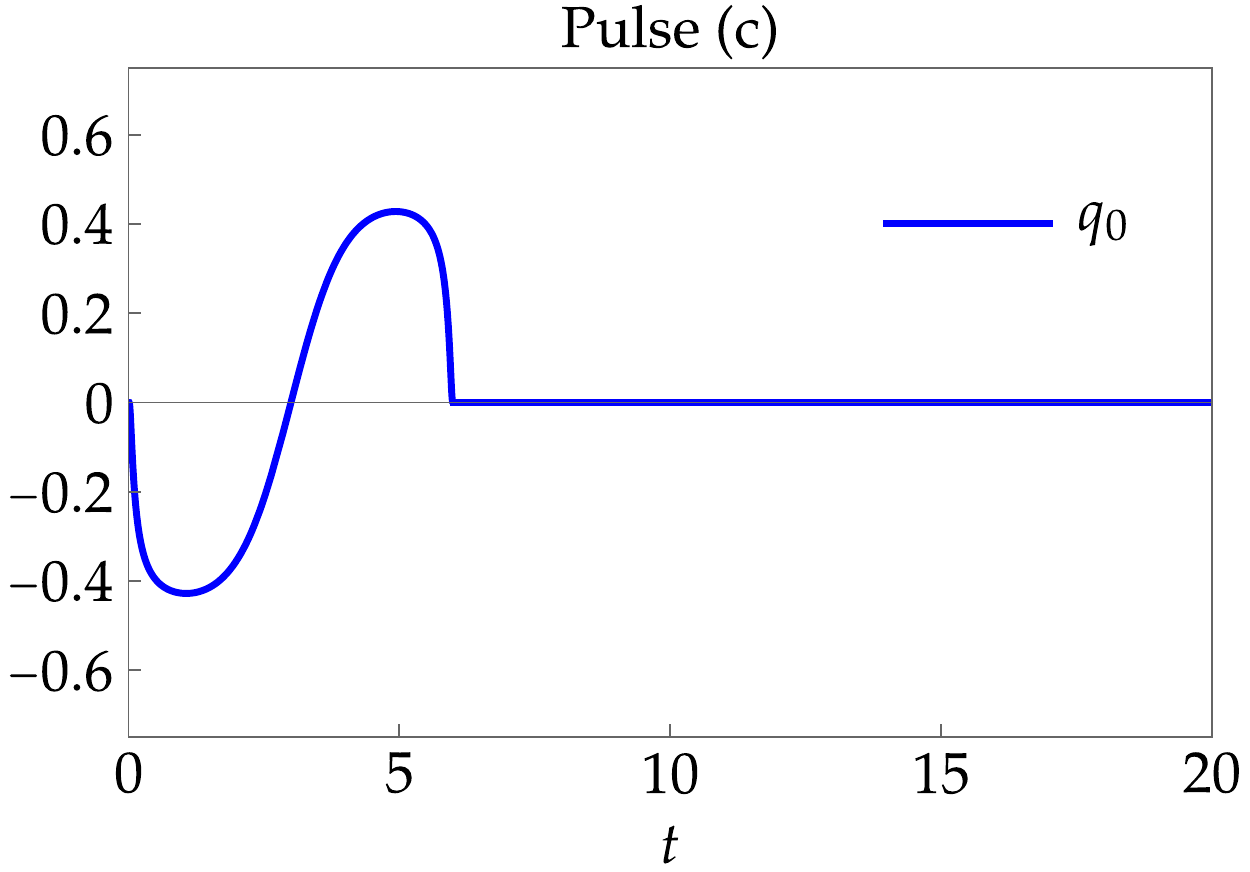}\hfill
\includegraphics[width=0.45\linewidth]{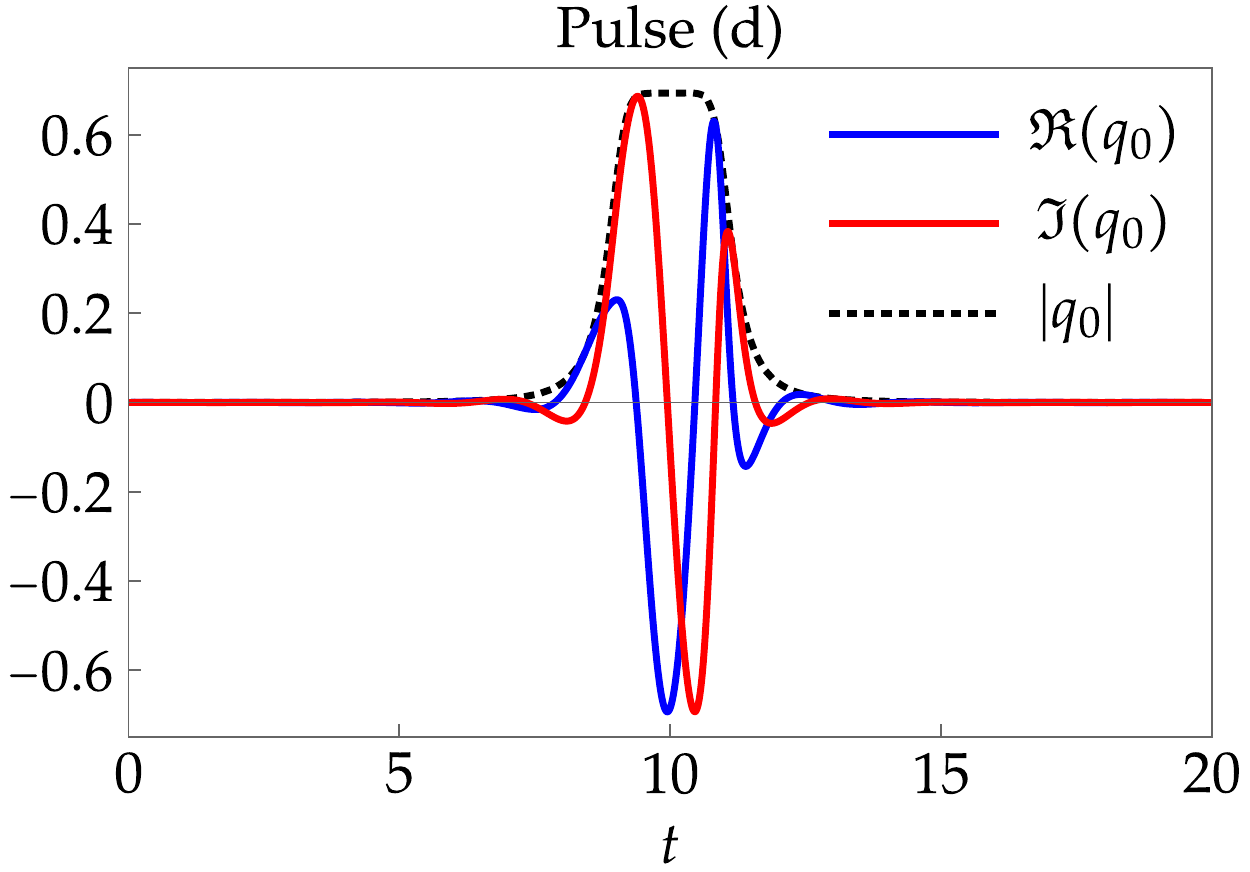}
\end{center}
\caption{The four incident pulses from Table~\ref{tab:ic-numerics}.}
\label{fig:ic-numerics}
\end{figure}

\subsubsection{Generic pulses}
Pulses (a) and (b) are consistent with the assumptions of Theorem~\ref{thm:reconstruction}, and they are generic, i.e., $r_0\neq 0$ and hence the index of the first nonzero moment is $M=0$.
Since pulse (a) is real-valued, the explicit formula \eqref{e:r0-real} can be used to compute the nonzero value of $r_0$ indicated in Table~\ref{tab:ic-numerics}.  For the same reason we obtain $\aleph=0$ for this pulse (see Remark~\ref{rem:aleph}).  Numerical integration of the Zakharov-Shabat problem was used to compute the nonzero value of $r_0$ indicated in Table~\ref{tab:ic-numerics} for the complex pulse (b).
Both pulses (a) and (b) are actually infinitely continuously differentiable for all $t\in\Real$ (the apparent sharp corners on the respective plots in Figure~\ref{fig:ic-numerics} actually disappear upon closer scrutiny).  To verify the hypothesis that pulse (a) does not generate any discrete spectrum or spectral singularities for the Zakharov-Shabat problem, note that this pulse satisfies the criteria of the Klaus-Shaw theory \cite{KlausS02}, allowing us to simply compute the $L^1(\mathbb{R})$-norm of $q_0$ and confirm that it lies below the threshold value of $\tfrac{1}{2}\pi$.  To do the same for pulse (b), we relied on numerical computations.

As pulses (a) and (b) are generic and satisfy the hypotheses of Theorem~\ref{thm:reconstruction}, Theorem~\ref{thm:global-M0-stable-and-unstable} applies.  We first illustrate the accuracy of this result by examining the numerical causal solutions of the Cauchy problem \eqref{e:mbe} for each pulse in the transition regime that $z$ is inversely proportional to $t$, where self-similar behavior is predicted.
In Figures~\ref{f:ica_PIII} and \ref{f:icb_PIII} the results are shown for pulses (a) and (b) respectively.
\begin{figure}[h]
    \centering
    \begin{minipage}[b]{.49\textwidth}
    \includegraphics[width = 1\textwidth]{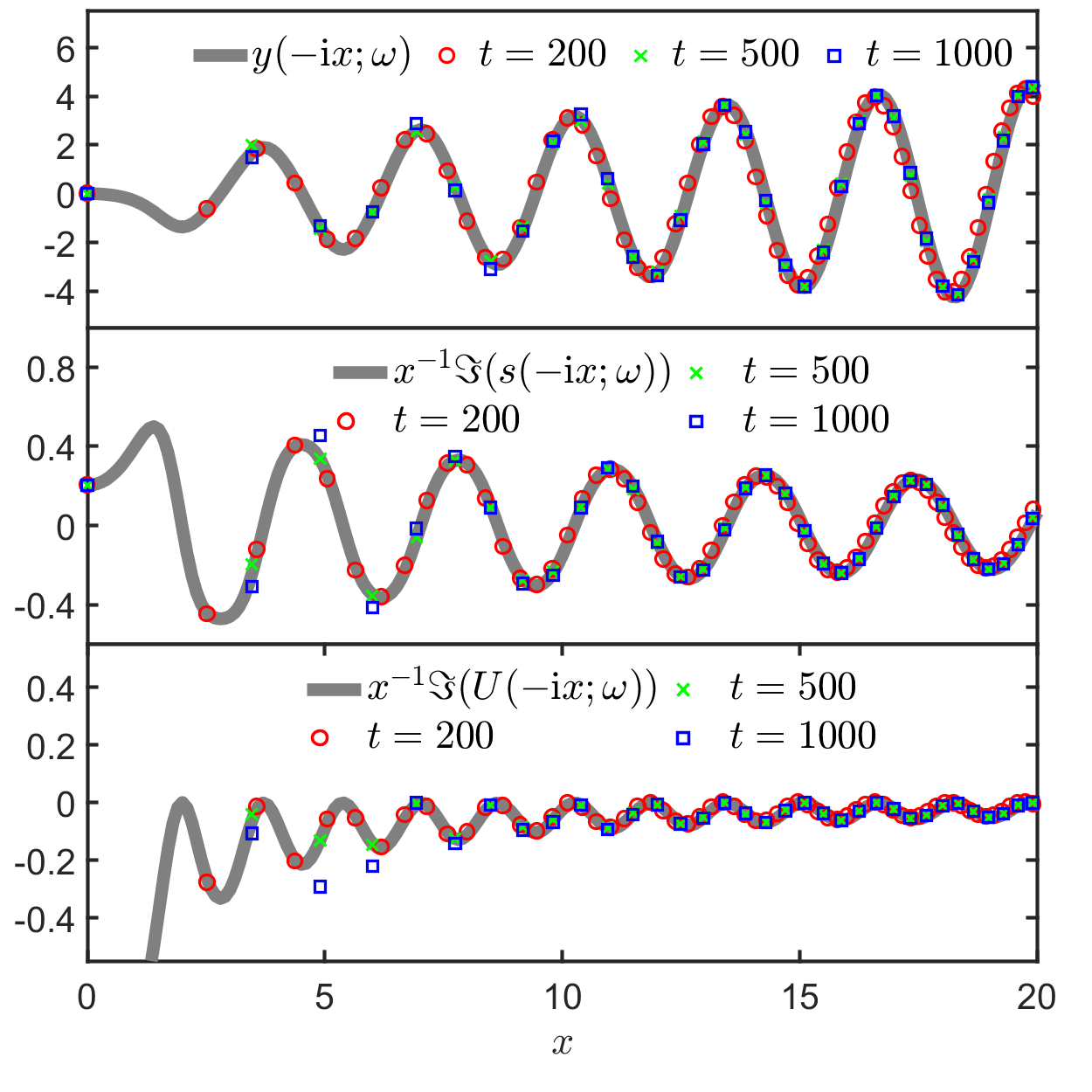}\\
    \includegraphics[width = 1\textwidth]{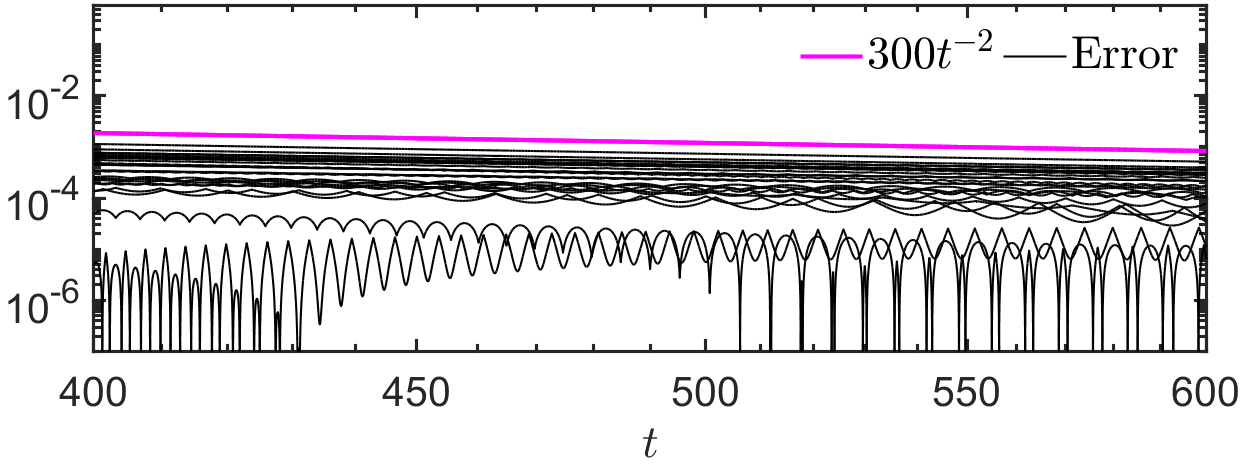}
    \end{minipage}
    \begin{minipage}[b]{.49\textwidth}
    \includegraphics[width = 1\textwidth]{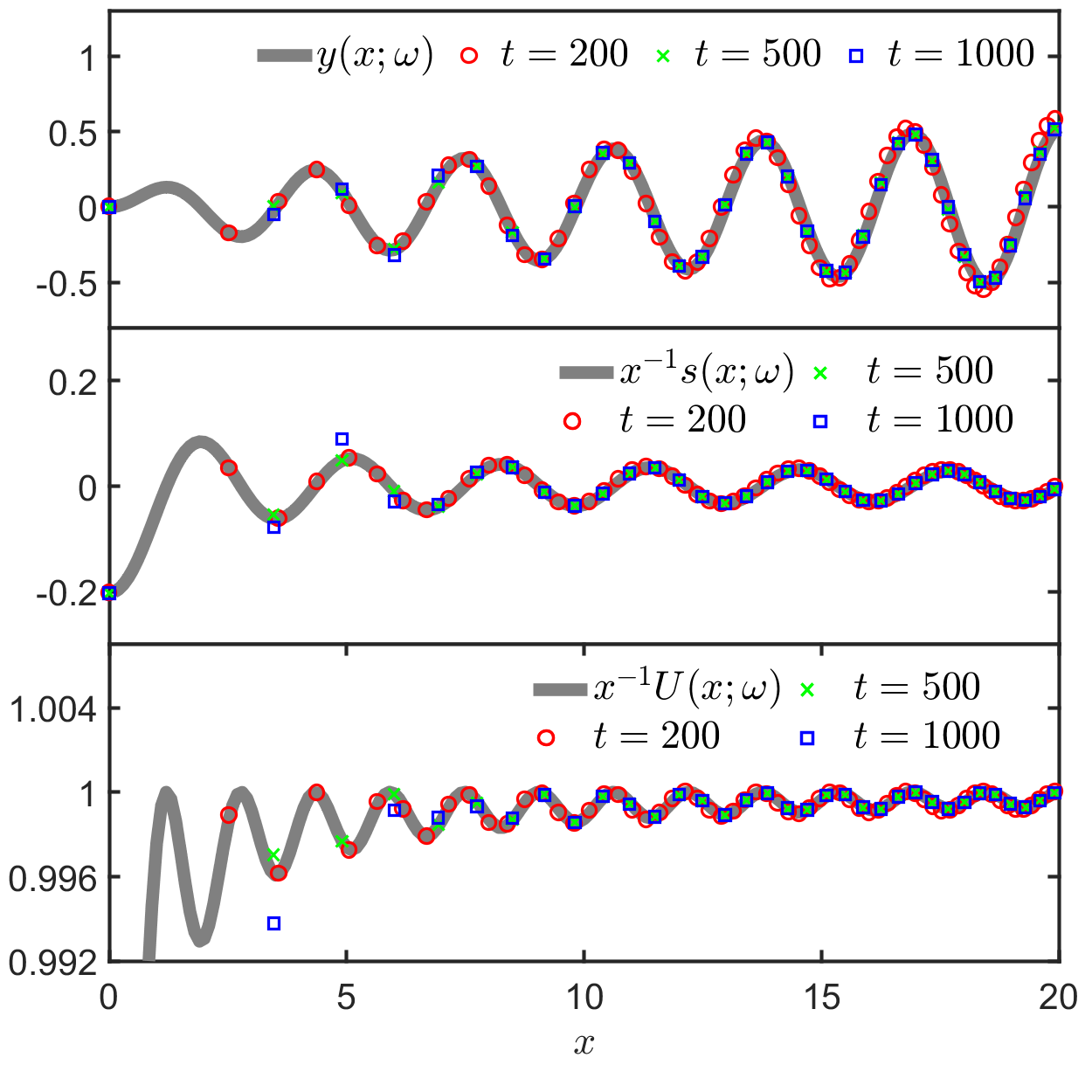}\\
    \includegraphics[width = 1\textwidth]{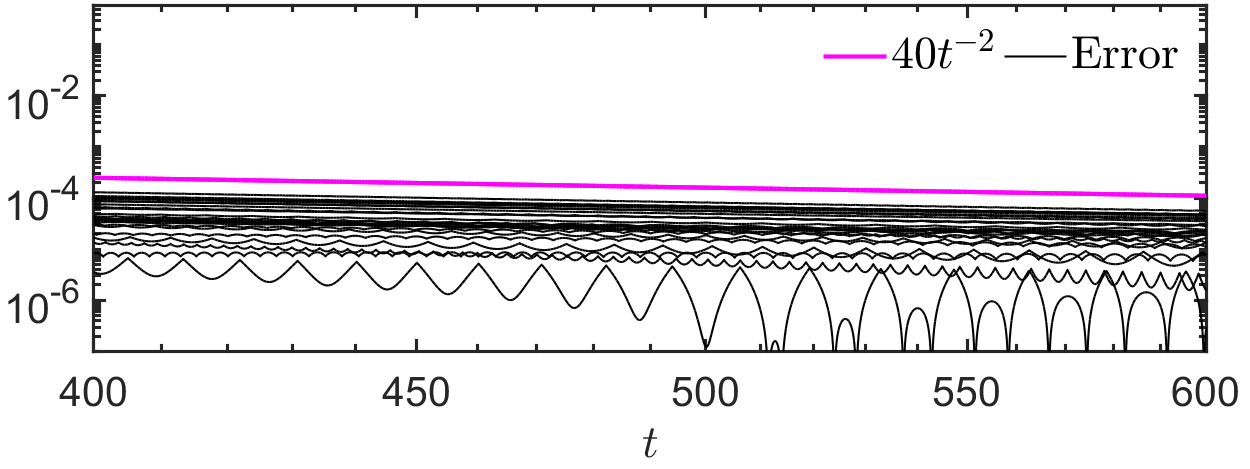}
    \end{minipage}
    \caption{
    Numerical study of incident pulse (a) in the transition regime for propagation in an initially-stable medium $D_-=-1$ (left) and an initially-unstable medium $D_-=1$ (right).
        See the main text for a full explanation.
    }
    \label{f:ica_PIII}
\end{figure}
\begin{figure}[h]
    \centering
    \begin{minipage}[b]{.49\textwidth}
    \includegraphics[width = 1\textwidth]{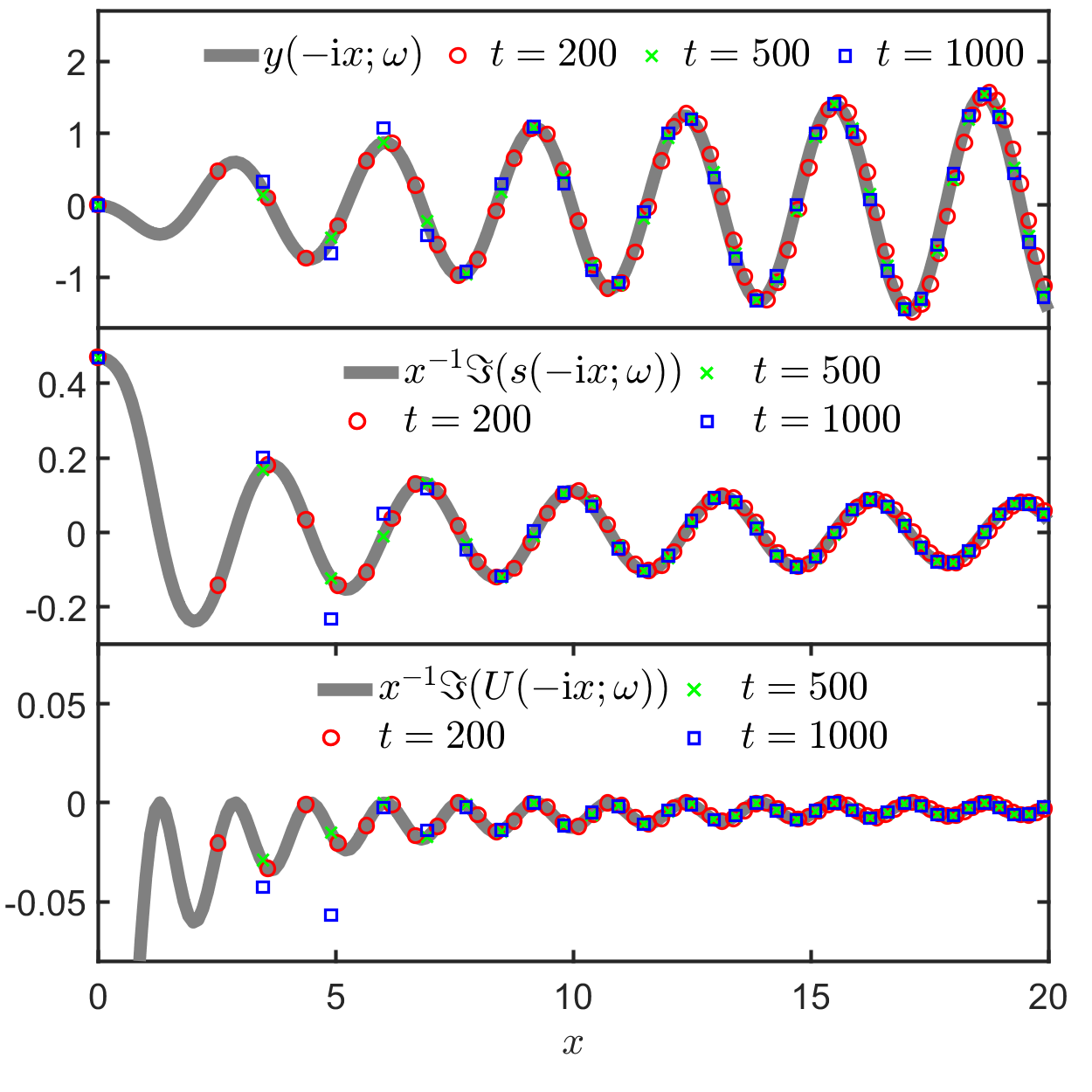}\\
        \includegraphics[width = 1\textwidth]{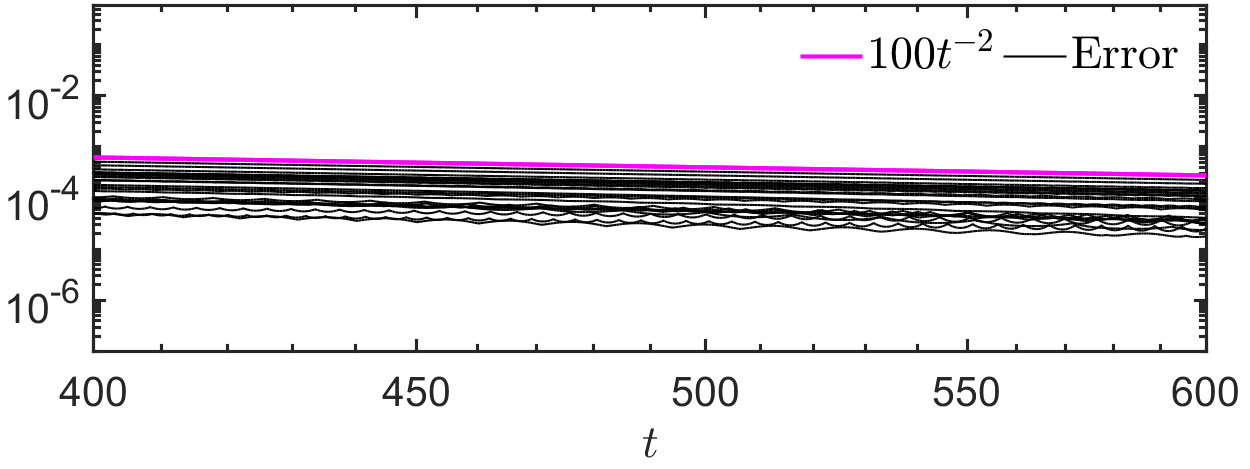}
    \end{minipage}
    \begin{minipage}[b]{.49\textwidth}
    \includegraphics[width = 1\textwidth]{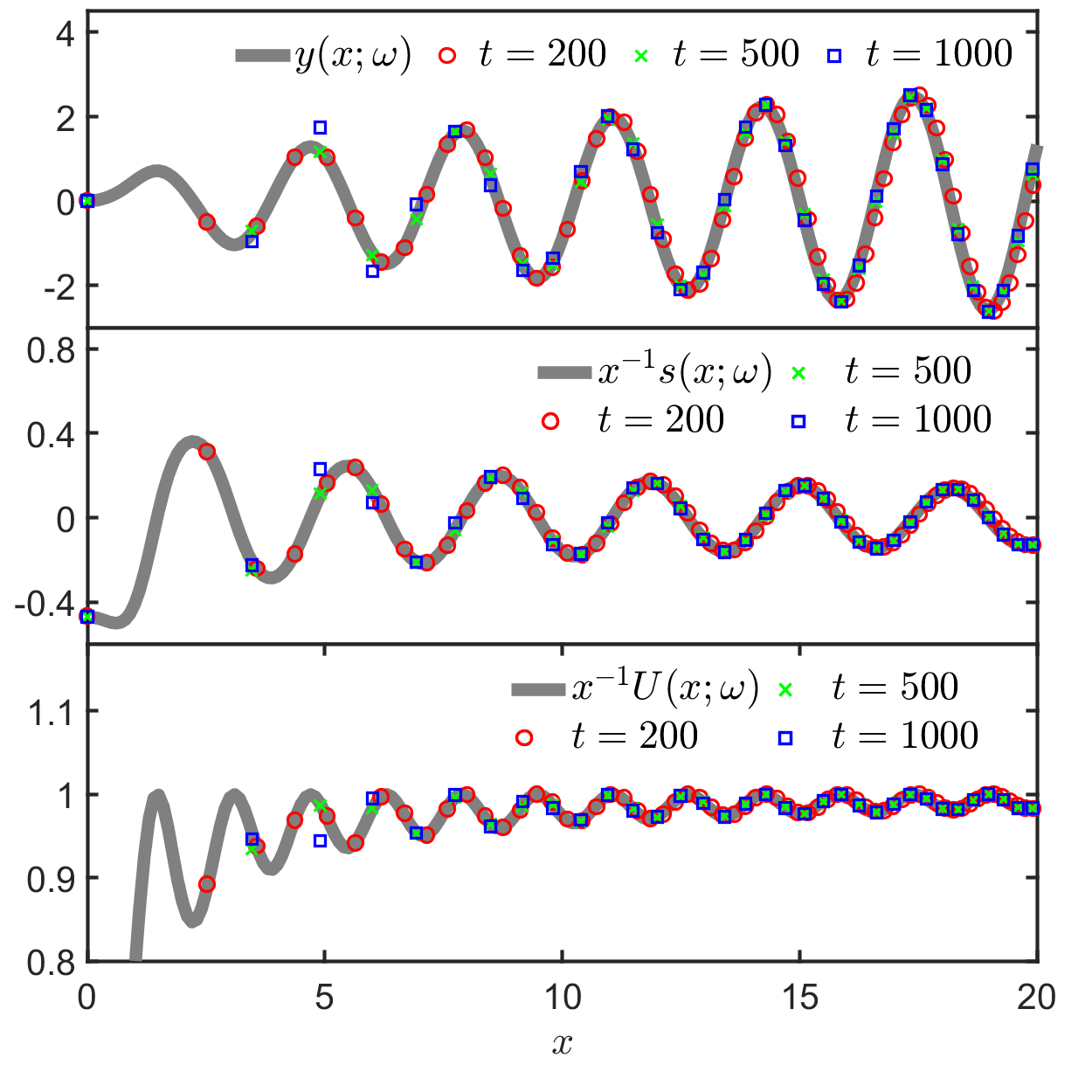}\\
        \includegraphics[width = 1\textwidth]{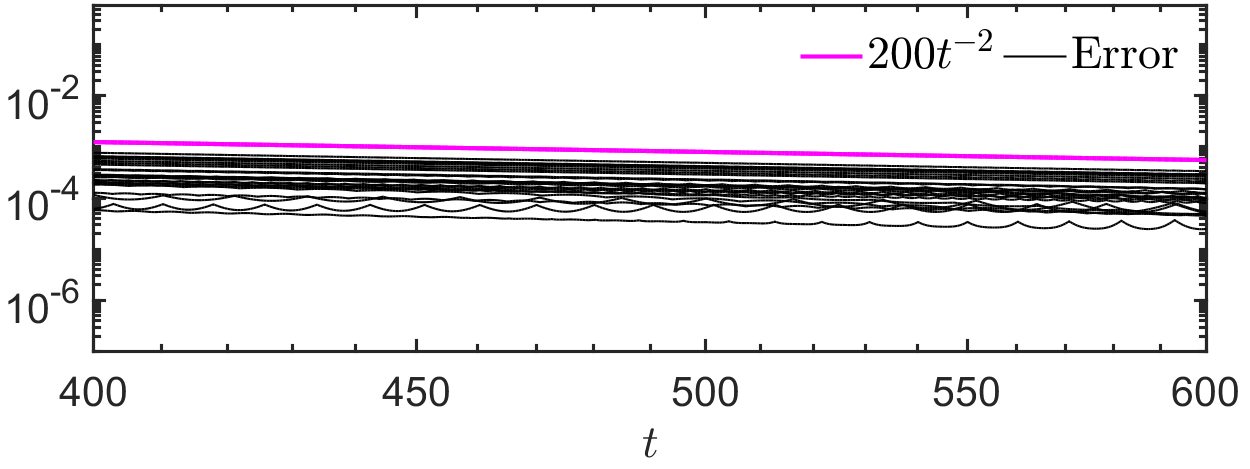}
    \end{minipage}
    \caption{As in Figure~\ref{f:ica_PIII} but for pulse (b).
    }
    \label{f:icb_PIII}
\end{figure}
In each of these figures, the left-hand (resp., right-hand) column corresponds to the case that the pulse is incident on an initially-stable (resp., initially-unstable) active medium.  In each column there are two panels as follows.
\begin{itemize}
\item
The upper panel compares the numerical results with theoretical predictions for three fixed values of $t=200, 500, 1000$ with the independent variable $z$ expressed in terms of the similarity variable $x$ by $z=z(x;t)=x^2/(2t)$, plotted as functions of $x\in [0,20]$.  Here we expect convergence of suitably renormalized versions of $q(t,z)$, $P(t,z)$, and $D(t,z)$ to limiting PIII functions whose graphs are shown with thick gray curves.  There are therefore three subplots, from top to bottom:
\begin{itemize}
\item a plot comparing numerical data for $\Re(t\ee^{\ii(\aleph+\arg(r_0))}q(t,z(x;t)))$ with its limiting function $y(-\ii x;\omega)$ in the stable-medium case or $y(x;\omega)$ in the unstable-medium case;
\item a plot comparing numerical data for $\Re(\tfrac{1}{2}\ee^{\ii(\aleph+\arg(r_0))}P(t,z(x;t)))$ with its limiting function $x^{-1}\Im(s(-\ii x;\omega))$ in the stable-medium case or $x^{-1}s(x;\omega)$ in the unstable-medium case;
\item a plot comparing numerical data for $\tfrac{1}{2}(D_--D(t,z(x;t)))$ with its limiting function $x^{-1}\Im(U(-\ii x;\omega))$ in the stable-medium case or $x^{-1}U(x;\omega)$ in the unstable-medium case.
\end{itemize}
\item The lower panel illustrates the accuracy of Theorem~\ref{thm:global-M0-stable-and-unstable} in the transition regime of fixed $tz$ and large $t>0$, in a more quantitative fashion than in the upper panel.  Here the absolute value of the difference between the numerical solution $q(t,z)$ and the relevant leading term given in Theorem~\ref{thm:global-M0-stable-and-unstable} is plotted as a function of $t$ for 25 different fixed values of $x=\sqrt{2tz}=8+\tfrac{1}{2}n$, $n=0,2,\dots,24$ on the same log-log axes.  The magenta line is a trend line for these errors and its slope indicates a decay rate proportional to $t^{-2}$ as is consistent with the prediction $\O(t^{-2})+\O(t^{1-N})$ valid for $z=\O(t^{-1})$, given that $N$ is arbitrarily large.
\end{itemize}
The accuracy on display in the upper panels of Figures~\ref{f:ica_PIII} and \ref{f:icb_PIII} is remarkable even for $t=200$, and it is clear that the accuracy improves as $t$ increases.  It might be observed that in the upper panel there is some deviation from the limiting curves for the largest value of $t=1000$; however this is occurring for smaller values of $x$ where for large $t$ there is simply insufficient numerical resolution of the self-similar solution for any accuracy to be expected.  In other words, this is a shortcoming of the numerical method, not of the asymptotic result.

We also investigated pulses (a) and (b) in the medium-bulk regime to make a comparison with Corollary~\ref{cor:medium-bulk-M0-stable-and-unstable}.  The medium-bulk regime in particular corresponds to bounded $z$ independent of $t$, so in the left-hand panel of each of Figures~\ref{f:real-stable}--\ref{f:complex-unstable} we first show a grayscale density plot of $\ln(|q(t,z)|)$ over a fixed portion of the first quadrant in the $(t,z)$-plane.
\begin{figure}[h]
    \centering
    \begin{minipage}[b]{.52\textwidth}
    \includegraphics[width = 1\textwidth]{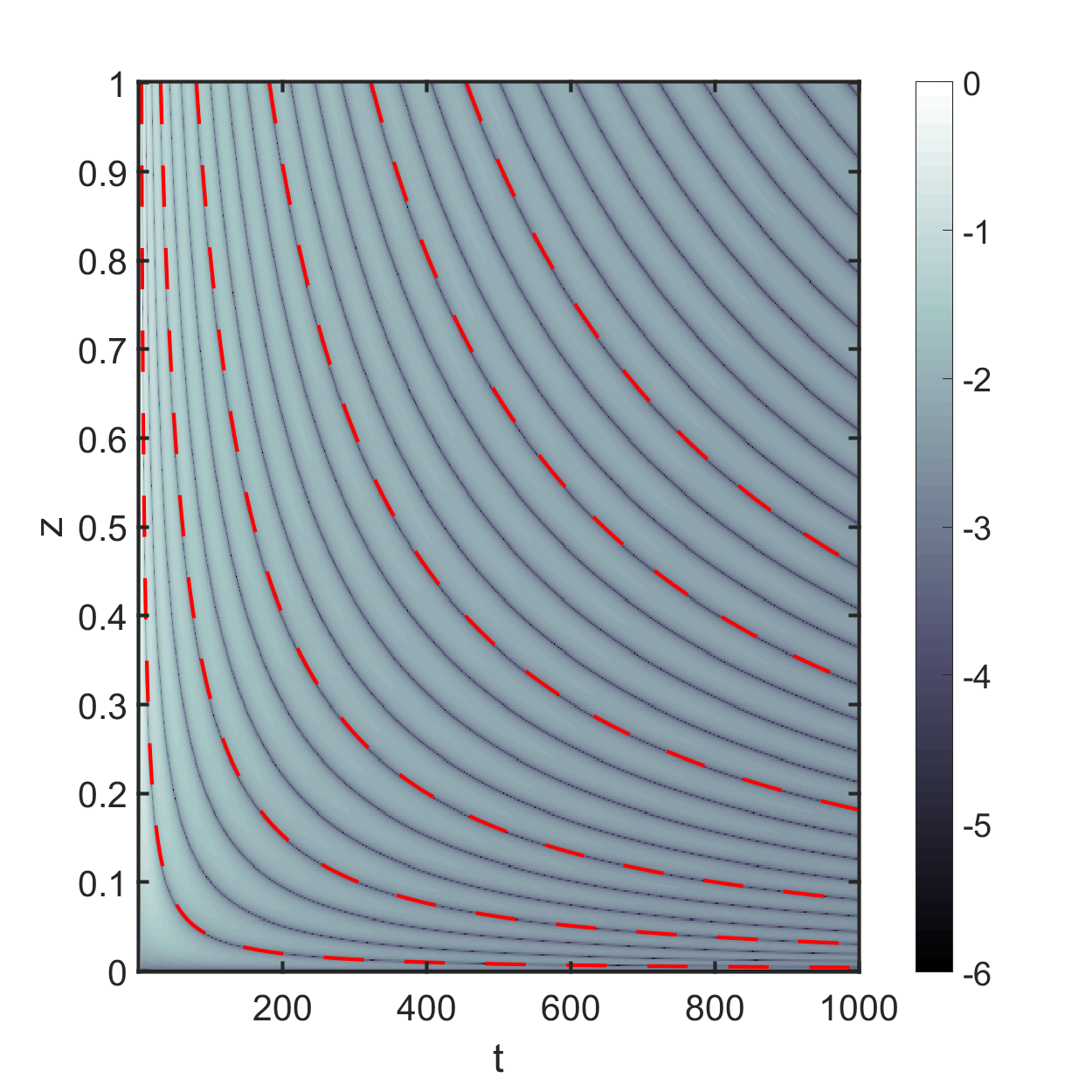}
    \end{minipage}\hspace{-1.5em}
    \begin{minipage}[b]{.5\textwidth}
    \includegraphics[width = 1\textwidth]{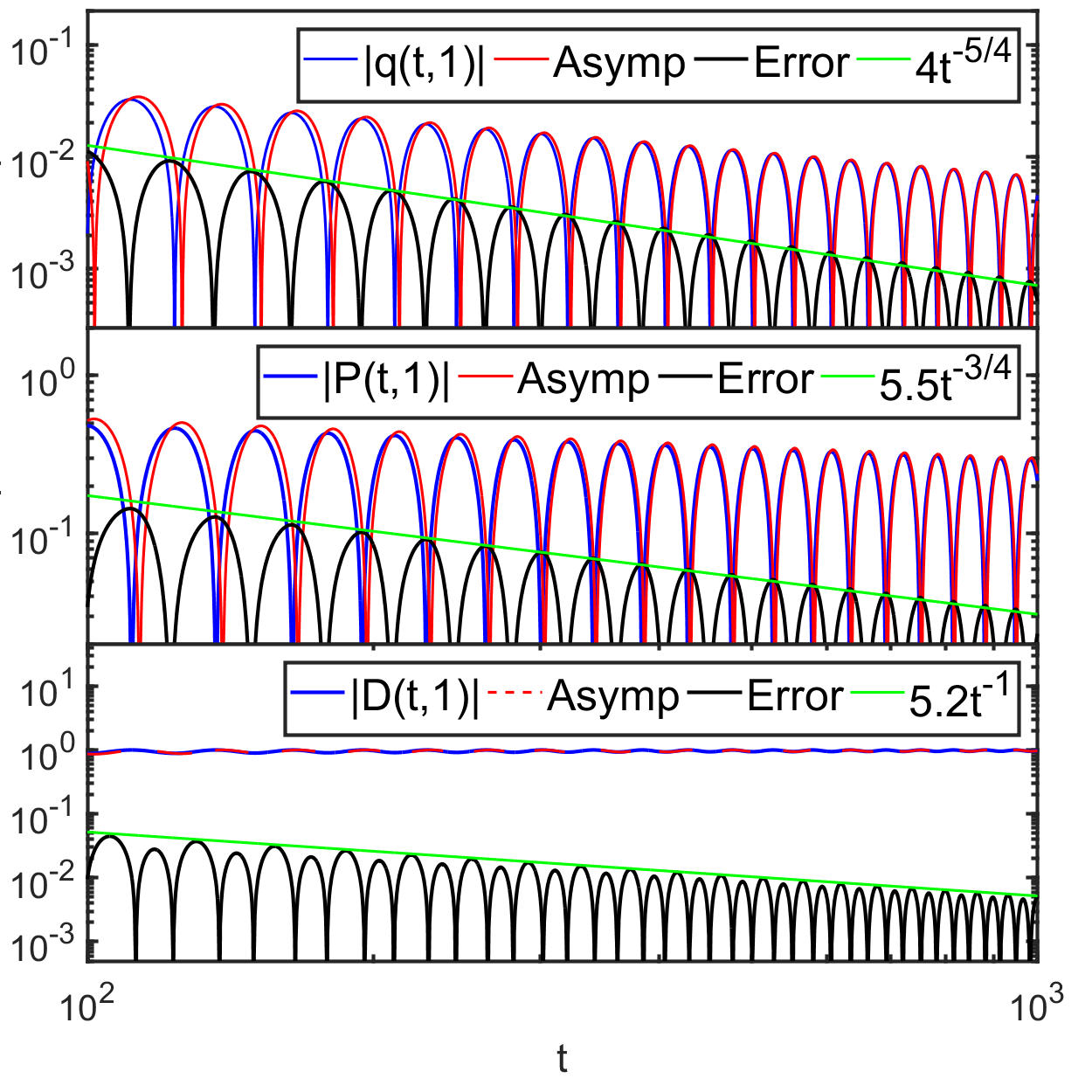}
    \end{minipage}
    \caption{
    Numerical study of incident pulse (a) in the medium-bulk regime for propagation in an initially-stable medium ($D_-=-1$).
    See the main text for a full explanation.
%
    }
    \label{f:real-stable}
\end{figure}
\begin{figure}[h]
    \centering
    \begin{minipage}[b]{.52\textwidth}
    \includegraphics[width = 1\textwidth]{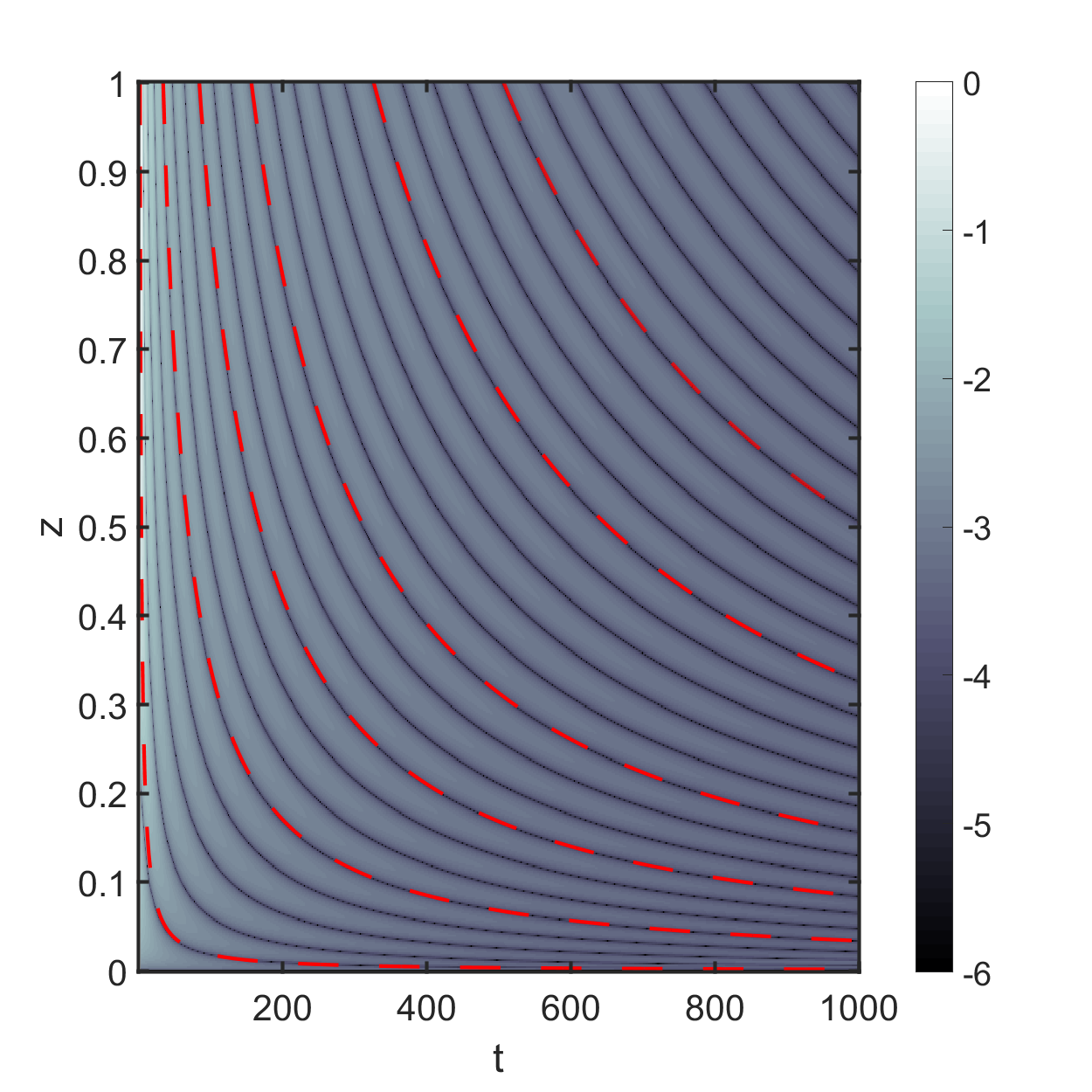}
    \end{minipage}\hspace{-1.5em}
    \begin{minipage}[b]{.5\textwidth}
    \includegraphics[width = 1\textwidth]{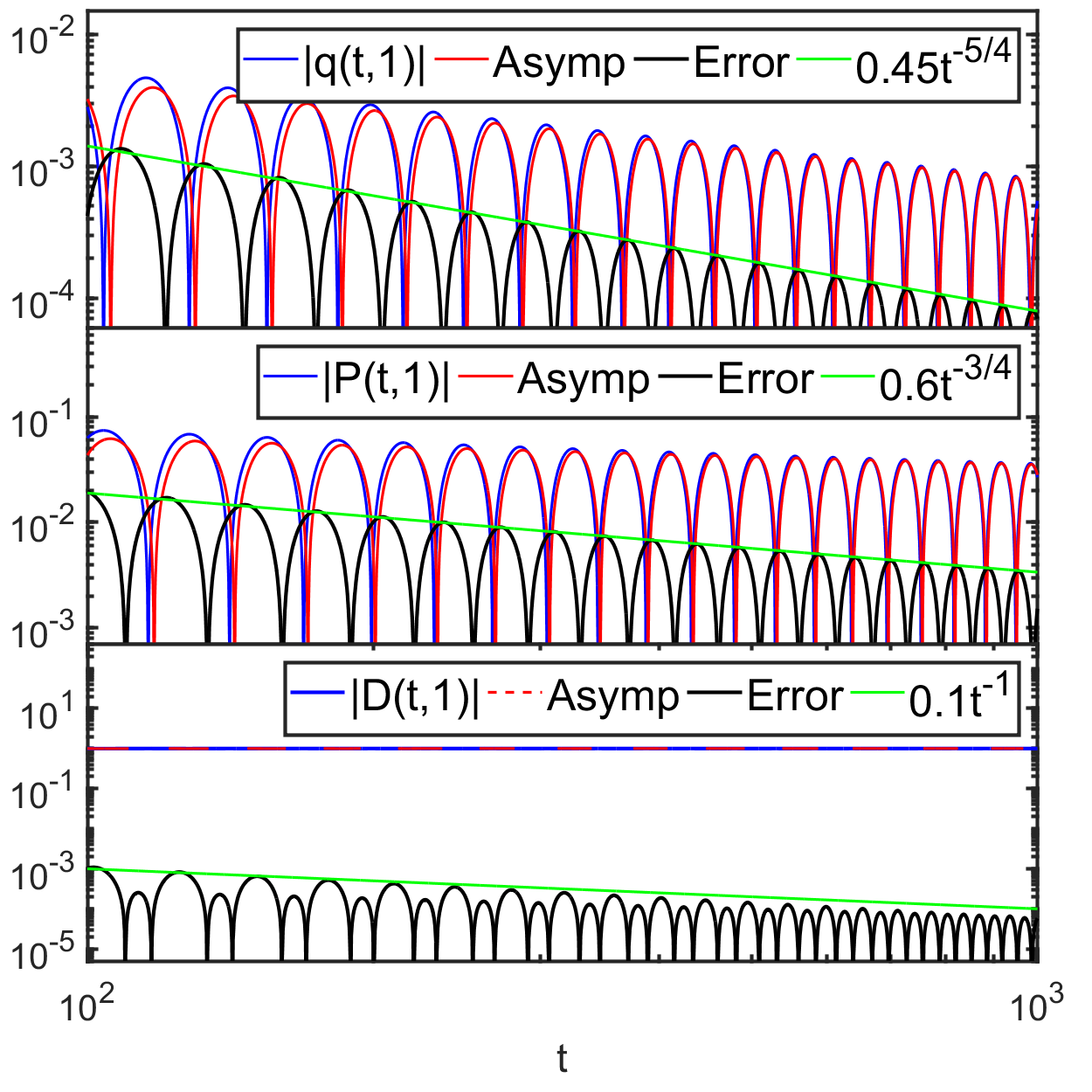}
    \end{minipage}
    \caption{
        As in Figure~\ref{f:real-stable} but for propagation in an initially-unstable medium ($D_-=1$).
%
    }
    \label{f:real-unstable}
\end{figure}
\begin{figure}[h]
    \centering
    \begin{minipage}[b]{.52\textwidth}
    \includegraphics[width = 1\textwidth]{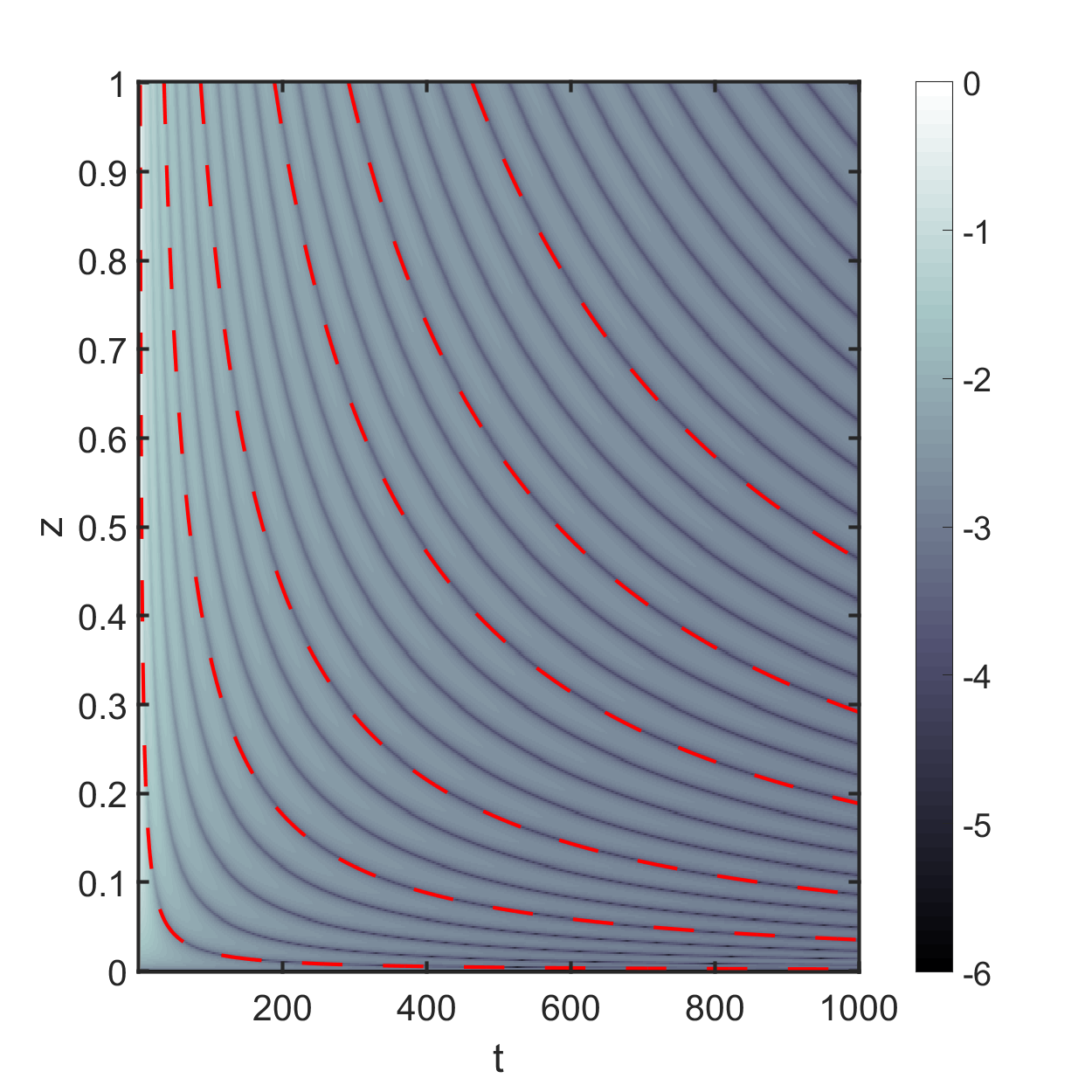}
    \end{minipage}\hspace{-1.5em}
    \begin{minipage}[b]{.5\textwidth}
    \includegraphics[width = 1\textwidth]{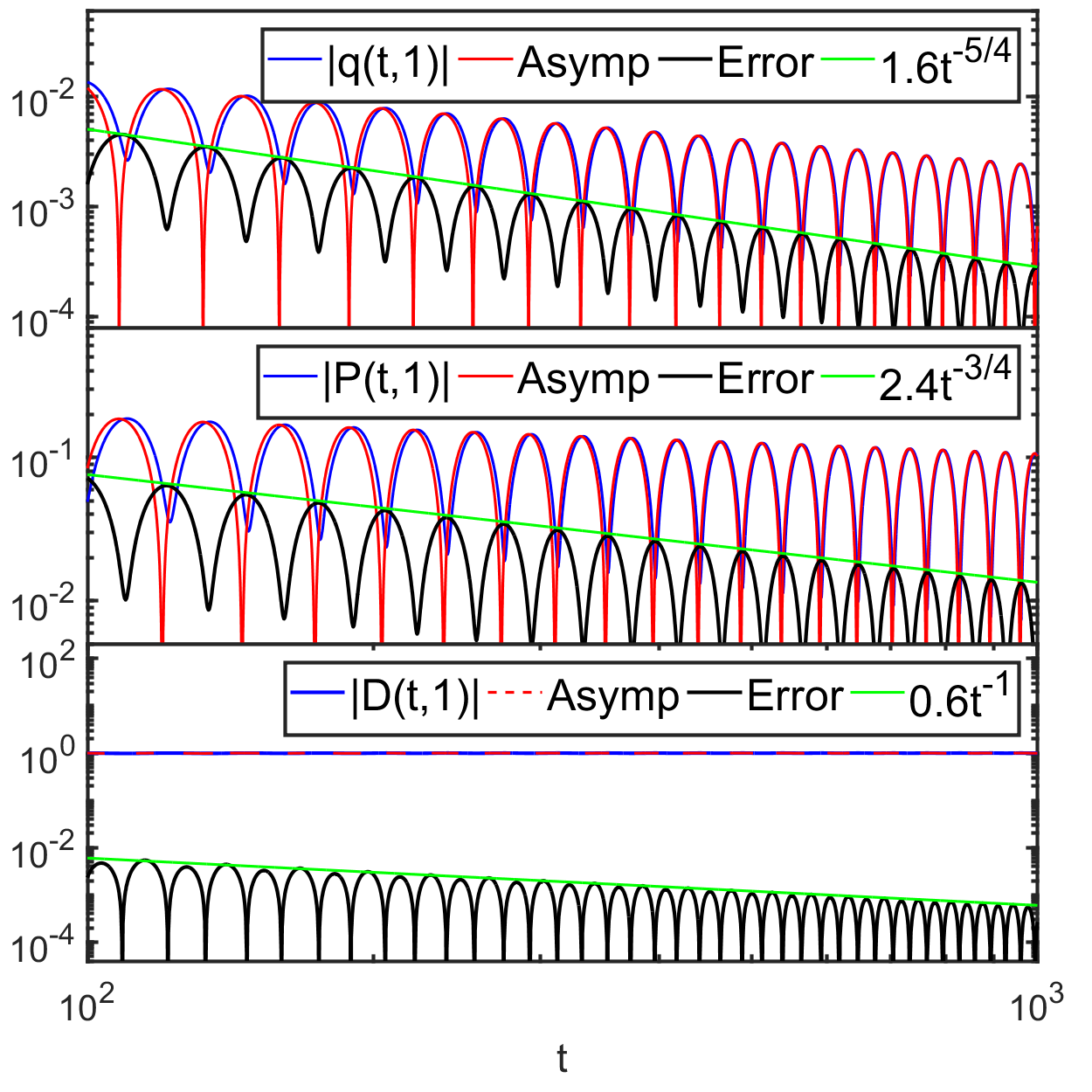}
    \end{minipage}
    \caption{
        Numerical study of incident pulse (b) in the medium-bulk regime for propagation in an initially-stable medium ($D_-=-1$).
            See the main text for a full explanation.
    }
    \label{f:complex-stable}
\end{figure}
\begin{figure}[h]
    \centering
    \begin{minipage}[b]{.52\textwidth}
    \includegraphics[width = 1\textwidth]{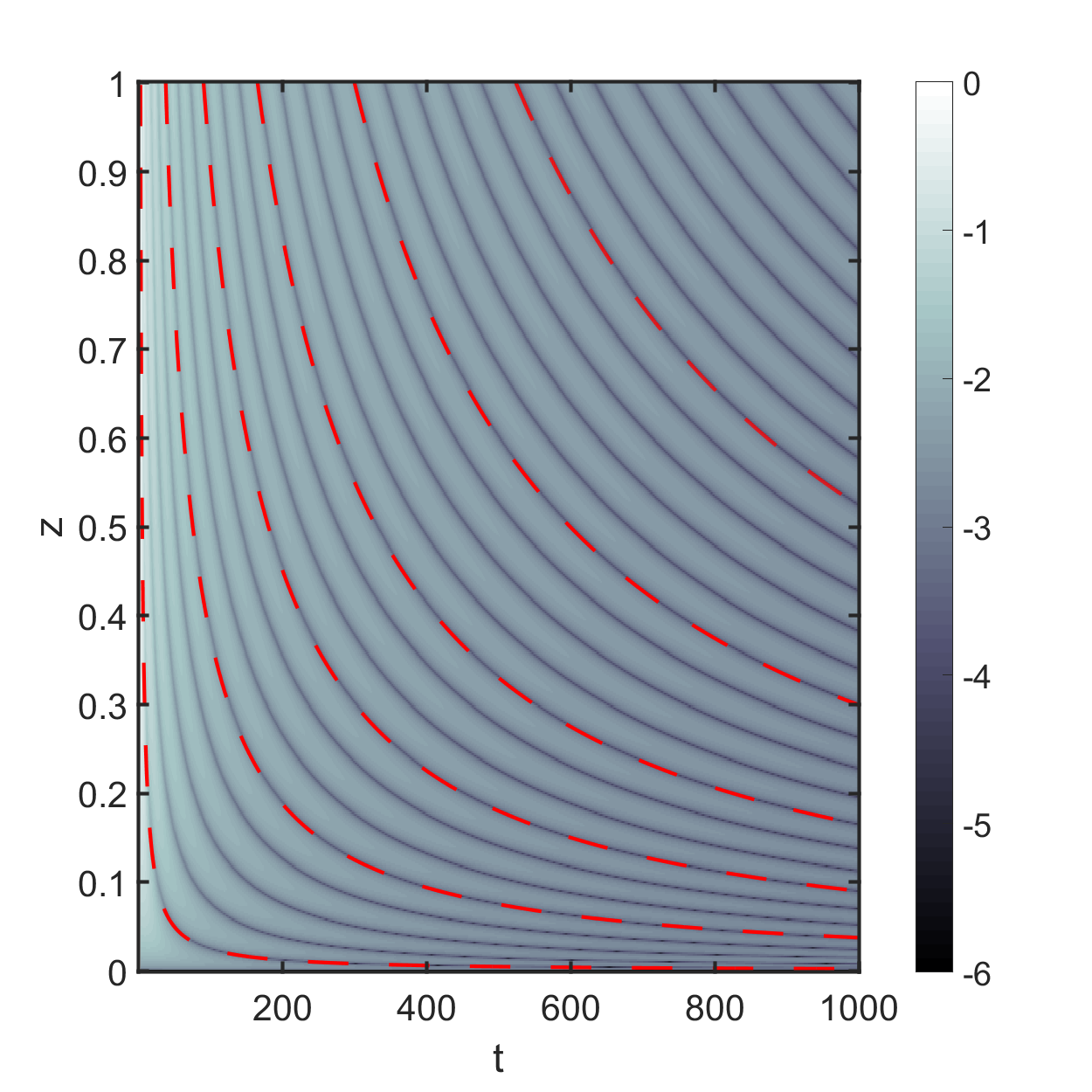}
    \end{minipage}\hspace{-1.5em}
    \begin{minipage}[b]{.5\textwidth}
    \includegraphics[width = 1\textwidth]{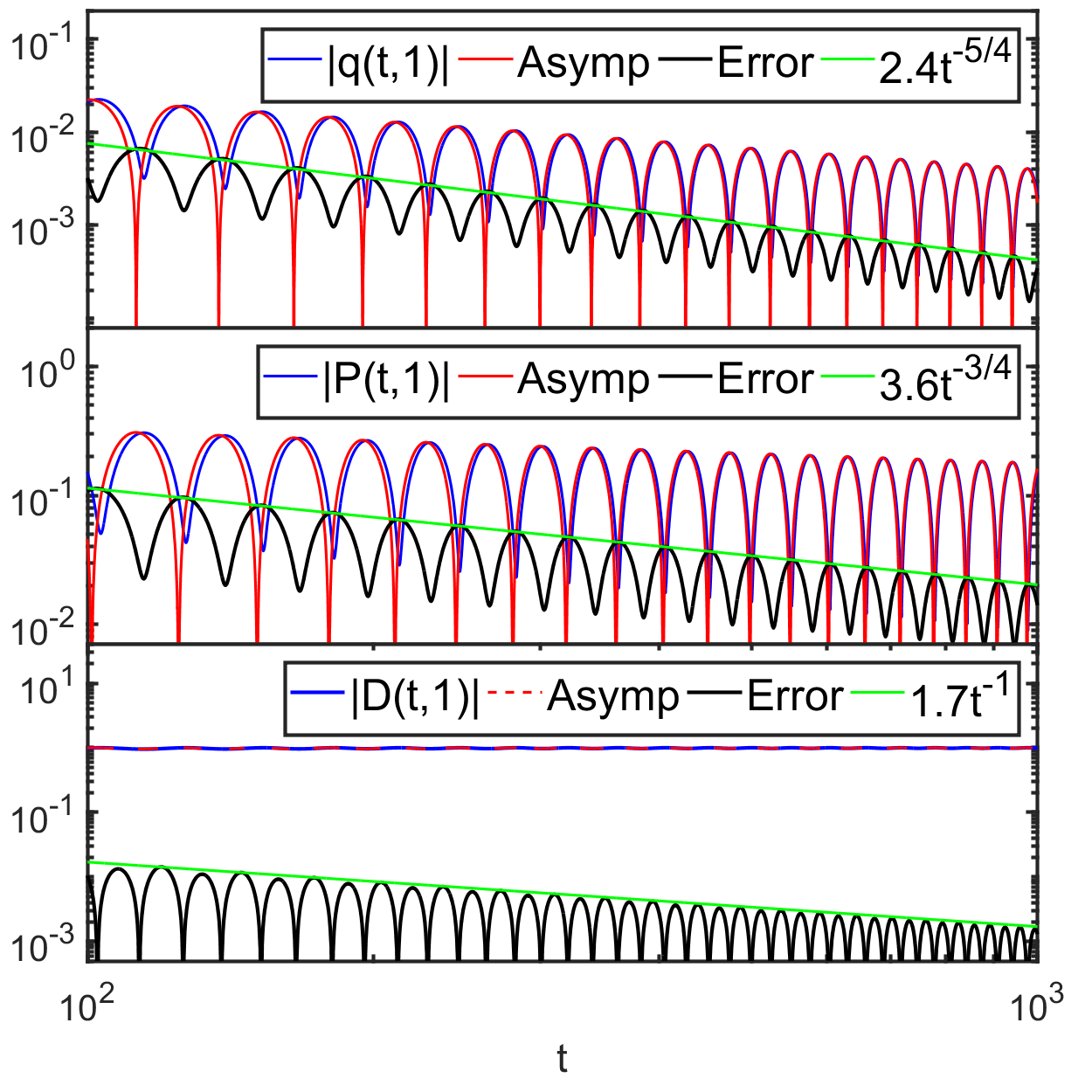}
    \end{minipage}
    \caption{
    As in Figure~\ref{f:complex-stable}, but
        for propagation in an initially-unstable medium ($D_-=1$).
    }
    \label{f:complex-unstable}
\end{figure}
Near the dark curves, the numerical value of $|q(t,z)|$ is very small (note that for incident pulse (a) the optical field $q(t,z)$ is real-valued and hence vanishes along curves while for incident pulse (b) the optical field is complex-valued and has only approximate zeros).  Superimposed with red dashed curves are selected hyperbol\ae\ $\sqrt{2tz}=x$ corresponding to exact roots of the relevant approximating function from Theorem~\ref{thm:global-M0-stable-and-unstable}.  These curves would be expected to predict approximate zeros of $|q(t,z)|$ when $t>0$ is large, but remarkably the agreement is also quite good for small $t>0$ and large $z>0$.  One interesting feature revealed by these plots is that (comparing Figure~\ref{f:real-stable} with Figure~\ref{f:real-unstable}, or comparing Figure~\ref{f:complex-stable} with Figure~\ref{f:complex-unstable}) the same incident pulse can produce an optical field $q(t,z)$ within the medium ($z>0$) of dramatically different amplitude.  Indeed, it appears that for pulse (a), $|q(t,z)|$ is an order of magnitude larger for $z>0$ in the stable medium than in the unstable medium.  On the other hand, for pulse (b) the effect is reversed and it is less dramatic.  This phenomenon can be explained by the asymptotic formul\ae\ in Corollary~\ref{cor:medium-bulk-M0-stable-and-unstable} which involve an amplitude factor $A$ (see \eqref{e:parameters-medium-bulk-M0-stable-and-unstable}) that for the same incident pulse takes different values in the stable and unstable cases.

The right-hand panel in each of Figures~\ref{f:real-stable}--\ref{f:complex-unstable} is a quantitative study of the accuracy of Corollary~\ref{cor:medium-bulk-M0-stable-and-unstable}.  We compare
numerical data for $q(t,z)$, $P(t,z)$, and $D(t,z)$ with the corresponding approximations from Corollary~\ref{cor:medium-bulk-M0-stable-and-unstable} for fixed $z=1$ and increasing $t>0$ in three similar log-log plots.  In this situation (fixed $z>0$), the error estimates for $q$, $P$, and $D$ in Corollary~\ref{cor:medium-bulk-M0-stable-and-unstable} simplify to $\O(t^{-1})$, $\O(t^{-\frac{1}{2}})$ and $\O(t^{-\frac{1}{2}})$ respectively.  The magenta line in each plot is a bound for the numerically calculated difference between the solution and the leading terms in Corollary~\ref{cor:medium-bulk-M0-stable-and-unstable}.  The slope of this line suggests that for $q(t,z)$, $P(t,z)$, the estimates in Corollary~\ref{cor:medium-bulk-M0-stable-and-unstable} may be improved by an additional factor of $t^{-\frac{1}{4}}$, while for $D(t,z)$, an additional factor of $t^{-\frac12}$ may be expected.  Looking in more detail at the error terms in \eqref{e:qPD-medium-bulk-M0-stable-and-unstable} and taking $\alpha=0$ as would be correct for the plots under consideration, we see that the first error term on the right-hand side in each cases matches the numerically-observed rate of decay, but it is dominated in each case by the second error term (and the third term may be regarded as beyond-all-orders).  This suggests that the apparently-dominant error term, which originates from the first error term on the right-hand sides of \eqref{e:qPD-generic-global-stable} and \eqref{e:qPD-generic-global-unstable} in Theorem~\ref{thm:global-M0-stable-and-unstable}, is not sharp, at least when $z>0$ is fixed.  This term originates from the very last step of our analysis, the solution of a small-norm RHP (see Section~\ref{s:small-norm-RHP-s} below).

\subsubsection{A nongeneric pulse}
Pulse (c) is also consistent with the assumptions of Theorem~\ref{thm:reconstruction}, but it is nongeneric.  Being a real-valued pulse that is odd about $t=3$, it follows from \eqref{e:r0-real} that
$r_0=0$, and by Remark~\ref{rem:aleph} we also have $\aleph=0$.  Similarly, from \eqref{e:rprime-zero} we obtain the nonzero value of $r_0^{(1)}=r'(0)$ indicated in Table~\ref{tab:ic-numerics}, and hence the index of the first nonzero moment is $M=1$.  We note that pulse (c) is Schwartz-class (again the apparent corners on the plot in Figure~\ref{fig:ic-numerics} are artifacts of insufficient resolution), and the fact that it generates no eigenvalues or spectral singularities was confirmed numerically.

Since it is nongeneric, for pulse (c) it is Theorems~\ref{thm:global-Mpos-stable} and \ref{thm:edge-Mpos-unstable} that are applicable (here in the case $M=1$), for propagation in an initially-stable medium and an initially-unstable medium respectively.  Both of these theorems characterize the solution in the transition regime, so we may begin by presenting an analogue of Figures~\ref{f:ica_PIII}--\ref{f:icb_PIII} in Figure~\ref{f:icc_PIII}.
\begin{figure}[h]
    \centering
    \begin{minipage}[b]{.49\textwidth}
    \includegraphics[width = 1\textwidth]{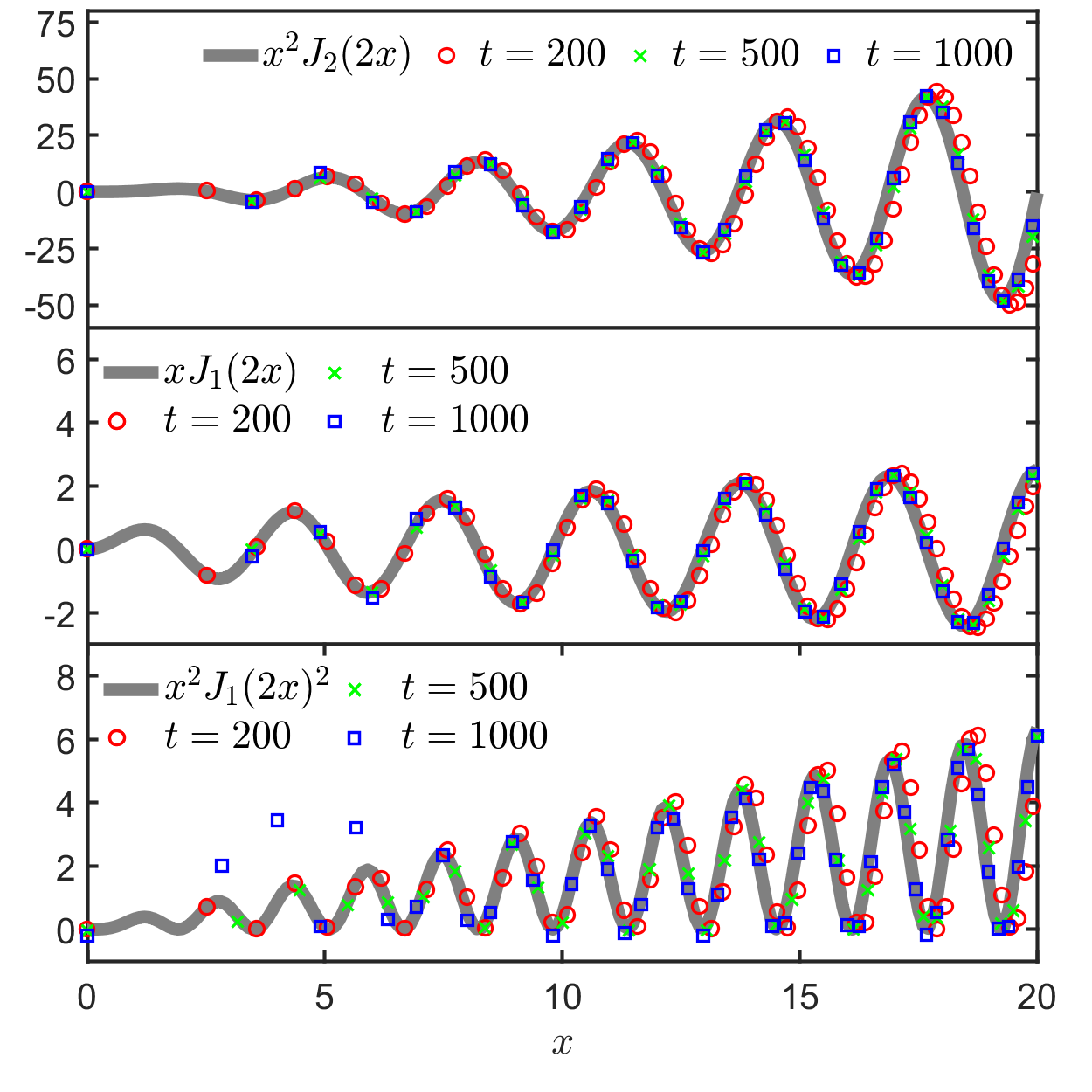}\\
            \includegraphics[width = 1\textwidth]{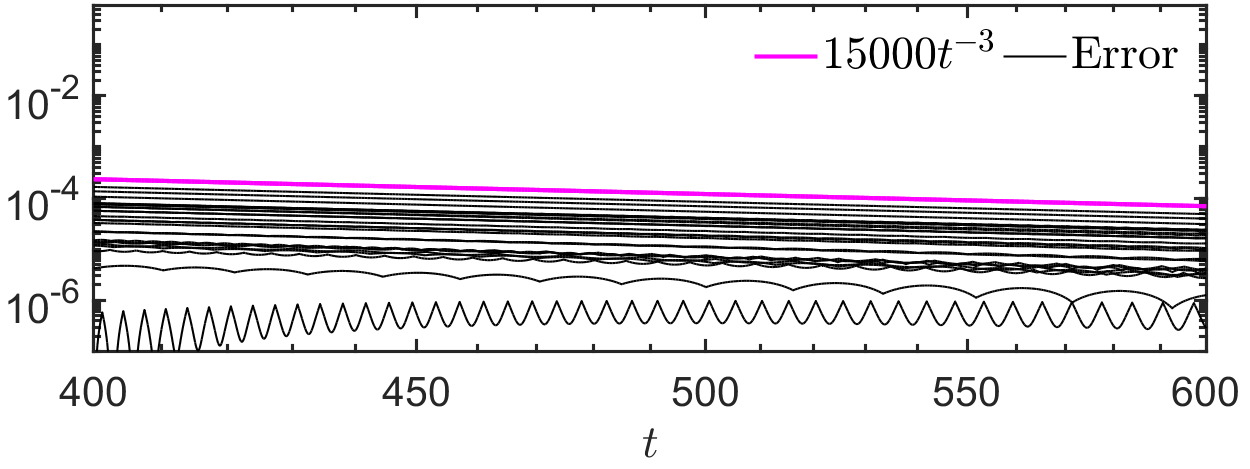}
    \end{minipage}
    \begin{minipage}[b]{.49\textwidth}
    \includegraphics[width = 1\textwidth]{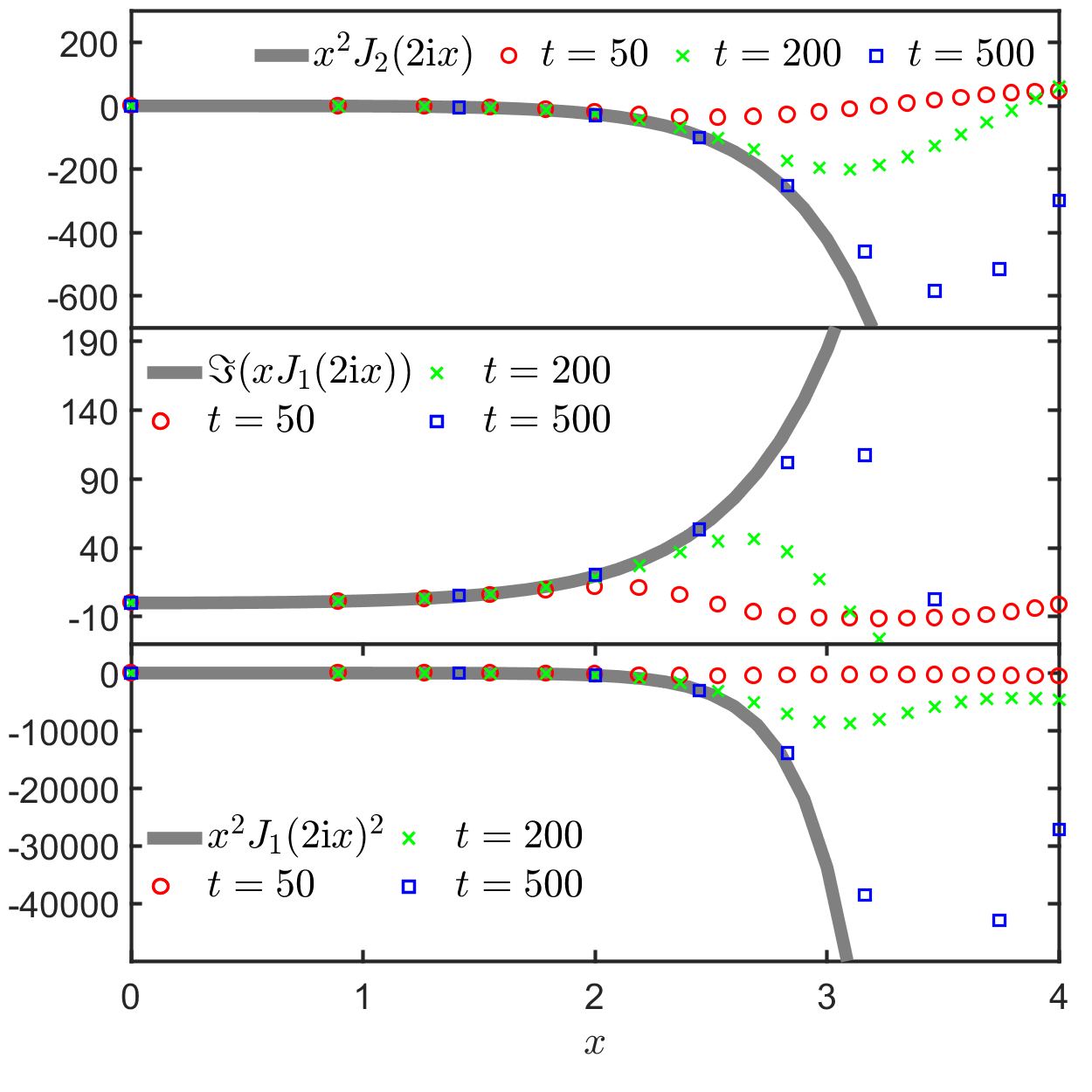}\\
            \includegraphics[width = 1\textwidth]{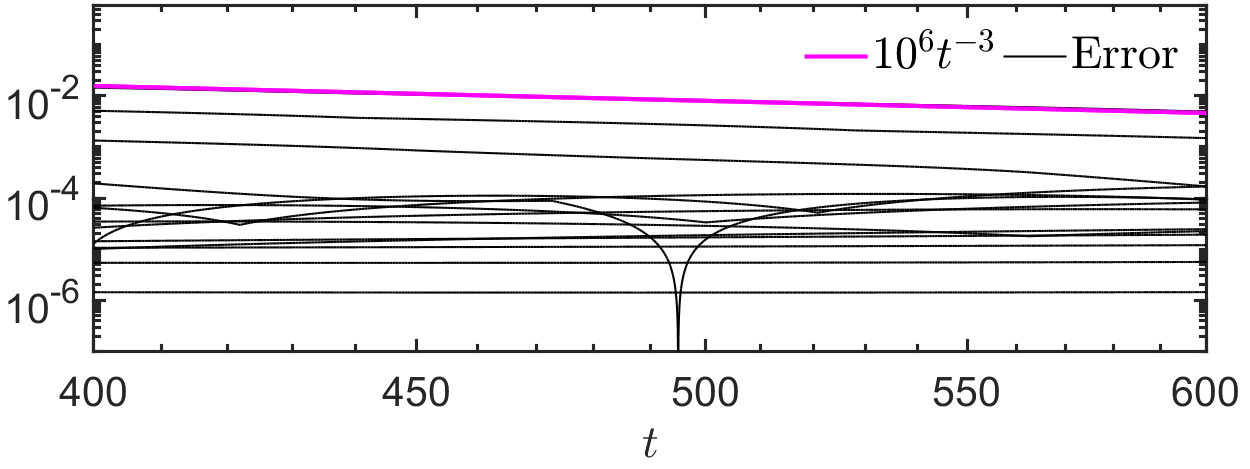}
    \end{minipage}
    \caption{
        Numerical study of incident pulse (c) in the transition regime for propagation in an initially-stable medium $D_-=-1$ (left) and an initially-unstable medium $D_-=1$ (right).
            See the main text for a full explanation.
    }
    \label{f:icc_PIII}
\end{figure}
Once again, the left-hand (respectively, right-hand) column corresponds to propagation in an initially-stable (respectively, initially-unstable) medium.  The three sub-plots of the upper panel in each column show numerical data at the indicated fixed values of $t$ as functions of the similarity variable $x=\sqrt{2tz}$ compared with the predicted limiting function plotted with a thick gray line:
\begin{itemize}
\item the top sub-plot compares $\Re(-2\ii D_- t^2 \ee^{\ii\aleph}q(t,z(x;t))/\b{r_0^{(1)}})$ with the limiting function $x^2J_2(2x)$ ($x^2J_2(2\ii x)$) in the stable (unstable) case;
\item the center sub-plot compares $\Re(-\ii t\ee^{\ii\aleph}P(t,z(x;t))/\b{r_0^{(1)}})$ with the limiting function $xJ_1(2x)$ ($\Im(xJ_1(2\ii x))$) in the stable (unstable) case;
\item the bottom sub-plot compares the quantity $2t^2(D(t,z(x;t))-D_-)/|r_0^{(1)}|^2$ with the limiting function $x^2J_1(2x)^2$ ($x^2J_1(2\ii x)^2$) in the stable (unstable) case.
\end{itemize}
At the bottom of each column is a plot of the absolute error between $q(t,z)$ and the leading term on the right-hand side in \eqref{e:qPD-global-Mpos-stable} (stable case) or in \eqref{e:qPD-edge-Mpos-unstable} (unstable case) for numerous fixed values of $x=\sqrt{2tz}$ plotted as functions of $t$.  The magenta trend line in each plot is consistent with a rate of decay of $\O(t^{-3})$ exactly as predicted in Theorems~\ref{thm:global-Mpos-stable} and \ref{thm:edge-Mpos-unstable} for the transition region corresponding to choosing $\alpha=-1$ in the exponent.

The sub-plots in the upper plot of the right-hand column suggest a nonuniform rate of convergence to the limiting functions, which in this (unstable) case exhibit exponential growth as $x\to+\infty$.  This in turn suggests that pulse (c) produces a strong response in the unstable medium that drives it away from the self-similar behavior that occurs near the edge $z=0$ for large $t>0$.  To understand the solutions away from the edge, we make plots similar to Figures~\ref{f:real-stable}--\ref{f:complex-unstable} for the nongeneric pulse (c).  In Figures~\ref{f:Bessel-stable} and \ref{f:Bessel-unstable} (for propagation in an initially-stable and unstable medium respectively), we show in the left-hand panel a density plot of $\ln(|q(t,z)|)$.  In the stable case, the leading approximation from Theorem~\ref{thm:global-Mpos-stable} has zeros corresponding to fixed values of $x=\sqrt{2tz}$ and some of these curves are overlaid; however for the unstable case the leading approximation has no zeros at all, even though the density plot shows some curves along which the real-valued solution $q(t,z)$ evidently vanishes.  We do not have an explanation for this phenomenon because it occurs in the medium-bulk regime where Theorem~\ref{thm:edge-Mpos-unstable} does not apply; it is related to the nonuniformity of convergence apparent in the upper right-hand panel of Figure~\ref{f:icc_PIII}.
\begin{figure}[h]
    \centering
    \begin{minipage}[b]{.52\textwidth}
    \includegraphics[width = 1\textwidth]{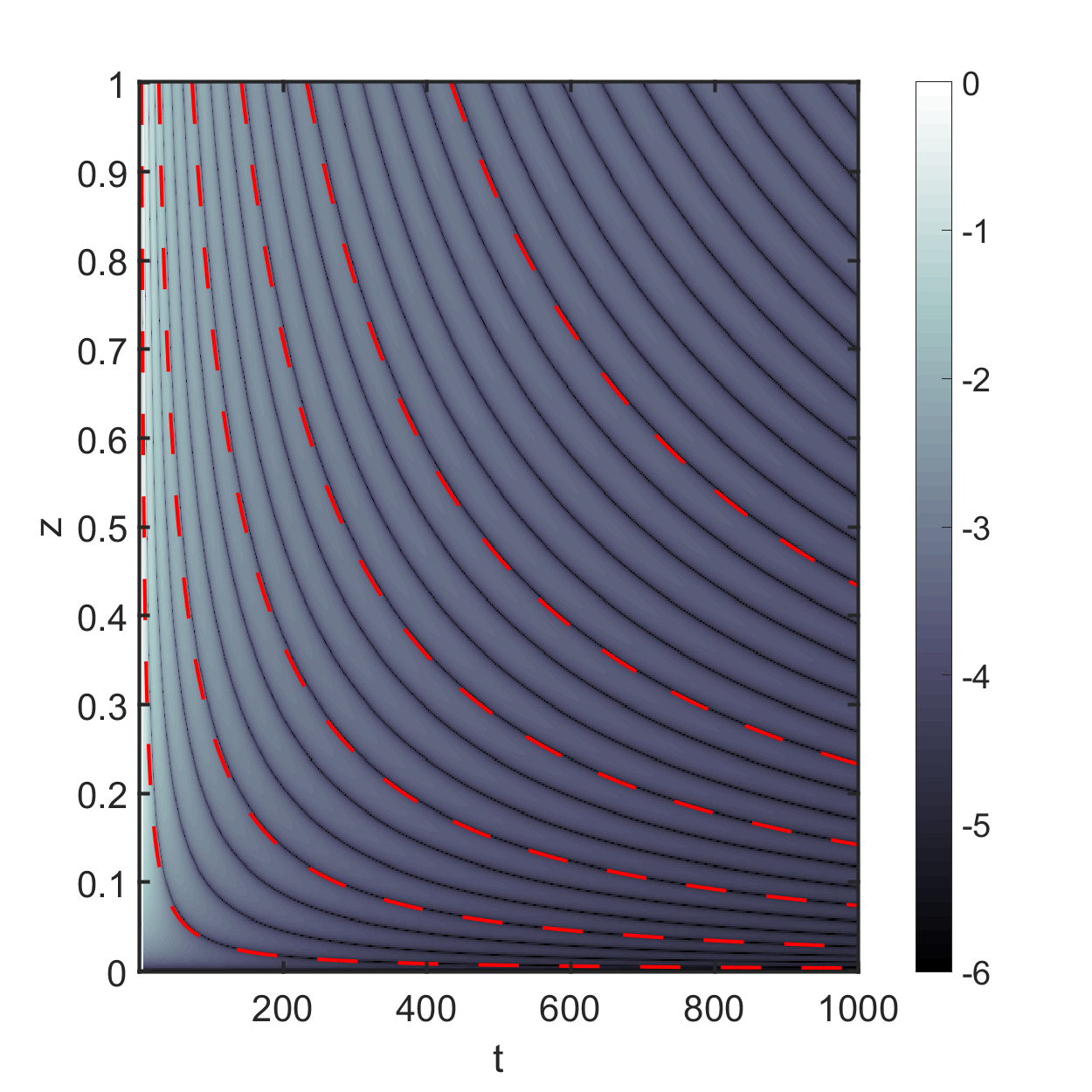}
    \end{minipage}\hspace{-1.5em}
    \begin{minipage}[b]{.5\textwidth}
    \includegraphics[width = 1\textwidth]{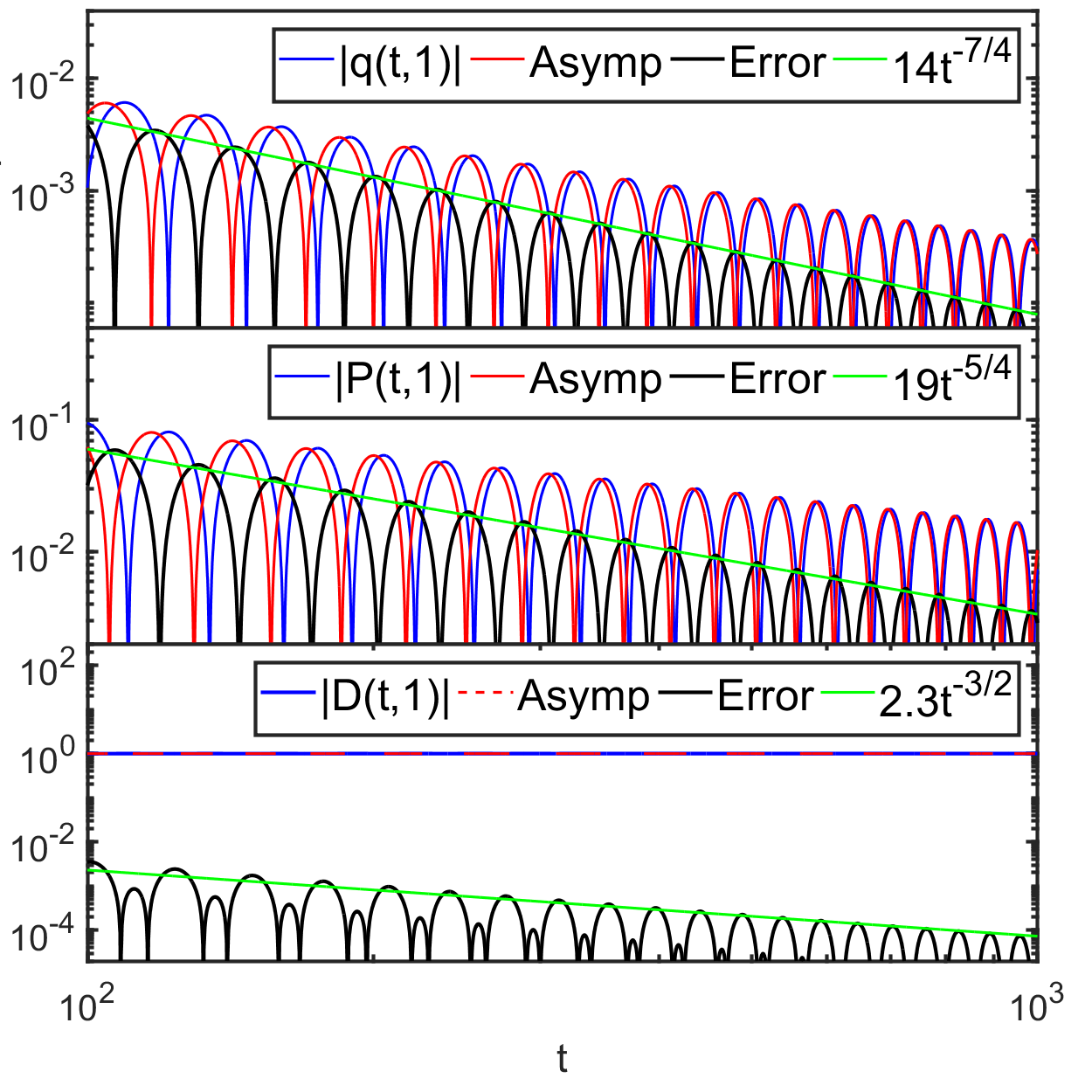}
    \end{minipage}
    \caption{
    Numerical study of nongeneric incident pulse (c) in the medium-bulk regime for propagation in an initially-stable medium ($D_-=-1$).  See the main text for a full explanation.
    }
    \label{f:Bessel-stable}
\end{figure}
\begin{figure}[h]
    \centering
    \begin{minipage}[b]{.52\textwidth}
    \includegraphics[width = 1\textwidth]{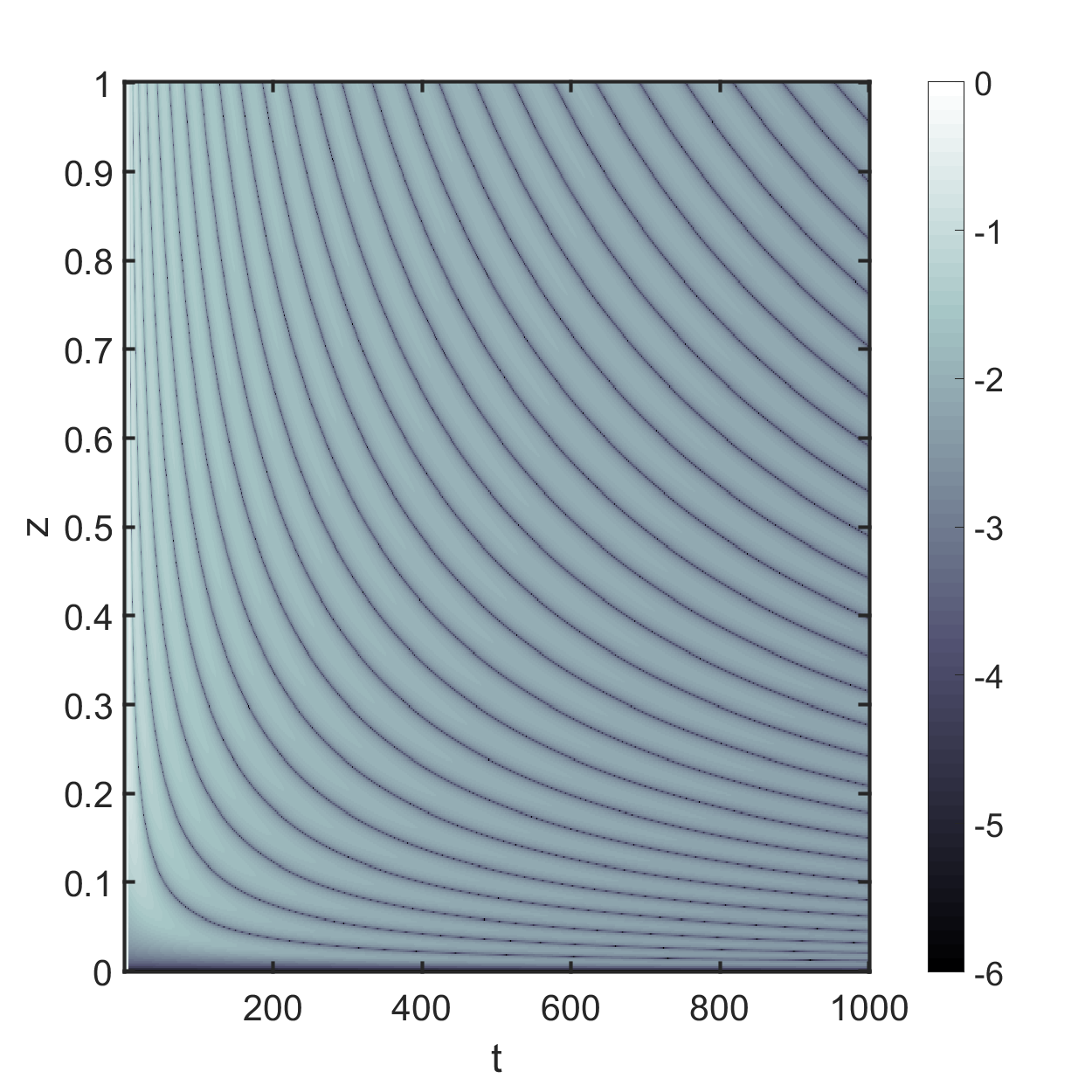}
    \end{minipage}\hspace{-1.5em}
    \begin{minipage}[b]{.5\textwidth}
    \includegraphics[width = 1\textwidth]{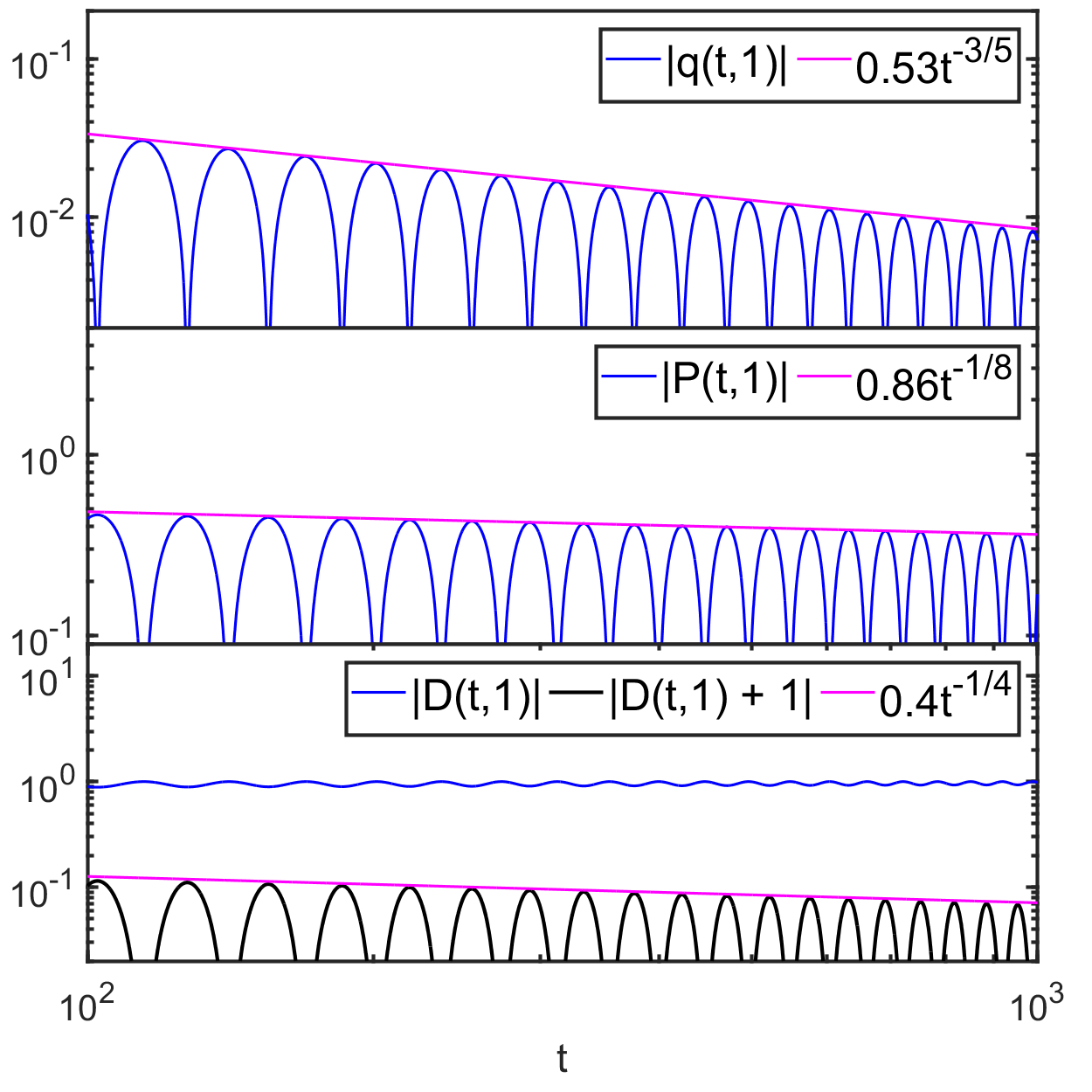}
    \end{minipage}
    \caption{
    As in Figure~\ref{f:Bessel-stable}, but
            for propagation in an initially-unstable medium ($D_-=1$).
    }
    \label{f:Bessel-unstable}
\end{figure}

The right-hand panel of Figures~\ref{f:Bessel-stable} and \ref{f:Bessel-unstable} is a quantitative study of the behavior of the solution generated by pulse (c) in the medium-bulk regime where $z=1$ is fixed for stable and unstable media respectively.   In the stable case illustrated in Figure~\ref{f:Bessel-stable}, we compare with the leading terms in Theorem~\ref{thm:global-Mpos-stable} (we could have expanded the Bessel functions for large arguments using \eqref{e:Bessel-expand} but since it is just as easy to evaluate the Bessel functions numerically, we did not do so here) and observe that for all three fields the measured rates of decay of the error seem to be faster than predicted:  $\O(t^{-\frac{7}{4}})$ versus $\O(t^{-\frac{3}{2}})$, $\O(t^{-\frac{5}{4}})$ versus $\O(t^{-1})$, and $\O(t^{-\frac{3}{2}})$ versus $\O(t^{-1})$ for $q(t,z)$, $P(t,z)$, and $D(t,z)$ respectively.  For the unstable case we have no result to compare with in the medium-bulk regime of fixed $z=1$, so we simply plot the numerical data against $t$ and hence provide evidence that the nongeneric pulse (c) indeed switches the unstable medium back to the stable state as $t\to+\infty$, although the decay is very gradual with $D(t,z)+1$ proportional to $t^{-\frac{1}{4}}$ only.

%

\subsubsection{A non-Schwartz class pulse}
We include pulse (d) as an example to illustrate the applicability of Theorem~\ref{thm:global-M0-stable-and-unstable} and its corollaries beyond the technically-convenient assumption that $q_0\in\mathscr{S}(\Real)$.  This pulse is in $L^1(\Real)$, allowing the numerical computation of the scattering matrix for all $\lambda\in\Real$ which shows that there are no spectral singularities or eigenvalues, and allows the (numerical) calculation of the value of $r_0\neq 0$ given in Table~\ref{tab:ic-numerics}.  In particular, this is a generic ($M=0$) pulse.  Although it is not in the Schwartz space, one can check that pulse (d) lies in the weighted Sobolev space $H^{1,N}(\Real)$ for all $N<3.5$, and by the weighted Sobolev bijection established in \cite{z1998}, as there are no spectral singularities, this implies that the reflection coefficient lies in $H^{N,1}(\Real)$.  As suggested in Remark~\ref{rem:Schwartz-too-strong}, this is almost enough for our proofs to go through with $N=3$; indeed the only additional requirements would be that $r'(\lambda)$ and $r''(\lambda)$ be absolutely integrable on $\Real$; we made no attempt to confirm numerically whether this is the case for pulse (d).

Since $r_0\neq 0$ for pulse (d), it would make sense to compare it with Theorem~\ref{thm:global-M0-stable-and-unstable} and its corollaries.  Figure~\ref{f:icd_PIII} is the analogue for pulse (d) of Figures~\ref{f:ica_PIII}--\ref{f:icb_PIII}.
\begin{figure}[h]
    \centering
    \begin{minipage}[b]{.49\textwidth}
    \includegraphics[width = 1\textwidth]{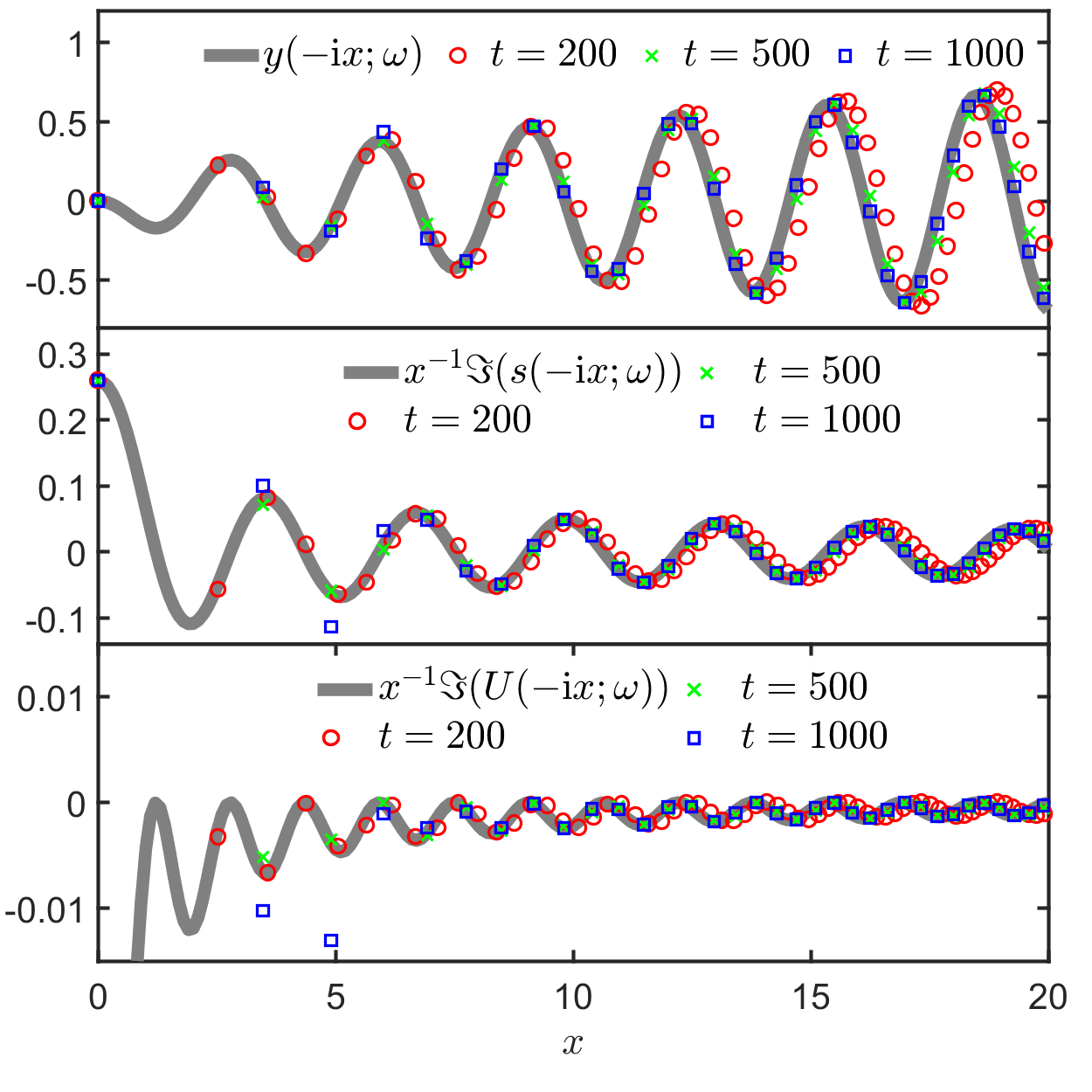}\\
                \includegraphics[width = 1\textwidth]{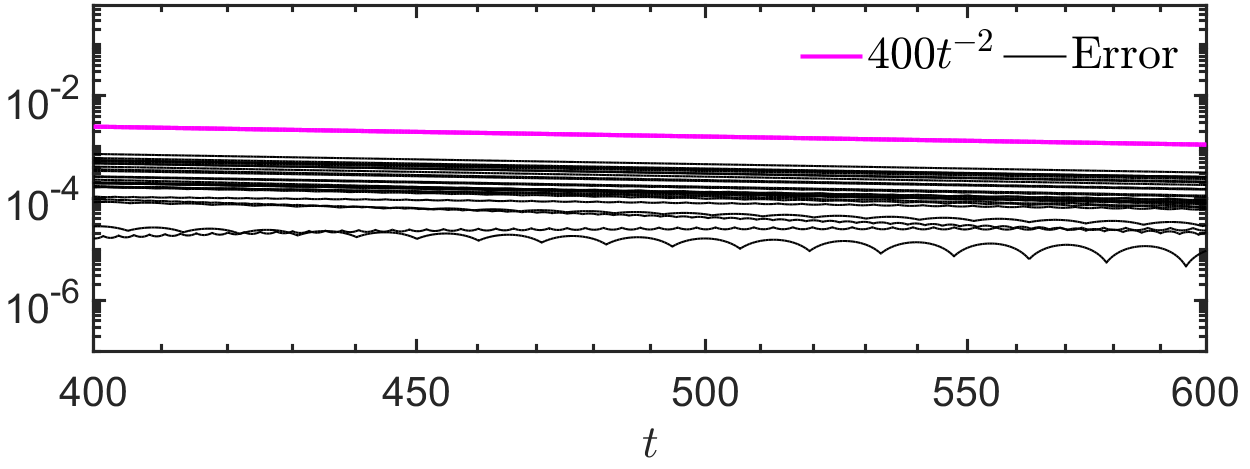}
    \end{minipage}
    \begin{minipage}[b]{.49\textwidth}
    \includegraphics[width = 1\textwidth]{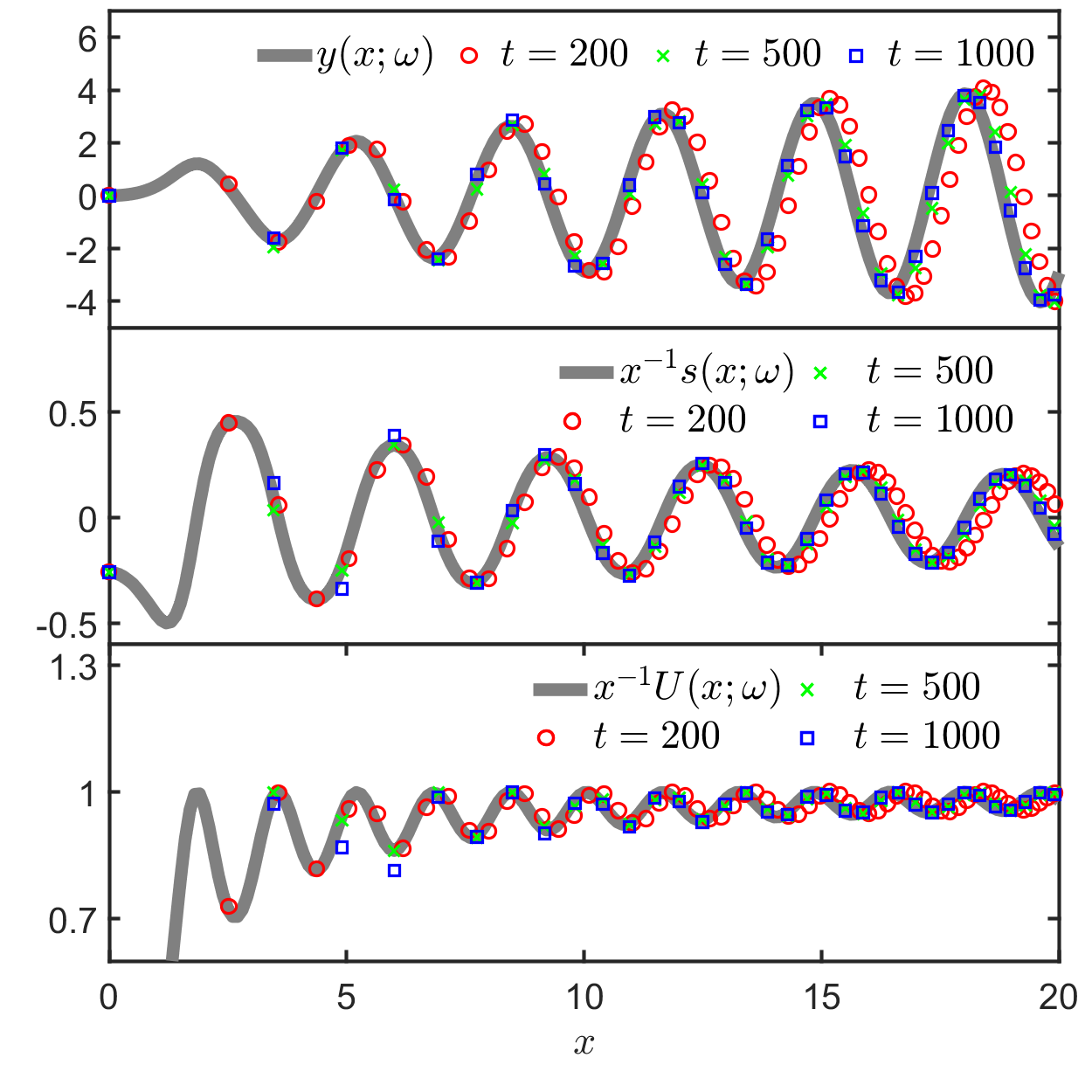}\\
                \includegraphics[width = 1\textwidth]{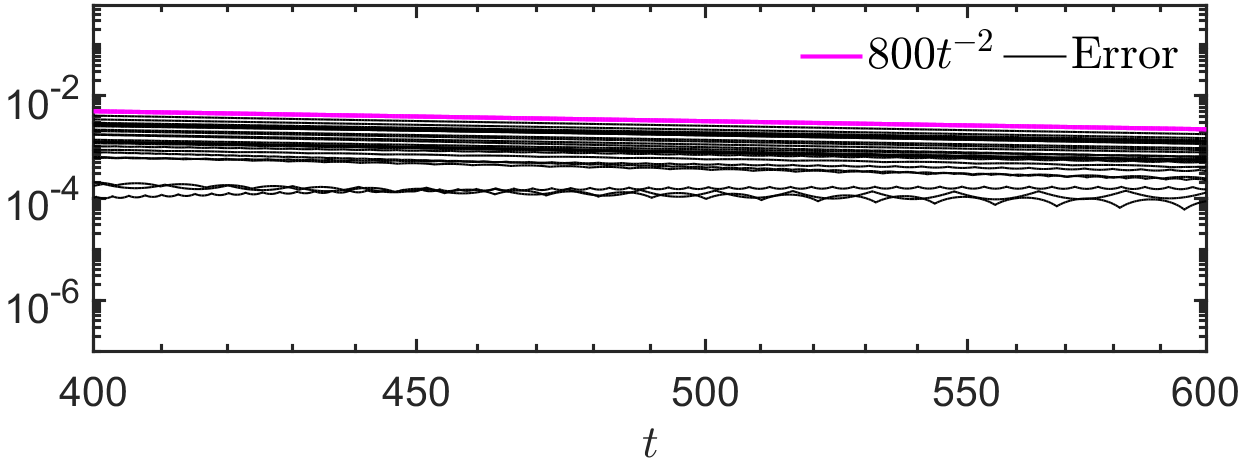}
    \end{minipage}
    \caption{
        Numerical study of incident pulse (d) in the transition regime for propagation in an initially-stable medium $D_-=-1$ (left) and an initially-unstable medium $D_-=1$ (right).
            See the main text for a full explanation.
    }
    \label{f:icd_PIII}
\end{figure}
Figures~\ref{f:rational-stable} and \ref{f:rational-unstable} are analogues for pulse (d) of Figures~\ref{f:real-stable} and \ref{f:complex-stable}, showing propagation in the medium-bulk regime of a stable medium, and of Figures~\ref{f:real-unstable} and \ref{f:complex-unstable}, showing propagation in the same regime of an unstable medium, respectively.
This experiment shows that our results indeed hold for some pulses that do not decay rapidly enough to lie in the Schwartz space.
\begin{figure}[h]
    \centering
    \begin{minipage}[b]{.52\textwidth}
    \includegraphics[width = 1\textwidth]{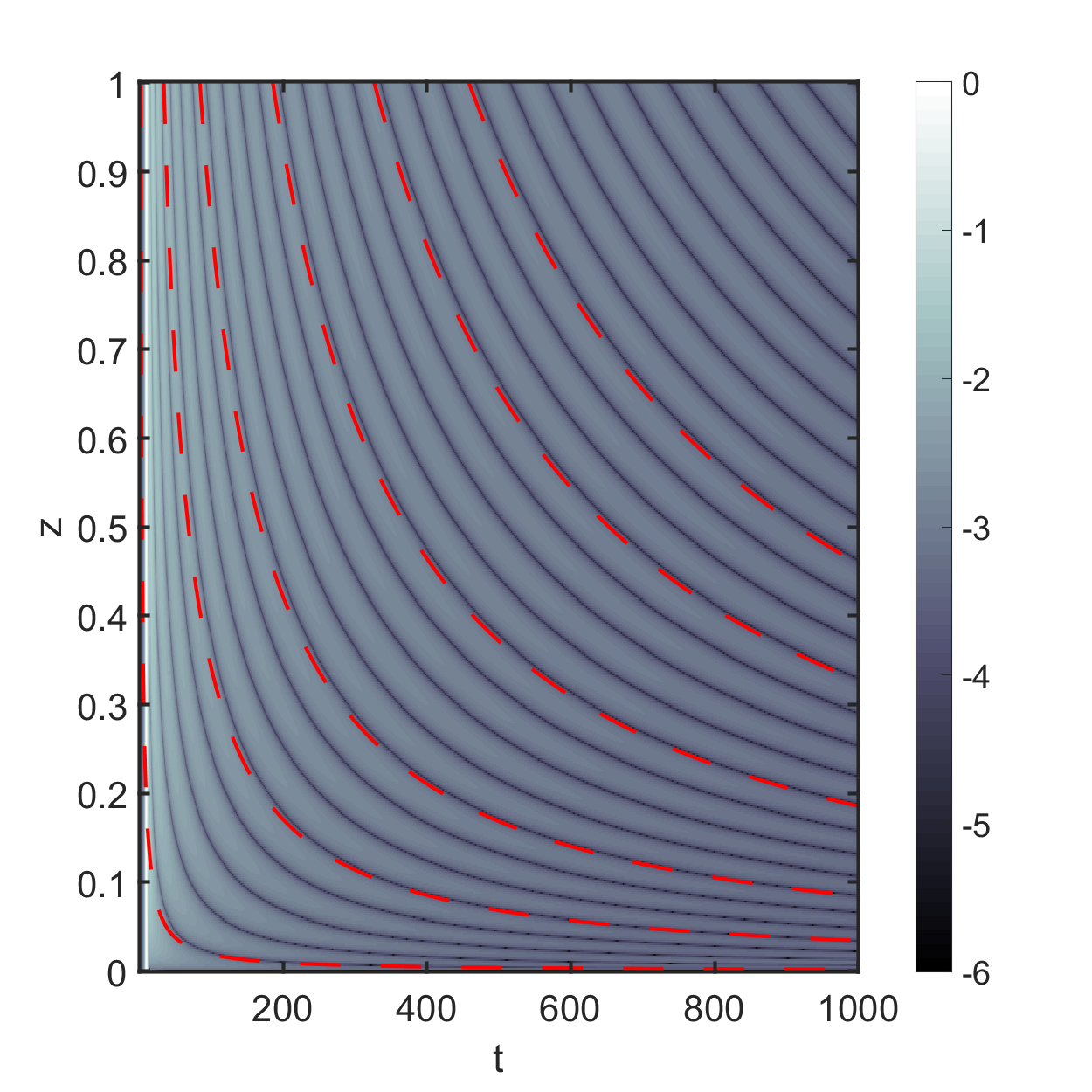}
    \end{minipage}\hspace{-1.5em}
    \begin{minipage}[b]{.5\textwidth}
    \includegraphics[width = 1\textwidth]{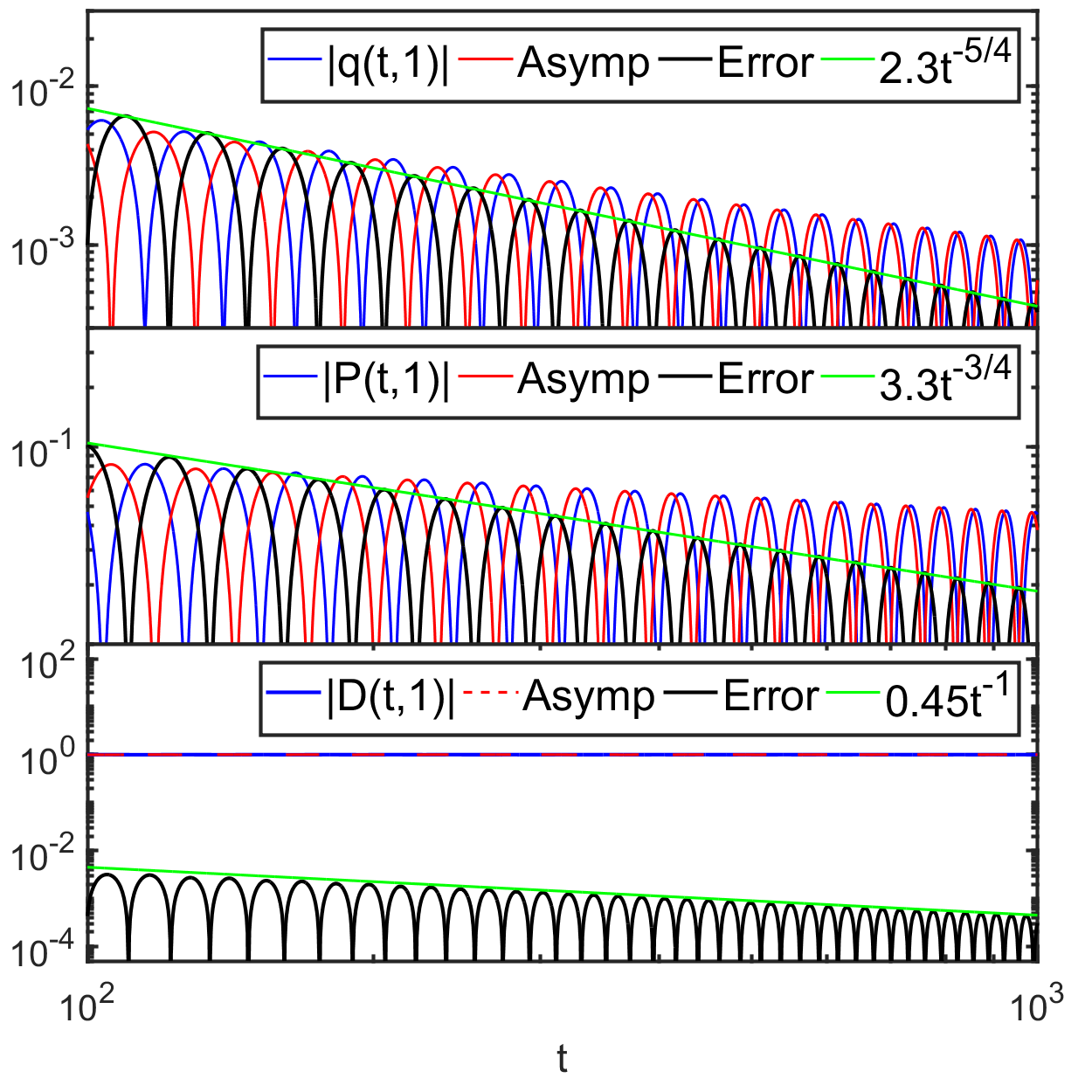}
    \end{minipage}
    \caption{
        Numerical study of incident pulse (d) in the medium-bulk regime for propagation in an initially-stable medium ($D_-=-1$).
            See the main text for a full explanation.
    }
    \label{f:rational-stable}
\end{figure}
\begin{figure}[h]
    \centering
    \begin{minipage}[b]{.52\textwidth}
    \includegraphics[width = 1\textwidth]{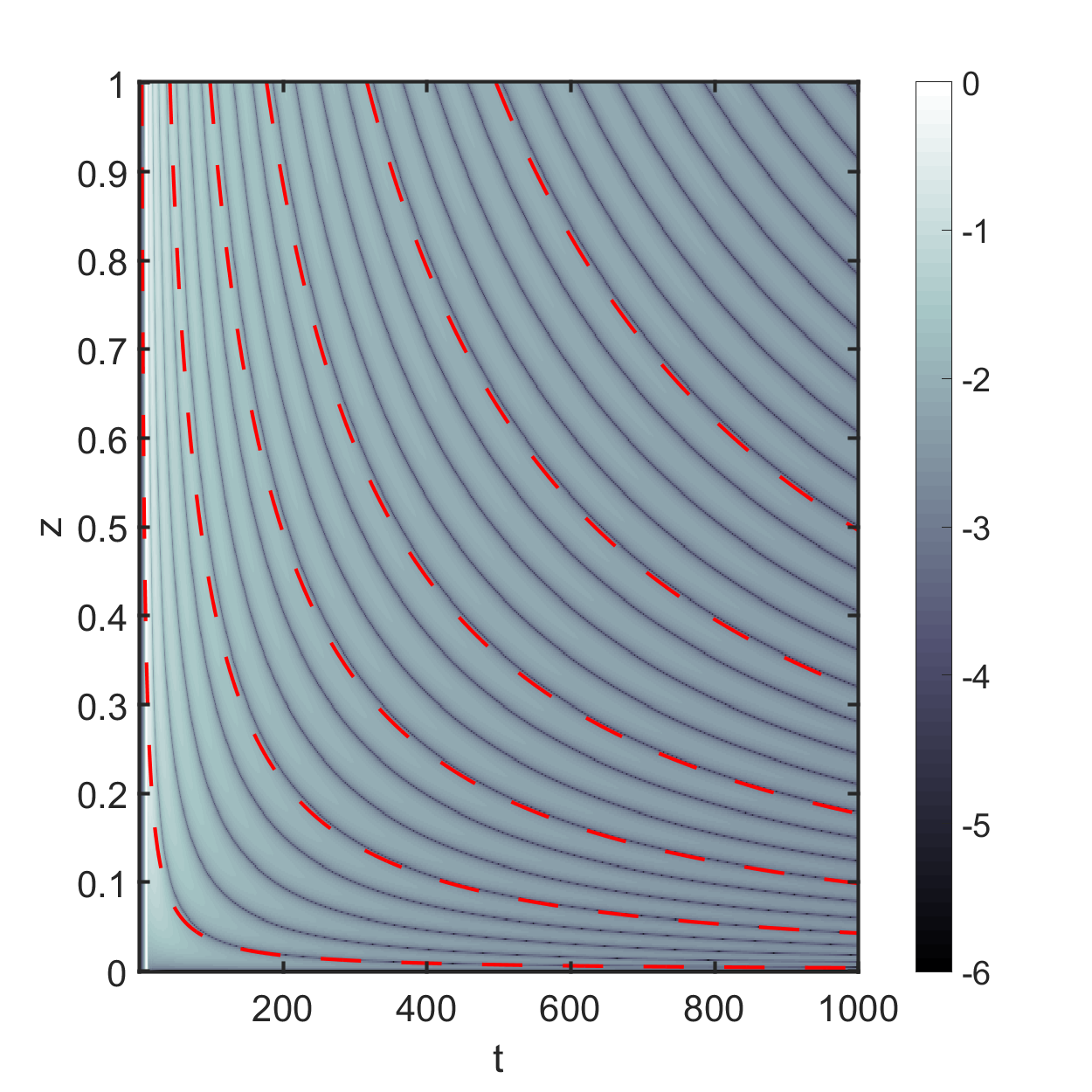}
    \end{minipage}\hspace{-1.5em}
    \begin{minipage}[b]{.5\textwidth}
    \includegraphics[width = 1\textwidth]{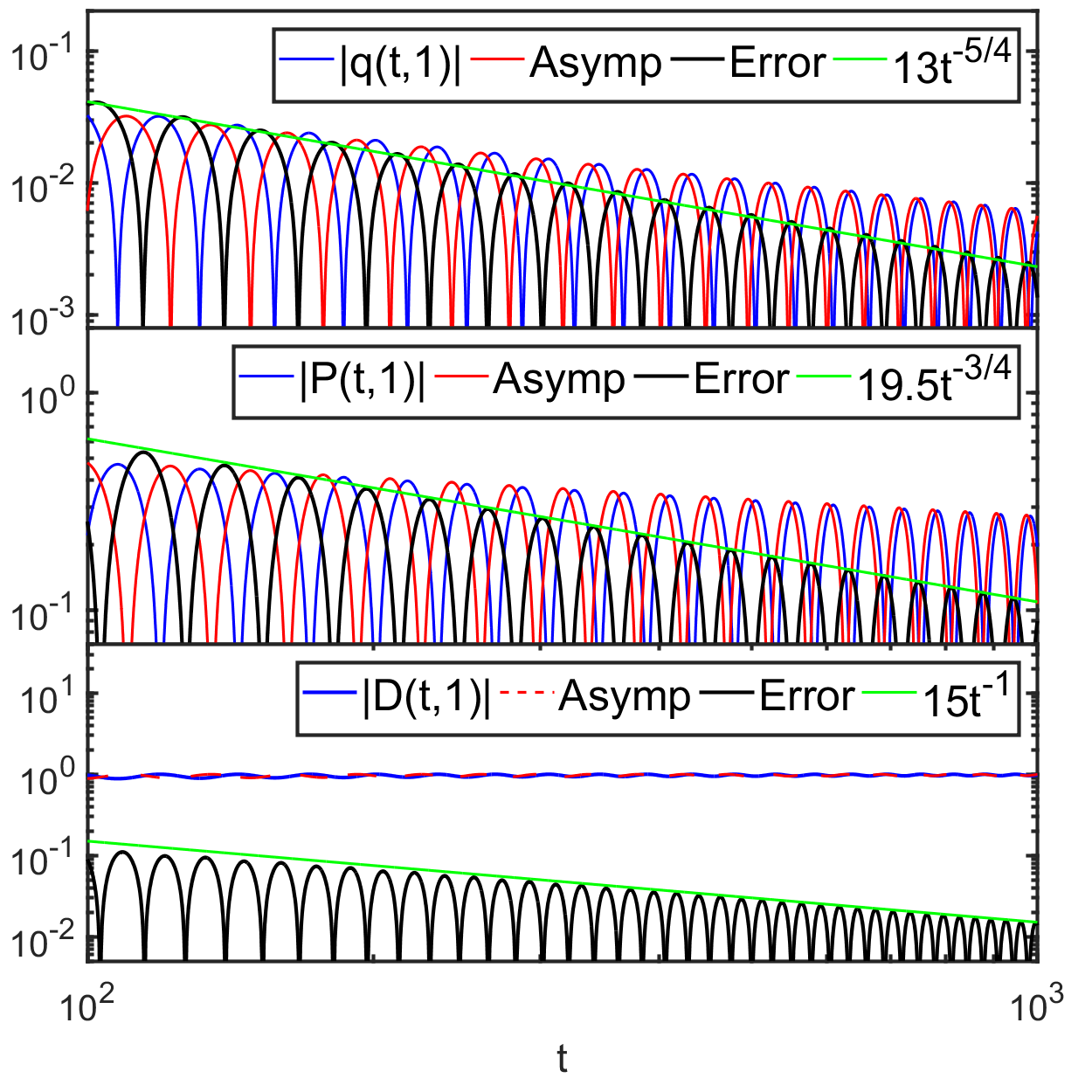}
    \end{minipage}
    \caption{
    As in Figure~\ref{f:rational-stable}, but
        for propagation in an initially-unstable medium ($D_-=1$).
    }
    \label{f:rational-unstable}
\end{figure}

\section{Analysis for propagation in an initially-stable medium}
\label{s:stable-positive}
This section concerns the analysis of RHP~\ref{rhp:M} in the case of an initially-stable medium with $D_-=-1$, in the limit that $t\to+\infty$ with $0\le z=o(t)$.
The whole analysis is driven by the sign structure of the real part of the exponent $\ii\theta(\lambda;t,z)$, for which we have the following specialized notation in the case $D_-=-1$.
\begin{definition}[the phase for $D_-=-1$]
In the stable case, we denote the phase $\theta(\lambda;t,z)$ appearing in \eqref{e:phase-def-general} as
$\theta(\lambda;t,z) = \theta_\mathrm{s}(\lambda;t,z)$, where
\begin{equation}
\label{e:thetas-def}
\theta_\mathrm{s}(\lambda;t,z) \coloneq \lambda t + \frac{z}{2\lambda}\,.
\end{equation}
\label{def:theta-s}
\end{definition}
The sign chart of $\Re(\ii\theta_\mathrm{s}(\lambda;t,z))$ is shown for $t>0$ and $z>0$ in the left-hand panel of Figure~\ref{f:inside-itheta}.
\begin{figure}[h]
\includegraphics[width = 0.4\textwidth]{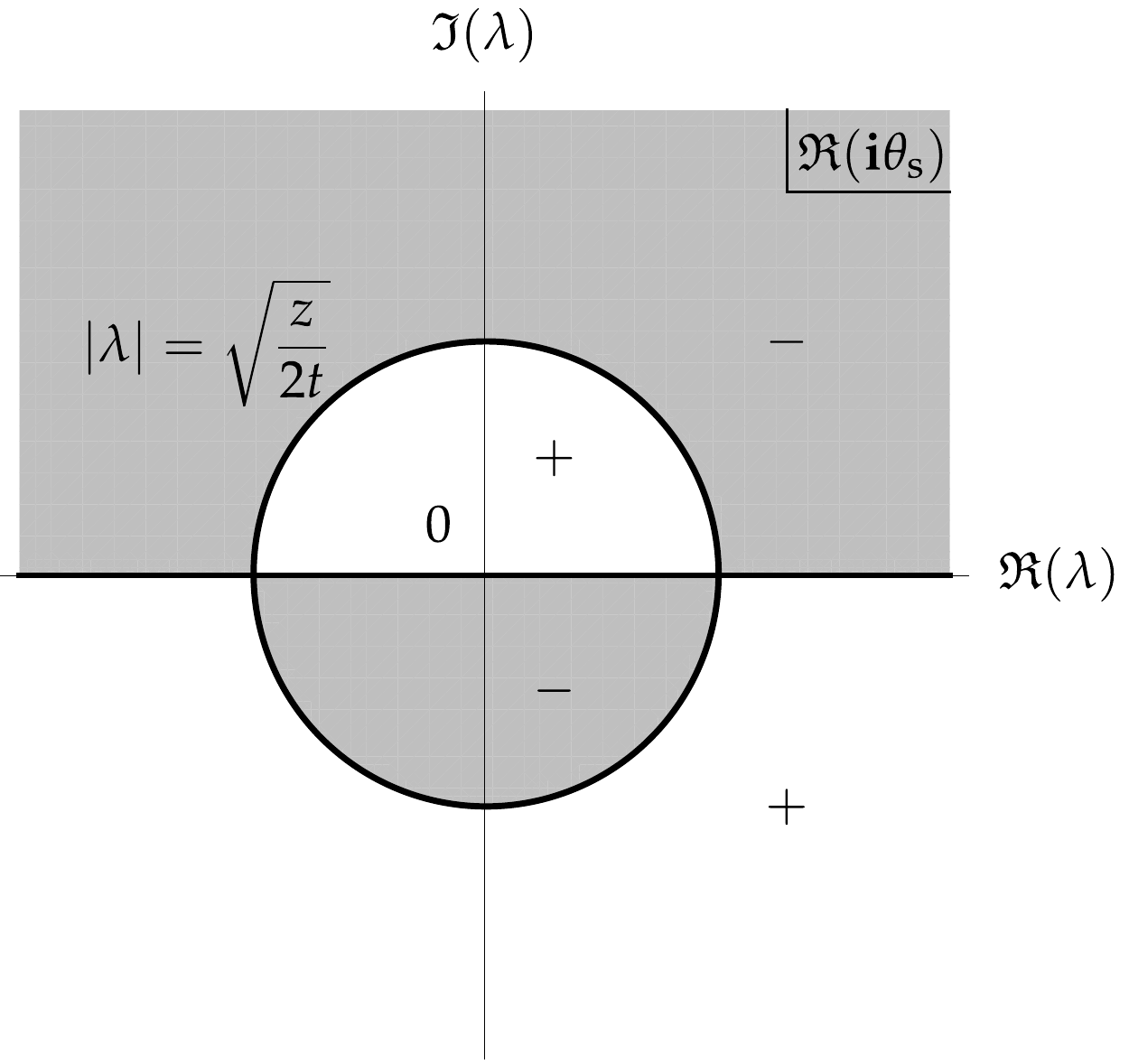}\qquad
\includegraphics[width = 0.4\textwidth]{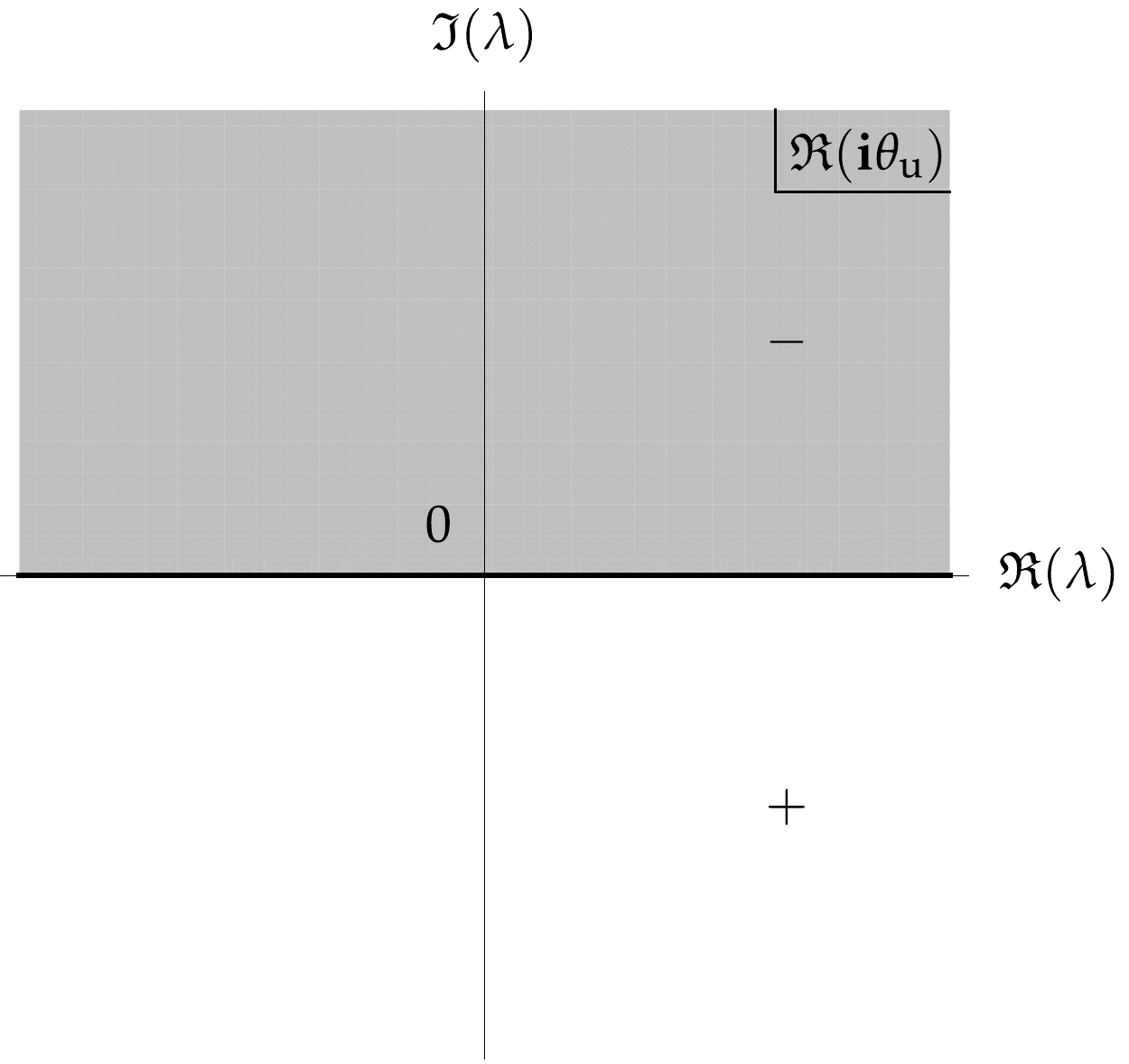}
\caption{For $t>0$ and $z>0$, the sign structure of $\Re(\ii\theta(\lambda;t,z))$ in the complex plane for an initially-stable medium $D_- = -1$ (left),
and for an initially-unstable medium $D_- = 1$ (right).
White (gray) shading corresponds to positive (negative) values of $\Re(\ii\theta(\lambda;t,z))$.
}
\label{f:inside-itheta}
\end{figure}
Note that the radius of the circle shown in that plot is
\begin{equation}
\label{e:lambdao-def}
\lambda_\circ \coloneq \sqrt{\frac{z}{2t}}\,.
\end{equation}
A key observation going forward is that under the assumption $z=o(t)$, $\lambda_\circ\to 0$ in the limit $t\to+\infty$.  This is why the moments and Taylor expansion of $r(\lambda)$ about $\lambda=0$ are of primary importance in our analysis.

\subsection{Setting up a Riemann-Hilbert-$\b{\partial}$ problem}
\label{s:stable-positive-1}
We begin by taking the arbitrary radius $\gamma>0$ in RHP~\ref{rhp:M} to coincide with $\lambda_\circ$. By the sign chart in the left-hand panel of Figure~\ref{f:inside-itheta}, this has the effect that the exponential factors in the jump matrix satisfy $|\ee^{\pm 2\ii\theta_\mathrm{s}(\lambda;t,z)}|=1$ on the jump contour $\Sigma_\M$.

Next, we remove the jumps across the real intervals $(-\infty,-\lambda_\circ)\cup (\lambda_\circ,+\infty)$ as follows.  The exponential factors $\ee^{\pm2\ii\theta_\mathrm{s}(\lambda;t,z)}$ can be algebraically separated by the jump matrix factorization:
\begin{multline}
\W^\dagger(\lambda;t,z)\W(\lambda;t,z)=\bpm 1 & 0\\R(\lambda)\ee^{-2\ii\theta_\mathrm{s}(\lambda;t,z)} & 1\epm (1+|r(\lambda)|^2)^{\sigma_3}\bpm 1&\b{R(\lambda)}\ee^{2\ii\theta_\mathrm{s}(\lambda;t,z)}\\0 & 1\epm,\\
\lambda\in (-\infty,-\lambda_\circ)\cup (\lambda_\circ,+\infty),
\label{e:three-factor}
\end{multline}
where
\begin{equation}
R(\lambda)\coloneq\frac{r(\lambda)}{1+|r(\lambda)|^2},\quad\lambda\in\Real.
\label{e:R-def}
\end{equation}
\begin{lemma}
\label{thm:reflection2}
If an incident pulse $q_0(t)$ satisfies Assumption~\ref{ass:q0-assumption} and generates no discrete eigenvalues or spectral singularities,
then
\begin{equation}
R(\lambda)\in \mathscr{S}(\Real)\,, 
\qquad\text{and}\qquad
\ln(1 + |r(\lambda)|^2)\in \mathscr{S}(\Real)\,. 
\end{equation}
\end{lemma}
\begin{proof}
We already know that
$r(\lambda)\in\mathscr{S}(\Real)$
from Lemma~\ref{lemma:reflection}.
Since
all derivatives of $r(\lambda)$
are continuous and decay rapidly,
by repeated differentiation using $|r(\lambda)|^2=r(\lambda)\b{r(\lambda)}$ and $1+|r(\lambda)|^2\ge 1$ one sees that $R(\lambda)$ and $\ln(1+|r(\lambda)|^2)$ are functions in
$\mathscr{S}(\Real)$.
\end{proof}
%
By the sign chart in the left-hand panel of Figure~\ref{f:inside-itheta}, the factor $\ee^{\pm 2\ii\theta_\mathrm{s}(\lambda;t,z)}$ has modulus less than $1$ in the part of the exterior region $|\lambda|>\lambda_\circ$ with $\pm\Im(\lambda)>0$.  It is therefore desirable to make a substitution to move the triangular factors in \eqref{e:three-factor} into the respective half-planes.  However, since $R(\lambda)$ generally has no analytic continuation from the real line, no substitution that accomplishes the stated goal can be analytic, so we will adapt the $\b{\partial}$ approach
from the works~\cite{mm2006,mm2008,dmm2019}.
We identify the complex plane having coordinate $\lambda\in\Complex$ with $\Real^2$ having real cartesian coordinates
\begin{equation}
\label{e:uv-def}
u \coloneq \Re(\lambda)\,,\qquad
v \coloneq \Im(\lambda)\,.
\end{equation}
Since by Lemma~\ref{thm:reflection2}
$R(\lambda)$ has any number of
continuous derivatives on the real line, for any $N\ge 2$, a continuous extension of $R(\lambda)$ from $\Real$ to $\Real^2$ can be defined by the formula:
\begin{equation}
Q_N(u,v)\coloneq\sum_{n=0}^{N-2}\frac{(\ii v)^n}{n!}\frac{\dd^nR}{\dd u^n}(u),\quad (u,v)\in\Real^2.
\label{e:qN}
\end{equation}
The differentiation here is along the real line $v=0$.
That $Q_N(u,v)$ is an extension of $R$ into the plane $\Real^2$ is easily seen by setting $v=0$ which yields $Q_N(u,0)=R(u)$.  In particular, the extension $Q_1(u,v)$ is just orthogonal projection to $v=0$:  $Q_1(u,v)=Q_1(u,0)=R(u)$.  The Schwarz reflection of $Q_N(u,v)$ is defined by
\begin{equation}
\b{Q_N}(u,v)\coloneq \b{Q_N(u,-v)}=\sum_{n=0}^{N-2}\frac{(\ii v)^n}{n!}\frac{\dd^n\b{R}}{\dd u^n}(u),\quad (u,v)\in\Real^2.
\end{equation}
Here $\b{R}(u)=\b{R(u)}$.
While these extensions are not analytic functions, they are nearly so near the real axis; indeed, recalling the Cauchy-Riemann operator
\begin{equation}
\dbar \coloneq \frac{1}{2}\bigg(\frac{\partial}{\partial u} + \ii\frac{\partial}{\partial v}\bigg)
\end{equation}
(annihilating all analytic functions), one sees by direct computation that $\b{\partial}Q_N(u,v)$ is a continuous function $\Real^2\to\Complex$ that vanishes to order $N-2$ at $v=0$:
\begin{equation}
\b{\partial}Q_N(u,v)=\frac{1}{2}\frac{(\ii v)^{N-2}}{(N-2)!}\frac{\dd^{N-1}R}{\dd u^{N-1}}(u),\quad (u,v)\in\Real^2.
\label{e:dbarQN}
\end{equation}
Likewise,
\begin{equation}
\b{\partial}\,\b{Q_N}(u,v)=\frac{1}{2}\frac{(\ii v)^{N-2}}{(N-2)!}\frac{\dd^{N-1}\b{R}}{\dd u^{N-1}}(u),\quad (u,v)\in\Real^2.
\label{e:dbarQNbar}
\end{equation}
%
%
The extensions $Q_N(u,v)$ and $\b{Q_N}(u,v)$ will be used to remove the triangular factors in \eqref{e:three-factor} from the jump condition on $(-\infty,-\lambda_\circ)\cup(\lambda_\circ,+\infty)$ at the cost of some non-analyticity measured by \eqref{e:dbarQN}--\eqref{e:dbarQNbar}.  The central factor $(1+|r(\lambda)|^2)^{\sigma_3}$ in \eqref{e:three-factor} can be factored into a ratio of functions analytic in the upper and lower half-planes; recalling the function $F(\lambda)$ defined in \eqref{e:f-lambda}, we set
\begin{equation}
\delta(\lambda)\coloneq \ee^{-F(\lambda)},\quad \lambda\in\Complex\setminus\Real.
\label{e:delta-def}
\end{equation}
Then, it is easy to verify that $\delta(\lambda)\to 1$ as $\lambda\to\infty$ and
\begin{equation}
\delta^+(\lambda)\delta^-(\lambda)^{-1}=1+|r(\lambda)|^2,\quad\lambda\in\Real.
\label{e:delta-equation}
\end{equation}
\begin{remark}
Note that, since the diagonal factor $(1+|r(\lambda)|^2)^{\sigma_3}$ only appears in the jump matrix in \eqref{e:three-factor} in the complement of the interval $(-\lambda_\circ,\lambda_\circ)$, one could omit this interval from the integration over $\Real$ in \eqref{e:f-lambda} and obtain another function $\widetilde{F}(\lambda)$ and from \eqref{e:delta-def} a function $\widetilde{\delta}(\lambda)$ that satisfies \eqref{e:delta-equation} exactly where \eqref{e:three-factor} holds.  However, since $\lambda_\circ\to 0$ as $t\to+\infty$, it turns out to be more convenient to keep the interval $(-\lambda_\circ,\lambda_\circ)$ in the integration.
\end{remark}
\begin{lemma}
\label{thm:delta}
Under the conditions of Lemma~\ref{thm:reflection2}, $|\delta(\lambda)|$ and $|\delta(\lambda)|^{-1}$ are uniformly bounded on their domain of definition $\Complex\setminus\Real$.
\end{lemma}
\begin{proof}
The proof is similar to that in~\cite{dz2003,lps2018,dmm2019}.
By Lemma~\ref{thm:reflection2}, there is a constant $C>0$ such that
$|\ln(1+|r(\lambda)|^2)|\le C$.
We then see that for $\lambda=u+\ii v$,
\begin{equation}
|\delta(\lambda)|
 = \exp\bigg(\frac{|v|}{2\pi}\int_\Real \frac{\ln(1+|r(s)|^2)}{(s - u)^2 + v^2}\dd s\bigg)
 \le \exp\bigg(\frac{C|v|}{2\pi}\int_\Real \frac{\dd s}{(s - u)^2 + v^2}\bigg)
 = 
 \ee^{\frac{1}{2}C}\,,
\end{equation}
so we have proved $\delta(\lambda)\in L^\infty(\Complex\setminus\Real)$.
The definition~\eqref{e:delta-def} directly yields $\delta(\lambda)\b{\delta}(\lambda) = 1$ for all $\lambda\in\Complex$, where $\b{\delta}(\lambda)$ denotes the Schwarz reflection $\b{\delta(\b{\lambda})}$.  Then, given $\lambda\in\Complex\setminus\Real$, we have $|\delta(\lambda)|^{-1} = |\delta(\b{\lambda})| \le \ee^{\frac{1}{2}C}$, so $\delta(\lambda)^{-1}\in L^\infty(\Complex\setminus\Real)$ as well.
\end{proof}
Let $B:\Real\to[0,1]$ be a ``bump'' function of class $C^\infty(\Real)$ with the additional properties:
\begin{itemize}
\item $B(-v)=B(v)$,
\item $B(v)\equiv 0$ for $|v|>2$, and
\item $B(v)\equiv 1$ for $|v|<1$.
\end{itemize}
Defining a matrix for $(u,v)\in\Real^2$ by setting
\begin{equation}
\mathbf{T}_\mathrm{s}(u,v;t,z) \coloneq \bpm 1 & 0 \\ B(v)Q_N(u,v) \ee^{-2\ii\theta_\mathrm{s}(u+\ii v;t,z)} & 1\epm\,,\quad (u,v)\in\Real^2
\end{equation}
the jump matrix factorization in \eqref{e:three-factor} can be rewritten as
\begin{multline}
\W^\dagger(\lambda;t,z)\W(\lambda;t,z) =
\mathbf{T}_\mathrm{s}(u,0;t,z)\delta^+(u)^{\sigma_3}\cdot\delta^-(u)^{-\sigma_3}\mathbf{T}_\mathrm{s}^\dagger(u,0;t,z),\\
\lambda=u\in (-\infty,-\lambda_\circ)\cup (\lambda_\circ,+\infty),\quad v=0.
%
\end{multline}
Here, $\mathbf{T}_\mathrm{s}^\dagger(u,v;t,z)$ denotes the Schwarz reflection $\mathbf{T}_\mathrm{s}(u,-v;t,z)^\dagger$.
This motivates one to define a new matrix function $\K_\mathrm{s}(u,v;t,z)$ explicitly in terms of the solution $\M(\lambda;t,z)$ of RHP~\ref{rhp:M} by setting
\begin{equation}
\label{e:Ks-def}
\K_\mathrm{s}(u,v;t,z) \coloneq
  \begin{cases}
  \M(u+\ii v;t,z)\,\delta(u+\ii v)^{-\sigma_3}\,, &\qquad |u+\ii v|<\lambda_\circ\,,\\
  \M(u+\ii v;t,z)\mathbf{T}_\mathrm{s}^{\dagger}(u,v;t,z)^{-1}\delta(u+\ii v)^{-\sigma_3}\,, &\qquad |u+\ii v| > \lambda_\circ\,,\qquad v > 0\,,\\
  \M(u+\ii v;t,z)\mathbf{T}_\mathrm{s}(u,v;t,z)\delta(u+\ii v)^{-\sigma_3}\,, &\qquad |u+\ii v|> \lambda_\circ\,,\qquad v < 0\,.
  \end{cases}
\end{equation}
This definition is shown schematically in the left-hand panel of Figure~\ref{f:Ks}.
\begin{figure}[h]
\centering
\includegraphics[width = 0.37\textwidth]{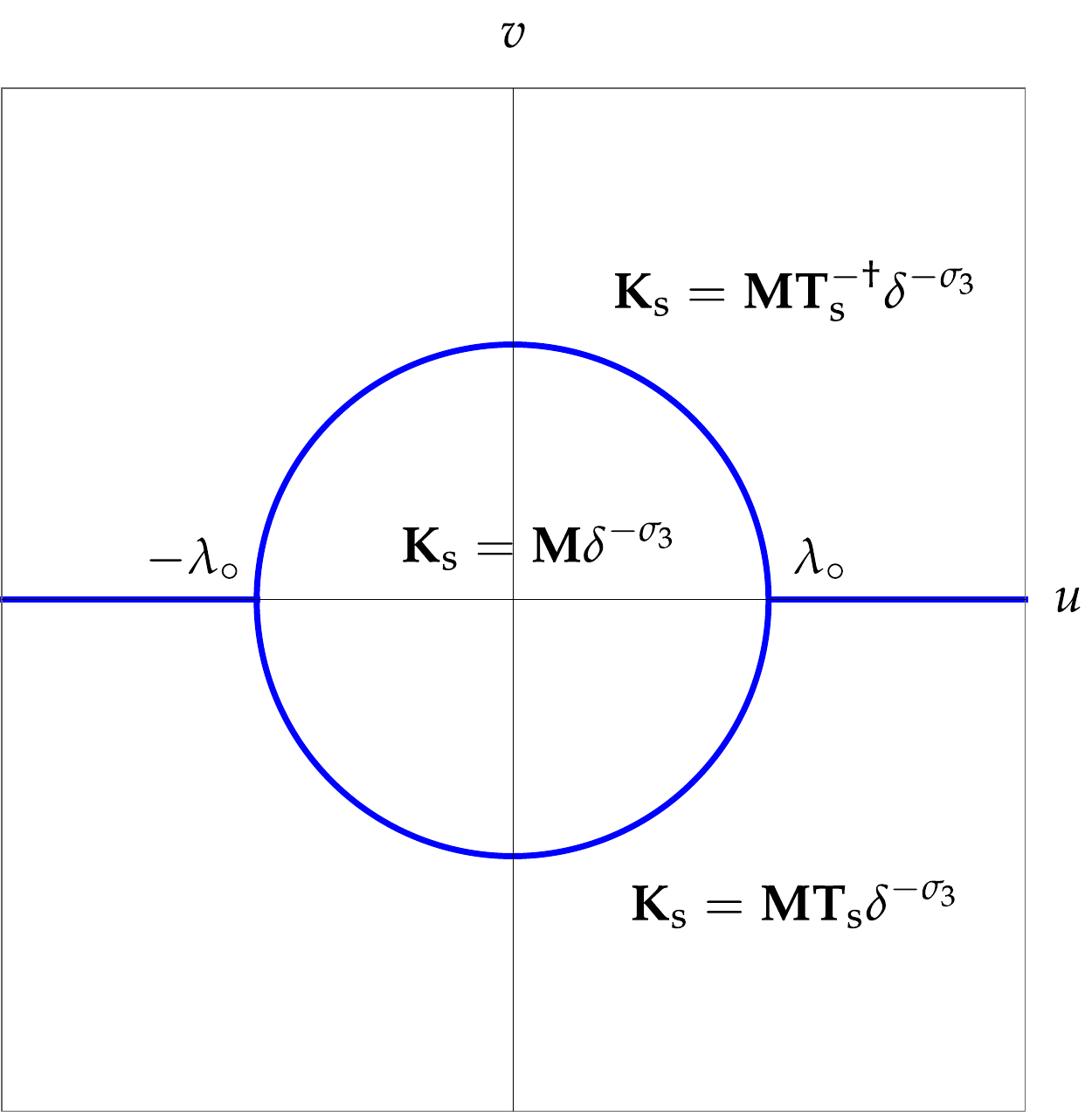}\quad
\includegraphics[width = 0.37\textwidth]{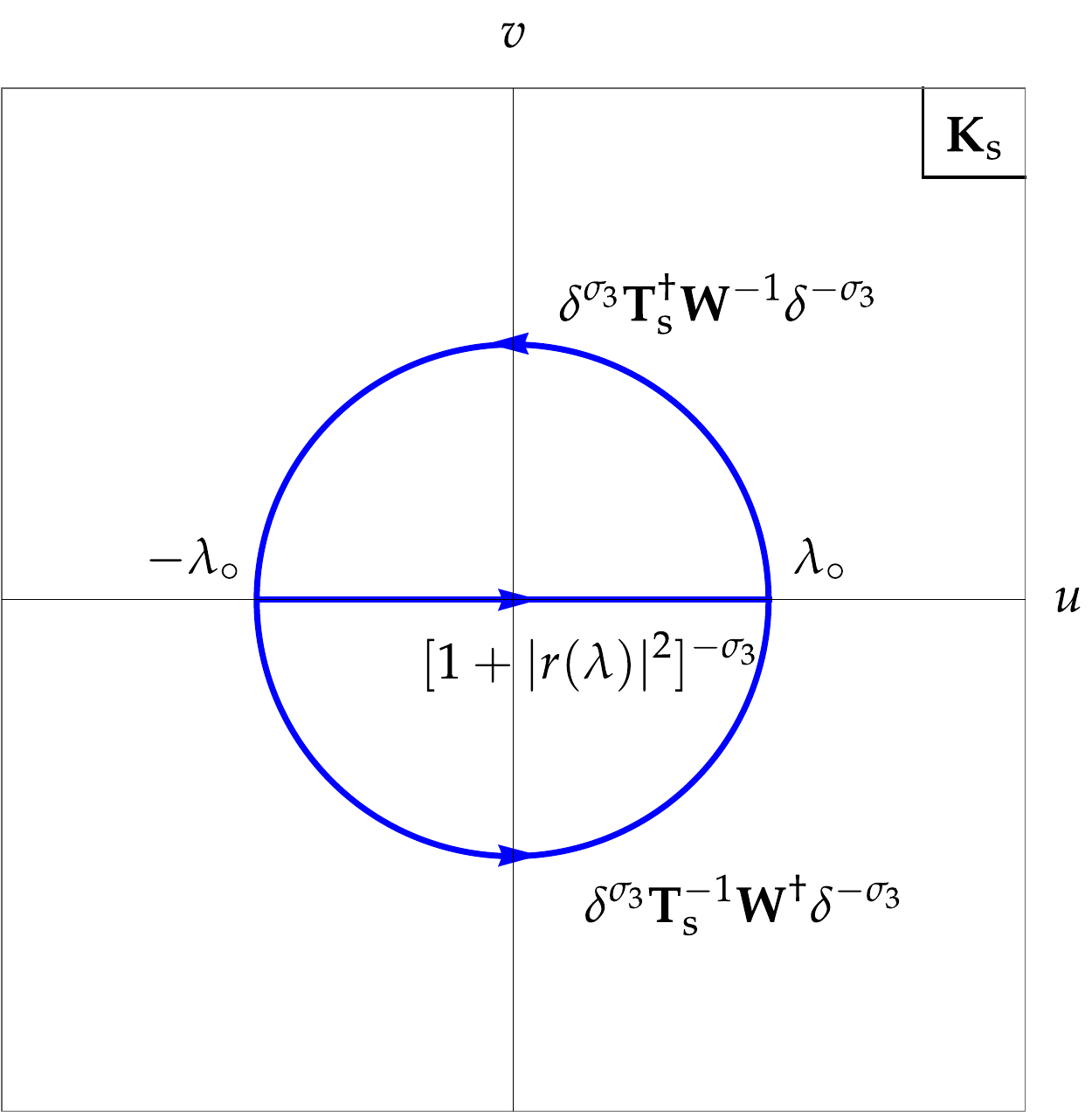}
\caption{
Left: The definition of $\K_\mathrm{s}(u,v;t,z)$.
Right: The jump for $\K_\mathrm{s}(u,v;t,z)$ with the jump contour $\Sigma$ and its orientation shown in blue.
}
\label{f:Ks}
\end{figure}
This definition implies in particular that $\K_\mathrm{s}(u,v;t,z)\to\I$ as $u+\ii v\to\infty$ because:
(i) $\M(u+\ii v;t,z)\to\I$ according to the conditions of RHP~\ref{rhp:M}; (ii) $\delta(u+\ii v)\to1$ according to \eqref{e:f-lambda} and \eqref{e:delta-def}; and
(iii) the off-diagonal entries of $\mathbf{T}_\mathrm{s}(u,v;t,z)$ and $\mathbf{T}_\mathrm{s}^{\dagger}(u,v;t,z)^{-1}$ decay rapidly at infinity in the corresponding half-planes (and actually vanish identically for $|v|>2$).
More generally, it is straightforward to confirm that the conditions of RHP~\ref{rhp:M} are equivalent to the following problem for $\K_\mathrm{s}(u,v;t,z)$.  Let $\Sigma$ denote the contour shown in the right-hand panel of Figure~\ref{f:Ks}.
\begin{rhdp}
\label{rhp:Ks}
Given $t\ge 0$ and $z\ge 0$, seek a $2\times 2$ matrix-valued function $\Real^2\ni (u,v)\mapsto \K_\mathrm{s}(u,v;t,z)$ that is continuous for $(u,v)\in\Real^2\setminus\Sigma$; that satisfies $\K_\mathrm{s}\to\I$ as $u+\ii v\to\infty$; that takes continuous boundary values on $\Sigma$ from each component of the complement related by the jump conditions $\K_\mathrm{s}^+(u,v;t,z)=\K_\mathrm{s}^-(u,v;t,z)\mathbf{J}_\mathrm{s}(u,v;t,z)$ where
\begin{equation}
\mathbf{J}_\mathrm{s}(u,v;t,z)\coloneq\begin{cases}
\delta(u+\ii v)^{\sigma_3}\mathbf{T}_\mathrm{s}^\dagger(u,v;t,z)\W(u+\ii v;t,z)^{-1}\delta(u+\ii v)^{-\sigma_3}\,,&
|u + \ii v| = \lambda_\circ\,,\quad v > 0\,,\\
\delta(u+\ii v)^{\sigma_3}\mathbf{T}_\mathrm{s}(u,v;t,z)^{-1}\W^\dagger(u+\ii v;t,z)\delta(u+\ii v)^{-\sigma_3}\,,& |u + \ii v| = \lambda_\circ\,,\quad v < 0\,,\\
(1 + |r(u)|^2)^{-\sigma_3}\,,& u\in(-\lambda_\circ,\lambda_\circ)\,,\quad v = 0\,;
\end{cases}
\label{e:Js-def}
\end{equation}
and that satisfies the following $\b{\partial}$ differential equation
\begin{equation}
\dbar \K_\mathrm{s}(u,v;t,z) = \K_\mathrm{s}(u,v;t,z)\mathbf{D}_\mathrm{s}(u,v;t,z)\,,\qquad (u,v)\in\Real^2\setminus\Sigma\,,
\label{e:dbar-equation}
\end{equation}
where the matrix $\mathbf{D}_\mathrm{s}(u,v;t,z)$ is given by
\begin{equation}
\begin{aligned}
\mathbf{D}_\mathrm{s}(u,v;t,z) \coloneq \begin{cases}
\bpm 0 & -\delta(u+\ii v)^2\ee^{2\ii\theta_\mathrm{s}(u+\ii v;t,z)}\dbar\left[B(v)\b{Q_N}(u,v)\right] \\ 0 & 0 \epm\,, & |u+\ii v| > \lambda_\circ\,,\quad v > 0\,,\\
\bpm 0 & 0 \\ \delta(u+\ii v)^{-2}\ee^{-2\ii\theta_\mathrm{s}(u+\ii v;t,z)}\dbar \left[B(v)Q_N(u,v)\right]  & 0 \epm\,, & |u+\ii v| > \lambda_\circ\,,\quad v < 0\,,\\
\mathbf0_{2\times2} \,, & |u+\ii v|<\lambda_\circ\,.
\end{cases}
\end{aligned}
\label{e:dbar-equation-D}
\end{equation}
\end{rhdp}
\begin{remark}
A Riemann-Hilbert-$\b{\partial}$ problem (RH$\b{\partial}$P) replaces the RHP condition of sectional analyticity with mere sectional continuity at the cost of an additional $\b{\partial}$ equation of the form \eqref{e:dbar-equation} as is necessary to restore well-posedness.
\end{remark}
Combining Theorem~\ref{thm:reconstruction} with the definition \eqref{e:Ks-def}, one can reconstruct the causal solution of the Cauchy problem \eqref{e:mbe} for propagation in a stable medium ($D_-=-1$) from $\K_\mathrm{s}(u,v;t,z)$ by
\begin{equation}
\label{e:Ks-reconstruction}
\begin{aligned}
q(t,z) & = -2\ii\lim_{u + \ii v\to\infty} (u+\ii v) K_{s,1,2}(u,v;t,z)\,,\\
\brho(t,z) & = -\lim_{u + \ii v\to0} \K_\mathrm{s}(u,v;t,z)\sigma_3\K_\mathrm{s}(u,v;t,z)^{-1}\,.
\end{aligned}
\end{equation}
\begin{remark}
Although $\K_\mathrm{s}(u,v;t,z)$ has a jump discontinuity across the segment $u\in (-\lambda_\circ,\lambda_\circ)$, $v=0$, its jump matrix is diagonal, so while the limit in the formula for $\brho(t,z)$ in \eqref{e:Ks-reconstruction} is necessary, it makes no difference whether it is taken from $v>0$ or from $v<0$.
\end{remark}

\subsection{Construction of a parametrix}
\label{s:stable-positive-parametrix}
We now construct a parametrix for $\K_\mathrm{s}(u,v;t,z)$ by the following steps:
\begin{itemize}
\item
we neglect the matrix $\mathbf{D}_\mathrm{s}(u,v;t,z)$ in the $\b{\partial}$ equation \eqref{e:dbar-equation} measuring deviation from analyticity, i.e., the parametrix will be sectionally analytic rather than merely sectionally continuous;
\item
we approximate the jump matrix $\mathbf{J}_\mathrm{s}(u,v;t,z)$, accounting for the fact that when $t\to+\infty$ with $z=o(t)$ the whole jump contour $\Sigma$ is small of size $\lambda_\circ\ll 1$.
\end{itemize}
Based on the first point, we will restore the complex variable $\lambda=u+\ii v$ and denote the parametrix for the solution of RH$\b{\partial}$P~\ref{rhp:Ks} by $\dot{\K}_\mathrm{s}(\lambda;t,z)$.

%
To accomplish the approximation mentioned in the second point, we begin with the following Lemma.
\begin{lemma}
\label{thm:fuv-difference}
Suppose $f(\cdot)\in C^k(\Real)$ and $0\le n_2 \le n_1 \le k$.
Recalling $u = \Re(\lambda)$ and $v = \Im(\lambda)$,
we have
\begin{equation}
\label{e:fuv-difference}
\sum_{n = 0}^{n_1-1} \frac{(\ii v)^n}{n!}f^{(n)}(u) - \sum_{n = 0}^{n_2-1} \frac{\lambda^n}{n!}f^{(n)}(0) = \O(\lambda^{n_2}),\quad\lambda\to 0.
\end{equation}
\end{lemma}
Lemma~\ref{thm:fuv-difference} is proved in Appendix~\ref{s:proof-fuv-difference}.  Now let $a_n$ denote the Taylor coefficients of $r(\lambda)$:
\begin{equation}
a_n:=\frac{r_0^{(n)}}{n!}
\label{e:avoid-factorials}
\end{equation}
and recalling the index $M\ge 0$ of the first nonzero moment of $r$, define
\begin{equation}
\label{e:DeltaM-def}
\Delta_M(\lambda)\coloneq 1 + |a_M|^2\lambda^{2M}\,,\qquad
\delta_0^\pm \coloneq \mathop{\lim_{\lambda\to 0}}_{\pm\Im(\lambda)>0}\delta(\lambda)\,.
\end{equation}
Note that $\Delta_M(\lambda)=\b{\Delta_M}(\lambda):=\b{\Delta_M(\b{\lambda})}$.
All roots of $\Delta_M(\lambda)$ lie on the circle of fixed radius $|a_M|^{-\frac{1}{M}}$, so $\Delta_M(\lambda)^{-1}$ is analytic on $\Sigma$ when its radius $\lambda_\circ$ is sufficiently small.
\begin{lemma}
\label{thm:qrdelta-leading}
Suppose that
$n\ge M+2$ and that
$|\lambda| = |u + \ii v| \le \lambda_\circ$.
Then as $\lambda_\circ\downarrow 0$,
\begin{equation}
\begin{aligned}
\label{e:qrdelta-leading}
1 + |r(u)|^2 & = \Delta_M(u) + \O(\lambda_\circ^{M+1})\,,\\
r(\lambda)
& = a_M\lambda^M + \O(\lambda_\circ^{M+1})\,,\quad\Im(\lambda)>0\,,\\
\delta(\lambda)^2 & = (\delta_0^\pm)^2 + \O(\lambda_\circ)\\
&=\ee^{-\ii\aleph}\Delta_M(\lambda)^{\pm 1}+\O(\lambda_\circ)\,,\quad\pm\Im(\lambda)>0\,,\\
Q_N(u,v)
 & = \frac{a_M(u+\ii v)^M}{\Delta_M(u+\ii v)} + \O(\lambda_\circ^{M+1})\,.
\end{aligned}
\end{equation}
\end{lemma}
\begin{proof}
The first equation follows from (real) Taylor expansion of $1+|r(u)|^2=1+r(u)\b{r(u)}$.
The second equation follows also from Taylor expansion and boundedness of $M+1$ derivatives of $r(\lambda)$ down to the real axis from the upper half-plane according to Lemma~\ref{lemma:reflection}.

For the third equation, we note that $\delta'(\lambda)$ is bounded from each half-plane in a neighborhood of $\lambda=0$; this follows because $\ln(1+|r(s)|^2)$ in \eqref{e:f-lambda} has a H\"older continuous derivative.  This establishes that $\delta(\lambda)^2=(\delta_0^\pm)^2+\O(\lambda_\circ)$ for $\pm\Im(\lambda)>0$.  Now the jump condition \eqref{e:delta-equation} taken at $\lambda=0$ implies that
$\delta_0^+/\delta_0^-=1+|r_0|^2$.  On one hand, if $M=0$ then this can be written exactly in the form $\delta_0^+/\delta_0^-=\Delta_M(\lambda)$.  On the other hand, if $M>0$ then $r_0=0$ and the same identity can be written $\delta_0^+/\delta_0^-=1=\Delta_M(\lambda)+\O(\lambda^{2M}) = \Delta_M(\lambda)+\O(\lambda_\circ)$.  So regardless of the index $M$, $\delta_0^+/\delta_0^-=\Delta_M(\lambda)+\O(\lambda_\circ)$. Next, recalling \eqref{e:f-lambda} and \eqref{e:delta-def}, and using $F_+(0)+F_-(0)=\ii\aleph$ where $\aleph$ is defined in Definition~\ref{def:aleph}, we have $\delta_0^+\delta_0^-=\ee^{-\ii\aleph}$.  Therefore $(\delta_0^\pm)^2 = \ee^{-\ii\aleph}\Delta_M(\lambda)^{\pm 1} +\O(\lambda_\circ)$.

In order to derive the last equation,
we first apply Lemma~\ref{thm:fuv-difference}, noting that $Q_N(u,v)$ has the form of the first term on the left-hand side of \eqref{e:fuv-difference}:
\begin{equation}
Q_N(u,v)
= \sum_{n = 0}^{M} \frac{\lambda^n}{n!}\frac{\dd^nR}{\dd u^n}(0)
+ \O(\lambda_\circ^{M + 1})
= \frac{\lambda^{M}}{M!}\frac{\dd^{M}R}{\dd u^{M}}(0)
+ \O(\lambda_\circ^{M + 1})\,.
\end{equation}
The reason that only the last term in the sum survives in the second equality is that all of the lower-order derivatives of $R(u)$ are proportional to derivatives of $r(u)$ of order strictly less than $M$, all of which vanish when $u=0$ (by definition of the index $M$).
If $M=0$, then the desired result holds.  For $M>0$,
we calculate the surviving term explicitly using \eqref{e:R-def} and the product rule:
\begin{equation}
\begin{aligned}
\frac{\dd^{M}R}{\dd u^{M}}(0)
 = \sum_{n = 0}^{M}\binom{M}{n} r^{(n)}(0)\frac{\dd^{M - n}(1 + |r(u)|^2)^{-1}}{\dd u^{M - n}}\bigg|_{u = 0}
 = \binom{M}{M} \frac{r_0^{(M)}}{1 + |r(0)|^2}
 = r_0^{(M)}\quad \text{(using $M>0$)}\,.
\end{aligned}
\end{equation}
Again, the reason that all terms except for the last one vanish is that
$r_0^{(n)} = 0$ for $0 \le n < M$.
Consequently, for $M>0$ we have
\begin{equation}
\label{e:q-leading0}
\begin{aligned}
Q_N(u,v)
 = \frac{r_0^{(M)}}{M!}\lambda^{M} + \O(\lambda_\circ^{M + 1})
 = a_M\lambda^M+\O(\lambda_\circ^{M+1})=\frac{a_M\lambda^M}{\Delta_M(\lambda)} + \O(\lambda_\circ^{M + 1})\,,
\end{aligned}
\end{equation}
which proves the desired statement.
(The insertion of $\Delta_M(\lambda)=1+\O(\lambda^{2M})$ in the denominator in the last step may seem artificial, but it is important in maintaining the unit-determinant condition on the jump matrices, and it is also useful in ensuring their compatibility at self-intersection points of the jump contour later on.)
\end{proof}

With the results of Lemma~\ref{thm:qrdelta-leading} in hand, we have the following \emph{analytic} approximation of the jump matrix $\mathbf{J}_\mathrm{s}(u,v;t,z)$ defined in \eqref{e:Js-def}:
\begin{equation}
\mathbf{J}_\mathrm{s}(u,v;t,z)=\dot{\mathbf{J}}_\mathrm{s}(u+\ii v;t,z)+\O(\lambda_\circ^{M+1})
\label{e:Js-approximation}
\end{equation}
holding uniformly for $u+\ii v\in\Sigma$, where
\begin{equation}
\dot{\mathbf{J}}_\mathrm{s}(\lambda;t,z)\coloneq
\begin{cases}
\bpm \Delta_M(\lambda)^{-1} & \b{a_M}\lambda^M\ee^{2\ii\theta_\mathrm{s}(\lambda;t,z)-\ii\aleph}\\
-\Delta_M(\lambda)^{-1}a_M\lambda^M\ee^{\ii\aleph-2\ii\theta_\mathrm{s}(\lambda;t,z)} & 1\epm,&\quad |\lambda|=\lambda_\circ,\quad \Im(\lambda)>0\,,\\
\bpm 1 & \Delta_M(\lambda)^{-1}\b{a_M}\lambda^M\ee^{2\ii\theta_\mathrm{s}(\lambda;z,t)-\ii\aleph}\\
-a_M\lambda^M\ee^{\ii\aleph-2\ii\theta_\mathrm{s}(\lambda;t,z)} & \Delta_M(\lambda)^{-1}\epm,&\quad |\lambda|=\lambda_\circ,\quad\Im(\lambda)<0\,,\\
\Delta_M(\lambda)^{-\sigma_3},&\quad \lambda\in (-\lambda_\circ,\lambda_\circ)\,.
\end{cases}
\label{e:dot-Js}
\end{equation}

Note that $\det(\dot{\mathbf{J}}_\mathrm{s}(\lambda;t,z))=1$.  We therefore arrive at the following specification of a parametrix.
\begin{rhp}[Parametrix for $\K_\mathrm{s}$]
\label{rhp:dotKs}
Given $t\ge 0$ and $z\ge 0$, seek a $2\times 2$ matrix-valued function $\lambda\mapsto\dot{\K}_\mathrm{s}(\lambda;t,z)$ that is analytic for $\lambda\in\Complex\setminus\Sigma$; that satisfies $\dot{\K}_\mathrm{s}\to\I$ as $\lambda\to\infty$; and that takes continuous boundary values on $\Sigma$ from each component of the complement related by the jump conditions $\dot{\K}_\mathrm{s}^+(\lambda;t,z)=\dot{\K}_\mathrm{s}^-(\lambda;t,z)\dot{\mathbf{J}}_\mathrm{s}(\lambda;t,z)$ where the jump matrix $\dot{\mathbf{J}}_\mathrm{s}(\lambda;t,z)$ is defined on $\Sigma$ by \eqref{e:dot-Js}.
\end{rhp}
While the conditions of RHP~\ref{rhp:dotKs} have been obtained from those of RH$\b{\partial}$P~\ref{rhp:Ks} by formal unjustified approximations, it is straightforward to check that
the jump matrix $\dot{\mathbf{J}}_\mathrm{s}(\lambda;t,z)$ satisfies the necessary Schwarz symmetry $\dot{\mathbf{J}}_\mathrm{s}(\b{\lambda};t,z)^{-1}=\dot{\mathbf{J}}_\mathrm{s}(\lambda;t,z)^\dagger$ for $\lambda\in\Sigma\setminus\Real$, that $\dot{\mathbf{J}}_\mathrm{s}(\lambda;t,z)^\dagger\dot{\mathbf{J}}_\mathrm{s}(\lambda;t,z)$ is positive definite for $\lambda\in\Sigma\cap\Real$, and that $\dot{\mathbf{J}}_\mathrm{s}(\lambda;t,z)$ satisfies the necessary consistency condition\footnote{i.e., that the clockwise product of jump matrices for arcs approaching a self-intersection point is the identity.} at the two self-intersection points $\lambda=\pm \lambda_\circ$ of $\Sigma$.  Therefore, by Zhou's vanishing lemma \cite{z1989}, RHP~\ref{rhp:dotKs} has a unique solution so the parametrix is well-defined.

Since for $\lambda_\circ$ sufficiently small, $\Delta_M(\lambda)$ is an analytic nonvanishing function on the disk $|\lambda|<\lambda_\circ$ satisfying $\Delta_M(0)=1$ for $M>0$ and $\Delta_0(0)=1+|r_0|^2$, we may define on this disk an analytic square root $\Delta_M(\lambda)^\frac{1}{2}$ by the condition $\Delta_M(0)^\frac{1}{2}>0$.  Using this function to transfer the jump from the interval $(-\lambda_\circ,\lambda_\circ)$ to the upper and lower semicircles $|\lambda|=\lambda_\circ$, conjugating off some constants, and rescaling the circle $|\lambda|=\lambda_\circ$ to a fixed size results in an explicit transformation:
\begin{equation}
\ddot{\K}_\mathrm{s}(k;t,z):=\begin{cases}
\ee^{\frac{1}{2}\ii[\arg(a_M)+\aleph]\sigma_3}\dot{\K}_\mathrm{s}(\lambda_\circ k;t,z)\ee^{-\frac{1}{2}\ii[\arg(a_M)+\aleph]\sigma_3},&\quad |k|>1,\\
\ee^{\frac{1}{2}\ii[\arg(a_M)+\aleph]\sigma_3}\dot{\K}_\mathrm{s}(\lambda_\circ k;t,z)\Delta_M(\lambda_\circ k)^{\frac{1}{2}\sigma_3}\ee^{-\frac{1}{2}\ii[\arg(a_M)+\aleph]\sigma_3},&\quad |k|<1,\quad\Im(k)>0,\\
\ee^{\frac{1}{2}\ii[\arg(a_M)+\aleph]\sigma_3}\dot{\K}_\mathrm{s}(\lambda_\circ k;t,z)\Delta_M(\lambda_\circ k)^{-\frac{1}{2}\sigma_3}\ee^{-\frac{1}{2}\ii[\arg(a_M)+\aleph]\sigma_3},&\quad |k|<1,\quad\Im(k)<0.
\end{cases}
\label{e:dotKs-ddotKs}
\end{equation}
Then, since combining \eqref{e:thetas-def}--\eqref{e:lambdao-def} yields $\theta_\mathrm{s}(\lambda_\circ k;t,z)=\tfrac{1}{2}x(k+k^{-1})$ where $x:=\sqrt{2tz}$, $\ddot{\K}_\mathrm{s}(k;t,z)$ is the unique solution of the following simplified RHP equivalent to RHP~\ref{rhp:dotKs} via the substitution \eqref{e:dotKs-ddotKs}:
\begin{rhp}[Modified parametrix for $\K_\mathrm{s}$]
Given $t\ge 0$ and $z\ge 0$, seek a $2\times 2$ matrix-valued function $k\mapsto\ddot{\K}_\mathrm{s}(k;t,z)$ that is analytic for $|k|\neq 1$; that satisfies $\ddot{\K}_\mathrm{s}\to\I$ as $k\to\infty$; and that takes continuous boundary values on $|k|=1$ from the interior and exterior related by the jump condition
\begin{equation}
\ddot{\K}_\mathrm{s}^+(k;t,z)=\ddot{\K}_\mathrm{s}^-(k;t,z)\bpm \Delta_M(\lambda_\circ k)^{-\frac{1}{2}} & \lambda_\circ^M |a_M|\Delta_M(\lambda_\circ k)^{-\frac{1}{2}} k^M\ee^{\ii x(k+k^{-1})}\\
-\lambda_\circ^M |a_M|\Delta_M(\lambda_\circ k)^{-\frac{1}{2}}k^M\ee^{-\ii x(k+k^{-1})} & \Delta_M(\lambda_\circ k)^{-\frac{1}{2}}\epm,\quad |k|=1,
\end{equation}
where $\lambda_\circ$ is defined in terms of $(t,z)$ by \eqref{e:lambdao-def} and $x=\sqrt{2tz}$.
\label{rhp:ddotKs}
\end{rhp}
To prove that the parametrix $\dot{\K}_\mathrm{s}(\lambda;t,z)$ is an accurate approximation to $\K_\mathrm{s}(u,v;t,z)$ when $t>0$ is large and $z=o(t)$, we will need to first prove that $\dot{\K}_\mathrm{s}(\lambda;t,z)$ is uniformly bounded in this limit; using \eqref{e:dotKs-ddotKs} and the fact that $\Delta_M(\lambda_\circ k)$ has a positive limit as $\lambda_\circ\to 0$ for $|k|<1$ it is sufficient to show instead that $\ddot{\K}_\mathrm{s}(k;t,z)$ is bounded.  In a different direction, for the parametrix $\dot{\K}_\mathrm{s}(\lambda;t,z)$ to be a useful approximation of $\K_\mathrm{s}(u,v;t,z)$, we will need to express it in terms of known functions (or equivalently do the same for $\ddot{\K}_\mathrm{s}(k;t,z)$).  We address both of these issues next.

\subsection{Properties of the modified parametrix:  $M=0$}
\label{s:Properties-of-Y}
When $M=0$, the jump matrix in RHP~\ref{rhp:ddotKs} becomes simpler because the only dependence on $(t,z)$ or $k$ enters via the exponential factors $\ee^{\pm\ii x(k+k^{-1})}$.  At the same time, the constants $|a_0|\Delta_0^{-\frac{1}{2}}=|r_0|/\sqrt{1+|r_0|^2}>0$ and $\Delta_0^{-\frac{1}{2}}=1/\sqrt{1+|r_0|^2}>0$ are respectively the sine and cosine of an angle $\eta\in (0,\tfrac{1}{2}\pi)$.
Indeed,
\begin{equation}
\ddot{\K}_\mathrm{s}(k;t,z)=\mathbf{Y}(\ii k;-\ii x,\arctan(|r_0|)),
\label{e:ddotK-Y}
\end{equation}
where $\mathbf{Y}(\Lambda;X,\eta)$ is the solution of
\begin{rhp}[Painlev\'e-III]
Given $X\in\Complex$ and $\eta\in (0,\tfrac{1}{2}\pi)$, seek a $2\times 2$ matrix-valued function $\Lambda\mapsto\mathbf{Y}(\Lambda;X,\eta)$ that is analytic for $|\Lambda|\neq 1$; that satisfies $\mathbf{Y}\to\I$ as $\Lambda\to\infty$; and that takes continuous boundary values on $|\Lambda|=1$ from the interior and exterior related by the jump condition
\begin{equation}
\mathbf{Y}^+(\Lambda;X,\eta)=\mathbf{Y}^-(\Lambda;X,\eta)\ee^{\ii\Theta(\Lambda,X)\sigma_3}\mathbf{E}(\eta)\ee^{-\ii\Theta(\Lambda,X)\sigma_3},
\end{equation}
where
\begin{equation}
\Theta(\Lambda,X)\coloneq \frac{1}{2}X(\Lambda-\Lambda^{-1})\quad\text{and}\quad
\mathbf{E}(\eta)\coloneq\bpm\cos(\eta) & \sin(\eta)\\-\sin(\eta) & \cos(\eta)\epm,
\end{equation}
and the unit circle $|\Lambda|=1$ has counterclockwise orientation.
\label{rhp:PIII}
\end{rhp}
It follows from Liouville's theorem that for given parameters $X$ and $\eta$ there is at most one solution of this problem, and it must have unit determinant.  Likewise, given $\eta\in (0,\tfrac{1}{2}\pi)$, it follows from analytic Fredholm theory that if there is a solution for any $X\in\Complex$ then the solution is meromorphic in $X$.  Since $\ddot{\K}_\mathrm{s}(k;t,z)$ exists for all $t\ge 0$ and $z\ge 0$, we deduce from \eqref{e:ddotK-Y} existence of $\mathbf{Y}(\Lambda;X,\eta)$ for all $X$ on the closed negative imaginary axis.

RHP~\ref{rhp:PIII} can easily be related to a RHP appearing in several recent papers on the topic of high-order solitons and rogue wave solutions of the focusing nonlinear Schr\"odinger equation.  For instance, comparing with the RHP satisfied by the matrix denoted $\mathbf{B}(\Lambda;X,T)$ in \cite[Eqn.\@ (4.14)]{bb2019}, one can check that (after making a suitable choice of the arbitrary radius of the circle across which $\mathbf{B}(\Lambda;X,T)$ experiences its jump discontinuity)
\begin{equation}
\mathbf{Y}(\Lambda;X,\eta)
 =
\begin{cases}
\ee^{\frac{1}{2}\ii\arg(c_2/c_1)\sigma_3}\sigma_3\mathbf{B}\left(-\frac{4\ii}{X}\Lambda;\frac{X^2}{8},0\right)\sigma_3\ee^{-\frac{1}{2}\ii\arg(c_2/c_1)\sigma_3}\,, & \quad
|\Lambda| > 1\,,\\
\ee^{\frac{1}{2}\ii\arg(c_2/c_1)\sigma_3}\sigma_3\mathbf{B}\left(-\frac{4\ii}{X}\Lambda;\frac{X^2}{8},0\right)\sigma_3\ee^{-\frac{1}{2}\ii\arg(c_1c_2)\sigma_3}\,, & \quad
|\Lambda| < 1\,,
\end{cases}
\label{e:map-to-RWIO}
\end{equation}
and $\tan(\eta) = |c_2|/|c_1|$,
where $(c_1,c_2)\in\Complex^2$ is a parameter vector for $\mathbf{B}(\Lambda;X,T)$.  The same RHP for a special case of $(c_1,c_2)$ appeared also in \cite{blm2020}.

\subsubsection{Isomonodromic interpretation of RHP~\ref{rhp:PIII}}
\label{s:Y-isomonodromy}
Comparing with \cite[Theorem 5.4]{fikn2006}, one sees that RHP~\ref{rhp:PIII} is a special case of the inverse monodromy problem for the Painlev\'e-III equation, where the Stokes matrices are all trivial and the formal monodromy parameters $\Theta_0$ and $\Theta_\infty$ about $\Lambda=0$ and $\Lambda=\infty$ respectively both vanish.  Indeed, for fixed $\eta$, setting $\boldsymbol{\Psi}(\Lambda,X):=\mathbf{Y}(\Lambda;X,\eta)\ee^{\ii\Theta(\Lambda,X)\sigma_3}$, one sees easily that the matrices
\begin{equation}
\boldsymbol{\Lambda}(\Lambda,X)\coloneq\frac{\partial\boldsymbol{\Psi}}{\partial\Lambda}(\Lambda,X)\boldsymbol{\Psi}(\Lambda,X)^{-1}\quad\text{and}\quad
\mathbf{X}(\Lambda,X)\coloneq\frac{\partial\boldsymbol{\Psi}}{\partial X}(\Lambda,X)\boldsymbol{\Psi}(\Lambda,X)^{-1}
\label{e:Lax-matrices-1}
\end{equation}
are both analytic for $\Lambda\in\Complex\setminus\{0\}$.  Moreover they are Laurent polynomials in $\Lambda$ of degrees $(0,2)$ and $(1,1)$ respectively, and their coefficients can be written explicitly in terms of the matrix function $\mathbf{Y}(\Lambda;X,\eta)$ as follows:
\begin{equation}
\begin{split}
\boldsymbol{\Lambda}(\Lambda,X)&=\frac{\ii X}{2}\sigma_3 + \bpm 0 & y(X)\\v(X) & 0\epm\frac{1}{\Lambda} +
\bpm \tfrac{1}{2}\ii X-\ii U(X) & \ii s(X)\\-\ii s(X)^{-1}U(X)(U(X)-X) & \ii U(X)-\tfrac{1}{2}\ii X\epm
\frac{1}{\Lambda^2}\\
\mathbf{X}(\Lambda,X)&=\frac{\ii}{2}\sigma_3\Lambda + \frac{1}{X}\bpm 0 & y(X)\\v(X) & 0\epm -\frac{1}{X}\bpm \tfrac{1}{2}\ii X-\ii U(X) & \ii s(X)\\-\ii s(X)^{-1}U(X)(U(X)-X) & \ii U(X)-\tfrac{1}{2}\ii X\epm
\frac{1}{\Lambda},
\end{split}
\label{e:Lax-matrices-2}
\end{equation}
where, indexing by the equivalent parameter $\omega=-3\cos(2\eta)$,
\begin{equation}
\begin{split}
y(X)=y(X;\omega)&\coloneq -\ii X\lim_{\Lambda\to\infty}\Lambda Y_{1,2}(\Lambda;X,\eta)\\
v(X)=y(X;\omega)&\coloneq \ii X\lim_{\Lambda\to\infty}\Lambda Y_{2,1}(\Lambda;X,\eta)\\
s(X)=s(X;\omega)&\coloneq -XY_{1,1}(0;X,\eta)Y_{1,2}(0;X,\eta)\\
U(X)=U(X;\omega)&\coloneq -XY_{1,2}(0;X,\eta)Y_{2,1}(0;X,\eta).
\end{split}
\label{e:Lax-potentials}
\end{equation}
With the forms \eqref{e:Lax-matrices-2} for $\boldsymbol{\Lambda}(\Lambda,X)$ and $\mathbf{X}(\Lambda,X)$, the equations \eqref{e:Lax-matrices-1} can be re-interpreted as a compatible first-order Lax system
on the unknown $\boldsymbol{\Psi}(\Lambda,X)$.  The compatibility condition for this Lax pair is equivalent to the statement that the functions \eqref{e:Lax-potentials} satisfy the following first-order nonlinear system:
\begin{equation}
\begin{split}
y'(X)&=-2s(X)\\
v'(X)&= 2Xs(X)^{-1}U(X)-2s(X)^{-1}U(X)^2\\
Xs'(X)&=s(X)-2Xy(X)+4y(X)U(X)\\
XU'(X)&=U(X)-2Xs(X)^{-1}y(X)U(X)+2s(X)^{-1}y(X)U(X)^2+2v(X)s(X).
\end{split}
\label{e:full-PIII-system}
\end{equation}
We note here that our parametrization of the matrices $\boldsymbol{\Lambda}(\Lambda,X)$ and $\mathbf{X}(\Lambda,X)$ differs from the Jimbo-Miwa parametrization used in \cite{fikn2006} (where $w(X)=-s(X)/U(X)$ is used in place of $s(X)$) as well as from the parametrization used in \cite{bms2018} (where $t(X)=U(X)/s(X)$ is used in place of $U(X)$).  However, for the MBE system it is more natural to work with both $U(X)$ and $s(X)$, which is why we have interpolated between these two parametrizations.

\subsubsection{Basic symmetries of RHP~\ref{rhp:PIII}}
\label{s:Y-symmetries}
It is easy to check that, if $\mathbf{Y}(\Lambda;X,\eta)$ is the solution of RHP~\ref{rhp:PIII} for some $X\in\Complex$ and $\eta\in (0,\tfrac{1}{2}\pi)$, then the matrices $\mathbf{Y}(-\Lambda;X,\eta)^{-\top}=\sigma_2\mathbf{Y}(-\Lambda;X,\eta)\sigma_2$ and $\mathbf{Y}(\b{\Lambda};\b{X},\eta)^{-\dagger}$ both satisfy the conditions of RHP~\ref{rhp:PIII} and hence by uniqueness are equal to $\mathbf{Y}(\Lambda;X,\eta)$.
Expanding the identity $\mathbf{Y}(\Lambda;X,\eta)=\mathbf{Y}(-\Lambda;X,\eta)^{-\top}$ for large $\Lambda$ and using \eqref{e:Lax-potentials} gives the identity
\begin{equation}
v(X)=-y(X).
\label{e:v-eliminate}
\end{equation}
Likewise expanding the identity $\mathbf{Y}(\Lambda;X,\eta)=\mathbf{Y}(\b{\Lambda};\b{X},\eta)^{-\dagger}$ for large $\Lambda$ gives
\begin{equation}
y(X)=-\b{v(\b{X})}=\b{y(\b{X})}
\end{equation}
(which also implies $s(X)=\b{s(\b{X})}$ since $y'(X)=-2s(X)$)
and evaluating at $\Lambda=0$ gives
\begin{equation}
U(X)=\b{U(\b{X})}.
\end{equation}

These are symmetries for fixed $\eta\in (0,\tfrac{1}{2}\pi)$.  Another useful symmetry relates solutions of RHP~\ref{rhp:PIII} for different values of $\eta$.  Indeed, by a similar uniqueness argument, if $\mathbf{Y}(\Lambda;X,\eta)$ is the solution of RHP~\ref{rhp:PIII} for some $X$ and $\eta$, then
\begin{equation}
\mathbf{Y}(\Lambda;-\ii X,\tfrac{1}{2}\pi-\eta)=\begin{cases}
\sigma_3\mathbf{Y}(-\ii\Lambda;X,\eta)\ee^{X\Lambda\sigma_3}\bpm 0 & -1\\1 & 0\epm\sigma_3,&\quad|\Lambda|<1\\
\sigma_3\mathbf{Y}(-\ii\Lambda;X,\eta)\ee^{X\Lambda^{-1}\sigma_3}\sigma_3,&\quad |\Lambda|>1,
\end{cases}
\label{e:symmetry-rotate-X}
\end{equation}
because the right-hand side satisfies the conditions of RHP~\ref{rhp:PIII} with the parameters $(X,\eta)$ replaced by $(-\ii X,\tfrac{1}{2}\pi-\eta)$.
Note that the mapping $\eta\mapsto \tfrac{1}{2}\pi-\eta$ corresponds to $|r_0|\mapsto |r_0|^{-1}$ or in terms of the parameter $\omega=-3\cos(2\eta)$, $\omega\mapsto -\omega$.  Expanding for small and large $\Lambda$ using \eqref{e:Lax-potentials}, we obtain from this symmetry the identities \eqref{e:PIII-rotate-X-identities}.
Since for all $\eta\in (0,\tfrac{1}{2}\pi)$, $\mathbf{Y}(\Lambda;X,\eta)$ exists for all $X$ on the closed negative imaginary axis, it follows from \eqref{e:symmetry-rotate-X} that also $\mathbf{Y}(\Lambda;X,\eta)$ exists for all positive real $X$.  In fact, since $\eta\mapsto \tfrac{1}{2}\pi-\eta$ is an involution, iteration of \eqref{e:symmetry-rotate-X} yields the identity $\mathbf{Y}(\Lambda;-X,\eta)=\mathbf{Y}(-\Lambda;X,\eta)$.  Therefore, for $\eta\in (0,\tfrac{1}{2}\pi)$ and $\Lambda$ with $|\Lambda|\neq 1$, $X\mapsto \mathbf{Y}(\Lambda;X,\eta)$ is analytic for $X^2\in\Real$.
Combining $\mathbf{Y}(\Lambda;-X,\eta)=\mathbf{Y}(-\Lambda;X,\eta)$ with \eqref{e:Lax-potentials} then also shows that
\begin{equation}
y(-X;\omega)=y(X;\omega),\quad s(-X;\omega)=-s(X;\omega),\quad\text{and}\quad U(-X;\omega)=-U(X;\omega).
\label{e:PIII-even-odd}
\end{equation}

Using \eqref{e:v-eliminate} to eliminate $v(X)$ from the system \eqref{e:full-PIII-system} shows that the functions $y(X)$, $s(X)$, and $U(X)$ solve the coupled system \eqref{e:coupled-PIII-0-0} presented in the introduction, and hence also that the function $u(X)=-y(X)/s(X)$ satisfies the Painlev\'e-III equation in the form \eqref{e:PIII}.  Next, we show how that system can be reduced to the self-similar Maxwell-Bloch system \eqref{e:coupled-self-similar}.

\subsubsection{Expansions of the functions $y(X)$, $s(X)$, $U(X)$, and $u(X)$ near $X=0$}
\label{s:expansions-near-zero}
Since $X\mapsto\mathbf{Y}(\Lambda;X,\eta)$ is analytic for $X^2\in\Real$, it follows from \eqref{e:Lax-potentials} that the functions $y(X)$, $s(X)$, and $U(X)$ are analytic at the origin $X=0$.  We now explain how to compute their Taylor expansions.

In particular, $X\mapsto \mathbf{Y}(\Lambda;X,\eta)$ is analytic at $X=0$, and moreover RHP~\ref{rhp:PIII} is trivial to solve explicitly when $X=0$:
\begin{equation}
\mathbf{Y}(\Lambda;0,\eta)=\begin{cases}\I,&\quad |\Lambda|>1,\\
\E(\eta),&\quad |\Lambda|<1.
\end{cases}
\label{e:at-X=0}
\end{equation}
Then, using \eqref{e:Lax-potentials} gives
\begin{equation}
y(0;\omega) = U(0;\omega) = s(0;\omega) = 0\,.
\end{equation}
It is straightforward to compute as many $X$-derivatives of $\mathbf{Y}(\Lambda;X,\eta)$ at $X=0$ as desired.  These derivatives solve an inhomogeneous form of RHP~\ref{rhp:PIII} that we solve recursively as follows.
Letting $\mathbf{V}(\Lambda,X)\coloneq \ee^{\ii\Theta(\Lambda,X)\sigma_3}\E(\eta)\ee^{-\ii\Theta(\Lambda,X)\sigma_3}$ denote the jump matrix in RHP~\ref{rhp:PIII}, we introduce the notation
\begin{equation}
\V_n(\Lambda,X):=\frac{\partial^n\V}{\partial X^n}(\Lambda,X)=
\sin(\eta)(\ii(\Lambda-\Lambda^{-1}))^n\sigma_3^{n+1}\ee^{\ii\Theta(\Lambda,X)\sigma_3}\sigma_1\ee^{-\ii\Theta(\Lambda,X)\sigma_3},\quad n=1,2,3,\dots
\end{equation}
Then also
\begin{equation}
\V_n(\Lambda,X)\V(\Lambda,X)^{-1}
=
\sin(\eta)(\ii(\Lambda-\Lambda^{-1}))^n\sigma_3^{n+1}\ee^{\ii\Theta(\Lambda,X)\sigma_3}\sigma_1\E(\eta)^{-1}\ee^{-\ii\Theta(\Lambda,X)\sigma_3},\quad n=1,2,3,\dots
\label{e:V-derivs}
\end{equation}
Define a sequence of matrix functions $\F_n(\Lambda;X,\eta)$ in terms of derivatives of the solution of RHP~\ref{rhp:PIII} by
\begin{equation}
\F_n(\Lambda;X,\eta):=\frac{\partial^n\mathbf{Y}}{\partial X^n}(\Lambda;X,\eta)\mathbf{Y}(\Lambda;X,\eta)^{-1},\quad n=0,1,2,\dots,
\end{equation}
so in particular $\F_0(\Lambda;X,\eta)=\I$.  For $n\ge 1$, $\Lambda\mapsto\F_n(\Lambda;X,\eta)$ is analytic for $|\Lambda|\neq 1$, satisfies the normalization condition $\F_n(\infty;X,\eta)=\mathbf{0}$, and the jump condition
\begin{equation}
\F_n^+(\Lambda;X,\eta)-\F_n^-(\Lambda;X,\eta)=\sum_{k=1}^n\binom{n}{k}\F_{n-k}^-(\Lambda;X,\eta)\mathbf{Y}^-(\Lambda;X,\eta)\V_k(\Lambda,X)\V(\Lambda,X)^{-1}\mathbf{Y}^-(\Lambda;X,\eta)^{-1}
\end{equation}
for $|\Lambda|=1$ with counterclockwise orientation.  This immediately leads to a recursive formula for $\F_n(\Lambda;X,\eta)$ valid for $|\Lambda|\neq 1$:
\begin{equation}
\F_n(\Lambda;X,\eta)=\sum_{k=1}^n\binom{n}{k}\frac{1}{2\pi\ii}\oint_{|\mu|=1}
\frac{\F_{n-k}^-(\mu;X,\eta)\mathbf{Y}^-(\mu;X,\eta)\V_k(\mu,X)\V(\mu,X)^{-1}\mathbf{Y}^-(\mu;X,\eta)^{-1}}{\mu-\Lambda}\dd\mu.
\end{equation}
By \eqref{e:at-X=0} we have $\mathbf{Y}^-(\mu;0,\eta)=\mathbb{I}$ while $\V_k(\mu,0)\V(\mu,0)^{-1}$ is given in \eqref{e:V-derivs}.  Therefore, for $X=0$ this simplifies as follows:
\begin{equation}
\F_n(\Lambda;0,\eta)=\sin(\eta)\sum_{k=1}^n\binom{n}{k}\frac{1}{2\pi\ii}\oint_{|\mu|=1}\frac{(\mu-\mu^{-1})^k}{\mu-\Lambda}\ii^k\F_{n-k}^-(\mu;0,\eta)\dd\mu\,\sigma_3^{k+1}\sigma_1\E(\eta)^{-1},\quad |\Lambda|\neq 1.
\end{equation}
It is convenient to rescale by $\F_n(\Lambda;0,\eta)=\ii^{n}\G_n(\Lambda;0,\eta)$, giving the modified recursion
\begin{equation}
\G_n(\Lambda;0,\eta)=\sin(\eta)\sum_{k=1}^n\binom{n}{k}\frac{1}{2\pi\ii}\oint_{|\mu|=1}\frac{(\mu-\mu^{-1})^k}{\mu-\Lambda}\G_{n-k}^-(\mu;0,\eta)\dd\mu\,\sigma_3^{k+1}\sigma_1\E(\eta)^{-1}\,,\quad
|\Lambda|\neq 1.
\end{equation}
So, using $\G_0^-(\mu;0,\eta)=\F_0^-(\mu;0,\eta)=\I$ gives
\begin{equation}
\G_1(\Lambda;0,\eta) =
\begin{cases}
\sin(\eta)\sigma_1\E(\eta)^{-1}\Lambda^{-1} ,&\quad |\Lambda|>1,\\
\sin(\eta)\sigma_1\E(\eta)^{-1}\Lambda,&\quad |\Lambda|<1.
\end{cases}
\end{equation}
In particular, the boundary value from $|\mu|>1$ is
$\G_1^-(\mu;0,\eta)=\sin(\eta)\sigma_1\E(\eta)^{-1}\mu^{-1}$, which then implies
\begin{equation}
\G_2(\Lambda,0) =
\begin{cases}
[2\sin^2(\eta)\I - \sin(\eta)\sigma_3\sigma_1\E(\eta)^{-1}]\Lambda^{-2},&\quad|\Lambda|>1,\\
[2\sin^2(\eta)\I - 2\sin(\eta)\sigma_3\sigma_1\E(\eta)^{-1}]+\sin(\eta)\sigma_3\sigma_1\E(\eta)^{-1}\Lambda^2,&\quad|\Lambda|<1.
\end{cases}
\end{equation}
In particular,
$\G_2^-(\mu;0,\eta)=[2\sin^2(\eta)\I - \sin(\eta)\sigma_3\sigma_1\E(\eta)^{-1}]\mu^{-2}$.
Using this and the further identity
\begin{equation}
\E(\eta)^{-1}\sigma_3\sigma_1 = \sigma_3\sigma_1\E(\eta)^{-1},
\end{equation}
gives
\begin{equation}
\G_3(\Lambda;0,\eta)=
\begin{cases}
\A(\eta)\Lambda^{-3} + \mathbf{B}(\eta)
\Lambda^{-1},&\quad|\Lambda|>1\\
\C(\eta)\Lambda + \mathbf{D}(\eta)\Lambda^3,&\quad|\Lambda|<1,
\end{cases}
\end{equation}
where
\begin{equation}
\mathbf{B}(\eta)
\coloneq -6\sin^3(\eta)\sigma_1\E(\eta)^{-1} + 3\sin^2(\eta)\sigma_3 - 6\sin^2(\eta)\sigma_3\E(\eta)^{-2} - 3\sin(\eta)\sigma_1\E(\eta)^{-1}\,,
%
\end{equation}
and
$\A(\eta)$, $\C(\eta)$, and $\mathbf{D}(\eta)$ are unneeded matrix coefficients.
Putting the results together so far, the Taylor expansion of the first moment at $\Lambda=\infty$ is
\begin{equation}
\begin{split}
\lim_{\Lambda\to\infty}\Lambda(\mathbf{Y}(\Lambda;X,\eta)-\I)&=
\lim_{\Lambda\to\infty}\Lambda\left(
\mathbf{Y}(\Lambda;0,\eta)-\I +\frac{\partial\mathbf{Y}}{\partial X}(\Lambda;0,\eta)X +\frac{1}{2}\frac{\partial^2\mathbf{Y}}{\partial X^2}(\Lambda;0,\eta)X^2\right.\\
&\qquad\qquad{}\left. +\frac{1}{6}\frac{\partial^3\mathbf{Y}}{\partial X^3}(\Lambda;0,\eta) X^3 + \O(X^4\Lambda^{-1})\right)\\
&=\lim_{\Lambda\to\infty}\Lambda\left(\F_1(\Lambda;0,\eta)X +\frac{1}{2}\F_2(\Lambda;0,\eta)X^2 +\frac{1}{6}\F_3(\Lambda;0,\eta)X^3+\O(X^4\Lambda^{-1})\right)\\
&=  \ii\sin(\eta)\sigma_1\E(\eta)^{-1}X -\frac{\ii}{6} \mathbf{B}(\eta)X^3 + \O(X^4),\quad X\to 0.
\end{split}
\end{equation}
In particular, it follows from \eqref{e:Lax-potentials} and the even symmetry of $y(X;\omega)$ in \eqref{e:PIII-even-odd} that
\begin{equation}
\begin{split}
y(X;\omega)&=
-\ii X \lim_{\Lambda\to\infty}\Lambda Y_{1,2}(\Lambda;X,\eta) \\
&= \sin(\eta)\cos(\eta)X^2 +\frac{1}{2}\sin(\eta)\cos(\eta)[\cos^2(\eta)-\sin^2(\eta)]X^4 + \O(X^6),\quad X\to 0,
\end{split}
\end{equation}
which matches \eqref{e:y-series} with
\begin{equation}
y_2:=y''(0;\omega)=2\sin(\eta)\cos(\eta)=\sin(2\eta)=\sqrt{1-\cos^2(2\eta)}=\sqrt{1-\left(\frac{\omega}{3}\right)^2}.
\end{equation}
Also, evaluating at $\Lambda=0$, the Taylor expansion of $X\mapsto\mathbf{Y}(0;X,\eta)$ reads
\begin{equation}
\begin{split}
\mathbf{Y}(0;X,\eta)&=\mathbf{Y}(0;0,\eta) + \frac{\partial\mathbf{Y}}{\partial X}(0;0,\eta)X +\frac{1}{2}
\frac{\partial^2\mathbf{Y}}{\partial X^2}(0;0,\eta)X^2 +\frac{1}{6}\frac{\partial^3\mathbf{Y}}{\partial X^3}(0;0,\eta)X^3
+\O(X^4)\\
&=\E(\eta)+[\sin(\eta)\sigma_3\sigma_1-\sin^2(\eta)\E(\eta)]X^2 + \O(X^4),\quad X\to 0.
\end{split}
\end{equation}
In particular, it follows from \eqref{e:Lax-potentials} and the odd symmetry of $s(X;\omega)$ and $U(X;\omega)$ in \eqref{e:PIII-even-odd} that
\begin{equation}
\begin{split}
s(X;\omega)&=-XY_{1,1}(0;X,\eta)Y_{1,2}(0;X,\eta)\\
&=-\sin(\eta)\cos(\eta)X-\sin(\eta)\cos(\eta)[\cos^2(\eta)-\sin^2(\eta)]X^3+\O(X^5),\quad X\to 0,
\end{split}
\end{equation}
and that
\begin{equation}
\begin{split}
U(X;\omega)&=-XY_{1,2}(0;X,\eta)Y_{2,1}(0;X,\eta)\\
&=\sin^2(\eta)X+2\sin^2(X)\cos^2(X)X^3+\O(X^5),\quad X\to 0.
\end{split}
\end{equation}
These match the series \eqref{e:s-series} and \eqref{e:U-series} respectively.  Finally, the function $u(X;\omega)=-y(X;\omega)/s(X;\omega)$ solving the Painlev\'e-III equation \eqref{e:PIII} is easily seen to be analytic at $X=0$ with the Taylor series \eqref{e:u-series}.

We can also use these expansions to evaluate the $X$-independent first integral $J$ given in \eqref{e:J-constant} of the system \eqref{e:coupled-PIII-0-0} by computing its value at $X=0$:
\begin{equation}
J=\frac{U(X;\omega)(U(X;\omega)-X)}{s(X;\omega)^2} = \lim_{X\to 0}\frac{U(X;\omega)(U(X;\omega)-X)}{s(X;\omega)^2}=-1.
\end{equation}
As pointed out in the introduction, the fact that $J=-1$ makes the system \eqref{e:coupled-PIII-0-0} equivalent to the self-similar Maxwell-Bloch system \eqref{e:coupled-self-similar}.

\subsubsection{Boundedness of the solution of RHP~\ref{rhp:PIII}}
\label{s:Boundedness-of-RHP-PIII}
Let $\eta\in (0,\tfrac{1}{2}\pi)$ be fixed but arbitrary.  By analyticity of $X\mapsto\mathbf{Y}(\Lambda;X,\eta)$ on the coordinate axes and the normalization $\mathbf{Y}\to\mathbb{I}$ as $\Lambda\to\infty$, it is clear that for every $L>0$ however large there is some constant $C=C(L)>0$ such that
\begin{equation}
\sup_{-L^2\le X^2\le L^2}\sup_{|\Lambda|\neq 1}\|\mathbf{Y}(\Lambda;X,\eta)\|\le C(L),
\label{e:supremum-bounded}
\end{equation}
i.e., for $X$ bounded on the real and imaginary axes, $\mathbf{Y}(\Lambda;X,\eta)$ is uniformly bounded on the complex $\Lambda$-plane.  Aside from the expansions for small $X$ described in Section~\ref{s:expansions-near-zero}, there is no further simplification of the Painlev\'e-III functions $y(X;\omega)$, $s(X;\omega)$, and $U(X;\omega)$ that takes place for bounded $X$.

For the purposes of application to the analysis of $\K_\mathrm{s}(u,v;t,z)$, it has to be proved that for arbitrary fixed $\eta\in (0,\tfrac{1}{2}\pi)$, $\mathbf{Y}(\Lambda;X,\eta)$ is bounded uniformly with respect to $X$ on the whole negative imaginary axis.  In light of the symmetry \eqref{e:symmetry-rotate-X}, this will automatically imply that the matrix
\begin{equation}
\widetilde{\mathbf{Y}}(\Lambda;X,\eta):=\begin{cases}\mathbf{Y}(\Lambda;X,\eta)\ee^{\ii X\Lambda\sigma_3},&\quad |\Lambda|<1\\
\mathbf{Y}(\Lambda;X,\eta)\ee^{-\ii X\Lambda^{-1}\sigma_3},&\quad|\Lambda|>1
\end{cases}
\label{e:tilde-Y}
\end{equation}
will, for arbitrary fixed $\eta\in (0,\tfrac{1}{2}\pi)$, be bounded uniformly with respect to $X$ on the whole positive real axis.  It is exactly this latter bound that will be needed to analyze an analogous matrix $\K_\mathrm{u}(u,v;t,z)$ that will be introduced to study the behavior of solutions in an initially-unstable medium in Section~\ref{s:unstable-positive} below.  Based on the identification \eqref{e:map-to-RWIO}, the following arguments will generalize the special case of $\eta=\tfrac{1}{4}\pi$ or $\omega=0$, which was analyzed for large $X$ in \cite[Section 4.1]{blm2020}.

Referring to the regions of the complex $k$-plane indicated in the left-hand panel of Figure~\ref{f:M3},
\begin{figure}[h]
    \centering
    \includegraphics[scale = 0.53]{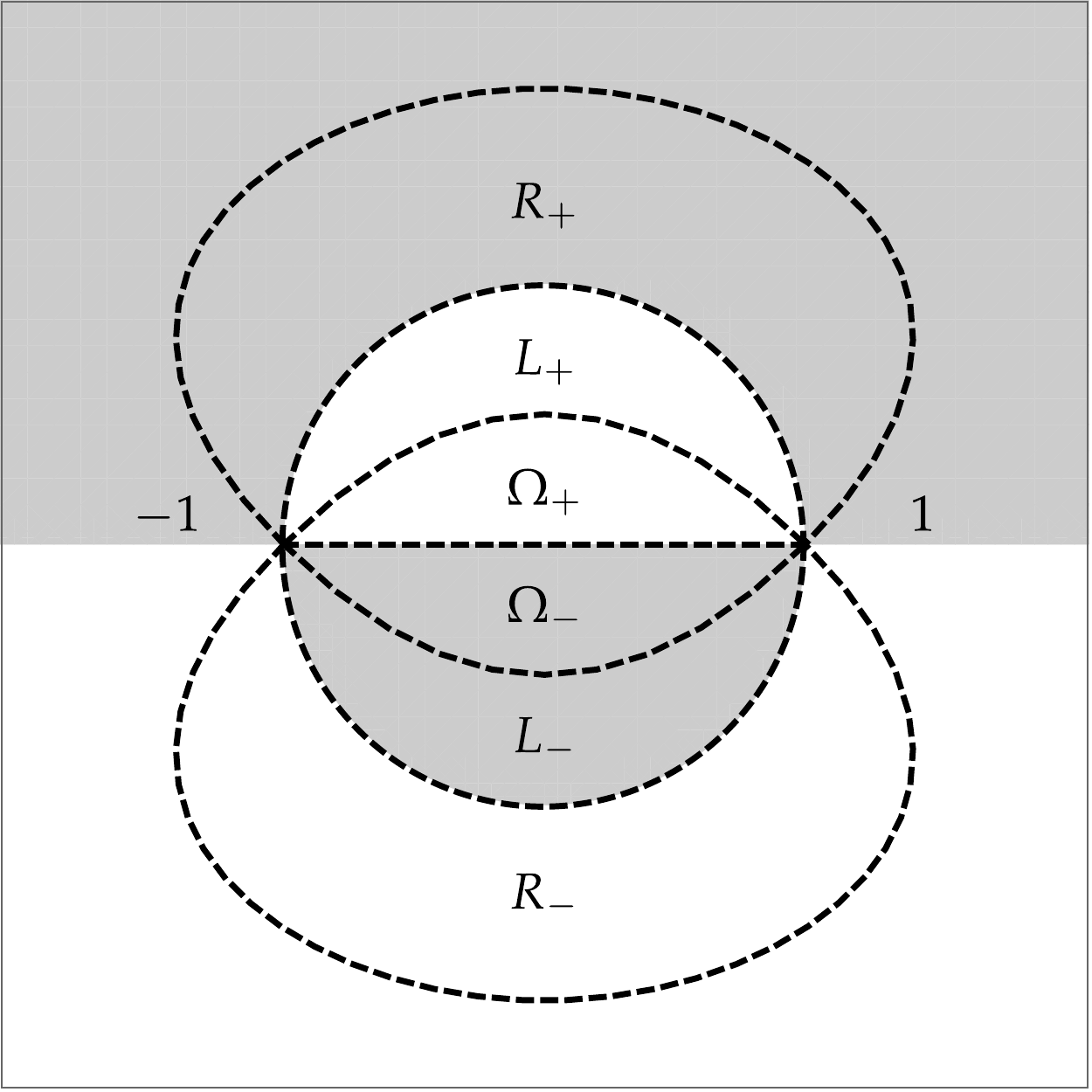}\quad
    \includegraphics[scale = 0.53]{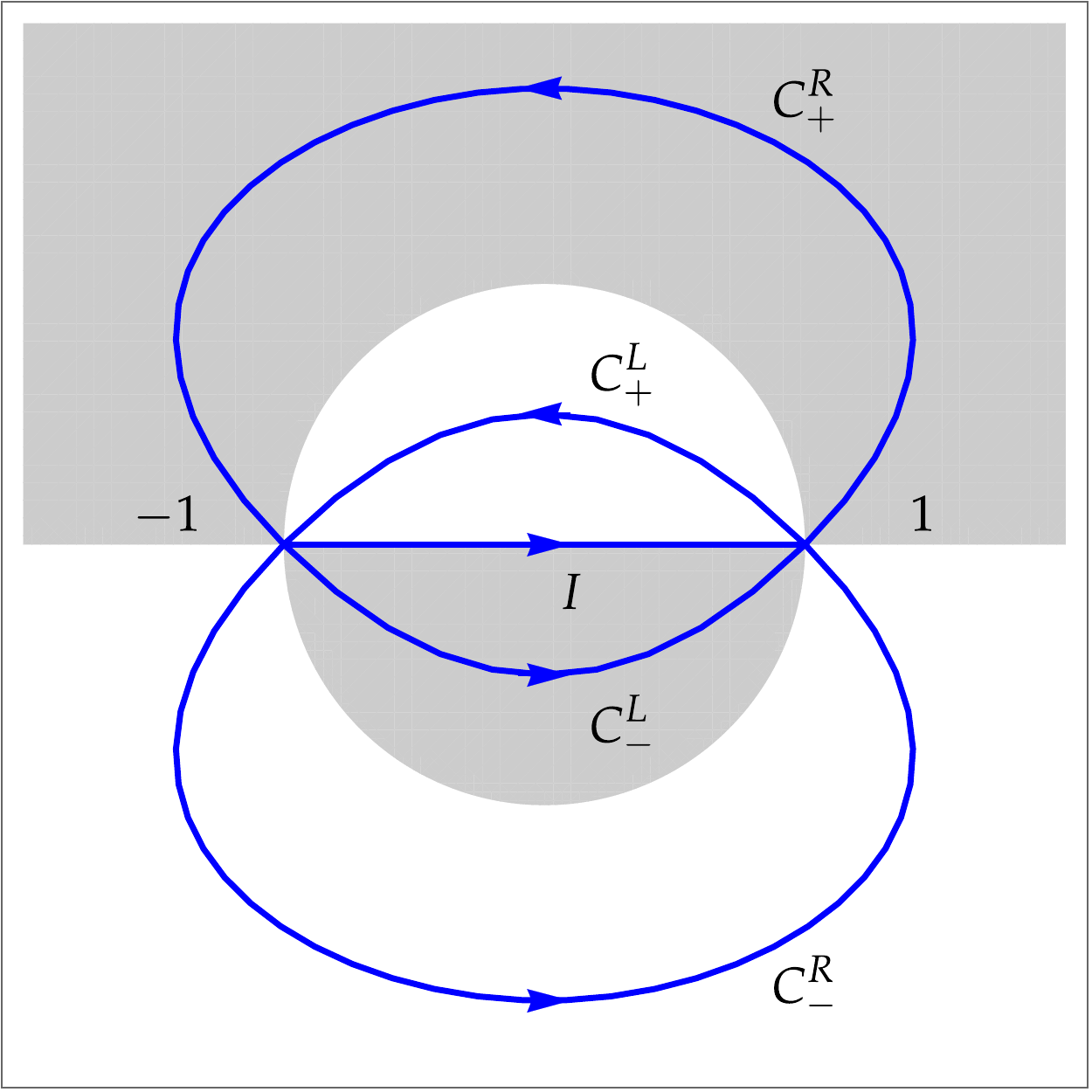}
    \caption{Left:  the definitions of the regions in the substitution \eqref{e:Y-to-Z-first}--\eqref{e:Y-to-Z-last} superimposed on the sign chart of $\Re(\ii x(k+k^{-1}))$.  Right:  the arcs of the jump contour for $\mathbf{Z}(k;x,\eta)$.
    }
    \label{f:M3}
\end{figure}
we define a new unknown from $\mathbf{Y}(\Lambda;-\ii x,\eta)$, $x>0$, as follows.
\begin{equation}
\mathbf{Z}(k;x,\eta)\coloneq \mathbf{Y}(\ii k;-\ii x,\eta)\bpm 1& \tan(\eta)\ee^{\ii x(k+k^{-1})}\\0 & 1\epm,\quad k\in R_+,
\label{e:Y-to-Z-first}
\end{equation}
\begin{equation}
\mathbf{Z}(k;x,\eta)\coloneq \mathbf{Y}(\ii k;-\ii x,\eta)\cos(\eta)^{\sigma_3}\bpm 1&0\\
\sin(\eta)\cos(\eta)\ee^{-\ii x(k+k^{-1})}&1\epm,\quad k\in L_+,
\end{equation}
\begin{equation}
\mathbf{Z}(k;x,\eta)\coloneq \mathbf{Y}(\ii k;-\ii x,\eta)\cos(\eta)^{\sigma_3},\quad k\in\Omega_+,
\end{equation}
\begin{equation}
\mathbf{Z}(k;x,\eta)\coloneq\mathbf{Y}(\ii k;-\ii x,\eta)\cos(\eta)^{-\sigma_3},\quad k\in\Omega_-,
\end{equation}
\begin{equation}
\mathbf{Z}(k;x,\eta)\coloneq\mathbf{Y}(\ii k;-\ii x,\eta)\cos(\eta)^{-\sigma_3}\bpm 1&-\sin(\eta)\cos(\eta)\ee^{\ii x(k+k^{-1})}\\ 0&1\epm,\quad k\in L_-,
\end{equation}
\begin{equation}
\mathbf{Z}(k;x,\eta)\coloneq\mathbf{Y}(\ii k;-\ii x,\eta)\bpm 1&0\\-\tan(\eta)\ee^{-\ii x(k+k^{-1})}&1\epm,\quad k\in R_-,
\label{e:Y-to-Z-last}
\end{equation}
and elsewhere we set $\mathbf{Z}(k;x,\eta)\coloneq\mathbf{Y}(\ii k;-\ii x,\eta)$.  This matrix is analytic on the complement of the contour illustrated in the right-hand panel of Figure~\ref{f:M3}, and due to the sign structure of $\Re(\ii x(k+k^{-1}))$, $\mathbf{Z}(k;x,\eta)$ is bounded uniformly with respect to $x>0$ and $k\in\mathbb{C}$ if and only if $\mathbf{Y}(\ii k;-\ii x,\eta)$ is.  The jump conditions satisfied by $\mathbf{Z}(k;x,\eta)$ across the arcs of its jump contour take the form $\mathbf{Z}^+(k;x,\eta)=\mathbf{Z}^-(k;x,\eta)\mathbf{V}_\mathbf{Z}(k;x,\eta)$, where the jump matrix $\mathbf{V}_\mathbf{Z}(k;x,\eta)$ is defined on the various arcs of jump contour as
as follows.
\begin{equation}
\mathbf{V}_\mathbf{Z}(k;x,\eta)\coloneq\bpm 1&\tan(\eta)\ee^{\ii x(k+k^{-1})}\\0&1\epm,\quad k\in C^R_+,
\label{e:Z-jump-first}
\end{equation}
\begin{equation}
\mathbf{V}_\mathbf{Z}(k;x,\eta)\coloneq\bpm 1&0\\-\sin(\eta)\cos(\eta)\ee^{-\ii x(k+k^{-1})} & 1\epm,\quad k\in C^L_+,
\end{equation}
\begin{equation}
\mathbf{V}_\mathbf{Z}(k;x,\eta)\coloneq\cos(\eta)^{2\sigma_3},\quad k\in I,
\end{equation}
\begin{equation}
\mathbf{V}_\mathbf{Z}(k;x,\eta)\coloneq\bpm 1&\sin(\eta)\cos(\eta)\ee^{\ii x(k+k^{-1})}\\0 & 1\epm,\quad k\in C_-^L,\quad\text{and}
\end{equation}
\begin{equation}
\mathbf{V}_\mathbf{Z}(k;x,\eta)\coloneq\bpm 1&0\\-\tan(\eta)\ee^{-\ii x(k+k^{-1})} & 1\epm,\quad k\in C_-^R.
\label{e:Z-jump-last}
\end{equation}
These jump matrices are all exponentially small perturbations of the identity except on the interval $I=[-1,1]$ and near its endpoints.

We deal with the jumps that are not close to identity by building a parametrix for $\mathbf{Z}(k;x,\eta)$.  For an outer parametrix, we solve the jump on $I$ exactly and satisfy the normalization condition $\mathbf{Z}(k;x,\eta)\to\I$ as $k\to\infty$ by defining
\begin{equation}
\dot{\mathbf{Z}}_{\mathrm{out}}(k;\eta):=\left(\frac{k-1}{k+1}\right)^{\ii \nu\sigma_3},\quad
\nu\coloneq -\frac{\ln(\cos(\eta))}{\pi}>0,\quad k\in \Complex\setminus I.
\label{e:Z-dot-out}
\end{equation}
Here the power function is defined using the principal branch, which is cut precisely in $I=[-1,1]$ and yields the desired asymptotic as $k\to\infty$.

The endpoints $k=\pm 1$ of $I$ are the two saddle points of the phase $k+k^{-1}$.  Let $y_{-1}(k)$ and $y_1(k)$ denote conformal coordinates defined near $k=-1$ and $k=1$ respectively:
\begin{equation}
\begin{split}
y_{-1}(k)&\coloneq (-k)^{\frac{1}{2}}-(-k)^{-\frac{1}{2}},\quad |k+1|\le\tfrac{1}{2},\\
y_1(k)&\coloneq k^{\frac{1}{2}}-k^{-\frac{1}{2}},\quad |k-1|\le\tfrac{1}{2},
\end{split}
\end{equation}
in which the principal branch square roots are meant, guaranteeing the analyticity of the coordinates.  Note that $y_{-1}(-1)=y_1(1)=0$.
Then, the exponent function appearing in \eqref{e:Z-jump-first}--\eqref{e:Z-jump-last} can be written in terms of the conformal coordinates as follows:
\begin{equation}
x(k+k^{-1}) = \begin{cases}-2x-\zeta^2,&\quad \zeta\coloneq \sqrt{x}y_{-1}(k),\quad |k+1|\le\tfrac{1}{2}\\
2x+\zeta^2,&\quad \zeta\coloneq\sqrt{x}y_{1}(k),\quad |k-1|\le\tfrac{1}{2}.
\end{cases}
\end{equation}
The outer parametrix can conveniently be expressed in terms of the coordinates $y_{-1}(k)$ and $y_1(k)$ as follows:
\begin{equation}
\dot{\mathbf{Z}}_\mathrm{out}(k;\eta)=\begin{cases}
\displaystyle x^{\frac{1}{2}\ii\nu\sigma_3}\left(\frac{1-k}{(-k)^{\frac{1}{2}}}\right)^{\ii\nu\sigma_3}\zeta^{-\ii\nu\sigma_3},&\quad \zeta\coloneq \sqrt{x}y_{-1}(k),\quad |k+1|\le\tfrac{1}{2},\\
\displaystyle x^{-\frac{1}{2}\ii\nu\sigma_3}\left(\frac{k^\frac{1}{2}}{k+1}\right)^{\ii\nu\sigma_3}\zeta^{\ii\nu\sigma_3},&\quad \zeta\coloneq\sqrt{x}y_1(k),\quad |k-1|\le \tfrac{1}{2}.
\end{cases}
\label{e:Zdot-out-zeta}
\end{equation}
Here, the factors to the left of $\zeta^{\pm\ii\nu\sigma_3}$ are analytic nonvanishing functions in the indicated disks taking the values $(2\sqrt{x})^{\mp\ii\nu\sigma_3}$ at $k=\pm 1$.  We now assume that the jump contours $C_L^\pm$ and $C_R^\pm$ are modified if necessary within the disks $|k\pm 1|\le\tfrac{1}{2}$ to lie along the rays $\arg(y_{\mp 1}(k))\in\{-\tfrac{3}{4}\pi,-\tfrac{1}{4}\pi,\tfrac{1}{4}\pi,\tfrac{3}{4}\pi\}$.  Then, in terms of the rescaled conformal coordinates $\zeta=\sqrt{x}y_{-1}(k)$ and $\zeta=\sqrt{x}y_1(k)$, the matrices $\sigma_3\ee^{\ii x\sigma_3}\mathbf{Z}(k;x,\eta)\ee^{-\ii x\sigma_3}\sigma_3$ and $\sigma_1\ee^{-\ii x\sigma_3}\mathbf{Z}(k;x,\eta)\ee^{\ii x\sigma_3}\sigma_1$ are, within the respective disks $|k+ 1|\le\tfrac{1}{2}$ and $|k-1|\le\tfrac{1}{2}$, analytic except on the same four rays along which they satisfy exactly the same jump conditions.  Those jump conditions are in turn the same as those of a matrix function $\zeta\mapsto\P(\zeta;\eta)$ that satisfies the following conditions:  $\P(\zeta;\eta)$ is analytic in five infinite sectors of the complex $\zeta$-plane bounded by the rays $\arg(\zeta)=\pm\tfrac{1}{4}\pi$, $\arg(\zeta)=\pm\tfrac{3}{4}\pi$, and $\arg(-\zeta)=0$; $\P(\zeta;\eta)$ takes continuous boundary values from each sector that are related along the five excluded rays by the jump conditions
\begin{equation}
\P^+(\zeta;\eta)=\P^-(\zeta;\eta)\bpm 1&0\\\tan(\eta)\ee^{\ii\zeta^2} & 1\epm,\quad \arg(\zeta)=\tfrac{1}{4}\pi,
\label{e:PC-jump-first}
\end{equation}
\begin{equation}
\P^+(\zeta;\eta)=\P^-(\zeta;\eta)\bpm 1&\tan(\eta)\ee^{-\ii\zeta^2}\\0 & 1\epm,\quad\arg(\zeta)=-\tfrac{1}{4}\pi,
\end{equation}
\begin{equation}
\P^+(\zeta;\eta)=\P^-(\zeta;\eta)\bpm 1&\sin(\eta)\cos(\eta)\ee^{-\ii\zeta^2}\\0&1\epm,\quad\arg(\zeta)=\tfrac{3}{4}\pi,
\end{equation}
\begin{equation}
\P^+(\zeta;\eta)=\P^-(\zeta;\eta)\bpm 1&0\\\sin(\eta)\cos(\eta)\ee^{\ii\zeta^2} & 1\epm,\quad \arg(\zeta)=-\tfrac{3}{4}\pi,\quad\text{and}
\end{equation}
\begin{equation}
\P^+(\zeta;\eta)=\P^-(\zeta;\eta)\cos(\eta)^{-2\sigma_3},\quad \arg(-\zeta)=0,
\label{e:PC-jump-last}
\end{equation}
in which all rays are taken to be oriented in the direction of increasing $\Re(\zeta)$
as indicated in the left-hand panel of Figure~\ref{f:PE1};
\begin{figure}[h]
    \centering
    \includegraphics[scale = 0.56]{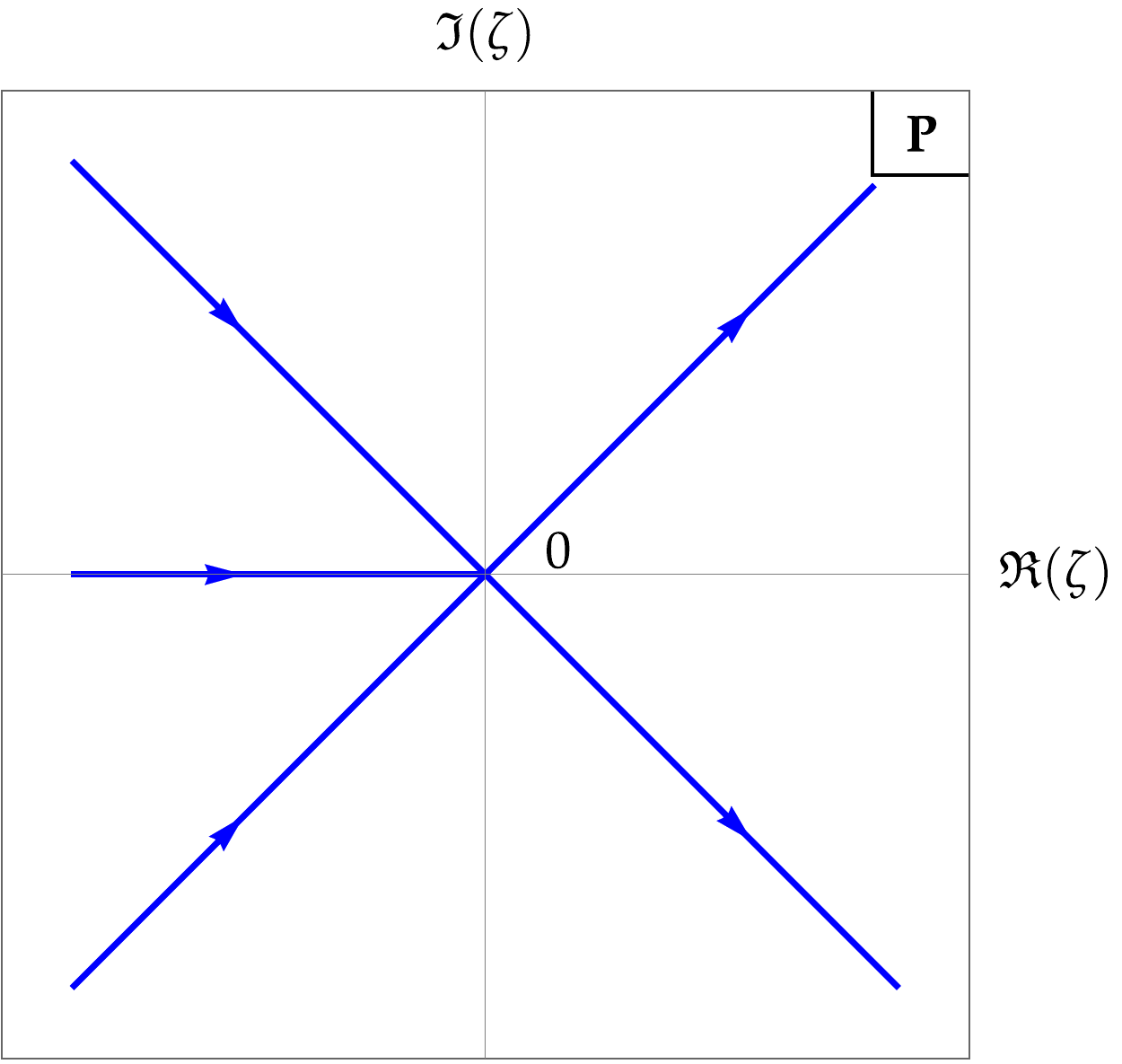}\quad
    \includegraphics[scale = 0.49]{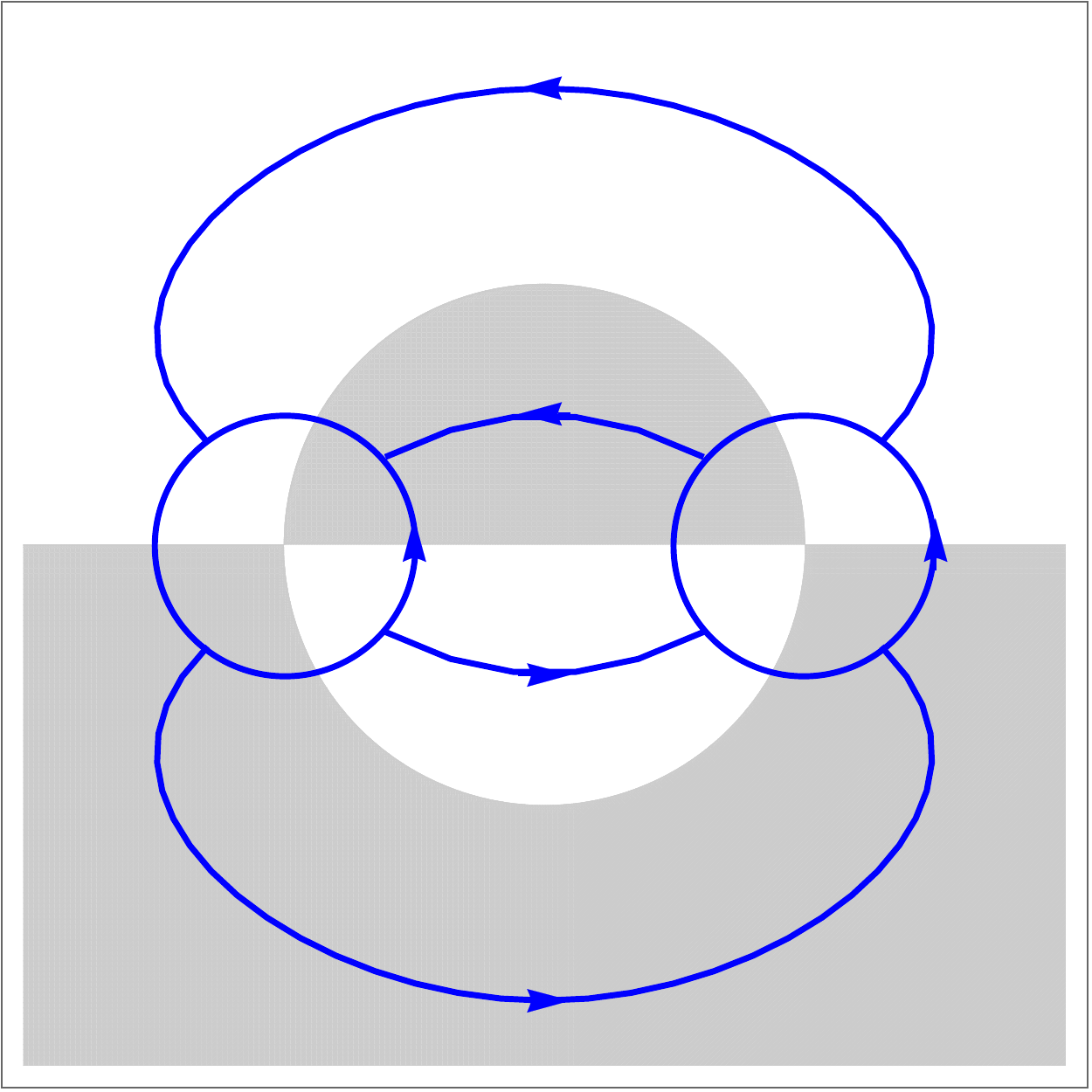}
    \caption{Left:  the jump contour for $\mathbf{P}(\zeta;\eta)$ in the $\zeta$-plane.  Right:  the jump contour $\Sigma_\mathbf{E}$ for $\mathbf{E}_\mathbf{Z}(k;x,\eta)$ in the $k$-plane.
    }
    \label{f:PE1}
\end{figure}
and $\P(\zeta;\eta)\zeta^{\ii\nu\sigma_3}\to\I$ as $\zeta\to\infty$ in every direction, where $\nu=\nu(\eta)$ is given in \eqref{e:Z-dot-out}.  It is well-known that these conditions uniquely determine $\P(\zeta;\eta)$, which can be written explicitly in terms of parabolic cylinder functions; a complete development of the solution and its properties can be found for instance in \cite[Appendix A]{m2018} in which the relevant parameters are $\tau=\tan(\eta)$ and $p=\nu$.  The defining properties of $\P(\zeta;\eta)$ listed above in fact also imply that for each $\eta\in (0,\tfrac{1}{2}\pi)$, $\zeta\mapsto \P(\zeta;\eta)$ is uniformly bounded, and that the matrix function $\zeta\mapsto \P(\zeta;\eta)\zeta^{\ii\nu\sigma_3}$ has a complete asymptotic expansion as $\zeta\to\infty$ in descending integer powers of $\zeta$ in which the leading terms read
\begin{equation}
\P(\zeta;\eta)\zeta^{\ii\nu\sigma_3}=\I +\frac{1}{2\ii\zeta}\bpm 0 & \varrho\\\b{\varrho} & 0\epm + \bpm
\O(\zeta^{-2}) & \O(\zeta^{-3})\\\O(\zeta^{-3}) & \O(\zeta^{-2})\epm,\quad \zeta\to\infty,
\label{e:PC-expansion}
\end{equation}
with
\begin{equation}
\varrho\coloneq \pi^{-\frac{3}{2}}\tan(\eta)\ln(\sec(\eta))\sqrt{\cos(\eta)}\Gamma\left(\ii\frac{\ln(\cos(\eta))}{\pi}\right)\exp\left(\ii\left[\frac{\pi}{4}-\frac{1}{\pi}\ln(2)\ln(\cos(\eta))\right]\right).
\label{e:varrho-def}
\end{equation}
We use the matrix function $\P(\zeta;\eta)$ to define local parametrices for $\mathbf{Z}(k;x,\eta)$ near $k=-1,1$ as follows:
\begin{equation}
\begin{split}
\dot{\mathbf{Z}}_{-1}(k;x,\eta)&\coloneq x^{\frac{1}{2}\ii\nu\sigma_3}\left(\frac{1-k}{(-k)^\frac{1}{2}}\right)^{\ii\nu\sigma_3}\ee^{-\ii x\sigma_3}\sigma_3\P(\sqrt{x}y_{-1}(k);\eta)\sigma_3\ee^{\ii x\sigma_3},\quad |k+1|\le\tfrac{1}{2}\\
\dot{\mathbf{Z}}_1(k;x,\eta)&\coloneq x^{-\frac{1}{2}\ii\nu\sigma_3}\left(\frac{k^\frac{1}{2}}{k+1}\right)^{\ii\nu\sigma_3}\ee^{\ii x\sigma_3}\sigma_1\P(\sqrt{x}y_1(k);\eta)\sigma_1\ee^{-\ii x\sigma_3},\quad |k-1|\le\tfrac{1}{2}.
\end{split}
\label{e:Zdot-pm1}
\end{equation}

The outer and local parametrices are combined to define a global parametrix for $\mathbf{Z}(k;x,\eta)$ by setting:
\begin{equation}
\dot{\mathbf{Z}}(k;x,\eta)\coloneq\begin{cases}
\dot{\mathbf{Z}}_{-1}(k;x,\eta),&\quad |k+1|<\tfrac{1}{2},\\
\dot{\mathbf{Z}}_1(k;x,\eta),&\quad |k-1|<\tfrac{1}{2},\\
\dot{\mathbf{Z}}_\mathrm{out}(k;\eta),&\quad \text{$|k+1|>\tfrac{1}{2}$ and $|k-1|>\tfrac{1}{2}$.}
\end{cases}
\end{equation}
To compare $\mathbf{Z}(k;x,\eta)$ with its parametrix $\dot{\mathbf{Z}}(k;x,\eta)$ we define the error by
\begin{equation}
\E_\mathbf{Z}(k;x,\eta)\coloneq\mathbf{Z}(k;x,\eta)\dot{\mathbf{Z}}(k;x,\eta)^{-1}
\label{e:EZ-def}
\end{equation}
wherever both factors are defined.  Observe that since both $\mathbf{Z}(k;x,\eta)$ and $\dot{\mathbf{Z}}(k;x,\eta)$ take continuous boundary values satisfying the same jump conditions within the disks $|k\pm 1|<\tfrac{1}{2}$ and on the segment $I$, $\E_\mathbf{Z}(k;x,\eta)$ may be extended to the latter contours as an analytic function, which we denote by the same symbol;  it is therefore analytic in the domain complementary to the jump contour illustrated in the right-hand panel of Figure~\ref{f:PE1}.  Moreover, $\E_\mathbf{Z}(k;x,\eta)\to\I$ as $k\to\infty$ because this is true of both factors on the right-hand side of \eqref{e:EZ-def}.  By direct computation, the jump conditions satisfied by $\E_\mathbf{Z}(k;x,\eta)$ take the form $\E_\mathbf{Z}^+(k;x,\eta)=\E_\mathbf{Z}^-(k;x,\eta)\mathbf{V}_\E(k;x,\eta)$ where the jump matrix $\mathbf{V}_\E(k;x,\eta)$ is defined on the arcs of its jump contour as follows.  Firstly, on all arcs outside of the two circles $|k\pm 1|=\tfrac{1}{2}$,
we have
$\mathbf{V}_\E(k;x,\eta)=\dot{\mathbf{Z}}_\mathrm{out}(k;\eta)\mathbf{V}_\mathbf{Z}(k;x,\eta)\dot{\mathbf{Z}}_\mathrm{out}(k;\eta)^{-1}$.
Note that $\dot{\mathbf{Z}}_\mathrm{out}(k;\eta)$ has unit determinant, is independent of $x$, and is bounded for $|k\pm 1|\ge \tfrac{1}{2}$; since the exponential factors in the triangular jump matrices defined in \eqref{e:Z-jump-first}--\eqref{e:Z-jump-last} are uniformly exponentially decaying as $x\to+\infty$ by virtue of the sign chart for $\Re(\ii x(k+k^{-1}))$, there is some constant $\epsilon>0$ such that $\mathbf{V}_\E(k;x,\eta)=\I+\O(\ee^{-\epsilon x})$ holds uniformly on these arcs as $x\to+\infty$.  Secondly, on the two circles $|k\pm 1|=\tfrac{1}{2}$ taken with counterclockwise orientation, the jump matrix can be written as $\mathbf{V}_\E(k;x,\eta)=\dot{\mathbf{Z}}_\mathrm{out}(k;\eta)\dot{\mathbf{Z}}_{\mp 1}(k;x,\eta)^{-1}$.  Combining \eqref{e:Zdot-out-zeta} with \eqref{e:Zdot-pm1} makes this more precise:
\begin{multline}
\mathbf{V}_\E(k;x,\eta)=x^{\frac{1}{2}\ii\nu\sigma_3}\left(\frac{1-k}{(-k)^{\frac{1}{2}}}\right)^{\ii\nu\sigma_3}\ee^{-\ii x\sigma_3}\sigma_3\left[\P(\zeta;\eta)\zeta^{\ii\nu\sigma_3}\right]^{-1}\sigma_3\ee^{\ii x\sigma_3}\left(\frac{1-k}{(-k)^\frac{1}{2}}\right)^{-\ii\nu\sigma_3}x^{-\frac{1}{2}\ii\nu\sigma_3},\\
\zeta\coloneq\sqrt{x}y_{-1}(k),\quad |k+1|=\tfrac{1}{2},\quad\text{and}
\label{e:VE-left-circle}
\end{multline}
\begin{multline}
\mathbf{V}_\E(k;x,\eta)=x^{-\frac{1}{2}\ii\nu\sigma_3}\left(\frac{k^\frac{1}{2}}{k+1}\right)^{\ii\nu\sigma_3}\ee^{\ii x\sigma_3}\sigma_1\left[\P(\zeta;\eta)\zeta^{\ii\nu\sigma_3}\right]^{-1}\sigma_1\ee^{-\ii x\sigma_3}\left(\frac{k^\frac{1}{2}}{k+1}\right)^{-\ii\nu\sigma_3}x^{\frac{1}{2}\ii\nu\sigma_3},\\ \zeta\coloneq\sqrt{x}y_1(k),\quad |k-1|=\tfrac{1}{2}.
\label{e:VE-right-circle}
\end{multline}
Due to the fact that $y_{\mp 1}(k)$ is bounded away from zero for $|k\pm 1|=\tfrac{1}{2}$, using \eqref{e:PC-expansion} immediately shows that both of these jump matrices can be written as $\mathbf{V}_\E(k;x,\eta)=\I+\O_\mathrm{OD}(x^{-\frac{1}{2}})+\O_\mathrm{D}(x^{-1})$ uniformly on $|k\pm 1|=\tfrac{1}{2}$ as $x\to+\infty$, where the subscripts on the error terms indicate off-diagonal and diagonal matrices respectively.

It follows that $\E_\mathbf{Z}(k;x,\eta)$ satisfies the conditions of a small-norm RHP.  In particular, the uniform estimate on $\mathbf{V}_\E(k;x,\eta)-\I$ carries over to a uniform estimate on $\E_\mathbf{Z}(k;x,\eta)-\I$ itself of the same order, $\O_\mathrm{OD}(x^{-\frac{1}{2}})+\O_\mathrm{D}(x^{-1})$.  Hence $\E_\mathbf{Z}(k;x,\eta)$ is bounded on the whole $k$-plane, uniformly with respect to the limit $x\to+\infty$.  Since $\dot{\mathbf{Z}}(k;x,\eta)$ also has this property, the same is also true of $\mathbf{Z}(k;x,\eta)=\E_\mathbf{Z}(k;x,\eta)\dot{\mathbf{Z}}(k;x,\eta)$.  Since the explicit transformation relating $\mathbf{Y}(\ii k;-\ii x;\eta)$ to $\mathbf{Z}(k;x,\eta)$ (see \eqref{e:Y-to-Z-first}--\eqref{e:Y-to-Z-last}) is invertible with bounded inverse, we have finally proven that $\mathbf{Y}(\Lambda;-\ii x,\eta)$ is bounded on the $\Lambda$-plane, uniformly with respect to $x\to+\infty$.

Since $\mathbf{Y}(\Lambda;X,\eta)$ is uniformly bounded for $|\Lambda|\neq 1$ and $X$ on the negative imaginary axis, it follows that also the matrix $\widetilde{\mathbf{Y}}(\Lambda;X,\eta)$ defined in \eqref{e:tilde-Y} is uniformly bounded for $|\Lambda|\neq 1$ and $X$ on the positive real axis.

\subsubsection{Asymptotic behavior of $y(X;\omega)$, $s(X;\omega)$, and $U(X;\omega)$}
The small-norm theory behind the uniform estimate $\E_\mathbf{Z}(k;x,\eta)-\I=\O_\mathrm{OD}(x^{-\frac{1}{2}})+\O_\mathrm{D}(x^{-1})$ also produces accurate asymptotic formul\ae\ for the functions $y(X;\omega)$, $s(X;\omega)$, and $U(X;\omega)$ that are valid for large $X=-\ii x$ on the negative imaginary axis.  From the jump condition $\E_\mathbf{Z}^+(k;x,\eta)=\E_\mathbf{Z}^-(k;x,\eta)\mathbf{V}_\E(k;x,\eta)$ and the condition $\E_\mathbf{Z}(k;x,\eta)\to\I$ as $k\to\infty$, the Plemelj formula implies that
\begin{equation}
\begin{split}
\E_\mathbf{Z}(k;x,\eta)&=\I+\frac{1}{2\pi\ii}\int_{\Sigma_\mathbf{E}}\frac{\E_\mathbf{Z}^-(s;x,\eta)(\mathbf{V}_\E(s;x,\eta)-\I)}{s-k}\dd s\\
&=\I+\frac{1}{2\pi\ii}\int_{\Sigma_\mathbf{E}}\frac{\mathbf{V}_\E(s;x,\eta)-\I}{s-k}\dd s +
\frac{1}{2\pi\ii}\int_{\Sigma_\mathbf{E}}\frac{(\E^-_\mathbf{Z}(s;x,\eta)-\I)(\mathbf{V}_\E(s;x,\eta)-\I)}{s-k}\dd s,\quad k\in\Complex\setminus\Sigma_\mathbf{E},
\end{split}
\label{e:E-formula}
\end{equation}
where $\Sigma_\mathbf{E}$ is the jump contour illustrated in the right-hand panel of Figure~\ref{f:PE1}.  From this formula and the uniform estimates $\E_\mathbf{Z}(\cdot;x,\eta)-\I=\O_\mathrm{OD}(x^{-\frac{1}{2}})+\O_\mathrm{D}(x^{-1})$ and $\mathbf{V}_\E(\cdot;x,\eta)-\I=\O_\mathrm{OD}(x^{-\frac{1}{2}})+\O_\mathrm{D}(x^{-1})$ we get
\begin{equation}
\E_\mathbf{Z}(k;x,\eta)=\I+\frac{1}{k}\E_\mathbf{Z}^{(1)}(x,\eta) + \O(k^{-2}),\quad k\to\infty,\quad\text{where}
\end{equation}
\begin{equation}
\E_\mathbf{Z}^{(1)}(x,\eta)=-\frac{1}{2\pi\ii}\oint_{|s\pm 1|=\frac{1}{2}}\left(\mathbf{V}_\E(s;x,\eta)-\I\right)\dd s + \O_\mathrm{D}(x^{-1})+\O_\mathrm{OD}(x^{-\frac{3}{2}}),\quad x\to +\infty,
\label{e:E1-expansion}
\end{equation}
and that
\begin{equation}
\E_\mathbf{Z}(0;x,\eta)=\I +\frac{1}{2\pi\ii}\oint_{|s\pm 1|=\frac{1}{2}}\left(\mathbf{V}_\E(s;x,\eta)-\I\right)s^{-1}\dd s
+ \O_\mathrm{D}(x^{-1})+\O_\mathrm{OD}(x^{-\frac{3}{2}}),\quad x\to+\infty.
\label{e:E-at-origin-expansion}
\end{equation}
In \eqref{e:E1-expansion}--\eqref{e:E-at-origin-expansion} we have retained the integration over just the two circles dominating the estimate for $\mathbf{V}_\E(k;x,\eta)-\I$, absorbing exponentially small contributions over the remaining arcs of $\Sigma_\E$ into the error terms.  We evaluate the remaining integrals by combining \eqref{e:PC-expansion} with \eqref{e:VE-left-circle}--\eqref{e:VE-right-circle}; the $\O_\mathrm{D}(\zeta^{-2})$ and $\O_\mathrm{OD}(\zeta^{-3})$ terms in \eqref{e:PC-expansion} are immediately absorbed into the $\O_\mathrm{D}(x^{-1})$ and $\O_\mathrm{OD}(x^{-\frac{3}{2}})$ error terms respectively, while those proportional to $\zeta^{-1}=x^{-\frac{1}{2}}y_{\mp 1}(k)^{-1}$ are integrated by residues using the fact that $y_{\mp 1}(k)$ is analytic on the indicated disk and vanishes to first order at $k=\mp 1$.  Therefore, as $x\to+\infty$
\begin{multline}
\E^{(1)}_\mathbf{Z}(x,\eta)=\frac{1}{2\ii \sqrt{x}y_{-1}'(-1)}(2\sqrt{x})^{\ii\nu\sigma_3}\ee^{-\ii x\sigma_3}\sigma_3\bpm 0 & \varrho\\\b{\varrho} & 0\epm\sigma_3\ee^{\ii x\sigma_3}(2\sqrt{x})^{-\ii\nu\sigma_3} \\
{}+\frac{1}{2\ii\sqrt{x}y_1'(1)}(2\sqrt{x})^{-\ii\nu\sigma_3}\ee^{\ii x\sigma_3}\sigma_1
\bpm 0 &\varrho\\\b{\varrho} & 0\epm\sigma_1\ee^{-\ii x\sigma_3}(2\sqrt{x})^{\ii\nu\sigma_3} + \O_\mathrm{D}(x^{-1})+\O_\mathrm{OD}(x^{-\frac{3}{2}}),
\end{multline}
and
\begin{multline}
\E_\mathbf{Z}(0;x,\eta)=\I + \frac{1}{2\ii \sqrt{x}y_{-1}'(-1)}(2\sqrt{x})^{\ii\nu\sigma_3}\ee^{-\ii x\sigma_3}\sigma_3\bpm 0 & \varrho\\\b{\varrho} & 0\epm\sigma_3\ee^{\ii x\sigma_3}(2\sqrt{x})^{-\ii\nu\sigma_3} \\
{}-\frac{1}{2\ii\sqrt{x}y_1'(1)}(2\sqrt{x})^{-\ii\nu\sigma_3}\ee^{\ii x\sigma_3}\sigma_1
\bpm 0 &\varrho\\\b{\varrho} & 0\epm\sigma_1\ee^{-\ii x\sigma_3}(2\sqrt{x})^{\ii\nu\sigma_3} + \O_\mathrm{D}(x^{-1})+\O_\mathrm{OD}(x^{-\frac{3}{2}}).
\end{multline}
Note that $y_{-1}'(-1)=-1$ and $y_1'(1)=1$.  Now recalling from \eqref{e:Lax-potentials} the definition of $y(-\ii x;\omega)$ gives
\begin{equation}
\begin{split}
y(-\ii x;\omega)&=-x\lim_{\Lambda\to\infty}\Lambda Y_{1,2}(\Lambda;-\ii x,\eta)\\
&=-\ii x\lim_{k\to\infty}kY_{1,2}(\ii k;-\ii x,\eta)\\
&=-\ii x\lim_{k\to\infty}kZ_{1,2}(k;x,\eta)\\
&=-\ii x\lim_{k\to\infty}k\left(E_{\mathbf{Z},1,1}(k;x,\eta)\dot{Z}_{\mathrm{out},1,2}(k;\eta)+
E_{\mathbf{Z},1,2}(k;x,\eta)\dot{Z}_{\mathrm{out},2,2}(k;\eta)\right)\\
&=-\ii xE^{(1)}_{1,2}(x,\eta)\\
&=-\sqrt{x}\Re\left((4x)^{\ii\nu}\ee^{-2\ii x}\varrho\right) + \O(x^{-\frac{1}{2}}),\quad x\to+\infty\\
&=-|\varrho|\sqrt{x}\sin\left(2x-\nu\ln(x)+\tfrac{1}{2}\pi-2\nu\ln(2)-\arg(\varrho)\right)+\O(x^{-\frac{1}{2}}),\quad x\to+\infty.
\end{split}
\label{e:y-im-asymp}
\end{equation}
Similarly,
\begin{equation}
\begin{split}
s(-\ii x;\omega)&=\ii xY_{1,1}(0;-\ii x,\eta)Y_{1,2}(0;-\ii x,\eta)\\
&=\ii x\mathop{\lim_{k\to 0}}_{\Im(k)\neq 0}Z_{1,1}(k;x,\eta)Z_{1,2}(k;x,\eta)\\
&=\ii xE_{\mathbf{Z},1,1}(0;x,\eta)E_{\mathbf{Z},1,2}(0;x,\eta)\mathop{\lim_{k\to 0}}_{\Im(k)\neq 0}\dot{Z}_{\mathrm{out},1,1}(k;\eta)
\dot{Z}_{\mathrm{out},2,2}(k;\eta)\\
&=\ii xE_{\mathbf{Z},1,1}(0;x,\eta)E_{\mathbf{Z},1,2}(0;x,\eta)\\
&=\ii\sqrt{x}\Im\left((4x)^{\ii\nu}\ee^{-2\ii x}\varrho\right)+\O(x^{-\frac{1}{2}}),\quad x\to+\infty\\
&=\ii|\varrho|\sqrt{x}\cos\left(2x-\nu\ln(x)+\tfrac{1}{2}\pi-2\nu\ln(2)-\arg(\varrho)\right)+\O(x^{-\frac{1}{2}}),\quad x\to+\infty,
\end{split}
\label{e:s-im-asymp}
\end{equation}
and
\begin{equation}
\begin{split}
U(-\ii x;\omega)&=\ii x Y_{1,2}(0;-\ii x,\eta)Y_{2,1}(0;-\ii x,\eta)\\
&=\ii x\mathop{\lim_{k\to 0}}_{\Im(k)\neq 0}Z_{1,2}(k;x,\eta)Z_{2,1}(k;x,\eta)\\
&=\ii xE_{\mathbf{Z},1,2}(0;x,\eta)E_{\mathbf{Z},2,1}(0;x,\eta)\mathop{\lim_{k\to 0}}_{\Im(k)\neq 0}
\dot{Z}_{\mathrm{out},2,2}(k;\eta)\dot{Z}_{\mathrm{out},1,1}(k;\eta)\\
&=\ii xE_{\mathbf{Z},1,2}(0;x,\eta)E_{\mathbf{Z},2,1}(0;x,\eta)\\
&=-\ii\left[\Im\left((4x)^{\ii\nu}\ee^{-2\ii x}\varrho\right)
\right]^2 + \O(x^{-1}),\quad x\to+\infty\\
&=-\ii |\varrho|^2\cos^2\left(2x-\nu\ln(x)+\tfrac{1}{2}\pi-2\nu\ln(2)-\arg(\varrho)\right)+\O(x^{-1}),\quad x\to+\infty.
\end{split}
\label{e:U-im-asymp}
\end{equation}
In all three of these asymptotic formul\ae, $\nu$ and $\varrho$ are defined in terms of $\eta\in (0,\tfrac{1}{2}\pi)$ by \eqref{e:Z-dot-out} and \eqref{e:varrho-def} respectively, and $\omega=-3\cos(2\eta)$.

The identities \eqref{e:PIII-rotate-X-identities} then allow translation of asymptotic formul\ae\ for $y(X;\omega)$, $s(X;\omega)$, and $U(X;\omega)$ valid for large imaginary $X$ into similar formul\ae\ valid for large real $X$.  We obtain the following results:
\begin{equation}
y(x;\omega)=|\varrho'|\sqrt{x}\sin\left(2x-\nu'\ln(x)+\tfrac{1}{2}\pi-2\nu'\ln(2)-\arg(\varrho')\right)+\O(x^{-\frac{1}{2}}),\quad x\to+\infty,
\label{e:y-re-asymp}
\end{equation}
\begin{equation}
s(x;\omega)=-|\varrho'|\sqrt{x}\cos\left(2x-\nu'\ln(x)+\tfrac{1}{2}\pi-2\nu'\ln(2)-\arg(\varrho')\right)+\O(x^{-\frac{1}{2}}),\quad x\to+\infty,
\label{e:s-re-asymp}
\end{equation}
and
\begin{equation}
U(x;\omega)=x-|\varrho'|^2\cos^2\left(2x-\nu'\ln(x)+\tfrac{1}{2}\pi-2\nu'\ln(2)-\arg(\varrho')\right)+\O(x^{-1}),\quad x\to+\infty,
\label{e:U-re-asymp}
\end{equation}
where again $\omega=-3\cos(2\eta)$ and the modified parameters $\nu'$ and $\varrho'$ are given by
\begin{equation}
\begin{split}
\nu'&\coloneq -\frac{\ln(\sin(\eta))}{\pi}>0\\
\varrho'&\coloneq \pi^{-\frac{3}{2}}\cot(\eta)\ln(\csc(\eta))\sqrt{\sin(\eta)}\Gamma\left(\ii\frac{\ln(\sin(\eta))}{\pi}\right)\exp\left(\ii\left[\frac{\pi}{4}-\frac{1}{\pi}\ln(2)\ln(\sin(\eta))\right]\right).
\end{split}
\label{e:nprime-varrhoprime}
\end{equation}

\subsection{Properties of the modified parametrix:  $M>0$}
\label{s:Properties-for-M-positive-s}
When $M>0$ (i.e., $r_0=0$), the jump matrix in RHP~\ref{rhp:ddotKs} tends uniformly to the identity matrix in the limit $t\to+\infty$ with $z=o(t)$, because $\lambda_\circ=\O((z/t)^\frac{1}{2})$.  Indeed, for the diagonal entries we have $\Delta_M(\lambda_\circ k)^{\pm 1}=1+\O((z/t)^{M})$, and since $|\ee^{-\ii x (k+k^{-1})}|=1$ for $|k|=1$ the off-diagonal elements are $\O((z/t)^{\frac{1}{2}M})$.  This makes RHP~\ref{rhp:ddotKs} a small-norm problem without further approximation when $M>0$ in the limit $t\to+\infty$ with $z=o(t)$.  As in Section~\ref{s:Boundedness-of-RHP-PIII}, it follows that $\ddot{\mathbf{K}}_\mathrm{s}(k;t,z)=\I + \O_{\mathrm{OD}}((z/t)^{\frac{1}{2}M}) + \O_{\mathrm{D}}((z/t)^{M})$ holds uniformly for $|k|\neq 1$ in this limit.  In particular, this establishes the uniform boundedness of $\ddot{\mathbf{K}}_\mathrm{s}(k;t,z)$ in this situation.

As in \eqref{e:E-formula}--\eqref{e:E-at-origin-expansion} we can then write
\begin{equation}
\ddot{\K}_\mathrm{s}(k;t,z)=\I + \frac{1}{k}\ddot{\K}_\mathrm{s}^{(1)}(t,z) + \O(k^{-2}),\quad k\to\infty,\quad \text{where}
\label{e:ddotKs-moment-define}
\end{equation}
\begin{equation}
\ddot{\K}_\mathrm{s}^{(1)}(t,z)=-\frac{|a_M|\lambda_\circ^M}{2\pi\ii}\oint_{|s|=1}\bpm
0 & s^M\ee^{\ii x(s+s^{-1})}\\-s^M\ee^{-\ii x(s+s^{-1})} & 0\epm\dd s +\O_\mathrm{D}(\lambda_\circ^{2M}) + \O_\mathrm{OD}(\lambda_\circ^{3M}),\quad\lambda_\circ\to 0,
\label{e:ddotKs-moment-approx}
\end{equation}
and
\begin{multline}
\ddot{\mathbf{K}}_\mathrm{s}(0;t,z)=\I+\frac{|a_M|\lambda_\circ^M}{2\pi\ii}\oint_{|s|=1}\bpm 0 & s^M\ee^{\ii x(s+s^{-1})}\\-s^M\ee^{-\ii x(s+s^{-1})} & 0\epm\frac{\dd s}{s} \\
{}+ \O_\mathrm{D}(\lambda_\circ^{2M})+
\O_\mathrm{OD}(\lambda_\circ^{3M}),\quad\lambda_\circ\to 0,
\label{e:ddotKs-origin-approx}
\end{multline}
where we recall that $\lambda_\circ\coloneq \sqrt{z/(2t)}>0$, that $a_M$ is a Taylor coefficient defined by \eqref{e:avoid-factorials}, and where the unit circle is oriented in the counterclockwise direction.  These formul\ae\ yield the Bessel approximations for the solution of the MBE problem when $M>0$ appearing in Theorem~\ref{thm:global-Mpos-stable}.  However first it is necessary to control the error in approximating $\K_\mathrm{s}(u,v;t,z)$ by $\dot{\K}_\mathrm{s}(\lambda;t,z)$ in both cases $M=0$ and $M>0$.  We address this next.

\subsection{Comparing $\K_\mathrm{s}(u,v;t,z)$ with its parametrix $\dot{\K}_\mathrm{s}(\lambda;t,z)$}
\label{s:Ks-vs-dotKs}
Now we return to the study of RH$\b{\partial}$P~\ref{rhp:Ks} and derive conditions on a matrix function defined in terms of its solution $\K_\mathrm{s}(u,v;t,z)$ and the parametrix $\dot{\K}_\mathrm{s}(\lambda;t,z)$ by
\begin{equation}
\mathbf{F}_\mathrm{s}(u,v;t,z)\coloneq\K_\mathrm{s}(u,v;t,z)\dot{\K}_\mathrm{s}(u+\ii v;t,z)^{-1},\quad (u,v)\in\Real^2\setminus\Sigma.
\label{e:K-dotK-F-s}
\end{equation}
This matrix function has jump discontinuities across the arcs of $\Sigma$ because the jump matrices for the two factors do not exactly agree; see \eqref{e:Js-approximation} in which $\lambda_\circ$ is small for $z/t=o(1)$.  The jump conditions for $\mathbf{F}_\mathrm{s}(u,v;t,z)$ read
\begin{equation}
\mathbf{F}_\mathrm{s}^+(u,v;t,z)=\mathbf{F}_\mathrm{s}^-(u,v;t,z)\dot{\K}^-_\mathrm{s}(u+\ii v;t,z)\mathbf{J}(u,v;t,z)\dot{\mathbf{J}}(u+\ii v;t,z)^{-1}\dot{\K}^-_\mathrm{s}(u+\ii v;t,z)^{-1},\quad (u,v)\in\Sigma.
\label{e:Fs-jump}
\end{equation}
It also fails to be analytic in the complement of the contour $\Sigma$ (at least outside of the small circle $|u+\ii v|=\lambda_\circ$ but within the strip $|v|\le 2$).  It is however continuous on each component of $\Real^2\setminus\Sigma$, and while the second factor is sectionally analytic, the first is subject to the $\b{\partial}$ differential equation \eqref{e:dbar-equation}.  This implies that on each component of $\Real^2\setminus\Sigma$, $\mathbf{F}_\mathrm{s}(u,v;t,z)$ satsifies its own $\b{\partial}$ equation:
\begin{equation}
\b{\partial}\mathbf{F}_\mathrm{s}(u,v;t,z)=\mathbf{F}_\mathrm{s}(u,v;t,z)\dot{\K}_\mathrm{s}(u+\ii v;t,z)\mathbf{D}_\mathrm{s}(u,v;t,z)\dot{\K}_\mathrm{s}(u+\ii v;t,z)^{-1},\quad (u,v)\in\Real^2\setminus\Sigma,
\label{e:Fs-dbar}
\end{equation}
in which the matrix $\mathbf{D}_\mathrm{s}(u,v;t,z)$ is given by \eqref{e:dbar-equation-D}.  We will deal with the non-analyticity of $\mathbf{F}_\mathrm{s}(u,v;t,z)$ measured by \eqref{e:Fs-dbar} and its near-identity jumps in two steps.

\subsubsection{A continuous parametrix for $\mathbf{F}_\mathrm{s}(u,v;t,z)$}
\label{s:dbar-s}
First, we model the non-analyticity of $\mathbf{F}_\mathrm{s}(u,v;t,z)$ by solving the following pure $\b{\partial}$-problem ($\dbar$P):
\begin{dbp}
Given $t\ge 0$ and $z\ge 0$, seek a $2\times 2$ matrix-valued continuous function $\Real^2\ni (u,v)\mapsto \dot{\mathbf{F}}_\mathrm{s}(u;v,t,z)$ that satisfies $\dot{\mathbf{F}}_\mathrm{s}(u,v;t,z)\to\I$ as $u+\ii v\to\infty$ and that satisfies the following $\b{\partial}$ differential equation:
\begin{equation}
\b{\partial}\dot{\mathbf{F}}_\mathrm{s}(u,v;t,z)=\dot{\mathbf{F}}_\mathrm{s}(u,v;t,z)\dot{\K}_\mathrm{s}(u+\ii v;t,z)\mathbf{D}_\mathrm{s}(u,v;t,z)\dot{\K}_\mathrm{s}(u+\ii v;t,z)^{-1},\quad (u,v)\in\Real^2.
\label{e:dotFs-dbar}
\end{equation}
\label{rhp:Fsdot}
\end{dbp}
This problem is converted into an integral equation with the help of the solid Cauchy transform:
\begin{equation}
\dot{\mathbf{F}}_\mathrm{s}(u,v)=\I + \S_{t,z}\dot{\mathbf{F}_\mathrm{s}}(u,v),\quad (u,v)\in\Real^2,\quad \dot{\mathbf{F}}_\mathrm{s}(u,v)=\dot{\mathbf{F}}_\mathrm{s}(u,v;t,z)
\label{e:dbar-integral-equation}
\end{equation}
where $\S_{t,z}$ is an integral operator defined by
\begin{equation}
\S_{t,z}\M(u,v)\coloneq -\frac{1}{\pi}\iint_{\Real^2}
\frac{\M(u',v')\dot{\K}_\mathrm{s}(u'+\ii v';t,z)\mathbf{D}_\mathrm{s}(u',v';t,z)\dot{\K}_\mathrm{s}(u'+\ii v';t,z)^{-1}}{(u'+\ii v')-(u+\ii v)}\dd A(u',v'),
\end{equation}
in which $\dd A(u',v')$ denotes the area differential in $\Real^2$.
\begin{lemma}
$\S_{t,z}$ is a bounded operator on $L^\infty(\Real^2)$ with operator norm $\|\S_{t,z}\|_\infty$ satisfying
\begin{equation}
\|\S_{t,z}\|_\infty=\O((z/t)^{\frac{1}{2}(N-1)})+\O(t^{\frac{3}{2}-N}),\quad t\to+\infty,\quad z=o(t).
\end{equation}
\label{lemma:S-bound}
\end{lemma}
\begin{proof}
Let $\|\cdot\|$ denote the norm on $2\times 2$ matrices induced from an arbitrary norm on $\Complex^2$, and for a matrix function $\Real^2\ni (u,v)\mapsto \M(u,v)$ take as the norm on $L^\infty(\Real^2)$
\begin{equation}
\|\M\|_\infty\coloneq \sup_{(u,v)\in\Real^2}\|\M(u,v)\|.
\end{equation}
Then, since it has been shown that $\dot{\K}_\mathrm{s}(\lambda;t,z)$ is bounded on $\Complex^2\setminus\Sigma$ uniformly with respect to $(t,z)$ with $t\ge 0$ and $0\le z\ll t$ as $t\to+\infty$, and since $\dot{\K}_\mathrm{s}(\lambda;t,z)$ has unit determinant, there is some constant $C>0$ independent of $(t,z)$ in the indicated regime, such that
\begin{equation}
\begin{split}
\|\S_{t,z}\M\|_\infty &\le
\sup_{(u,v)\in\Real^2}\iint_{\Real^2}\frac{\|\M(u',v')\|\|\dot{\K}_\mathrm{s}(u'+\ii v';t,z)\|\|\mathbf{D}_\mathrm{s}(u',v';t,z)\|\|\dot{\K}_\mathrm{s}(u'+\ii v';t,z)^{-1}\|}{\sqrt{(u'-u)^2+(v'-v)^2}}\dd A(u',v')
\\
&\le \|\dot{\K}_\mathrm{s}\|_\infty\|\dot{\K}_\mathrm{s}^{-1}\|_\infty
\sup_{(u,v)\in\Real^2}\iint_{\Real^2}\frac{\|\mathbf{D}_\mathrm{s}(u',v';t,z)\|\dd A(u',v')}{\sqrt{(u'-u)^2+(v'-v)^2}}\cdot \|\M\|_\infty \\
&\le C\sup_{(u,v)\in\Real^2}\iint_{\Real^2}\frac{\|\mathbf{D}_\mathrm{s}(u',v';t,z)\|\dd A(u',v')}{\sqrt{(u'-u)^2+(v'-v)^2}}\cdot \|\M\|_\infty\\
&= C\sup_{(u,v)\in\Real^2}\mathop{\iint_{|u'+\ii v'|\ge\lambda_\circ}}_{|v'|\le 2}\frac{\|\mathbf{D}_\mathrm{s}(u',v';t,z)\|\dd A(u',v')}{\sqrt{(u'-u)^2+(v'-v)^2}}\cdot \|\M\|_\infty,
\end{split}
\end{equation}
where the final equality comes from the fact that, according to \eqref{e:dbar-equation-D}, $\mathbf{D}_\mathrm{s}(u,v;t,z)\equiv \mathbf{0}_{2\times 2}$ for $|u+\ii v|<\lambda_\circ$ as well as for $|v|>2$ where $B(v)\equiv 0$ identically.
Hence an upper bound for the operator norm of $\S_{t,z}$ is
\begin{equation}
\|\S_{t,z}\|_\infty\le C\sup_{(u,v)\in\Real^2}\mathop{\iint_{|u'+\ii v'|\ge\lambda_\circ}}_{|v'|\le 2}\frac{\|\mathbf{D}_\mathrm{s}(u',v';t,z)\|\dd A(u',v')}{\sqrt{(u'-u)^2+(v'-v)^2}}.
\label{e:S-operator-norm}
\end{equation}

Again referring to \eqref{e:dbar-equation-D} and using Lemma~\ref{thm:delta}, the differentiation identities \eqref{e:dbarQN}--\eqref{e:dbarQNbar}, and the properties of the bump function $B(v)$, there is a constant $C_N>0$ independent of $(t,z)$ such that
\begin{multline}
\|\mathbf{D}_\mathrm{s}(u',v';t,z)\|\le
C_N \left(\left|R^{(N-1)}(u')\right||v'|^{N-2} + \chi_{[1,2]}(v')\sum_{n=0}^{N-2}\left|R^{(n)}(u')\right|\right)\left|\ee^{\pm 2\ii\theta_\mathrm{s}(u'+\ii v';t,z)}\right|,\\
 |u'+\ii v'|\ge\lambda_\circ,\quad  0\le \pm v'\le 2,
\label{e:D-matrix-upper-bound}
\end{multline}
in which $R(\cdot)$ is defined by \eqref{e:R-def}.
To estimate further, it is useful to split the exterior of the circle $|u'+\ii v'|=\lambda_\circ$ into the union of three regions:
\begin{equation}
\label{e:dbar-regions-def}
\begin{aligned}
\Omega_1:&\qquad \lambda_\circ \le |u'+\ii v'| < 2\lambda_\circ\,, \\ 
\Omega_2:&\qquad |u'+\ii v'| \ge 2\lambda_\circ,\quad |v'|<1\,, \\ 
\Omega_3:&\qquad |u'+\ii v'| \ge 2\lambda_\circ,\quad 1\le |v'|<2\,.
\end{aligned}
\end{equation}
Note that for $\lambda_\circ$ sufficiently small, $\Omega_1$ is completely contained in the strip $|v'|<1$ and $\Omega_3$ is a union of two horizontal strips bounded away from the real line.
Using this assumption, we can find simple and useful upper bounds for $\|\mathbf{D}_\mathrm{s}(u',v';t,z)\|$ on each of these regions as follows.
\begin{itemize}
\item For $(u',v')\in\Omega_1$ with $\pm v'\ge 0$, we will use the inequality $|\ee^{\pm 2\ii\theta_\mathrm{s}(u'+\ii v';t,z)}|\le 1$ and we have $\chi_{[1,2]}(v')\equiv 0$.
Recall Lemma~\ref{thm:reflection2} implying that $R^{(N-1)}\in \mathscr{S}(\Real)\subset L^\infty(\Real)$.
Hence from \eqref{e:D-matrix-upper-bound}, we get
\begin{equation}
\|\mathbf{D}_\mathrm{s}(u',v';t,z)\|\le C_N |v'|^{N-2},\quad (u',v')\in\Omega_1
\label{e:D-Omega-1}
\end{equation}
for some other constant $C_N$.
\item
For $(u',v')\in\Omega_2$, we use the inequality $|u'+\ii v'|^2\ge 4\lambda_\circ^2=2z/t$ and recall Definition~\ref{def:theta-s} to give for $\pm v'\ge 0$,
\begin{equation}
\left|\ee^{\pm 2\ii\theta_\mathrm{s}(u'+\ii v';t,z)}\right|=\exp\left(-2t|v'|+\frac{z|v'|}{|u'+\ii v'|^2}\right)
\le\exp\left(-2t |v'|+\frac{1}{2}t|v'|\right)\le \ee^{-t|v'|},
\label{e:exponential-estimate}
\end{equation}
and therefore since again $\chi_{[1,2]}(v')\equiv 0$ on $\Omega_2$,
\begin{equation}
\|\mathbf{D}_\mathrm{s}(u',v';t,z)\|\le C_N \left|R^{(N-1)}(u')\right||v'|^{N-2}\ee^{-t|v'|},\quad (u',v')\in\Omega_2.
\label{e:D-Omega-2}
\end{equation}
\item For $(u',v')\in\Omega_3$, we again have the inequality \eqref{e:exponential-estimate}, but we also have $|v'|\ge 1$ so $\ee^{-t|v'|}\le\ee^{-t}$, as well as the upper bound $|v'|\le 2$, so for some different constant $C_N>0$,
\begin{equation}
\|\mathbf{D}_\mathrm{s}(u',v';t,z)\|\le C_N\ee^{-t}\sum_{n=0}^{N-1}\left|R^{(n)}(u')\right|,\quad (u',v')\in\Omega_3.
\label{e:D-Omega-3}
\end{equation}
\end{itemize}

Now we use these upper bounds to estimate the double integral in \eqref{e:S-operator-norm}.  From \eqref{e:D-Omega-1} and the fact that $\Omega_1\subset [-2\lambda_\circ,2\lambda_\circ]^2\coloneq[-2\lambda_\circ,2\lambda_\circ]\times[-2\lambda_\circ,2\lambda_\circ]$,
\begin{equation}
\begin{split}
\iint_{\Omega_1}\frac{\|\mathbf{D}_\mathrm{s}(u',v';t,z)\|\dd A(u',v')}{\sqrt{(u'-u)^2+(v'-v)^2}}&\le
C_N\iint_{\Omega_1}\frac{|v'|^{N-2}\dd A(u',v')}{\sqrt{(u'-u)^2+(v'-v)^2}}\\
&\le C_N\iint_{[-2\lambda_\circ,2\lambda_\circ]^2}\frac{|v'|^{N-2}\dd A(u',v')}{\sqrt{(u'-u)^2+(v'-v)^2}}\\
&= C_N\lambda_\circ^{N-1}\iint_{[-2,2]^2}\frac{|v''|^{N-2}\dd A(u'',v'')}{\sqrt{(u''-u/\lambda_\circ)^2+(v''-v/\lambda_\circ)^2}},
\end{split}
\label{e:Omega1-integral-bound}
\end{equation}
where in the last step we rescaled by $(u',v')=\lambda_\circ(u'',v'')$.  Using iterated integration for the resulting double integral, we first integrate over $u''\in [-2,2]$ and apply the Cauchy-Schwarz inequality:
\begin{equation}
\begin{split}
\int_{-2}^2\frac{\dd u''}{\sqrt{(u''-u/\lambda_\circ)^2+(v''-v/\lambda_\circ)^2}}&\le
\left(\int_{-2}^2\dd u''\right)^\frac{1}{2}\left(\int_{-2}^2\frac{\dd u''}{(u''-u/\lambda_\circ)^2+(v''-v/\lambda_\circ)^2}\right)^\frac{1}{2}\\
&\le\left(\int_\Real\frac{4\dd u''}{(u''-u/\lambda_\circ)^2+(v''-v/\lambda_\circ)^2}\right)^\frac{1}{2}\\
&=\sqrt{\frac{4\pi}{|v''-v/\lambda_\circ|}}.
\end{split}
\label{e:Cauchy-Schwarz-I}
\end{equation}
Using this in \eqref{e:Omega1-integral-bound} gives
\begin{equation}
\iint_{\Omega_1}\frac{\|\mathbf{D}_\mathrm{s}(u',v';t,z)\|\dd A(u',v')}{\sqrt{(u'-u)^2+(v'-v)^2}}\le C_N\sqrt{4\pi}\lambda_\circ^{N-1}\int_{-2}^2\frac{|v''|^{N-2}\dd v''}{\sqrt{|v''-v/\lambda_\circ|}}.
\end{equation}
The remaining integral on $[-2,2]$ is bounded uniformly for $v/\lambda_\circ\in\Real$, which proves that
\begin{equation}
\sup_{(u,v)\in\Real^2}\iint_{\Omega_1}\frac{\|\mathbf{D}_\mathrm{s}(u',v';t,z)\|\dd A(u',v')}{\sqrt{(u'-u)^2+(v'-v)^2}} = \O(\lambda_\circ^{N-1})=\O((z/t)^{\frac{1}{2}(N-1)}).
\label{e:Omega1-final-estimate}
\end{equation}

The contribution to \eqref{e:S-operator-norm} from $\Omega_2$ is a bit more involved.  From \eqref{e:D-Omega-2} we have
\begin{equation}
\begin{split}
\iint_{\Omega_2}\frac{\|\mathbf{D}_\mathrm{s}(u',v';t,z)\|\dd A(u',v')}{\sqrt{(u'-u)^2+(v'-v)^2}}&\le C_N\iint_{\Omega_2}\frac{|R^{(N-1)}(u')||v'|^{N-2}\ee^{-t|v'|}\dd A(u',v')}{\sqrt{(u'-u)^2+(v'-v)^2}}\\
&=C_N\int_\Real|v'|^{N-2}\ee^{-t|v'|}\int_{(u')^2>4\lambda_\circ^2-(v')^2}\frac{|R^{(N-1)}(u')|\dd u'}{\sqrt{(u'-u)^2+(v'-v)^2}}\dd v'\\
&\le C_N\int_\Real |v'|^{N-2}\ee^{-t|v'|}\int_\Real\frac{|R^{(N-1)}(u')|\dd u'}{\sqrt{(u'-u)^2+(v'-v)^2}}\dd v'.
\end{split}
\label{e:Omega2-integral-bound}
\end{equation}
Since
$R^{(N-1)}(\cdot)\in L^2(\Real)$
by Lemma~\ref{thm:reflection2}, another application of the Cauchy-Schwarz inequality gives
\begin{equation}
\int_\Real\frac{|R^{(N-1)}(u')|\dd u'}{\sqrt{(u'-u)^2+(v'-v)^2}}\le \frac{\sqrt{\pi}\|R^{(N-1)}\|_{L^2(\Real)}}{\sqrt{|v'-v|}}.
\label{e:inner-integral-N-1}
\end{equation}
Using this in \eqref{e:Omega2-integral-bound} gives
\begin{equation}
\iint_{\Omega_2}\frac{\|\mathbf{D}_\mathrm{s}(u',v';t,z)\|\dd A(u',v')}{\sqrt{(u'-u)^2+(v'-v)^2}}\le C_N\sqrt{\pi}\|R^{(N-1)}\|_{L^2(\Real)}\int_\Real\frac{|v'|^{N-2}\ee^{-t|v'|}\dd v'}{\sqrt{|v'-v|}}.
\label{e:Omega2-integral-bound-II}
\end{equation}
This upper bound is obviously invariant under $v\mapsto -v$, so without loss of generality we assume $v\ge 0$ and divide up the integral over $v'$ as follows:
\begin{equation}
\int_\Real\frac{|v'|^{N-2}\ee^{-t|v'|}\dd v'}{\sqrt{|v'-v|}}=I_A(v)+I_B(v)+I_C(v),
\label{e:integral-split}
\end{equation}
where
\begin{equation}
I_A(v)\coloneq\int_{-\infty}^0\frac{|v'|^{N-2}\ee^{-t|v'|}\dd v'}{\sqrt{|v'-v|}},\;\;
I_B(v)\coloneq\int_0^v\frac{|v'|^{N-2}\ee^{-t|v'|}\dd v'}{\sqrt{|v'-v|}},\;\;
I_C(v)\coloneq\int_v^{+\infty} \frac{|v'|^{N-2}\ee^{-t|v'|}\dd v'}{\sqrt{|v'-v|}}.
\end{equation}
Firstly, writing $v'=-t^{-1}s$, $I_A(v)$ becomes
\begin{equation}
I_A(v)=t^{\frac{3}{2}-N}\int_0^{+\infty}\frac{s^{N-2}\ee^{-s}\dd s}{\sqrt{s+tv}}\le t^{\frac{3}{2}-N}\int_{0}^{+\infty}s^{N-\frac{5}{2}}\ee^{-s}\dd s = \Gamma(N-\tfrac{3}{2})t^{\frac{3}{2}-N}.
\label{e:IA-bound}
\end{equation}
Then, making the substitution $v'=vs$, $I_B(v)$ becomes
\begin{equation}
\begin{split}
I_B(v)&=v^{N-\frac{3}{2}}\int_0^1\frac{\ee^{-tvs}s^{N-2}}{\sqrt{1-s}}\dd s\\
&=t^{\frac{3}{2}-N}\psi_B(tv),\quad \psi_B(w)\coloneq w^{N-\frac{3}{2}}\int_0^1\frac{\ee^{-ws}s^{N-2}}{\sqrt{1-s}}\dd s.
\end{split}
\end{equation}
Because $N\ge 3$, $\psi_B(w)$ is a continuous function of $w\ge 0$, and it follows from Watson's Lemma that $\psi_B(w)=\O(w^{-\frac{1}{2}})$ as $w\to+\infty$.  Therefore $\psi_B(tv)$ is bounded uniformly and so
\begin{equation}
I_B(v)=\O(t^{\frac{3}{2}-N}),\quad t\to+\infty.
\label{e:IB-bound}
\end{equation}
Similarly, making the substitution $v'=v(1+s)$, $I_C(v)$ becomes
\begin{equation}
\begin{split}
I_C(v)&=v^{N-\frac{3}{2}}\int_0^{+\infty}\frac{\ee^{-tv(1+s)}(1+s)^{N-2}}{\sqrt{s}}\dd s\\
&=t^{\frac{3}{2}-N}\psi_C(tv),\quad \psi_C(w)\coloneq w^{N-\frac{3}{2}}\ee^{-w}\int_0^{+\infty}\frac{\ee^{-ws}(1+s)^{N-2}}{\sqrt{s}}\dd s.
\end{split}
\end{equation}
Clearly $\psi_C(w)$ is continuous at least for $w>0$ and by Watson's Lemma $\psi_C(w)\to 0$ rapidly as $w\to+\infty$.  By the further substitution $s=\tau/w$ we obtain
\begin{equation}
\psi_C(w)=\ee^{-w}\int_0^{+\infty}\frac{\ee^{-\tau}(\tau+w)^{N-2}}{\sqrt{\tau}}\dd\tau
\end{equation}
and by Lebesgue dominated convergence the integral factor has a finite limit as $w\to 0$ and hence so does $\psi_C(w)$.  Therefore $\psi_C(w)$ is again bounded for $w\ge 0$ and it follows that
\begin{equation}
I_C(v)=\O(t^{\frac{3}{2}-N}),\quad t\to+\infty.
\label{e:IC-bound}
\end{equation}
Combining \eqref{e:IA-bound}, \eqref{e:IB-bound}, and \eqref{e:IC-bound}, and using \eqref{e:integral-split} in \eqref{e:Omega2-integral-bound-II} gives
\begin{equation}
\sup_{(u,v)\in\Real^2}\iint_{\Omega_2}\frac{\|\mathbf{D}_\mathrm{s}(u',v';t,z)\|\dd A(u',v')}{\sqrt{(u'-u)^2+(v'-v)^2}}=\O(t^{\frac{3}{2}-N}),\quad t\to+\infty.
\label{e:Omega2-final-estimate}
\end{equation}

Finally, for the integral over $\Omega_3$, we use \eqref{e:D-Omega-3} to get
\begin{equation}
\begin{split}
\iint_{\Omega_3}\frac{\|\mathbf{D}_\mathrm{s}(u',v';t,z)\|\dd A(u',v')}{\sqrt{(u'-u)^2+(v'-v)^2}}&\le C_N\ee^{-t}\sum_{n=0}^{N-1}\iint_{\Omega_3}\frac{|R^{(n)}(u')|\dd A(u',v')}{\sqrt{(u'-u)^2+(v'-v)^2}}\\
&=C_N\ee^{-t}\sum_{n=0}^{N-1}\int_{1<|v'|<2}\int_\Real\frac{|R^{(n)}(u')|\dd u'}{\sqrt{(u'-u)^2+(v'-v)^2}}\dd v'.
\end{split}
\end{equation}
For the inner integral we again appeal to Lemma~\ref{thm:reflection2} to get $R^{(n)}(\cdot)\in L^2(\Real)$ for $n=0,\dots,N-1$, so we conclude that the inequality \eqref{e:inner-integral-N-1} holds with $R^{(N-1)}$ replaced by $R^{(n)}$.  Therefore
\begin{equation}
\iint_{\Omega_3}\frac{\|\mathbf{D}_\mathrm{s}(u',v';t,z)\|\dd A(u',v')}{\sqrt{(u'-u)^2+(v'-v)^2}}\le C_N\ee^{-t}\sum_{n=0}^{N-1}\sqrt{\pi}\|R^{(n)}\|_{L^2(\Real)}\int_{1<|v'|<2}\frac{\dd v'}{\sqrt{|v'-v|}}.
\end{equation}
The remaining integral over $v'$ is obviously bounded independently of $v\in\Real$, so clearly
\begin{equation}
\sup_{(u,v)\in\Real^2}\iint_{\Omega_3}\frac{\|\mathbf{D}_\mathrm{s}(u',v';t,z)\|\dd A(u',v')}{\sqrt{(u'-u)^2+(v'-v)^2}}=\O(\ee^{-t}),\quad t\to+\infty.
\label{e:Omega3-final-estimate}
\end{equation}

Using \eqref{e:Omega1-final-estimate}, \eqref{e:Omega2-final-estimate}, and \eqref{e:Omega3-final-estimate} in \eqref{e:S-operator-norm} completes the proof.
\end{proof}
It follows that for $t>0$ sufficiently large with $z/t\ge 0$ sufficiently small, the integral equation \eqref{e:dbar-integral-equation} has a unique solution in $L^\infty(\Real^2)$ given by the uniformly convergent Neumann series
\begin{equation}
\dot{\mathbf{F}}_\mathrm{s}(u,v)=\dot{\mathbf{F}}_\mathrm{s}(u,v;t,z)\coloneq(\mathrm{id}-\S_{t,z})^{-1}\I(u,v) = \sum_{n=0}^\infty \S_{t,z}^n\I(u,v).
\end{equation}
It follows that
\begin{equation}
\dot{\mathbf{F}}_\mathrm{s}(u,v;t,z)-\I=\O((z/t)^{\frac{1}{2}(N-1)})+\O(t^{\frac{3}{2}-N})
\label{e:dotFs-Linfty}
\end{equation}
holds in the $L^\infty(\Real^2)$ sense.

However more is true.  Firstly, $(u,v)\mapsto\dot{\mathbf{F}}_\mathrm{s}(u,v;t,z)$ is (H\"older) continuous on $\Real^2$, and $\dot{\mathbf{F}}(u,v;t,z)\to \I$ as $u+\ii v\to\infty$, which proves that $\dot{\mathbf{F}}_\mathrm{s}(u,v;t,z)$ is indeed the solution of $\dbar$P~\ref{rhp:Fsdot}.  These facts follow from the integral equation \eqref{e:dbar-integral-equation} viewed as a formula for $\dot{\mathbf{F}}_\mathrm{s}(u,v)-\I$ as the solid Cauchy transform $\dbar^{-1}$ acting on the density $\dot{\mathbf{F}}_\mathrm{s}(u,v)\dot{\K}_\mathrm{s}(u+\ii v)\mathbf{D}_\mathrm{s}(u,v)\dot{\K}_\mathrm{s}(u+\ii v)^{-1}$.  We start with the following lemma.
\begin{lemma}[$L^p(\Real^2)$ estimate of $\mathbf{D}_\mathrm{s}(u,v;t,z)$]
The matrix-valued function $(u,v)\mapsto\mathbf{D}_\mathrm{s}(u,v;t,z)$ is in $L^p(\Real^2)$ for $1\le p\le\infty$, and $\|\mathbf{D}_\mathrm{s}(\cdot,\cdot;t,z)\|_{L^1(\Real^2)} = \O((z/t)^{\frac{1}{2}N}) + \O(t^{1-N})$ as $t\to +\infty$ with $z=o(t)$.
\label{lem:Lp}
\end{lemma}
\begin{proof}
Given $t>0$ and $z/t>0$ small enough, the estimates \eqref{e:D-Omega-1}, \eqref{e:D-Omega-2}, and \eqref{e:D-Omega-3} show that $\|\mathbf{D}_\mathrm{s}(\cdot,\cdot;t,z)\|_{L^\infty(\Real^2)}<\infty$.
By interpolation, it suffices to control $\|\mathbf{D}_\mathrm{s}(\cdot,\cdot;t,z)\|_{L^1(\Real^2)}$.  We use the same three estimates in turn to calculate the contributions to the $L^1$-norm arising from integration over the corresponding sub-regions $\Omega_j$, $j=1,2,3$.  Using \eqref{e:D-Omega-1} we get
\begin{equation}
\begin{split}
\iint_{\Omega_1}\|\mathbf{D}_\mathrm{s}(u,v;t,z)\|\dd A(u,v) &\le C_N\iint_{\Omega_1}|v|^{N-2}\dd A(u,v)\\
&\le C_N\iint_{[-2\lambda_\circ,2\lambda_\circ]^2}|v|^{N-2}\dd A(u,v) = \O(\lambda_\circ^N)=\O((z/t)^{\frac{1}{2}N}).
\end{split}
\end{equation}
Then from \eqref{e:D-Omega-2},
\begin{equation}
\begin{split}
\iint_{\Omega_2}\|\mathbf{D}_\mathrm{s}(u,v;t,z)\|\dd A(u,v)&\le C_N\iint_{\Omega_2}\left|R^{(N-1)}(u)\right||v|^{N-2}\ee^{-t|v|}\dd A(u,v)\\
&\le\iint_{\Real^2}\left|R^{(N-1)}(u)\right||v|^{N-2}\ee^{-t|v|}\dd A(u,v)\\
&=\|R^{(N-1)}\|_{L^1(\Real)}\int_\Real |v|^{N-2}\ee^{-t|v|}\dd v = \O(t^{1-N})
\end{split}
\end{equation}
where we used the fact again following from Lemma~\ref{thm:reflection2} that $R(u)$ together with all of its derivatives are in $L^1(\Real)$.  Similarly, from \eqref{e:D-Omega-3},
\begin{equation}
\begin{split}
\iint_{\Omega_3}\|\mathbf{D}_\mathrm{s}(u,v;t,z)\|\dd A(u,v)&\le C_N\ee^{-t}\sum_{n=0}^{N-1}\iint_{\Omega_3}|R^{(n)}(u)|\dd A(u,v)\\
&=2C_N\ee^{-t}\sum_{n=0}^{N-1}\|R^{(n)}\|_{L^1(\Real)} = \O(\ee^{-t}).
\end{split}
\end{equation}
Summing these estimates gives the claimed bound for $\|\mathbf{D}_\mathrm{s}(\cdot,\cdot;t,z)\|_{L^1(\Real^2)}$, which then completes the proof.
\end{proof}
In particular, since the factors $\dot{\mathbf{F}}_\mathrm{s}(\cdot,\cdot;t,z)$, $\dot{\K}_\mathrm{s}(\cdot;t,z)$ and $\dot{\K}_\mathrm{s}(\cdot;t,z)^{-1}$ are all in $L^\infty(\Real^2)$, Lemma~\ref{lem:Lp} shows that $\dot{\mathbf{F}}_\mathrm{s}(u,v;t,z)-\I$ is the solid Cauchy transform $\dbar^{-1}$ acting on a function lying in $L^q(\Real^2)\cap L^p(\Real^2)$ for a H\"older pair of conjugate exponents $1< q<2<p<\infty$.  Then, according to \cite[Theorem 4.3.11]{AstalaIM09}, $(u,v)\mapsto \dot{\mathbf{F}}_\mathrm{s}(u,v;t,z)-\I$ is a continuous function that tends to zero as $u+\ii v\to\infty$ as claimed.  Moreover, \cite[Theorem 4.3.13]{AstalaIM09} shows that, since the function acted on by $\dbar^{-1}$ is in $L^p(\Real^2)$ for arbitrarily large $p$, $(u,v)\mapsto \dot{\mathbf{F}}_\mathrm{s}(u,v;t,z)-\I$ lies in $C^{0,\alpha}(\Real^2)$ for all $\alpha<1$, giving improved H\"older continuity, arbitrarily close to the threshold of Lipschitz continuity.

Secondly, we can prove the existence of the limit
\begin{equation}
\dot{\mathbf{F}}_\mathrm{s}^{(1)}(t,z)\coloneq \lim_{v\to\infty} (u+\ii v)\left[\dot{\mathbf{F}}_\mathrm{s}(u,v;t,z)-\I\right]
\label{e:F1-def}
\end{equation}
independent of $u\in\Real$, and obtain the estimate
\begin{equation}
\dot{\mathbf{F}}_\mathrm{s}^{(1)}(t,z) = \O((z/t)^{\frac{1}{2}N}) + \O(t^{1-N})
\label{e:F1-estimate}
\end{equation}
valid as $t\to+\infty$ with $z=o(t)$.  To this end, we again appeal to \eqref{e:dbar-integral-equation} to write
\begin{equation}
(u+\ii v)\left[\dot{\mathbf{F}}_\mathrm{s}(u,v;t,z)-\I\right] =
\frac{1}{\pi}\iint_{\Real^2}\beta(u'+\ii v';u+\ii v)\mathbf{H}(u',v';t,z)\dd A(t,z)
\end{equation}
where
\begin{equation}
\beta(\lambda';\lambda)\coloneq \frac{\lambda}{\lambda-\lambda'}
\end{equation}
and where
\begin{equation}
\mathbf{H}(u',v';t,z)\coloneq
\dot{\mathbf{F}}_\mathrm{s}(u',v';t,z)\dot{\K}_\mathrm{s}(u'+\ii v';t,z)\mathbf{D}_\mathrm{s}(u',v';t,z)\dot{\mathbf{K}}_\mathrm{s}(u'+\ii v';t,z)^{-1}.
\end{equation}
Now, the support of $\mathbf{H}(u',v';t,z)$ lies in the strip $|v'|\le 2$, so if $|v|\ge 4$ we get
\begin{equation}
\begin{split}
|\beta(u'+\ii v';u+\ii v)| &= \left|1 + \frac{u'+\ii v'}{(u-u')+\ii (v-v')}\right|\\
&\le 1+\sqrt{\frac{(u')^2+(v')^2}{(u-u')^2+(v-v')^2}} \\
&\le 1+\sqrt{\frac{(u')^2+4}{(u-u')^2+4}}.
\end{split}
\end{equation}
Given $u\in\Real$, this is in $L^\infty(\Real)$ with respect to $u'$.  So there is a constant $C_u$ such that $|v|\ge 4$ implies
\begin{equation}
\mathop{\sup_{(u',v')\in\Real^2}}_{|v'|\le 2}|\beta(u'+\ii v';u+\ii v)|\le C_u.
\end{equation}
Since also $\beta(u'+\ii v';u+\ii v)\to 1$ as $v\to\infty$ in $\Real$ pointwise with respect to $(u',v')$ in the support of $\mathbf{H}(u',v';t,z)$, and since by Lemma~\ref{lem:Lp} and the boundedness of the factors $\dot{\mathbf{F}}_\mathrm{s}(u',v';t,z)$, $\dot{\K}_\mathrm{s}(u'+\ii v';t,z)$, and $\dot{\K}_\mathrm{s}(u'+\ii v';t,z)^{-1}$ we have $\mathbf{H}(\cdot,\cdot;t,z)\in L^1(\Real^2)$, it follows by Lebesgue dominated convergence that the limit in \eqref{e:F1-def} exists, and that
\begin{equation}
\dot{\mathbf{F}}_\mathrm{s}^{(1)}(t,z)=\frac{1}{\pi}\iint_{\Real^2}\mathbf{H}(u',v';t,z)\dd A(u',v').
\end{equation}
Then using the $L^\infty(\Real^2)$ bounds $\|\dot{\mathbf{F}}_\mathrm{s}(\cdot,\cdot;t,z)\|_{L^\infty(\Real^2)}=\O(1)$, $\|\dot{\K}_\mathrm{s}(\cdot;t,z)\|_{L^\infty(\Real^2)} = \O(1)$, and a corresponding bound for the inverse following from $\det(\dot{\K}_\mathrm{s}(u'+\ii v';t,z))=1$, the $L^1(\Real^2)$ control of $\mathbf{D}_\mathrm{s}(u,v;t,z)$ in Lemma~\ref{lem:Lp} proves the estimate \eqref{e:F1-estimate}.


\subsubsection{Small-norm RHP}
\label{s:small-norm-RHP-s}
Comparing $\mathbf{F}_\mathrm{s}(u,v;t,z)$ with its continuous parametrix $\dot{\mathbf{F}}_\mathrm{s}(u,v;t,z)$ solving $\dbar$P~\ref{rhp:Fsdot}, we define
\begin{equation}
\ddot{\mathbf{F}}_\mathrm{s}(u+\ii v;t,z)\coloneq\mathbf{F}_\mathrm{s}(u,v;t,z)\dot{\mathbf{F}}_\mathrm{s}(u,v;t,z)^{-1},\quad (u,v)\in\Real^2.
\label{e:F-dotF-ddotF-s}
\end{equation}
Note that this definition makes sense because $\dot{\mathbf{F}}_\mathrm{s}(u,v;t,z)-\I$ is uniformly small for $t>0$ large and $z/t$ sufficiently small.  This matrix function tends to the identity as $u+\ii v\to\infty$ because this is true of both factors (by hypothesis for $\mathbf{F}_\mathrm{s}(u,v;t,z)$ and by construction for $\dot{\mathbf{F}}_\mathrm{s}(u,v;t,z)$), and for $u+\ii v\in\Complex\setminus\Sigma$ we have
\begin{multline}
\dbar\ddot{\mathbf{F}}_\mathrm{s}(u+\ii v;t,z)=\dbar\mathbf{F}_\mathrm{s}(u,v;t,z)\cdot\dot{\mathbf{F}}_\mathrm{s}(u,v;t,z)^{-1}\\ -
\mathbf{F}_\mathrm{s}(u,v;t,z)\dot{\mathbf{F}}_\mathrm{s}(u,v;t,z)^{-1}\dbar\dot{\mathbf{F}}_\mathrm{s}(u,v;t,z)\cdot
\dot{\mathbf{F}}_\mathrm{s}(u,v;t,z)^{-1},
\end{multline}
which vanishes identically by using \eqref{e:Fs-dbar} and \eqref{e:dotFs-dbar} (justifying our use of the complex notation $u+\ii v$ for the argument).  The relation between the boundary values taken by $\ddot{\mathbf{F}}_\mathrm{s}(u+\ii v;t,z)$ from opposite sides on the three arcs of $\Sigma$ can be easily computed from the established continuity of $\dot{\mathbf{F}}_\mathrm{s}(u,v;t,z)$ on the $(u,v)$-plane and the jump condition \eqref{e:Fs-jump}.  These facts imply that $\ddot{\mathbf{F}}_\mathrm{s}(u+\ii v;t,z)$ is the solution of the following RHP.
\begin{rhp}[Small norm problem]
Given $t\ge 0$ and $z\ge 0$, seek a $2\times 2$ matrix-valued function $\lambda\mapsto\ddot{\mathbf{F}}_\mathrm{s}(\lambda;t,z)$ that is analytic for $\lambda\in\Complex\setminus\Sigma$; that satisfies $\ddot{\mathbf{F}}_\mathrm{s}\to\I$ as $\lambda\to\infty$; and that takes continuous boundary values on $\Sigma$ from each component of the complement related by the jump condition $\ddot{\mathbf{F}}_\mathrm{s}^+(\lambda;t,z)=\ddot{\mathbf{F}}_\mathrm{s}^-(\lambda;t,z)\mathbf{G}_\mathrm{s}(u,v;t,z)$, $\lambda=u+\ii v$, where the jump matrix $\mathbf{G}_\mathrm{s}(u,v;t,z)$ is defined on the arcs of $\Sigma$ by
\begin{equation}
\mathbf{G}_\mathrm{s}(u,v;t,z)\coloneq
\dot{\mathbf{F}}_\mathrm{s}(u,v;t,z)\dot{\K}^-_\mathrm{s}(u+\ii v;t,z)\mathbf{J}_\mathrm{s}(u,v;t,z)\dot{\mathbf{J}}_\mathrm{s}(u+\ii v;t,z)^{-1}\dot{\K}^-_\mathrm{s}(u+\ii v;t,z)^{-1}\dot{\mathbf{F}}_\mathrm{s}(u,v;t,z)^{-1}.
\end{equation}
Here $\dot{\mathbf{F}}_\mathrm{s}(u,v;t,z)$ denotes the solution of $\dbar$P~\ref{rhp:Fsdot} analyzed in Section~\ref{s:dbar-s}; $\dot{\K}_\mathrm{s}(\lambda;t,z)$ denotes the solution of RHP~\ref{rhp:dotKs} analyzed (by means of the equivalent RHP~\ref{rhp:ddotKs}) in Sections~\ref{s:Properties-of-Y} and \ref{s:Properties-for-M-positive-s}; and $\dot{\mathbf{J}}_\mathrm{s}(\lambda;t,z)$ is a bounded, unit determinant jump matrix that is explicitly defined in \eqref{e:dot-Js} and is compared with $\mathbf{J}_\mathrm{s}(u,v;t,z)$ in \eqref{e:Js-approximation}.
\label{rhp:s-small-norm}
\end{rhp}

It follows from the uniform bound \eqref{e:dotFs-Linfty}, the uniform bounds for $\dot{\K}_\mathrm{s}(\lambda;t,z)$ and its inverse, the boundedness of $\dot{\mathbf{J}}_\mathrm{s}(\lambda;t,z)^{-1}$ and the estimate \eqref{e:Js-approximation} that $\mathbf{G}_\mathrm{s}(u,v;t,z)=\I + \O(\lambda_\circ^{M+1}) = \I+\O((z/t)^{\frac{1}{2}(M+1)})$ holds uniformly on $\Sigma$ as $t\to+\infty$ with $z=o(t)$.  The jump matrix $\mathbf{G}_\mathrm{s}(u,v;t,z)$ is also clearly H\"older continuous on each arc of $\Sigma$ with H\"older index $\alpha$ as close to $1$ as desired, and satisfies the necessary consistency condition at the two self-intersection points to admit a classical solution.  By the theory of RHPs in H\"older spaces (see, e.g., \cite[Appendix A]{KamvissisMM03}), using also the fact that although the contour $\Sigma$ depends on the parameters $(t,z)\in\Real^2$ it does so by scaling which commutes with the Cauchy integral operators on $\Sigma$ hence leaving their norms invariant, RHP~\ref{rhp:s-small-norm} is a classical small-norm problem, having a unique classical solution taking H\"older-continuous boundary values on the arcs of $\Sigma$ and satisfying the estimates
\begin{equation}
\ddot{\mathbf{F}}_\mathrm{s}(\lambda;t,z)=\I + \O((z/t)^{\frac{1}{2}(M+1)})
\label{e:ddotF-uniform-bound-s}
\end{equation}
holding uniformly for $\lambda\in\Complex\setminus\Sigma$, and
\begin{equation}
\begin{split}
\ddot{\mathbf{F}}_\mathrm{s}^{(1)}(t,z)\coloneq\lim_{\lambda\to\infty}\lambda\left[\ddot{\mathbf{F}}_\mathrm{s}(\lambda;t,z)-\I\right] &= -\frac{1}{2\pi\ii}\int_\Sigma \ddot{\mathbf{F}}_\mathrm{s}^-(\lambda;t,z)\left[\mathbf{G}_\mathrm{s}(\Re(\lambda),\Im(\lambda);t,z)-\I\right]\dd\lambda \\
&= \O((z/t)^{\frac{1}{2}(M+2)})
\end{split}
\label{e:ddotF-moment-bound-s}
\end{equation}
(the extra factor of $\lambda_\circ=(z/t)^\frac{1}{2}$ in the estimate \eqref{e:ddotF-moment-bound-s} comes from the total arc length of $\Sigma$) both in the limit that $t\to\infty$ with $z=o(t)$.

\begin{remark}
Based on the numerical experiments reported in Section~\ref{s:numerical-verification}, the estimates \eqref{e:ddotF-uniform-bound-s}--\eqref{e:ddotF-moment-bound-s} might not be sharp.  The same comment applies to the estimates \eqref{e:ddotF-uniform-bound-u}--\eqref{e:ddotF-moment-bound-u} below.
\label{rem:not-sharp}
\end{remark}

\subsubsection{Combining the parametrices}
From \eqref{e:K-dotK-F-s} and \eqref{e:F-dotF-ddotF-s} we obtain a representation of the solution $\K_\mathrm{s}(u,v;t,z)$ of RH$\dbar$P~\ref{rhp:Ks} in the form
\begin{equation}
\begin{split}
\K_\mathrm{s}(u,v;t,z)&=\mathbf{F}_\mathrm{s}(u,v;t,z)\dot{\K}_\mathrm{s}(u+\ii v;t,z)\\ &=\ddot{\mathbf{F}}_\mathrm{s}(u+\ii v;t,z)\dot{\mathbf{F}}_\mathrm{s}(u,v;t,z)\dot{\K}_\mathrm{s}(u+\ii v;t,z).
\end{split}
\label{e:Ks-formula}
\end{equation}

\subsection{Asymptotic formul\ae\ for $q(t,z)$, $P(t,z)$, and $D(t,z)$}
\subsubsection{Expressing $q$, $P$, and $D$ in terms of the modified parametrix $\ddot{\K}_\mathrm{s}$}
\label{s:potentials-in-terms-of-dotKs}
Since each of the three factors on the right-hand side of \eqref{e:Ks-formula} tends to $\I$ and has a finite first moment at $\lambda=u+\ii v=\infty$ for $t>0$ sufficiently large and $z/t>0$ sufficiently small, combining \eqref{e:Ks-reconstruction} with
\eqref{e:F1-def}--\eqref{e:F1-estimate} and \eqref{e:ddotF-moment-bound-s}
gives
\begin{equation}
\begin{split}
q(t,z)&=-2\ii\lim_{\lambda\to\infty}\lambda\ddot{F}_{s,1,2}(\lambda;t,z) - 2\ii\lim_{u+\ii v\to\infty}(u+\ii v)\dot{F}_{s,1,2}(\lambda;t,z) - 2\ii\lim_{\lambda\to\infty}\lambda \dot{K}_{s,1,2}(\lambda;t,z)\\
&=- 2\ii\lim_{\lambda\to\infty}\lambda \dot{K}_{s,1,2}(\lambda;t,z) + \O((z/t)^{\frac{1}{2}(M+2)}) + \O((z/t)^{\frac{1}{2}N}) + \O(t^{1-N}).
\end{split}
\label{e:q-formula1-s}
\end{equation}
Likewise, using the fact that the boundary values $\dot{\K}_\mathrm{s}^\pm(0;t,z)$ are bounded and have unit determinant, combining \eqref{e:Ks-reconstruction} with \eqref{e:dotFs-Linfty} and \eqref{e:ddotF-uniform-bound-s}
gives
\begin{equation}
\brho(t,z)=-\lim_{\lambda\to 0}\dot{\K}_\mathrm{s}(\lambda;t,z)\sigma_3\dot{\K}_\mathrm{s}(\lambda;t,z)^{-1} +
\O((z/t)^{\frac{1}{2}(M+1)}) + \O((z/t)^{\frac{1}{2}(N-1)}) + \O(t^{\frac{3}{2}-N}).
\label{e:rho-formula1-s}
\end{equation}
Now we calculate the explicit terms in \eqref{e:q-formula1-s} and \eqref{e:rho-formula1-s}.  From \eqref{e:dotKs-ddotKs} we get
\begin{equation}
\begin{split}
-2\ii\lim_{\lambda\to\infty}\lambda\dot{K}_{s,1,2}(\lambda;t,z) &= -2\ii\lambda_\circ\ee^{-\ii(\arg(a_M)+\aleph)}\lim_{k\to\infty}k\ddot{K}_{s,1,2}(k;t,z)\\
-\lim_{\lambda\to 0}\dot{\K}_\mathrm{s}(\lambda;t,z)\sigma_3\dot{\K}_\mathrm{s}(\lambda;t,z)^{-1}&=
-\ee^{-\frac{1}{2}\ii[\arg(a_M)+\aleph]\sigma_3}\ddot{\K}_\mathrm{s}(0;t,z)\sigma_3\ddot{\K}_\mathrm{s}(0;t,z)^{-1}\ee^{\frac{1}{2}\ii[\arg(a_M)+\aleph]\sigma_3}.
\end{split}
\label{e:explicit-terms-1}
\end{equation}
To continue the calculation, for the modified parametrix $\ddot{\K}_\mathrm{s}(k;t,z)$ we need to distinguish the case $M=0$ (see Section~\ref{s:Properties-of-Y}) from $M>0$ (see Section~\ref{s:Properties-for-M-positive-s}).

\subsubsection{Expansions for $M=0$}
\label{s:potentials-s-M-zero}
If $M=0$, we recall $a_0=r_0\neq 0$ and use \eqref{e:ddotK-Y} and \eqref{e:Lax-potentials} to get
\begin{equation}
\begin{split}
-2\ii\lambda_\circ\ee^{-\ii(\arg(a_0)+\aleph)}\lim_{k\to\infty}k\ddot{K}_{s,1,2}(k;t,z) &=
-2\lambda_\circ\ee^{-\ii(\arg(r_0)+\aleph)}\lim_{\Lambda\to\infty}\Lambda Y_{1,2}(\Lambda;-\ii\sqrt{2tz},\arctan(|r_0|))\\
&=2\frac{\lambda_\circ}{\sqrt{2tz}}\ee^{-\ii(\arg(r_0)+\aleph)}y(-\ii \sqrt{2tz};\omega)\\
&=t^{-1}\ee^{-\ii(\arg(r_0)+\aleph)}y(-\ii\sqrt{2tz};\omega),
\end{split}
\end{equation}
where $\omega=-3\cos(2\arctan(|r_0|))$ is given equivalently by \eqref{e:omega-r0}.  Recalling from \eqref{e:mbe-matrix} that $P(t,z)=\rho_{1,2}(t,z)$ and $D(t,z)=\rho_{1,1}(t,z)$, we take advantage of the fact that $\ddot{\K}_\mathrm{s}(k;t,z)$ has unit determinant to similarly compute
\begin{multline}
\left[-\ee^{-\frac{1}{2}\ii[\arg(a_0)+\aleph]\sigma_3}\ddot{\K}_\mathrm{s}(0;t,z)\sigma_3\ddot{\K}_\mathrm{s}(0;t,z)^{-1}\ee^{\frac{1}{2}\ii[\arg(a_0)+\aleph]\sigma_3}\right]_{1,2}\\
\begin{aligned}
&=
2\ee^{-\ii(\arg(r_0)+\aleph)}Y_{1,1}(0;-\ii\sqrt{2tz},\arctan(|r_0|))Y_{1,2}(0;-\ii\sqrt{2tz},\arctan(|r_0|))\\
&=-\frac{2\ii}{\sqrt{2tz}}\ee^{-\ii(\arg(r_0)+\aleph)}s(-\ii\sqrt{2tz};\omega)
\end{aligned}
\end{multline}
and
\begin{multline}
\left[-\ee^{-\frac{1}{2}\ii[\arg(a_0)+\aleph]\sigma_3}\ddot{\K}_\mathrm{s}(0;t,z)\sigma_3\ddot{\K}_\mathrm{s}(0;t,z)^{-1}\ee^{\frac{1}{2}\ii[\arg(a_0)+\aleph]\sigma_3}\right]_{1,1}\\
\begin{aligned}
&=-1-2Y_{1,2}(0;-\ii\sqrt{2tz},\arctan(|r_0|))Y_{2,1}(0;-\ii\sqrt{2tz},\arctan(|r_0|))\\
&=-1+\frac{2\ii}{\sqrt{2tz}}U(-\ii\sqrt{2tz};\omega).
\end{aligned}
\end{multline}
Using these formul\ae\ in \eqref{e:q-formula1-s}--\eqref{e:explicit-terms-1} for $M=0$ and comparing with Definition~\ref{def:self-similar} then gives \eqref{e:qPD-generic-global-stable} and proves Theorem~\ref{thm:global-M0-stable-and-unstable} in the case of propagation in an initially-stable medium.
To prove Corollary~\ref{cor:medium-bulk-M0-stable-and-unstable} for this case, we simply use the asymptotic formul\ae\ \eqref{e:y-im-asymp}--\eqref{e:U-im-asymp} in the expressions \eqref{e:stable-self-similar} for the leading self-similar terms in terms of PIII functions on the imaginary axis.

\subsubsection{Expansions for $M>0$}
\label{s:potentials-s-M-positive}
For $M>0$, we recall $a_M=r_0^{(M)}/M!$ and use \eqref{e:ddotKs-moment-define}--\eqref{e:ddotKs-moment-approx} with $x=\sqrt{2tz}$ and $\lambda_\circ=\sqrt{z/(2t)}$ to get
\begin{multline}
-2\ii\lambda_\circ\ee^{-\ii(\arg(a_M)+\aleph)}\lim_{k\to\infty}k\ddot{K}_{s,1,2}(k;t,z)\\
\begin{aligned}
&=
\frac{\b{r_0^{(M)}}\ee^{-\ii\aleph}}{\pi M!}\left(\frac{z}{2t}\right)^{\frac{1}{2}(M+1)}\oint_{|s|=1}s^M\ee^{\ii \sqrt{2tz}(s+s^{-1})}\dd s + \O((z/t)^{\frac{1}{2}(3M+1)})\\
&=-\ii^M\frac{\b{r_0^{(M)}}\ee^{-\ii\aleph}}{M!}2^{\frac{1}{2}(1-M)}\left(\frac{z}{t}\right)^{\frac{1}{2}(M+1)}J_{M+1}(2\sqrt{2tz})+\O((z/t)^{\frac{1}{2}(3M+1)}),
\end{aligned}
\end{multline}
where we used the integral representation of the Bessel function $J_{n}(\cdot)$ given in \cite[Eqn.\@ 10.9.2]{dlmf}.  Similarly, from \eqref{e:ddotKs-origin-approx} and $\det(\ddot{\K}_\mathrm{s}(k;t,z))=1$ we get
\begin{multline}
\left[-\ee^{-\frac{1}{2}\ii[\arg(a_M)+\aleph]\sigma_3}\ddot{\K}_\mathrm{s}(0;t,z)\sigma_3\ddot{\K}_\mathrm{s}(0;t,z)^{-1}\ee^{\frac{1}{2}\ii[\arg(a_M)+\aleph]\sigma_3}\right]_{1,2}\\
\begin{aligned}
&=2\ee^{-\ii(\arg(r_0^{(M)})+\aleph)}\ddot{K}_{s,1,1}(0;t,z)\ddot{K}_{s,1,2}(0;t,z)\\
&=\frac{\b{r_0^{(M)}}\ee^{-\ii\aleph}}{\ii\pi M!}\left(\frac{z}{2t}\right)^{\frac{1}{2}M}\oint_{|s|=1}s^{M-1}\ee^{\ii\sqrt{2tz}(s+s^{-1})}\dd s + \O((z/t)^{\frac{3}{2}M})\\
&=\ii^M\frac{\b{r_0^{(M)}}\ee^{-\ii\aleph}}{M!}2^{1-\frac{1}{2}M}\left(\frac{z}{t}\right)^{\frac{1}{2}M}J_M(2\sqrt{2tz}) + \O((z/t)^{\frac{3}{2}M})
\end{aligned}
\end{multline}
and
\begin{multline}
\left[-\ee^{-\frac{1}{2}\ii[\arg(a_M)+\aleph]\sigma_3}\ddot{\K}_\mathrm{s}(0;t,z)\sigma_3\ddot{\K}_\mathrm{s}(0;t,z)^{-1}\ee^{\frac{1}{2}\ii[\arg(a_M)+\aleph]\sigma_3}\right]_{1,1}\\
\begin{aligned}
&=-1-2\ddot{K}_{s,1,2}(0;t,z)\ddot{K}_{s,2,1}(0,t;z)\\
&=-1-\left(\frac{|r_0^{(M)}|}{2\pi M!}\right)^22^{1-M}\left(\frac{z}{t}\right)^{M}\left[\oint_{|s|=1}s^{M-1}\ee^{\ii\sqrt{2tz}(s+s^{-1})}\dd s\right]\left[\oint_{|s|=1}s^{M-1}\ee^{-\ii\sqrt{2tz}(s+s^{-1})}\dd s\right] + \O((z/t)^{2M})\\
&= -1+\left(\frac{|r_0^{(M)}|}{M!}\right)^22^{1-M}\left(\frac{z}{t}\right)^MJ_M(2\sqrt{2tz})^2 + \O((z/t)^{2M}),
\end{aligned}
\end{multline}
where we also used the fact that $J_n(-\cdot)=(-1)^nJ_n(\cdot)$.  In all three of these formul\ae, the error terms can be absorbed into those already present in \eqref{e:q-formula1-s}--\eqref{e:rho-formula1-s} under the condition $M\ge 1$ and $z=o(t)$.  Also, the condition $N\ge M+2$ allows two of the error terms in each of \eqref{e:q-formula1-s}--\eqref{e:rho-formula1-s} to be combined.  This establishes the asymptotic formul\ae\ \eqref{e:qPD-global-Mpos-stable} and completes the proof of Theorem~\ref{thm:global-Mpos-stable}.

\section{Analysis for propagation in an initially-unstable medium}
\label{s:unstable-positive}
For propagation in an initially-unstable medium, the coefficient of $\lambda^{-1}$ in the phase $\theta(\lambda;t,z)$ has opposite sign compared to the stable medium case.  We make the following definition:
\begin{definition}[the phase for $D_-=1$]
In the unstable case, we denote the phase $\theta(\lambda;t,z)$ appearing in \eqref{e:phase-def-general} as
$\theta(\lambda;t,z) = \theta_\mathrm{u}(\lambda;t,z)$, where
\begin{equation}
\label{e:thetau-def}
\theta_\mathrm{u}(\lambda;t,z) \coloneq \lambda t - \frac{z}{2\lambda}\,.
\end{equation}
\label{def:theta-u}
\end{definition}
The analysis is again driven by the sign chart of the real part of $\ii\theta(\lambda;t,z)$, but this is now changed compared to the stable case, as seen in the right-hand panel of Figure~\ref{f:inside-itheta}.

\subsection{Setting up a Riemann-Hilbert-$\dbar$ problem}
\label{s:unstable-positive-1}

If we directly mimic the transformation from $\M(\lambda;t,z)$ to $\K_\mathrm{s}(u,v;t,z)$ as in \eqref{e:Ks-def} simply replacing $\theta_\mathrm{s}(\lambda;t,z)$ with $\theta_\mathrm{u}(\lambda;t,z)$, then the resulting jump matrix on the circle $|\lambda|=\lambda_\circ$ has an exponential rather than oscillatory character.  This problem requires further stabilization of exactly the type that is supplied by the mechanism of a ``$g$-function'' in the Deift-Zhou steepest descent theory.  Here, the implementation is particularly simple:  we multiply on the right by $\ee^{\ii g(\lambda;t,z)\sigma_3}$ where
\begin{equation}
g(\lambda;t,z)\coloneq\begin{cases} -z\lambda^{-1},&\quad |\lambda|>\lambda_\circ\\
2t\lambda,&\quad |\lambda|<\lambda_\circ.
\end{cases}
\end{equation}
Thus, rather than \eqref{e:Ks-def} we define in this case
\begin{equation}
\label{e:Ku-def}
\K_\mathrm{u}(u,v;t,z) \coloneq
  \begin{cases}
  \M(u+\ii v;t,z)\,\delta(u+\ii v)^{-\sigma_3}\ee^{2\ii t\lambda\sigma_3}\,, &\qquad |u+\ii v|<\lambda_\circ\,,\\
  \M(u+\ii v;t,z)\mathbf{T}_\mathrm{u}^{\dagger}(u,v;t,z)^{-1}\delta(u+\ii v)^{-\sigma_3}\ee^{-\ii z\lambda^{-1}\sigma_3}\,, &\qquad |u+\ii v| > \lambda_\circ\,,\qquad v > 0\,,\\
  \M(u+\ii v;t,z)\mathbf{T}_\mathrm{u}(u,v;t,z)\delta(u+\ii v)^{-\sigma_3}\ee^{-\ii z\lambda^{-1}\sigma_3}\,, &\qquad |u+\ii v|> \lambda_\circ\,,\qquad v < 0\,.
  \end{cases}
\end{equation}
Here the matrix $\mathbf{T}_\mathrm{u}(u,v;t,z)$ is given by
\begin{equation}
\mathbf{T}_\mathrm{u}(u,v;t,z) \coloneq \bpm 1 & 0 \\ B(v)Q_N(u,v) \ee^{-2\ii\theta_\mathrm{u}(u+\ii v;t,z)} & 1\epm\,,\quad (u,v)\in\Real^2,
\end{equation}
in which the ``bump'' function $B$ and the non-analytic extension $Q_N(u,v)$ of $R(u)$ are exactly as defined in Section~\ref{s:stable-positive-1}.  The main point of introducing the $g$-function in the transformation is that $\K_\mathrm{u}(u,v;t,z)$ satisfies a RH$\dbar$P that \emph{only involves the stable-medium phase $\theta_\mathrm{s}(\lambda;t,z)$, which is real-valued on the jump contour $\Sigma$}.  That problem is the following.
\begin{rhdp}
\label{rhp:Ku}
Given $t\ge 0$ and $z\ge 0$, seek a $2\times 2$ matrix-valued function $\Real^2\ni (u,v)\mapsto \K_\mathrm{u}(u,v;t,z)$ that is continuous for $(u,v)\in\Real^2\setminus\Sigma$; that satisfies $\K_\mathrm{u}\to\I$ as $u+\ii v\to\infty$; that takes continuous boundary values on $\Sigma$ from each component of the complement related by the jump conditions $\K_\mathrm{u}^+(u,v;t,z)=\K_\mathrm{u}^-(u,v;t,z)\mathbf{J}_\mathrm{u}(u,v;t,z)$ where
\begin{multline}
\mathbf{J}_\mathrm{u}(u,v;t,z)\\
\coloneq\begin{cases}
\ee^{\ii z\lambda^{-1}\sigma_3}\delta(u+\ii v)^{\sigma_3}\mathbf{T}_\mathrm{u}^\dagger(u,v;t,z)\W(u+\ii v;t,z)^{-1}\delta(u+\ii v)^{-\sigma_3}\ee^{2\ii t\lambda\sigma_3}\,,&
|u + \ii v| = \lambda_\circ\,,\quad v > 0\,,\\
\ee^{\ii z\lambda^{-1}\sigma_3}\delta(u+\ii v)^{\sigma_3}\mathbf{T}_\mathrm{u}(u,v;t,z)^{-1}\W^\dagger(u+\ii v;t,z)\delta(u+\ii v)^{-\sigma_3}\ee^{2\ii t\lambda\sigma_3}\,,& |u + \ii v| = \lambda_\circ\,,\quad v < 0\,,\\
(1 + |r(u)|^2)^{-\sigma_3}\,,& u\in(-\lambda_\circ,\lambda_\circ)\,,\quad v = 0\,;
\end{cases}
\label{e:Ju-def}
\end{multline}
and that satisfies the following $\b{\partial}$ differential equation
\begin{equation}
\dbar \K_\mathrm{u}(u,v;t,z) = \K_\mathrm{u}(u,v;t,z)\mathbf{D}_\mathrm{s}(u,v;t,z)\,,\qquad (u,v)\in\Real^2\setminus\Sigma\,,
\label{e:dbar-equation-u}
\end{equation}
where the matrix $\mathbf{D}_\mathrm{s}(u,v;t,z)$ is given by \eqref{e:dbar-equation-D}.
\end{rhdp}
Thus, the $\dbar$ equation \eqref{e:dbar-equation-u} is \emph{exactly the same as in the stable-medium case}.  The jump matrix $\mathbf{J}_\mathrm{u}(u,v;t,z)$ is the same as $\mathbf{J}_\mathrm{s}(u,v;t,z)$ for $-\lambda_\circ<u<\lambda_\circ$ and $v=0$, while for $|u+\ii v|=\lambda_\circ$ it is obtained from $\mathbf{J}_\mathrm{s}(u,v;t,z)$ defined in \eqref{e:Js-def} by simply omitting the oscillatory exponentials $\ee^{\pm 2\ii\theta_\mathrm{s}(u+\ii v;t,z)}$ from the off-diagonal elements and instead moving them onto the diagonal elements.

Since $\ee^{\ii g(\lambda;t,z)\sigma_3}\to\I$ as $\lambda\to\infty$ and as $\lambda\to 0$, the optical field $q(t,z)$ and density matrix $\brho(t,z)$ are obtained from $\K_\mathrm{u}(u,v;t,z)$ by
formul\ae\ similar to \eqref{e:Ks-reconstruction} (but accounting for the change of sign of $D_-$ in the formula for $\brho(t,z)$):
\begin{equation}
\begin{split}
q(t,z)&=-2\ii\lim_{u+\ii v\to\infty}(u+\ii v)K_{u,1,2}(u,v;t,z),\\
\brho(t,z)&=\lim_{u+\ii v\to 0}\K_\mathrm{u}(u,v;t,z)\sigma_3\K_\mathrm{u}(u,v;t,z)^{-1}.
\end{split}
\label{e:Ku-reconstruction}
\end{equation}

\subsection{Construction of a parametrix}
To obtain a parametrix, we neglect the matrix $\mathbf{D}_\mathrm{s}(u,v;t,z)$ in the $\dbar$ equation and approximate the jump matrix $\mathbf{J}_\mathrm{u}(u,v;t,z)$ accurately on $\Sigma$ by exactly the same arguments as in Section~\ref{s:stable-positive-parametrix} leading to the uniform approximation
\begin{equation}
\mathbf{J}_\mathrm{u}(u,v;t,z)=\dot{\mathbf{J}}_\mathrm{u}(u+\ii v;t,z)+\O(\lambda_\circ^{M+1}),\quad u+\ii v\in\Sigma,
\label{e:Ju-approximation}
\end{equation}
where
\begin{equation}
\dot{\mathbf{J}}_\mathrm{u}(\lambda;t,z)\coloneq
\begin{cases}
\bpm \Delta_M(\lambda)^{-1}\ee^{2\ii\theta_\mathrm{s}(\lambda;t,z)} & \b{a_M}\lambda^M\ee^{-\ii\aleph}\\
-\Delta_M(\lambda)^{-1}a_M\lambda^M\ee^{\ii\aleph} & \ee^{-2\ii\theta_\mathrm{s}(\lambda;t,z)}\epm,&\quad |\lambda|=\lambda_\circ,\quad \Im(\lambda)>0\,,\\
\bpm \ee^{2\ii\theta_\mathrm{s}(\lambda;t,z)} & \Delta_M(\lambda)^{-1}\b{a_M}\lambda^M\ee^{-\ii\aleph}\\
-a_M\lambda^M\ee^{\ii\aleph} & \Delta_M(\lambda)^{-1}\ee^{-2\ii\theta_\mathrm{s}(\lambda;t,z)}\epm,&\quad |\lambda|=\lambda_\circ,\quad\Im(\lambda)<0\,,\\
\Delta_M(\lambda)^{-\sigma_3},&\quad \lambda\in (-\lambda_\circ,\lambda_\circ)\,,
\end{cases}
\label{e:dot-Ju}
\end{equation}
in which $a_M$ and $\Delta_M(\lambda)$ are defined in \eqref{e:avoid-factorials}--\eqref{e:DeltaM-def} and $\aleph$ is from Definition~\ref{def:aleph}.
We are therefore led to consider the following RHP:
\begin{rhp}[Parametrix for $\K_\mathrm{u}$]
\label{rhp:dotKu}
Given $t\ge 0$ and $z\ge 0$, seek a $2\times 2$ matrix-valued function $\lambda\mapsto\dot{\K}_\mathrm{u}(\lambda;t,z)$ that is analytic for $\lambda\in\Complex\setminus\Sigma$; that satisfies $\dot{\K}_\mathrm{u}\to\I$ as $\lambda\to\infty$; and that takes continuous boundary values on $\Sigma$ from each component of the complement related by the jump conditions $\dot{\K}_\mathrm{u}^+(\lambda;t,z)=\dot{\K}_\mathrm{u}^-(\lambda;t,z)\dot{\mathbf{J}}_\mathrm{u}(\lambda;t,z)$ where the jump matrix $\dot{\mathbf{J}}_\mathrm{u}(\lambda;t,z)$ is defined on $\Sigma$ by \eqref{e:dot-Ju}.
\end{rhp}
By following the same procedure as indicated for RHP~\ref{rhp:dotKs} in Section~\ref{s:stable-positive-parametrix}, one easily checks that this problem has a unique solution.  Then, by the exact same substitution \eqref{e:dotKs-ddotKs} replacing $\dot{\K}_\mathrm{s}(\lambda;t,z)$ and $\ddot{\K}_\mathrm{s}(k;t,z)$ by $\dot{\K}_\mathrm{u}(\lambda;t,z)$ and $\ddot{\K}_\mathrm{u}(k;t,z)$ respectively we obtain a simpler RHP equivalent to RHP~\ref{rhp:dotKu}:
\begin{rhp}[Modified parametrix for $\K_\mathrm{u}$]
Given $t\ge 0$ and $z\ge 0$, seek a $2\times 2$ matrix-valued function $k\mapsto\ddot{\K}_\mathrm{u}(k;t,z)$ that is analytic for $|k|\neq 1$; that satisfies $\ddot{\K}_\mathrm{u}\to\I$ as $k\to\infty$; and that takes continuous boundary values on $|k|=1$ from the interior and exterior related by the jump condition
\begin{equation}
\ddot{\K}_\mathrm{u}^+(k;t,z)=\ddot{\K}_\mathrm{u}^-(k;t,z)\bpm \Delta_M(\lambda_\circ k)^{-\frac{1}{2}}\ee^{\ii x(k+k^{-1})} & \lambda_\circ^M |a_M|\Delta_M(\lambda_\circ k)^{-\frac{1}{2}} k^M\\
-\lambda_\circ^M |a_M|\Delta_M(\lambda_\circ k)^{-\frac{1}{2}}k^M & \Delta_M(\lambda_\circ k)^{-\frac{1}{2}}\ee^{-\ii x(k+k^{-1})}\epm,\quad |k|=1,
\label{e:ddotKu-jump}
\end{equation}
where $\lambda_\circ$ is defined in terms of $(t,z)$ by \eqref{e:lambdao-def} and $x=\sqrt{2tz}$.
\label{rhp:ddotKu}
\end{rhp}

\subsection{Properties of the modified parametrix:  $M=0$}
It is easy to check that when $M=0$,
\begin{equation}
\ddot{\K}_\mathrm{u}(k;t,z)=
\widetilde{\mathbf{Y}}(k;x,\arctan(|r_0|)),\quad x=\sqrt{2tz},
\label{e:ddotKu-tilde-Y}
\end{equation}
where $\widetilde{\mathbf{Y}}(\Lambda;X,\eta)$ is explicitly defined by \eqref{e:tilde-Y} in terms of the solution $\mathbf{Y}(\Lambda;X,\eta)$ of RHP~\ref{rhp:PIII} described in detail in Section~\ref{s:Properties-of-Y}.  It follows from the discussion in Section~\ref{s:Boundedness-of-RHP-PIII} that $\ddot{\K}_\mathrm{u}(k;t,z)$ is uniformly bounded for $|k|\neq 1$ and $x>0$.

\subsection{Properties of the modified parametrix:  $M>0$}
Unlike $\ddot{\K}_\mathrm{s}(k;t,z)$, when $M>0$, $\ddot{\K}_\mathrm{u}(k;t,z)$ does \emph{not} satisfy the conditions of a small-norm RHP, the main obstruction being the oscillatory factors $\ee^{\pm\ii x(k+k^{-1})}$ on the diagonal elements of the jump matrix in \eqref{e:ddotKu-jump}.  These are easily removed, essentially by reversing the introduction of the $g$-function, but the cost is that non-oscillatory (exponentially growing in $x$) factors $\ee^{\pm\ii (k-k^{-1})}$ will appear on the off-diagonal elements of the jump matrix.  However, these potentially-dangerous exponential factors will be multiplied by $\lambda_\circ^M$, which is small.  More precisely, setting
\begin{equation}
\dddot{\K}_\mathrm{u}(k;t,z)\coloneq
\begin{cases}
\ddot{\K}_\mathrm{u}(k;t,z)\ee^{-\ii xk\sigma_3},&\quad |k|<1,\\
\ddot{\K}_\mathrm{u}(k;t,z)\ee^{\ii xk^{-1}\sigma_3},&\quad |k|>1,
\end{cases}
\label{e:dddotKu-ddotKu}
\end{equation}
we see that $\dddot{\K}_\mathrm{u}(k;t,z)$ is analytic for $|k|\neq 1$ and tends to the identity as $k\to\infty$, and that the jump
condition \eqref{e:ddotKu-jump} on the counterclockwise-oriented unit circle $|k|=1$ can be written in the form
\begin{equation}
\dddot{\K}^+_\mathrm{u}(k;t,z)=\dddot{\K}^-_\mathrm{u}(k;t,z)
\bpm 1+\O(\lambda_\circ^{2M}) & \left[\lambda_\circ^M|a_M|k^M+\O(\lambda_\circ^{3M})\right]\ee^{\ii x(k-k^{-1})}\\
-\left[\lambda_\circ^M|a_M|k^M+\O(\lambda_\circ^{3M})\right]\ee^{-\ii x(k-k^{-1})} & 1+\O(\lambda_\circ^{2M})\epm.
\end{equation}
Since $\pm\ii(k-k^{-1})$ is real-valued for $|k|=1$ with maximum value of $2$, $\dddot{\K}_\mathrm{u}(k;t,z)$ will indeed satisfy the conditions of a small-norm RHP provided that
\begin{itemize}
\item $\lambda_\circ$ is small, i.e., $z=o(t)$ as $t\to+\infty$, and
\item $\lambda_\circ^M\ee^{2x}$ is small, i.e., $x=\sqrt{2tz}=o(\ln(t/z))$.
\end{itemize}
For simplicity, we will assume that $x=\O(1)$ as $t\to+\infty$, i.e., $z=\O(t^{-1})$.  Then, as in Section~\ref{s:Properties-for-M-positive-s}, we obtain $\dddot{\K}_\mathrm{u}(k;t,z)=\I +\O_\mathrm{OD}((z/t)^{\frac{1}{2}M}) + \O_\mathrm{D}((z/t)^M) = \I + \O_\mathrm{OD}(t^{-M}) + \O_\mathrm{D}(t^{-2M})$ uniformly for $|k|\neq 1$ as $t\to+\infty$ with $z=\O(t^{-1})$.  Since the exponential factors in \eqref{e:dddotKu-ddotKu} are bounded on their respective domains of the $k$-plane when $x$ is bounded, it follows that $\ddot{\K}_\mathrm{u}(k;t,z)$ is uniformly bounded and hence so is $\dot{\K}_\mathrm{u}(\lambda;t,z)$.  It also follows from the same estimate of $\dddot{\K}_\mathrm{u}(k;t,z)$ that the analogues of \eqref{e:ddotKs-moment-define}--\eqref{e:ddotKs-origin-approx} are:
\begin{equation}
\dddot{\K}_\mathrm{u}(k;t,z)=\I+\frac{1}{k}\dddot{\K}_\mathrm{u}^{(1)}(t,z)+\O(k^{-2}),\quad k\to\infty\quad\text{where}
\label{e:dddotKu-moment-define}
\end{equation}
\begin{equation}
\dddot{\K}_\mathrm{u}^{(1)}(t,z)=-\frac{|a_M|\lambda_\circ^M}{2\pi\ii}\oint_{|s|=1}\bpm
0 & s^M\ee^{\ii x(s-s^{-1})}\\-s^M\ee^{-\ii x(s-s^{-1})} & 0\epm\dd s + \O_\mathrm{D}(\lambda_\circ^{2M}) +
\O_{\mathrm{OD}}(\lambda_\circ^{3M}),\quad\lambda_\circ\to 0
\label{e:dddotKu-moment-approx}
\end{equation}
and
\begin{multline}
\dddot{\K}_\mathrm{u}(0;t,z)=\I + \frac{|a_M|\lambda_\circ^M}{2\pi\ii}\oint_{|s|=1}\bpm
0 & s^M\ee^{\ii x(s-s^{-1})}\\-s^M\ee^{-\ii x(s-s^{-1})}&0\epm\frac{\dd s}{s}\\
+\O_\mathrm{D}(\lambda_\circ^{2M})+\O_{\mathrm{OD}}(\lambda_\circ^{3M}),\quad\lambda_\circ\to 0,
\label{e:dddotKu-origin-approx}
\end{multline}
assuming that $x=\sqrt{2tz}$ is bounded, where the unit circle has counterclockwise orientation.

\subsection{Comparing $\K_\mathrm{u}(u,v;t,z)$ with its parametrix $\dot{\K}_\mathrm{u}(\lambda;t,z)$}
As in Section~\ref{s:Ks-vs-dotKs}, we have an analogue of \eqref{e:Ks-formula} for the solution $\K_\mathrm{u}(u,v;t,z)$ of RH$\dbar$P~\ref{rhp:Ku}:
\begin{equation}
\K_\mathrm{u}(u,v;t,z)=\ddot{\mathbf{F}}_\mathrm{u}(u+\ii v;t,z)\dot{\mathbf{F}}_\mathrm{u}(u,v;t,z)\dot{\K}_\mathrm{u}(u+\ii v;t,z)
\label{e:Ku-formula}
\end{equation}
where $\dot{\mathbf{F}}_\mathrm{u}(u,v;t,z)$ satisfies an analogue of $\dbar$P~\ref{rhp:Fsdot} in which the $\dbar$ equation \eqref{e:dotFs-dbar} is modified slightly to read
\begin{equation}
\dbar\dot{\mathbf{F}}_\mathrm{u}(u,v;t,z)=\dot{\mathbf{F}}_\mathrm{u}(u,v;t,z)\dot{\K}_\mathrm{u}(u+\ii v;t,z)\mathbf{D}_\mathrm{s}(u,v;t,z)\dot{\K}_\mathrm{u}(u+\ii v;t,z)^{-1},\quad (u,v)\in\Real^2,
\label{e:dotFu-dbar}
\end{equation}
and where $\ddot{\mathbf{F}}_\mathrm{u}(\lambda;t,z)$ satisfies an analogue of RHP~\ref{rhp:s-small-norm} in which the jump condition reads $\ddot{\mathbf{F}}_\mathrm{u}^+(\lambda;t,z)=\ddot{\mathbf{F}}_\mathrm{u}^-(\lambda;t,z)\mathbf{G}_\mathrm{u}(u,v;t,z)$, $\lambda=u+\ii v$, with jump matrix
\begin{equation}
\mathbf{G}_\mathrm{u}(u,v;t,z)\coloneq
\dot{\mathbf{F}}_\mathrm{u}(u,v;t,z)\dot{\K}_\mathrm{u}^-(u+\ii v;t,z)\mathbf{J}_\mathrm{u}(u,v;t,z)\dot{\mathbf{J}}_\mathrm{u}(u+\ii v;t,z)^{-1}\dot{\K}_\mathrm{u}^-(u+\ii v;t,z)\dot{\mathbf{F}}_\mathrm{u}(u,v;t,z)^{-1}.
\end{equation}

Since the conjugating factors $\dot{\K}_\mathrm{s}(u+\ii v;t,z)$ and $\dot{\K}_\mathrm{s}(u+\ii v;t,z)^{-1}$ in \eqref{e:dotFs-dbar} played no role whatsoever in the analysis of $\dot{\mathbf{F}}_\mathrm{s}(u,v;t,z)$ in Section~\ref{s:dbar-s} once it was noted that both factors were uniformly bounded, so also under the conditions that either $M=0$ or $M>0$ and $t>0$ is large while $z=O(t^{-1})$ guaranteeing the same properties of $\dot{\K}_\mathrm{u}(u+\ii v;t,z)$ and $\dot{\K}_\mathrm{u}(u+\ii v;t,z)^{-1}$ all of the analysis from that section applies also to $\dot{\mathbf{F}}_\mathrm{u}(u,v;t,z)$ because the ``core'' matrix $\mathbf{D}_\mathrm{s}(u,v;t,z)$ is common to \eqref{e:dotFs-dbar} and \eqref{e:dotFu-dbar}.  It follows that under these conditions $\dot{\mathbf{F}}_\mathrm{u}(u,v;t,z)$ is H\"older continuous and tends to $\I$ as $u+\ii v\to\infty$,
\begin{equation}
\dot{\mathbf{F}}_\mathrm{u}(u,v;t,z)-\I=\O((z/t)^{\frac{1}{2}(N-1)}) + \O(t^{\frac{3}{2}-N})
\label{e:dotFu-uniform}
\end{equation}
uniformly for $(u,v)\in \Real^2$, and
\begin{equation}
\dot{\mathbf{F}}_\mathrm{u}^{(1)}(t,z)\coloneq\lim_{v\to\infty}(u+\ii v)\left[\dot{\mathbf{F}}_\mathrm{u}(u,v;t,z)-\I\right]=\O((z/t)^{\frac{1}{2}N})+\O(t^{1-N})
\label{e:dotF1u}
\end{equation}
with both estimates \eqref{e:dotFu-uniform} and \eqref{e:dotF1u} valid as $t\to+\infty$ with $z=\O(t)$ (or $z=\O(t^{-1})$ if $M>0$), where the limit in \eqref{e:dotF1u}
exists independently of $u\in\Real$ fixed.

One then checks that $\dot{\mathbf{J}}_\mathrm{u}(\lambda;t,z)$ defined by \eqref{e:dot-Ju} has unit determinant and is bounded on $\Sigma$, which implies via \eqref{e:Ju-approximation} that $\mathbf{J}_\mathrm{u}(u,v;t,z)\dot{\mathbf{J}}_\mathrm{u}(u+\ii v;t,z)^{-1}=\I+\O(\lambda_\circ^{M+1})$.  Then, under the conditions yielding the above estimates of $\dot{\mathbf{F}}_\mathrm{u}(u,v;t,z)$ it follows that also $\mathbf{G}_\mathrm{u}(u,v;t,z)=\I+\O(\lambda_\circ^{M+1})$.  Hence, by the same arguments as given in Section~\ref{s:small-norm-RHP-s} we
obtain also
\begin{equation}
\ddot{\mathbf{F}}_\mathrm{u}(\lambda;t,z)=\I+\O((z/t)^{\frac{1}{2}(M+1)})
\label{e:ddotF-uniform-bound-u}
\end{equation}
holding uniformly for $\lambda\in\Complex\setminus\Sigma$, and
\begin{equation}
\ddot{\mathbf{F}}_\mathrm{u}^{(1)}(t,z)\coloneq\lim_{\lambda\to\infty}\lambda\left[\ddot{\mathbf{F}}_\mathrm{u}(\lambda;t,z)-\I\right]=\O((z/t)^{\frac{1}{2}(M+2)}).
\label{e:ddotF-moment-bound-u}
\end{equation}

\subsection{Asymptotic formul\ae\ for $q(t,z)$, $P(t,z)$, and $D(t,z)$}
Assuming that $t\to+\infty$ and $z>0$ satisfies either $z=o(t)$ if $M=0$ or $z=O(t^{-1})$ if $M>0$ as needed to control the factors in \eqref{e:Ku-formula}, the calculations in Section~\ref{s:potentials-in-terms-of-dotKs} go through in the current setting, \emph{mutatis mutandis}.  We obtain that
\begin{equation}
q(t,z)=-2\ii\lambda_\circ\ee^{-\ii(\arg(a_M)+\aleph)}\lim_{k\to\infty}k\ddot{K}_{u,1,2}(k;t,z) + \O((z/t)^{\frac{1}{2}(M+2)})+\O((z/t)^{\frac{1}{2}N})+\O(t^{1-N})
\label{q:q-u-formula}
\end{equation}
and
\begin{multline}
\brho(t,z)=\ee^{-\frac{1}{2}\ii[\arg(a_M)+\aleph]\sigma_3}\ddot{\K}_\mathrm{u}(0;t,z)\sigma_3\ddot{\K}_\mathrm{u}(0;t,z)^{-1}
\ee^{\frac{1}{2}\ii[\arg(a_M)+\aleph]\sigma_3} \\
{}+ \O((z/t)^{\frac{1}{2}(M+1)})+\O((z/t)^{\frac{1}{2}(N-1)})+\O(t^{\frac{3}{2}-N}).
\label{e:rho-u-formula}
\end{multline}
We are primarily concerned with the first row of $\brho(t,z)$, which gives $P(t,z)=\rho_{1,2}(t,z)$ and $D(t,z)=\rho_{1,1}(t,z)$.

For the case that $M=0$, we then use \eqref{e:tilde-Y} and \eqref{e:ddotKu-tilde-Y} together with \eqref{e:Lax-potentials} to get
\begin{equation}
\begin{split}
-2\ii\lambda_\circ\ee^{-\ii(\arg(a_0)+\aleph)}\lim_{k\to\infty}k\ddot{K}_{u,1,2}(k;t,z)&=
-2\ii\lambda_\circ\ee^{-\ii(\arg(r_0)+\aleph)}\lim_{\Lambda\to\infty}\Lambda\widetilde{Y}_{1,2}(\Lambda;\sqrt{2tz},\arctan(|r_0|))\\
&=-2\ii\lambda_\circ\ee^{-\ii(\arg(r_0)+\aleph)}\lim_{\Lambda\to\infty}\Lambda Y_{1,2}(\Lambda;\sqrt{2tz},\arctan(|r_0|))\\
&=\frac{2\lambda_\circ}{\sqrt{2tz}}\ee^{-\ii(\arg(r_0)+\aleph)}y(\sqrt{2tz};\omega)\\
&=t^{-1}\ee^{-\ii(\arg(r_0)+\aleph)}y(\sqrt{2tz};\omega)
\end{split}
\end{equation}
where $\omega$ is defined in terms of $|r_0|$ by \eqref{e:omega-r0}.  Similarly,
\begin{multline}
\left[\ee^{-\frac{1}{2}\ii[\arg(a_0)+\aleph]\sigma_3}\ddot{\K}_\mathrm{u}(0;t,z)\sigma_3\ddot{\K}_\mathrm{u}(0;t,z)^{-1}\ee^{\frac{1}{2}\ii[\arg(a_0)+\aleph]\sigma_3}\right]_{1,2}\\
\begin{aligned}
&=
-2\ee^{-\ii(\arg(r_0)+\aleph)}Y_{1,1}(0;\sqrt{2tz},\arctan(|r_0|))Y_{1,2}(0;\sqrt{2tz},\arctan(|r_0|))\\
&=\frac{2}{\sqrt{2tz}}\ee^{-\ii(\arg(r_0)+\aleph)}s(\sqrt{2tz};\omega)
\end{aligned}
\end{multline}
and
\begin{multline}
\left[\ee^{-\frac{1}{2}\ii[\arg(a_0)+\aleph]\sigma_3}\ddot{\K}_\mathrm{u}(0;t,z)\sigma_3\ddot{\K}_\mathrm{u}(0;t,z)^{-1}\ee^{\frac{1}{2}\ii[\arg(a_0)+\aleph]\sigma_3}\right]_{1,1}\\
\begin{aligned}
&=1+2Y_{1,2}(0;\sqrt{2tz},\arctan(|r_0|))Y_{2,1}(0;\sqrt{2tz},\arctan(|r_0|))\\
&=1-\frac{2}{\sqrt{2tz}}U(\sqrt{2tz};\omega).
\end{aligned}
\end{multline}
Comparing with \eqref{e:unstable-self-similar} in Definition~\ref{def:self-similar} verifies the asymptotic formul\ae\ \eqref{e:qPD-generic-global-unstable} and hence finishes the proof of Theorem~\ref{thm:global-M0-stable-and-unstable} in the unstable-medium case of $D_-=1$.  Using the asymptotic formul\ae\ \eqref{e:y-re-asymp}--\eqref{e:U-re-asymp} with parameters \eqref{e:nprime-varrhoprime} and evaluating the error terms in Theorem~\ref{thm:global-M0-stable-and-unstable} for $z=O(t^\alpha)$ with $-1<\alpha<1$ then proves Corollary~\ref{cor:medium-bulk-M0-stable-and-unstable} in the unstable-medium case as well.

For the case $M>0$ with the additional assumption that $z=\O(t^{-1})$ we obtain from \eqref{e:dddotKu-ddotKu} and \eqref{e:dddotKu-moment-define}--\eqref{e:dddotKu-moment-approx} that
\begin{equation}
\begin{split}
-2\ii\lambda_\circ\ee^{-\ii(\arg(a_M)+\aleph)}\lim_{k\to\infty}k\ddot{K}_{u,1,2}(k;t,z)&=
-2\ii\lambda_\circ\ee^{-\ii(\arg(r_0^{(M)})+\aleph)}\lim_{k\to\infty}k\dddot{K}_{u,1,2}(k;t,z)\\
&=-2\ii\lambda_\circ\ee^{-\ii(\arg(r_0^{(M)})+\aleph)}\dddot{K}_{u,1,2}^{(1)}(t,z)\\
&=\frac{\b{r_0^{(M)}}\ee^{-\ii\aleph}}{\pi M!}\lambda_\circ^{M+1}\oint_{|s|=1}s^M\ee^{\ii \sqrt{2tz}(s-s^{-1})}\dd s + \O(\lambda_\circ^{3M+1})\\
&=\ii(-1)^{M+1}\frac{\b{r_0^{(M)}}\ee^{-\ii\aleph}}{M!}2^{\frac{1}{2}(1-M)}\left(\frac{z}{t}\right)^{\frac{1}{2}(M+1)}J_{M+1}(2\ii\sqrt{2tz}) \\
&\qquad\qquad\qquad{}+ \O((z/t)^{\frac{1}{2}(3M+1)}).
\end{split}
\end{equation}
Similarly, from \eqref{e:dddotKu-ddotKu} and \eqref{e:dddotKu-origin-approx} we find
\begin{multline}
\left[\ee^{-\frac{1}{2}[\arg(a_M)+\aleph]\sigma_3}\ddot{\K}_\mathrm{u}(0;t,z)\sigma_3\ddot{\K}_\mathrm{u}(0;t,z)^{-1}
\ee^{\frac{1}{2}\ii[\arg(a_M)+\aleph]\sigma_3}\right]_{1,2}\\
\begin{aligned}
&=-2\ee^{-\ii(\arg(r_0^{(M)})+\aleph)}\ddot{K}_{u,1,1}(0;t,z)\ddot{K}_{u,1,2}(0;t,z)\\
&=-2\ee^{-\ii(\arg(r_0^{(M)})+\aleph)}\dddot{K}_{u,1,1}(0;t,z)\dddot{K}_{u,1,2}(0;t,z)\\
&=-\frac{\b{r_0^{(M)}}\ee^{-\ii\aleph}}{\ii \pi M!}\lambda_\circ^{M}\oint_{|s|=1}s^{M-1}\ee^{\ii\sqrt{2tz}(s-s^{-1})}\dd s + \O(\lambda_\circ^{3M})\\
&=(-1)^{M+1}\frac{\b{r_0^{(M)}}\ee^{-\ii\aleph}}{M!}2^{1-\frac{1}{2}M}\left(\frac{z}{t}\right)^{\frac{1}{2}M}J_M(2\ii\sqrt{2tz}) + O((z/t)^{\frac{3}{2}M}),
\end{aligned}
\end{multline}
and
\begin{multline}
\left[\ee^{-\frac{1}{2}[\arg(a_M)+\aleph]\sigma_3}\ddot{\K}_\mathrm{u}(0;t,z)\sigma_3\ddot{\K}_\mathrm{u}(0;t,z)^{-1}
\ee^{\frac{1}{2}\ii[\arg(a_M)+\aleph]\sigma_3}\right]_{1,1}\\
\begin{aligned}
&=1+2\ddot{K}_{u,1,2}(0;t,z)\ddot{K}_{u,2,1}(0;t,z)\\
&=1+2\dddot{K}_{u,1,2}(0;t,z)\dddot{K}_{u,2,1}(0;t,z)\\
&=1+\left(\frac{|r_0^{(M)}|}{2\pi M!}\right)^22^{1-M}\left(\frac{z}{t}\right)^M\left[\oint_{|s|=1}s^{M-1}\ee^{\ii\sqrt{2tz}(s-s^{-1})}\dd s\right]\left[\oint_{|s|=1}s^{M-1}\ee^{-\ii \sqrt{2tz}(s-s^{-1})}\dd s\right]+\O((z/t)^{2M})\\
&=1+(-1)^{M+1}\left(\frac{|r_0^{(M)}|}{M!}\right)^22^{1-M}\left(\frac{z}{t}\right)^MJ_M(2\ii\sqrt{2tz}) +
\O((z/t)^{2M})^2.
\end{aligned}
\end{multline}
These computations then establish the asymptotic formul\ae\ \eqref{e:qPD-edge-Mpos-unstable}, which proves Theorem~\ref{thm:edge-Mpos-unstable}.

\appendix
\section{Uniqueness of causal solutions}
\label{s:causal-uniqueness}
In this appendix, we prove Theorem~\ref{thm:causal-uniqueness}.  Let $(q^{(j)}(t,z),P^{(j)}(t,z),D^{(j)}(t,z))$, $j=1,2$ denote two causal solutions of the
same Cauchy problem~\eqref{e:mbe}.
%
If we introduce the notation
\begin{equation}
\begin{split}
\Delta q(t,z)&:=q^{(2)}(t,z)-q^{(1)}(t,z)\\
\Delta P(t,z)&:=P^{(2)}(t,z)-P^{(1)}(t,z)\\
\Delta D(t,z)&:=D^{(2)}(t,z)-D^{(1)}(t,z),
\end{split}
\end{equation}
then because the incident pulse $q_0(\cdot)$ is the same for both solutions,
\begin{equation}
\Delta q(t,z)=0\quad\text{for $z=0$ and $t\ge 0$, }
\end{equation}
and because both solutions are causal,
\begin{equation}
\Delta q(t,z)=\Delta P(t,z)=\Delta D(t,z)=0\quad\text{for $t=0$ and $z\ge 0$.}
\end{equation}
The goal is to show that these boundary conditions on the quarter plane $t\ge 0$ and $z\ge 0$ imply that
$\Delta q(t,z)=\Delta P(t,z)=\Delta D(t,z)=0$ in the interior of the quarter plane.  In fact, we will not even require the condition $\Delta q(0,z)=0$ for $z\ge 0$ to prove this result.


From 
the Maxwell equation and Bloch subsystem governing $(q^{(j)}(t,z),P^{(j)}(t,z),D^{(j)}(t,z))$, $j=1,2$,
we deduce equations satisfied by the differences of two solutions:
\begin{equation}
\begin{split}
\Delta q_z(t,z)&=-\Delta P(t,z),\\
\Delta P_t(t,z)&=-2[q^{(2)}(t,z)D^{(2)}(t,z)-q^{(1)}(t,z)D^{(1)}(t,z)]\\
&=-2[D^{(2)}(t,z)\Delta q(t,z)+q^{(1)}(t,z)\Delta D(t,z)],\\
\Delta D_t(t,z)&=2\Re(\b{q^{(2)}(t,z)}P^{(2)}(t,z)-\b{q^{(1)}(t,z)}P^{(1)}(t,z))\\
&=2\Re(P^{(2)}(t,z)\b{\Delta q(t,z)}+\b{q^{(1)}(t,z)}\Delta P(t,z)).
\end{split}
\end{equation}
From the conservation law $D^{(j)}(t,z)^2+|P^{(j)}(t,z)|^2=1$ and the Maxwell equation governing $q^{(j)}(t,z)$, we obtain the following a priori estimates:
\begin{equation}
|P^{(j)}(t,z)|\le 1,\quad |D^{(j)}(t,z)|\le 1,\quad\text{and}\quad |q^{(j)}(t,z)|\le |q_0(t)| +z
\label{eq:apriori}
\end{equation}
for $t\ge 0$ and $z\ge 0$.
\begin{lemma}
Suppose that for $0\le t_0<t_1$ and $0\le z_0<z_1$ it holds that
\begin{equation}
\Delta q(t,z_0)=0\quad\text{for $t_0\le t\le t_1$, and}
\end{equation}
\begin{equation}
\Delta P(t_0,z)=\Delta D(t_0,z)=0\quad\text{for $z_0\le z\le z_1$.}
\end{equation}
Then, setting
\begin{equation}
M(t_1,z_1):=\int_{t_0}^{t_1}|q_0(\tau)|\dd\tau + z_1(t_1-t_0),
\label{eq:M-define}
\end{equation}
the condition $z_1-z_0+4(t_1-t_0)+4M(t_1,z_1)<1$ implies that $\Delta q(t,z)=\Delta P(t,z)=\Delta D(t,z)=0$ for all $(t,z)\in [t_0,t_1]\times [z_0,z_1]$.
\label{lem:complex-rectangle}
\end{lemma}
\begin{proof}
Whenever $(t,z)\in [t_0,t_1]\times [z_0,z_1]$,
\begin{equation}
\begin{split}
\Delta q(t,z)&=-\int_{z_0}^z\Delta P(t,\zeta)\dd\zeta\\
\Delta P(t,z)&=-2\int_{t_0}^t D^{(2)}(\tau,z)\Delta q(\tau,z)\dd\tau -2\int_{t_0}^t q^{(1)}(\tau,z)\Delta D(\tau,z)\dd\tau\\
\Delta D(t,z)&=2\int_{t_0}^t\Re(P^{(2)}(\tau,z)\b{\Delta q(t,z)})\dd\tau + 2\int_{t_0}^t\Re(\b{q^{(1)}(\tau,z)}\Delta P(\tau,z))\dd\tau.
\end{split}
\end{equation}
Using \eqref{eq:apriori} we then get the inequalities
\begin{equation}
\begin{split}
|\Delta q(t,z)|&\le\int_{z_0}^z|\Delta P(t,\zeta)|\dd\zeta\\
|\Delta P(t,z)|&\le 2\int_{t_0}^t|\Delta q(\tau,z)|\dd\tau + 2\int_{t_0}^t(|q_0(\tau)|+z)|\Delta D(\tau,z)|\dd\tau\\
|\Delta D(t,z)|&\le 2\int_{t_0}^t|\Delta q(\tau,z)|\dd\tau + 2\int_{t_0}^t(|q_0(\tau)|+z)|\Delta P(\tau,z)|\dd\tau.
\end{split}
\end{equation}
If we denote
\begin{equation}
\begin{split}
S_q(t,z)&:=\sup_{(\tau,\zeta)\in [t_0,t]\times[z_0,z]}|\Delta q(\tau,\zeta)|\\
S_P(t,z)&:=\sup_{(\tau,\zeta)\in [t_0,t]\times[z_0,z]}|\Delta P(\tau,\zeta)|\\
S_D(t,z)&:=\sup_{(\tau,\zeta)\in [t_0,t]\times[z_0,z]}|\Delta D(\tau,\zeta)|,
\end{split}
\end{equation}
then for $(t,z)\in [t_0,t_1]\times[z_0,z_1]$ we get
\begin{equation}
\begin{split}
|\Delta q(t,z)|&\le (z-z_0)S_P(t,z)\\
&\le (z_1-z_0)S_P(t_1,z_1)\\
|\Delta P(t,z)|&\le 2(t-t_0)S_q(t,z) + 2\left[\int_{t_0}^t|q_0(\tau)|\dd\tau + z(t-t_0)\right]S_D(t,z)\\
&\le
2(t_1-t_0)S_q(t_1,z_1)+2M(t_1,z_1)S_D(t_1,z_1)\\
|\Delta D(t,z)|&\le 2(t-t_0)S_q(t,z)+2\left[\int_{t_0}^t|q_0(\tau)|\dd\tau + z(t-t_0)\right]S_P(t,z)\\
&\le
2(t_1-t_0)S_q(t_1,z_1)+2M(t_1,z_1)S_P(t_1,z_1),
\end{split}
\end{equation}
where we recall the definition \eqref{eq:M-define}.
Since the final upper bounds are independent of $(t,z)\in [t_0,t_1]\times[z_0,z_1]$, taking the supremum in each case over this enclosing rectangle gives
\begin{equation}
\begin{split}
S_q(t_1,z_1)&\le(z_1-z_0)S_P(t_1,z_1)\\
S_P(t_1,z_1)&\le
2(t_1-t_0)S_q(t_1,z_1)+2M(t_1,z_1)S_D(t_1,z_1)\\
S_D(t_1,z_1)&\le
2(t_1-t_0)S_q(t_1,z_1)+2M(t_1,z_1)S_P(t_1,z_1).
\end{split}
\label{eq:3-inequalities}
\end{equation}
Now set
\begin{equation}
S(t,z):=S_q(t,z)+S_P(t,z)+S_D(t,z).
\end{equation}
Then since $S_q(t,z)\le S(t,z)$, $S_P(t,z)\le S(t,z)$, and $S_D(t,z)\le S(t,z)$, summing the three inequalities in \eqref{eq:3-inequalities} gives
\begin{equation}
S(t_1,z_1)\le \left( z_1-z_0 + 4(t_1-t_0) + 4M(t_1,z_1)\right)S(t_1,z_1).
\end{equation}
Suppose that $z_1-z_0+4(t_1-t_0)+4M(t_1,z_1)<1$.
If $S(t_1,z_1)>0$, then dividing by $S(t_1,z_1)$ gives $z_1-z_0+4(t_1-t_0)+4M(t_1,z_1)\ge 1$, a contradiction.  Since $S(t_1,z_1)\ge 0$, it follows that $S(t_1,z_1)=0$, which implies that also $S_q(t_1,z_1)=S_P(t_1,z_1)=S_D(t_1,z_1)=0$, i.e., $\Delta q(t,z)$, $\Delta P(t,z)$, and $\Delta D(t,z)$ all vanish identically on $[t_0,t_1]\times[z_0,z_1]$.
\end{proof}

\begin{lemma}
Suppose that for $0\le z_0<z_1$ it holds that
\begin{equation}
\Delta q(t,z_0)=0\quad\text{for all $t\ge 0$, and}
\end{equation}
\begin{equation}
\Delta P(0,z)=\Delta D(0,z)=0\quad\text{for $z_0\le z\le z_1$.}
\end{equation}
If $z_1-z_0<1$ then $\Delta q(t,z)=\Delta P(t,z)=\Delta D(t,z)=0$ for all $t\ge 0$ and $z\in [z_0,z_1]$.
\label{lem:horizontal-strip}
\end{lemma}
\begin{proof}
Since $z_1-z_0<1$, let $\epsilon>0$ satisfy $\epsilon < 1-(z_1-z_0)$.  Then, since $q_0\in L^1(\Real)$, there exists $\Delta t>0$ such that
\begin{equation}
4\Delta t + 4\int_{n\Delta t}^{(n+1)\Delta t}|q_0(\tau)|\dd\tau + 4z_1\Delta t\le\epsilon,\quad\forall n\in\mathbb{Z}_{\ge 0}.
\end{equation}
Note that $\Delta t$ will depend on $z_1$.

We prove the lemma by showing that given $n\in\mathbb{Z}_{\ge 0}$, $\Delta q(t,z)=\Delta P(t,z)=\Delta D(t,z)=0$ holds for $(t,z)\in [n\Delta t,(n+1)\Delta t]\times[z_0,z_1]$.  The base case for induction on $n$ is $n=0$.  By hypothesis, we have $\Delta q(t,z_0)=0$ for $0\le t\le \Delta t$ and $\Delta P(0,z)=\Delta D(0,z)=0$ for $z_0\le z\le z_1$.  Then taking $t_0=0$ and $t_1=\Delta t$,
\begin{equation}
z_1-z_0 + 4(t_1-t_0) + 4M(t_1,z_1)\le z_1-z_0+\epsilon<1
\label{eq:piece-of-the-strip}
\end{equation}
so Lemma~\ref{lem:complex-rectangle} implies that $\Delta q(t,z)=\Delta P(t,z)=\Delta D(t,z)=0$ for $(t,z)\in [t_0,t_1]\times[z_0,z_1]=[0,\Delta t]\times[z_0,z_1]$, which establishes the $n=0$ statement.  Taking the inductive hypothesis that $\Delta q(t,z)=\Delta P(t,z)=\Delta D(t,z)=0$ for $(t,z)\in [(n-1)\Delta t,n\Delta t]\times[z_0,z_1]$, we have in particular that $\Delta P(n\Delta t,z)=\Delta D(n\Delta t,z)=0$ for $z_0\le z\le z_1$.  We are also given that $\Delta q(t,z_0)=0$ for $n\Delta t\le t\le (n+1)\Delta t$.
Taking $t_0=n\Delta t$ and $t_1=(n+1)\Delta t$ and applying Lemma~\ref{lem:complex-rectangle} shows that
since \eqref{eq:piece-of-the-strip} holds, $\Delta q(t,z)=\Delta P(t,z)=\Delta D(t,z)=0$ holds for $(t,z)\in [t_0,t_1]\times[z_0,z_1]=[n\Delta t,(n+1)\Delta t]\times[z_0,z_1]$.  This completes the induction argument and the proof.
\end{proof}

Now we can give the proof of Theorem~\ref{thm:causal-uniqueness}.  If $(q^{(j)}(t,z),P^{(j)}(t,z),D^{(j)}(t,z))$, $j=1,2$ are two causal solutions of the MBE Cauchy problem~\eqref{e:mbe}, then in particular
$\Delta q(t,0)=0$ for all $t\ge 0$ and $\Delta P(0,z)=\Delta D(0,z)=0$ for all $z\ge 0$.  Combining Lemma~\ref{lem:horizontal-strip} with an induction argument then shows that
$\Delta q(t,z)=\Delta P(t,z)=\Delta D(t,z)=0$ holds on any horizontal strip of the form $t\ge 0$ and $\tfrac{1}{2}n\le z\le \tfrac{1}{2}(n+1)$ for $n\in\mathbb{Z}_{\ge 0}$.  Therefore the solutions
$(q^{(j)}(t,z),P^{(j)}(t,z),D^{(j)}(t,z))$, $j=1,2$ agree on the whole quarter plane $t\ge 0$ and $z\ge 0$, which completes the proof.

\section{Proof of Theorem~\ref{thm:reconstruction}}
\label{s:proof-reconstruction}

This proof that the reconstruction formula~\eqref{e:reconstruction} yields a causal solution of the Cauchy problem~\eqref{e:mbe} consists of three steps:
\begin{itemize}
\item[(i)]
The Lax pair~\eqref{e:LP} in which the coefficients are expressed in terms of $\M(\lambda;t,z)$ via the reconstruction formula~\eqref{e:reconstruction} is derived directly from RHP~\ref{rhp:M},
implying by compatibility that the two functions $q(t,z)$ and $\brho(t,z)$ from the reconstruction formula solve the Maxwell-Bloch equations;
\item[(ii)] The optical field, when evaluated at $z=0$, is shown to reproduce the required incident pulse,
i.e., $q(t,0) = q_0(t)$;
\item[(iii)] The density matrix is shown to satisfy the required initial values in the distant past,
i.e., $\lim_{t\to-\infty}\brho(t,z) = D_-\sigma_3$.  This part will be proved by showing first that the solution is causal.
\end{itemize}
The three parts combined prove that the reconstruction formula~\eqref{e:reconstruction} solves the Cauchy problem~\eqref{e:mbe} and is causal (and hence, by Theorem~\ref{thm:causal-uniqueness}, unique).

For step (i),
one defines in terms of the solution of RHP~\ref{rhp:M} a new matrix function $\bphi(t,z;\lambda) = \M(\lambda;t,z)\ee^{\ii\theta(\lambda;t,z)\sigma_3}$, which is analytic for $\lambda\in\Complex\setminus\Sigma_\M$ except at the origin, and which has the asymptotic behavior that $\bphi(t,z;\lambda)\ee^{-\ii\lambda t\sigma_3}=\mathbb{I}+\lambda^{-1}\M^{(1)}(t,z) + o(\lambda^{-1})$ as $\lambda\to\infty$ and that $\bphi(t,z;\lambda)\ee^{D_-z\sigma_3/(2\lambda)}$ is analytic at $\lambda=0$.
The jump conditions for $\bphi(t,z;\lambda)$ induced on $\Sigma_\M$ are easily checked to be independent of $(t,z)\in\Complex^2$.  Therefore, the matrix function $\bphi_t(t,z;\lambda)\bphi(t,z;\lambda)^{-1}$ is analytic in the whole complex plane with the possible exception of $\lambda=0$, however one easily checks that the singularity at the origin is removable.  Computing the asymptotic behavior as $\lambda\to\infty$ shows via Liouville's theorem that $\bphi_t(t,z;\lambda)\bphi(t,z;\lambda)^{-1}$ is a linear function of the form $\ii\lambda\sigma_3+\Q(t,z)$, where $\Q(t,z)$ has the form shown in \eqref{e:mbe-matrix} in which the function $q(t,z)$ is obtained from $\M(\lambda;t,z)$ via the reconstruction formula \eqref{e:reconstruction}.  Similarly, the matrix function $\bphi_z(t,z;\lambda)\bphi(t,z;\lambda)^{-1}$ is analytic except for a simple pole at the origin, and it vanishes for large $\lambda$.  Therefore by Liouville's theorem again it has the form $-\ii\brho(t,z)/(2\lambda)$, and the matrix coefficient $\brho(t,z)$ is obtained from $\M(\lambda;t,z)$ by the reconstruction formula \eqref{e:reconstruction}.
It follows that $\bphi(t,z;\lambda)$ is a simultaneous fundamental solution matrix of the Lax pair equations \eqref{e:LP}, and by compatibility, that the matrix coefficients $\Q(t,z)$ and $\brho(t,z)$
extracted from $\M(\lambda;t,z)$ via \eqref{e:reconstruction} constitute a solution of the MBE system in matrix form \eqref{e:mbe-matrix}.


For step (ii), we start by setting $z=0$ in RHP~\ref{rhp:M}, which results in a simplification of the phase:  $\theta(\lambda;t,0)=\lambda t$.  Therefore, the essential singularity at the origin $\lambda=0$ is removed from the jump matrices.  The analyticity of $r(\lambda)$ for $\Im(\lambda)>0$ and continuity for $\Im(\lambda)\ge 0$ then allows for a simple analytic deformation:
\begin{equation}
\N(\lambda;t)\coloneq
  \begin{cases}
    \M(\lambda;t,0)\W(\lambda;t,0)\,, & \quad |\lambda| < \gamma\,,\quad \Im(\lambda) > 0\,,\\
    \M(\lambda;t,0)\W^{\dagger}(\lambda;t,0)^{-1}\,, & \quad |\lambda| < \gamma\,,\quad \Im(\lambda) < 0\,,\\
    \M(\lambda;t,0)\,, & \quad |\lambda| > \gamma\,.
  \end{cases}
\end{equation}
The new function has the asymptotics $\N(\lambda;t)\to\I$ as $\lambda\to\infty$,
and a jump on the real line only:
\begin{equation}
\everymath{\displaystyle}
\setlength\arraycolsep{4pt}
\N^+(\lambda;t) = \N^-(\lambda;t) \bpm 1 + |r(\lambda)|^2 & \b{r(\lambda)}\ee^{2\ii\lambda t} \\ r(\lambda)\ee^{-2\ii\lambda t} & 1 \epm\,,\qquad \lambda\in\Real\,.
\end{equation}
This is the classic RHP associated with the non-selfadjoint Zakharov-Shabat scattering problem.
Also, Eq.~\eqref{e:reconstruction} yields $q(t,0) = -2\ii\lim_{\lambda\to\infty}\lambda M_{1,2}(\lambda;t,0) = -2\ii\lim_{\lambda\to\infty}\lambda N_{1,2}(\lambda;t)$.
Hence, the well-known Schwartz-space bijection between $q_0(t)$ and $r(\lambda)$ in absence of discrete spectrum or spectral singularities yields $q(t,0) = q_0(t)$.

For step (iii), we will prove that the solution is causal.
With $t < 0$, regardless of the value of $D_-=\pm 1$ the phase $\theta(\lambda;t,z)$ behaves in a similar way for $|\lambda| \gg 1$ where the linear term dominates.
The sign structure of $\Re(\ii\theta(\lambda;t,z))$ is shown in Figure~\ref{f:outside-itheta} in the case $D_-=-1$ (left panel) and $D_-=1$ (right panel), and
in both cases the sign is the same outside the circle $|\lambda| = \lambda_\circ=\sqrt{z/(2t)}$.
Hence, both stable and unstable cases with $t < 0$ can be treated in an identical way.
In particular, one notices that $\ee^{\mp 2\ii\theta(\lambda;t,z)}\to 0$ as $\pm \Im(\lambda)\to +\infty$ regardless of the sign $D_-$.
%
%
We also make use of the following standard result:
\begin{lemma}
\label{lem:r-asymptotics}
If $q_0$ satisfies Assumption~\ref{ass:q0-assumption}, then there exists $L>0$ such that $r(\lambda)$ is analytic for $|\lambda|>L$ and $\Im(\lambda)>0$, continuous for $|\lambda|>L$ and $\Im(\lambda)\ge 0$, and $r(\lambda)\to 0$ as $\lambda\to\infty$ uniformly with $\Im(\lambda)\ge 0$.
\end{lemma}
\begin{proof}
Recall the Jost matrix $\bphi_+(t;\lambda)$ satisfying the Zakharov-Shabat equation in \eqref{e:LP} and the boundary condition $\bphi_+(t;\lambda)=\ee^{\ii\lambda t\sigma_3}+o(1)$ as $t\to+\infty$ with $\lambda\in\mathbb{R}$.  Setting
\begin{equation}
u(t;\lambda):=\phi_{+,1,1}(t;\lambda)\ee^{-\ii\lambda t}\quad\text{and}\quad v(t;\lambda):=\phi_{+,2,1}(t;\lambda)\ee^{\ii \lambda t},
\end{equation}
the functions $u(\cdot;\lambda)$ and $v(\cdot;\lambda)$ satisfy the coupled differential equations
\begin{equation}
u'(t;\lambda)=q_0(t)\ee^{-2\ii\lambda t}v(t;\lambda)\quad\text{and}\quad
v'(t;\lambda)=-\b{q_0(t)}\ee^{2\ii\lambda t}u(t;\lambda),
\label{e:coupled-eqns}
\end{equation}
or, building in the boundary conditions, the coupled Volterra equations
\begin{equation}
u(t;\lambda)=1-\int_t^{+\infty}q_0(s)\ee^{-2\ii s\lambda}v(s;\lambda)\dd s,\quad v(t;\lambda)=\int_t^{+\infty}\b{q_0(s)}\ee^{2\ii\lambda s}u(s;\lambda)\dd s.
\label{e:coupled-Volterra}
\end{equation}
Eliminating $v(t;\lambda)$ gives
\begin{equation}
u(t;\lambda)=1+\int_t^{+\infty}K(t,w;\lambda)u(w;\lambda)\dd w,\quad K(t,w;\lambda):=-\b{q_0(w)}\int_t^wq_0(s)\ee^{2\ii \lambda(w-s)}\dd s.
\end{equation}
Noting that Assumption~\ref{ass:q0-assumption} implies that $q_0\in L^1(\mathbb{R})$, the kernel $K$ satisfies the estimate $|K(t,w;\lambda)|\le \|q_0\|_1\cdot |q_0(w)|$ whenever $w\ge t$ and $\Im(\lambda)\ge 0$. Straightforward analysis of the iterates proves that for each $t\in\mathbb{R}$, $u(t;\lambda)$ is analytic for $\Im(\lambda)>0$ and continuous for $\Im(\lambda)\ge 0$, and that
\begin{equation}
\|u(\cdot;\lambda)\|_\infty\le\ee^{\|q_0\|_1^2},\quad \Im(\lambda)\ge 0.
\label{e:u-bound}
\end{equation}
Using the uniformly-convergent iterates for $u(t;\lambda)$ in the integral equation for $v(t;\lambda)$ in \eqref{e:coupled-Volterra} shows that for each $t\in\mathbb{R}$, $v(t;\lambda)$ is also a function analytic for $\Im(\lambda)>0$ and continuous for $\Im(\lambda)\ge 0$.

Now let $\epsilon>0$ be given, arbitrarily small, and let $q_\epsilon(t)$ be an infinitely differentiable function with compact support for which
\begin{equation}
\|q_0-q_\epsilon\|_1\le\frac{\epsilon}{4} \ee^{-\|q_0\|_1^2}.
\label{e:q0-approximate}
\end{equation}
Then using \eqref{e:coupled-Volterra} we can express $\ee^{-2\ii\lambda t}v(t;\lambda)$ in terms of $u(t;\lambda)$ in the form
$\ee^{-2\ii\lambda t}v(t;\lambda)=I_1(t;\lambda)+I_2(t;\lambda)$, where
\begin{equation}
I_1(t;\lambda)=\int_t^{+\infty}[\b{q_0(s)}-\b{q_\epsilon(s)}]\ee^{2\ii\lambda (s-t)}u(s;\lambda)\dd s,\quad
I_2(t;\lambda)=\int_t^{+\infty}\b{q_\epsilon(s)}\ee^{2\ii\lambda(s-t)}u(s;\lambda)\dd s.
\end{equation}
Obviously, if $\Im(\lambda)\ge 0$, combining \eqref{e:u-bound} and \eqref{e:q0-approximate} shows that
\begin{equation}
|I_1|\le \|q_\epsilon-q_0\|_1\cdot\|u(\cdot;\lambda)\|_\infty \le \frac{1}{4}\epsilon.
\end{equation}
For $I_2$, we take advantage of smoothness and compact support of $q_\epsilon$ to integrate by parts:
\begin{equation}
I_2(t;\lambda) = -\frac{1}{2\ii\lambda}\left[\b{q_\epsilon(t)}u(t;\lambda) + \int_t^{+\infty}\b{q_\epsilon'(s)}u(s;\lambda)\ee^{2\ii\lambda (s-t)}\dd s + \int_t^{+\infty}\b{q_\epsilon(s)}u'(s;\lambda)\ee^{2\ii\lambda(s-t)}\dd s\right].
\end{equation}
Using \eqref{e:coupled-eqns} to eliminate $u'(s;\lambda)$ gives
\begin{multline}
I_2(t;\lambda)=-\frac{1}{2\ii\lambda}\left[\b{q_\epsilon(t)}u(t;\lambda) + \int_t^{+\infty}\b{q_\epsilon'(s)}u(s;\lambda)\ee^{2\ii\lambda (s-t)}\dd s\right.\\{}\left. +\int_t^{+\infty}\b{q_\epsilon(s)}q_0(s)\ee^{-2\ii\lambda s}v(s;\lambda)\ee^{2\ii\lambda(s-t)}\dd s\right].
\end{multline}
Therefore for $\Im(\lambda)\ge 0$,
\begin{equation}
\begin{split}
|\ee^{-2\ii\lambda t}v(t;\lambda)|&\le |I_1(t;\lambda)|+|I_2(t;\lambda)|\\
&\le \frac{\epsilon}{4} + \frac{\ee^{\|q_0\|_1^2}}{2|\lambda|}\left(\|q_\epsilon\|_\infty + \|q_\epsilon'\|_1\right) + \frac{1}{2|\lambda|}\|q_\epsilon\|_\infty\|q_0\|_1\|\ee^{-2\ii\lambda(\cdot)}v(\cdot;\lambda)\|_\infty,
\end{split}
\end{equation}
so
\begin{equation}
\left(1-\frac{1}{2|\lambda|}\|q_\epsilon\|_\infty\|q_0\|_1\right)\|\ee^{-2\ii\lambda(\cdot)}v(\cdot;\lambda)\|_\infty\le \frac{\epsilon}{4} + \frac{\ee^{\|q_0\|_1^2}}{2|\lambda|}\left(\|q_\epsilon\|_\infty+ \|q_\epsilon'\|_1\right).
\end{equation}
Therefore, taking $\Im(\lambda)\ge 0$ and $|\lambda|>L=L(\epsilon)$ where
\begin{equation}
L(\epsilon):=\max\left\{\|q_\epsilon\|_\infty\|q_0\|_1,\frac{2}{\epsilon}\left(\|q_\epsilon\|_\infty + \|q_\epsilon'\|_1\right)\ee^{\|q_0\|_1^2}\right\}
\end{equation}
gives $\|\ee^{-2\ii\lambda(\cdot)}v(\cdot;\lambda)\|_\infty<\epsilon$.  From \eqref{e:coupled-Volterra} it then follows that also $\|u(\cdot;\lambda)-1\|_\infty<\epsilon\|q_0\|_1$.

Now, under the cutoff condition on $q_0$ in Assumption~\ref{ass:q0-assumption}, the reflection coefficient is given by
\begin{equation}
r(\lambda):=\frac{\phi_{+,2,1}(0;\lambda)}{\phi_{+,1,1}(0;\lambda)} = \left.\frac{\ee^{-2\ii\lambda t}v(t;\lambda)}{u(t;\lambda)}\right|_{t=0},
\end{equation}
from which the claim follows.
\end{proof}
Now suppose that $t\le 0$.  Combining the boundedness of the exponential factors $\ee^{\pm2\ii\theta(\lambda;t,z)}$ in appropriate half-planes away from the origin with Lemma~\ref{lem:r-asymptotics},
one concludes that $\W(\lambda;t,z)\to \I$ as $\lambda\to\infty$ with $\Im(\lambda)\ge 0$ while $\W^\dagger(\lambda;t,z)\to\I$ as $\lambda\to\infty$ with $\Im(\lambda)\le 0$.
Thus, RHP~\ref{rhp:M} can be solved exactly for $t\le 0$ as follows,
\begin{equation}
\M(\lambda;t,z) = \begin{cases}
\W(\lambda;t,z)\,, & |\lambda| > \gamma \quad\text{and}\quad \Im(\lambda) > 0\,,\\
\I\,, & |\lambda| < \gamma\,,\\
\W^\dagger(\lambda;t,z)^{-1}\,, & |\lambda| > \gamma \quad\text{and}\quad \Im(\lambda) < 0\,,
\end{cases}
\end{equation}
where $\gamma > 0$ is the radius of the circular component of the jump contour $\Sigma_\M$.
Assuming further that $t<0$ and using the reconstruction formula~\eqref{e:reconstruction} (taking the limit $\lambda\to\infty$ for $q(t,z)$ in any direction non-tangential to the real axis to exploit the exponential decay of the factors $\ee^{\pm 2\ii\theta(\lambda;t,z)}$ in appropriate half-planes) one gets that
\begin{equation}
q(t,z) = 0\,,\qquad \brho(t,z) = D_-\sigma_3\,,
\end{equation}
for all $z \in\Real$ and $t < 0$.  This proves causality of the solution.
In particular, one notes
\begin{equation}
\lim_{t\to-\infty}\brho(t,z) = D_-\sigma_3\,,
\end{equation}
which establishes the prescribed initial state of the medium in the limit $t\to -\infty$.
This finishes the proof of Theorem~\ref{thm:reconstruction}.
\begin{figure}[t]
\centering
\includegraphics[width = 0.4\textwidth]{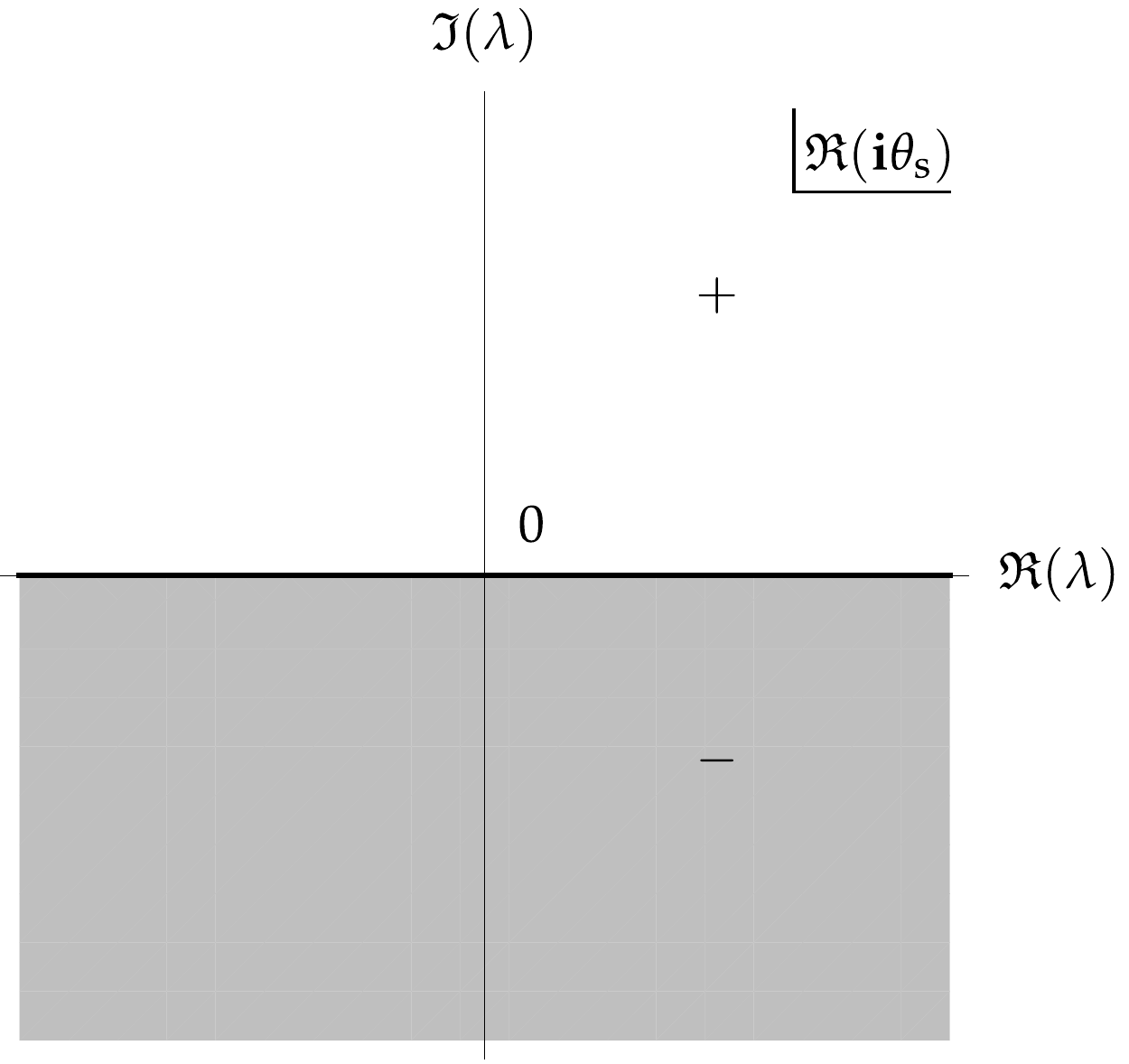}\qquad
\includegraphics[width = 0.4\textwidth]{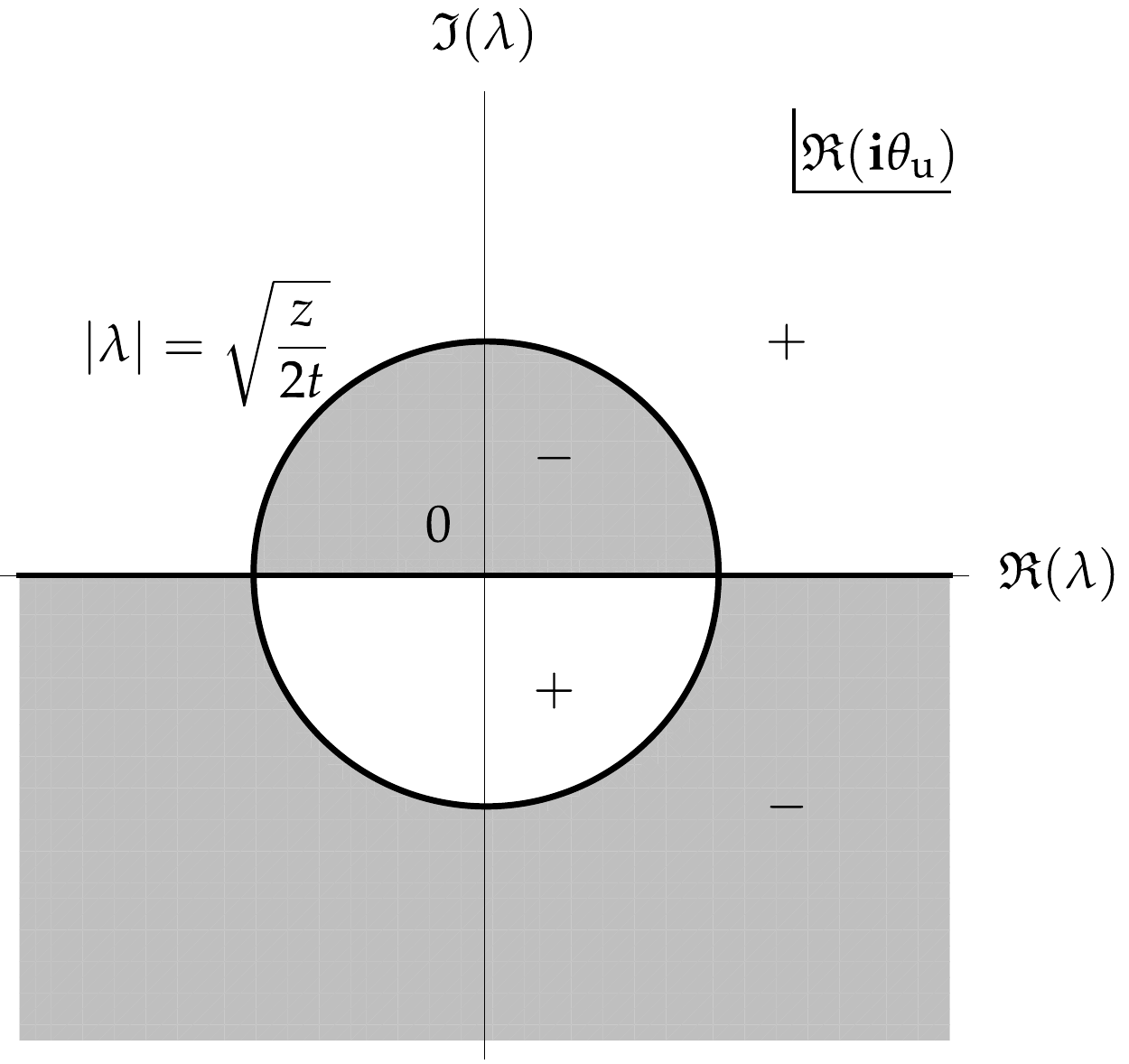}
\caption{For $t<0$ and $z>0$, the sign structure of $\Re(\ii\theta(\lambda;t,z))$ in the complex plane for an initially-stable medium $D_- = -1$ (left) and for an initially-unstable medium $D_- = 1$ (right).
White (gray) shading corresponds to positive (negative) values of $\Re(\ii\theta(\lambda;t,z))$.
}
\label{f:outside-itheta}
\end{figure}
%

\section{Proof of Lemma~\ref{thm:fuv-difference}}
\label{s:proof-fuv-difference}

We consider the Taylor expansion of $f^{(n)}(u)$ about $u=0$, truncated at the derivative of order $n_1$,  for $n=0,\dots,n_1-1$:
\begin{equation}
\begin{split}
f^{(n)}(u)&=\sum_{j=0}^{n_1-n-1}\frac{f^{(n+j)}(0)}{j!}u^j+\frac{f^{(n_1)}(\xi_n)}{(n_1-n)!}u^{n_1-n}\\
&=\sum_{m=n}^{n_1-1}\frac{f^{(m)}(0)}{(m-n)!}u^{m-n}+\frac{f^{(n_1)}(\xi_n)}{(n_1-n)!}u^{n_1-n},\quad n=0,\dots,n_1-1.
\end{split}
\end{equation}
Here $\xi_n$, $n=0,\dots,n_1-1$, are parameters lying between $0$ and $u$ on the real line.
Then substituting into the first term in \eqref{e:fuv-difference} gives
\begin{equation}
\sum_{n=0}^{n_1-1}\frac{(\ii v)^n}{n!}f^{(n)}(u) = \sum_{n=0}^{n_1-1}\sum_{m=n}^{n_1-1}\frac{f^{(m)}(0)(\ii v)^nu^{m-n}}{n!(m-n)!} + \sum_{n=0}^{n_1-1}\frac{f^{(n_1)}(\xi_n)}{n!(n_1-n)!}(\ii v)^nu^{n_1-n}.
\end{equation}
Exchanging the order of summation in the double sum gives
\begin{equation}
\begin{split}
\sum_{n=0}^{n_1-1}\frac{(\ii v)^n}{n!}f^{(n)}(u) &= \sum_{m=0}^{n_1-1}f^{(m)}(0)\sum_{n=0}^{m}\frac{(\ii v)^nu^{m-n}}{n!(m-n)!} + \sum_{n=0}^{n_1-1}\frac{f^{(n_1)}(\xi_n)}{n!(n_1-n)!}(\ii v)^nu^{n_1-n}\\
&= \sum_{m=0}^{n_1-1}\frac{f^{(m)}(0)}{m!}(u+\ii v)^m + \sum_{n=0}^{n_1-1}\frac{f^{(n_1)}(\xi_n)}{n!(n_1-n)!}(\ii v)^nu^{n_1-n},
\end{split}
\end{equation}
where on the second line we used the binomial formula.
Therefore, relabeling an index and recalling $\lambda=u+\ii v$ we have, for $n_2\le n_1$,
\begin{equation}
\begin{split}
\sum_{n=0}^{n_1-1}\frac{(\ii v)^n}{n!}f^{(n)}(u)&=\sum_{n=0}^{n_1-1}\frac{f^{(n)}(0)}{n!}\lambda^n +\O(\lambda^{n_1})\\
&=\sum_{n=0}^{n_2-1}\frac{f^{(n)}(0)}{n!}\lambda^n +\O(\lambda^{n_2})
\end{split}
\end{equation}
as $\lambda\to 0$.  Here we used the fact that
$f(u)\in C^k(\Real)$ so
$f^{(n_1)}(u)$ is continuous and therefore
$|f^{(n_1)}(\xi_n)|\le \max\limits_{u\in [-1,1]}|f^{(n_1)}(u)| = \O(1)$.
The lemma is proved.

\section{Numerical methods}
\label{s:numerics}

This appendix describes the numerical methods used in this paper.  Aside from evaluations of classical special functions and integration of ordinary differential equations by means of built-in tools available in many familiar platforms, we need to compute the values of Painlev\'e-III functions and to numerically compute causal solutions of the Cauchy problem for the MBE system.
%

\subsection{Computation of Painlev\'e-III functions}
The functions $y(X;\omega)$, $s(X;\omega)$, and $U(X;\omega)$ are defined by \eqref{e:Lax-potentials} in terms of the solution $\mathbf{Y}(\Lambda;X;\eta)$ of RHP~\ref{rhp:PIII}, where $\omega=-3\cos(2\eta)$.
Given $X\in\Complex$ and $\omega\in (-3,3)$, they may therefore be computed by finding a sufficiently accurate numerical approximation of the solution of RHP~\ref{rhp:PIII}.

To solve RHP~\ref{rhp:PIII} numerically, we use S. Olver's method \cite{olver2012} as implemented in the \textit{Mathematica} package \texttt{RHPackage} \cite{rhpackage}.  This method takes as input data a parametrization of the jump contour for the problem and the corresponding jump matrix, and it converts that given data into a system of singular integral equations which is then solved numerically with high accuracy for ``reasonable'' data.  In particular, a ``reasonable'' jump matrix should not have very large elements, so the case that $X\in\ii\Real$ is particularly good for applying the method to RHP~\ref{rhp:PIII} as the jump matrix is bounded on the unit circle.  When $X\in\Real$ instead, the jump matrix elements grow rapidly with $X$, so we may use the relation \eqref{e:symmetry-rotate-X} to convert RHP~\ref{rhp:PIII} into an equivalent problem with a bounded jump matrix on the unit circle.  Equivalently, it suffices to compute $y(X;-\omega)$, $s(X;-\omega)$, and $U(X;-\omega)$ for imaginary $X$ and then use \eqref{e:PIII-rotate-X-identities} to obtain $y(X;\omega)$, $s(X;\omega)$, and $U(X;\omega)$ for $X$ on the real axis.

When $X\in\ii\Real$ becomes large, the jump matrix becomes highly oscillatory, which is another difficulty in the numerical solution of RHPs that leads to large numerical errors.  In this situation, a strategy that has been used successfully \cite{to2014,to2016,blm2020} is to follow the lead of the Deift-Zhou steepest descent method by first introducing suitable contour deformations like the transformation between $\mathbf{Y}$ and $\mathbf{Z}$ given in \eqref{e:Y-to-Z-first}--\eqref{e:Y-to-Z-last} and illustrated in Figure~\ref{f:M3}.  Then, the numerical method can be applied to the deformed RHP which requires more complicated input data but yields superior results.  In our paper, we refrained from computing solutions for large values of $X$, so we did not use this modified technique.  Indeed for the plots in Section~\ref{s:numerical-verification} of this paper, we took $t\in [-1,1000]$ and $z\in [0,1]$, implying that $x=\sqrt{2tz}\in [0,20\sqrt{5}]\approx[0,45]$.  For the plots in Figures~\ref{f:unstable-self-similar}--\ref{f:stable-self-similar} we
only needed $x\in [0,8]$.

\subsection{A numerical method for causal solutions of the Cauchy problem}
\label{s:simulations}
The key to developing a numerical method that produces the unique causal solution of the generally ill-posed Cauchy problem \eqref{e:mbe} is to interpret the Maxwell equation and the Bloch subsystem as \emph{ordinary differential equations} (ODEs) in $z$ and $t$ respectively, and to select numerical methods for the initial-value problems of these two ODEs that advance the solution in the positive $z$ and $t$ directions.  Although the causal solution for an incident pulse $q_0(t)$ that vanishes for $t\le 0$ is theoretically trivial for $t\le 0$ and all $z\ge 0$, we test this numerically by solving the system on a time interval $t\in [-1,T]$ and a spatial interval $z\in [0,Z]$ for some given $T>0$ and $Z>0$.

Our numerical method is based on finite differences, so we introduce a grid on the independent variables $(t,z)\in [-1,T]\times[0,Z]$ by setting $t_m=-1+m\Delta t$ and $z_n=n\Delta z$ for integers $(m,n)\in\mathbb{Z}_{\ge 0}^2$, where $\Delta t$ and $\Delta z$ are sufficiently small lattice spacings.  We then introduce approximations for the causal solution at the grid points:
\begin{equation}
q_m^n \approx q(t_m,z_n)\,,\qquad
D_m^n \approx D(t_m,z_n)\,,\qquad
P_m^n \approx P(t_m,z_n)\,.
\end{equation}
The asymptotic formul\ae\ obtained in this paper prove that solutions have a multiscale structure with slowly-varying envelopes and rapid oscillations.  To deal with the stiffness expected on these grounds and to maintain accuracy for large $t$ to be able to compare with the asymptotics, we need implicit methods of high order for the two ODEs.  We use the backward differentiation formulas of various orders $j=1,\dots,4$ denoted BDF$j$ \cite{ch1952,Iserles1996,sm2003}.  Generally, BDF$j$ is a rule equating the right-hand side of a first-order equation evaluated at the current grid point to finite difference approximation of the derivative that is of order $h^j$ where $h=\Delta t$ or $h=\Delta z$, and it involves the unknown at the current grid point as well as the $j$ previous grid points relative to the direction of integration.  In particular, BDF1 is the well-known backward/implicit Euler scheme.  For increased accuracy for $(m,n)$ sufficiently large, we select BDF3 to integrate the Bloch subsystem and BDF4 to integrate the Maxwell equation, in the direction of increasing $t$ and $z$ respectively.  (These choices are somewhat arbitrary, but their accuracy is sufficient for our purposes.)  The basic scheme for such $(m,n)$ reads as follows:
\begin{align}
q_m^n&=\frac{1}{25}\left(48q_m^{n-1}-36q_m^{n-2}+16q_m^{n-3}-3q_m^{n-4}-12\Delta z\,P_m^n\right),
\label{e:q-numerical}\\
P_m^n&=\frac{1}{11}\left(18P_{m-1}^n-9P_{m-2}^n+2P_{m-3}^n-12\Delta t\, q_m^nD_m^n\right),
\label{e:P-numerical}\\
D_m^n&=\frac{1}{11}\left(18D_{m-1}^n-9D_{m-2}^n+2D_{m-3}^n+12\Delta t\,\Re\left(\b{q_m^n}P_m^n\right)\right).
\label{e:D-numerical-1}
\end{align}

To make the implicit nature of this scheme less daunting, we substitute the right-hand side of \eqref{e:P-numerical} for $P_m^n$ in the right-hand side of \eqref{e:D-numerical-1} which then becomes a linear equation for $D_m^n$ that is easily solved.  In this way, we replace \eqref{e:D-numerical-1} with the explicit update rule
\begin{equation}
D_m^n  = \frac{11(18D_{m-1}^n - 9D_{m-2}^n + 2D_{m-3}^n) + 12\Delta t\,\Re(18\b{P_{m-1}^n}q_m^n - 9\b{P_{m-2}^n}q_m^n + 2\b{P_{m-3}^n}q_m^n)}
 {121 + 144\Delta t^2|q_m^n|^2}\,.
\label{e:D-numerical-2}
\end{equation}
It is clear that for the scheme consisting of \eqref{e:q-numerical}, \eqref{e:P-numerical}, and \eqref{e:D-numerical-2}, $(q_m^n,P_m^n,D_m^n)$ only depends on approximations at grid points $(t_\mu,z_\nu)$ with $\mu\le m$ and $\nu\le n$.  This captures precisely the domain of dependence for the causal solution of the Cauchy problem \eqref{e:mbe}, and it is why the scheme produces a numerical approximation of the unique causal solution.

Even with \eqref{e:D-numerical-1} replaced by \eqref{e:D-numerical-2}, the scheme is still implicit, so we use an iteration to solve it.  For each $(m,n)$ sufficiently large for the values in the domain of dependence to be known, we use the following algorithm.
\begin{enumerate}
\item[1.] We use a low-order explicit method for the Maxwell equation $q_z=-P$ such as the forward Euler method to obtain an initial approximation for $q_m^n$ denoted $q_{m,0}^n$:
\begin{equation}
q_{m,0}^n := q_m^{n-1} - \Delta z\,P_m^{n-1}.
\end{equation}
\item[2.] Then, set $k=0$ and repeat:
\begin{enumerate}
\item[2a.] set $k\coloneq k+1$;
\item[2b.] from \eqref{e:D-numerical-2}, set
\begin{equation}
D_{m,k}^n \coloneq \frac{11(18D_{m-1}^n - 9D_{m-2}^n + 2D_{m-3}^n) + 12\Delta t\,\Re(18\b{P_{m-1}^n}q_{m,k-1}^n - 9\b{P_{m-2}^n}q_{m,k-1}^n + 2\b{P_{m-3}^n}q_{m,k-1}^n)}
 {121 + 144\Delta t^2|q_{m,k-1}^n|^2};
\end{equation}
\item[2c.] from \eqref{e:P-numerical}, set
\begin{equation}
P_{m,k}^n\coloneq\frac{1}{11}\left(18P^n_{m-1}-9P_{m-1}^n+2P_{m-3}^n-12\Delta t\,q_{m,k-1}^nD_{m,k}^n\right);
\end{equation}
\item[2d.] from \eqref{e:q-numerical}, set
\begin{equation}
q_{m,k}^n\coloneq\frac{1}{25}\left(48q_m^{n-1}-36q_m^{n-2}+16q_m^{n-3}-3q_m^{n-4}-12\Delta z\,P_{m,k}^n\right);
\end{equation}
\end{enumerate}
until $|q_{m,k}^n-q_{m,k-1}^n|<10^{-15}$.  Set $(q_m^n,P_m^n,D_m^n):=(q_{m,k}^n,P_{m,k}^n,D_{m,k}^n)$.
\end{enumerate}
This scheme assumes that $m\ge 3$ and $n\ge 4$.  For $m=1,2$ (respectively, for $n=1,2,3$), there is less data in the domain of dependence, so we use a scheme based on BDF$m$ for the Bloch subsystem (respectively, based on BDF$n$ for the Maxwell equation).  The use of lower-order BDFs near the initial boundaries of the domain $[-1,T]\times[0,Z]$ does not contaminate the overall accuracy of the scheme because
\begin{itemize}
\item for integration of the Bloch subsystem near $t=-1$ the exact solution is trivial until $t=0$, and the numerical method here is exact to machine precision errors obtained in the iteration, regardless of the order of the method;
\item for integration of the Maxwell equation near $z=0$ any errors introduced by the use of low-order methods are not given much chance to grow because we always keep $Z$ small compared to $T$.
\end{itemize}

The accuracy of the algorithm can be illustrated by taking an initial pulse $q_0(t)$ corresponding to an exact soliton solution of the MBE system:
\begin{equation}
\label{e:soliton}
\begin{aligned}
q(t,z) & = \ii\,\sech(t - 2z - 40)\,,\\
P(t,z) & = -2\ii\tanh(t - 2z - 40)\,\sech(t - 2z - 40)\,,\\
D(t,z) & = -1 + 2\,\sech^2(t- 2z - 40)\,.\\
\end{aligned}
\end{equation}
For this experiment, we choose $t \in [-1,100]$, $z\in[0,1]$, $\Delta t = 0.01$, and $\Delta z = 0.002$.
(In fact, we used the same values of $\Delta t$ and $\Delta z$ for all numerical solutions of the MBE system in this paper.)  Although technically this is not a causal solution, causality holds to machine precision in the domain $(t,z)\in [-1,0]\times[0,1]$ so the numerical method would be expected to reproduce the exact solution with high accuracy.
The results are shown in Figure~\ref{f:soliton}.
\begin{figure}[tp]
    \centering
    \begin{minipage}[b]{.52\textwidth}
    \includegraphics[width = 1\textwidth]{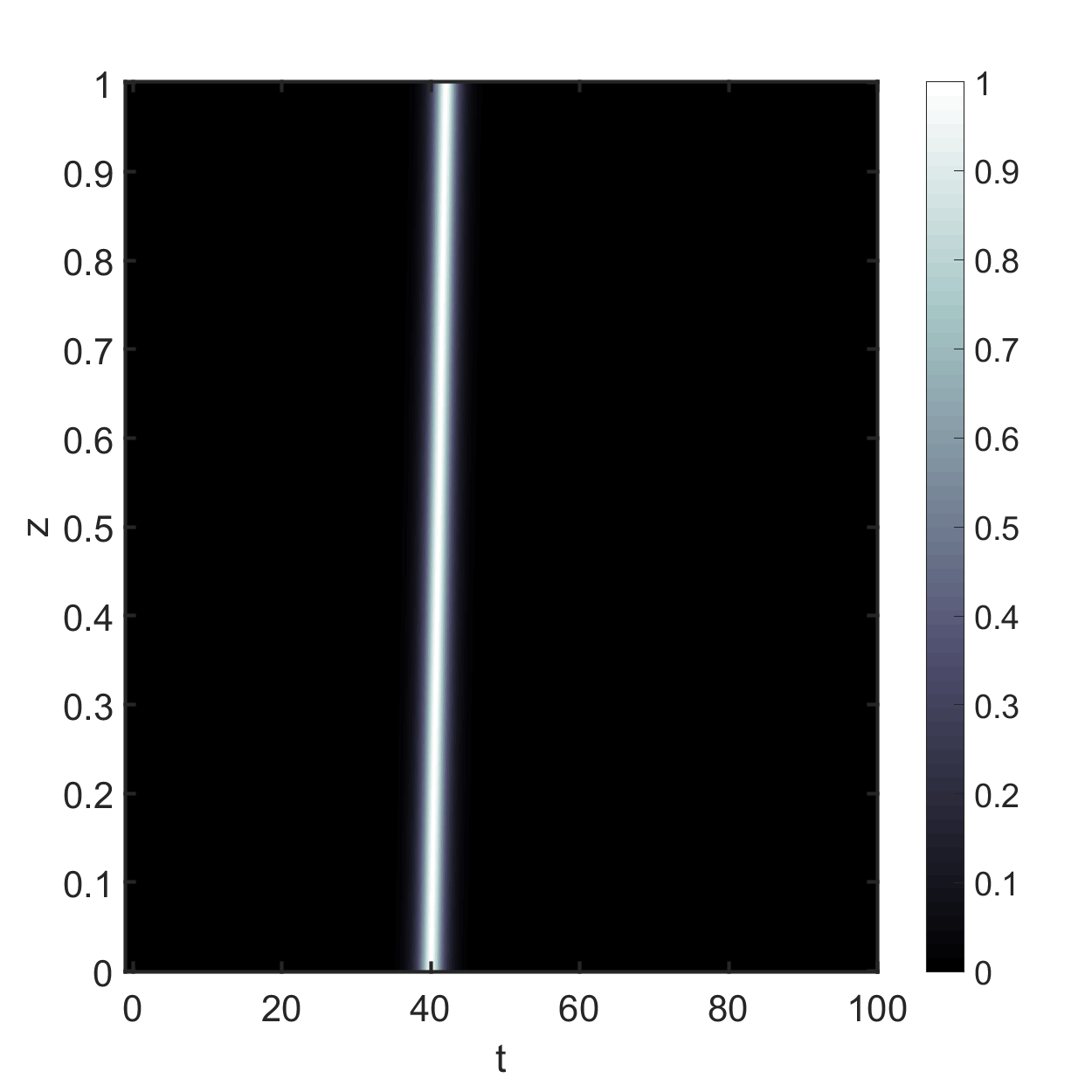}
    \end{minipage}\hspace{-1.5em}
    \begin{minipage}[b]{.49\textwidth}
    \includegraphics[width = 1\textwidth]{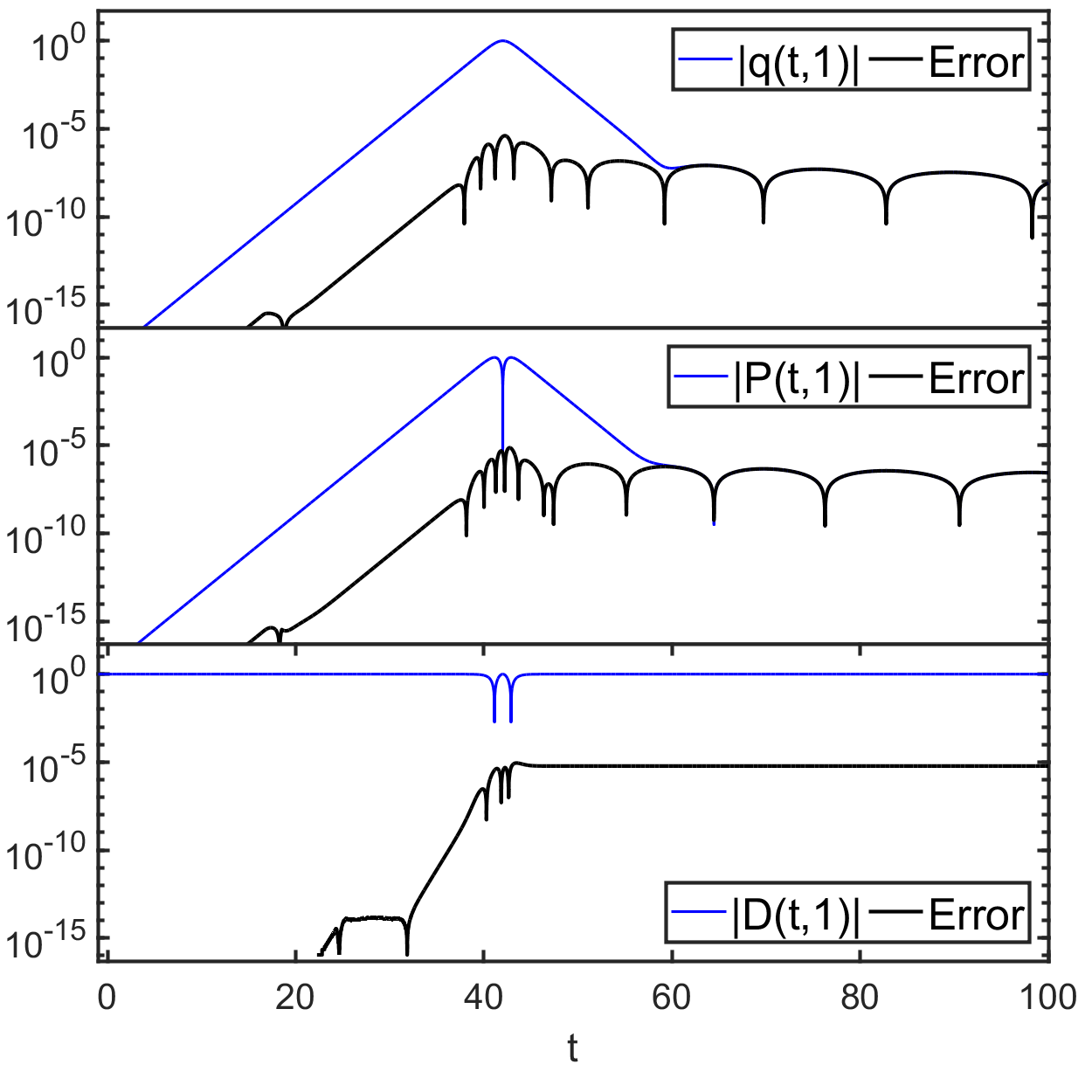}
    \end{minipage}
    \caption{
    Test of the numerical method for the exact soliton solution~\eqref{e:soliton}.
    Left: Density plot of $|q(t,z)|$.
    Right: For all three functions, the absolute errors (black curves) comparing the numerical approximations (blue curves) with the exact solution formul\ae\ \eqref{e:soliton} at $z = 1$.
    }
    \label{f:soliton}
\end{figure}
%



\end{document}